\documentclass[11pt,reqno]{amsart}

\usepackage[utf8]{inputenc}
\usepackage[T1]{fontenc}
\usepackage{lmodern}
\usepackage{amsmath,amssymb,amsthm,mathrsfs,mathtools}
\usepackage[margin=1in]{geometry}
\usepackage{xcolor}
\usepackage{cite}
\usepackage{hyperref}
\hypersetup{
  colorlinks=true,
  linkcolor={red!50!black},
  citecolor={blue!50!black},
  urlcolor={blue!80!black},
  pdftitle={On the Hardy--Littlewood--Chowla Conjecture},
  pdfauthor={Guangshi Lü, Zhiwei Wang, Jie Wu}
}



\allowdisplaybreaks[2]
\numberwithin{equation}{section}

\theoremstyle{plain}
\newtheorem*{Theorem A}{Theorem A}
\newtheorem*{Theorem B}{Theorem B}

\newtheorem{lemma}{Lemma}[section]
\newtheorem{theorem}{Theorem}

\newtheorem{proposition}[lemma]{Proposition}

\theoremstyle{definition}

\newtheorem{remark}{Remark}
\newtheorem{conjecture}[lemma]{Conjecture}

\def\leq{\leqslant}
\def\geq{\geqslant}
\def\le{\leqslant}
\def\ge{\geqslant}

\newcommand{\Z}{\mathbb Z}
\newcommand{\R}{\mathbb R}

\newcommand{\SL}{\operatorname{SL}}
\newcommand{\ind}{\mathbf 1}

\begin{document}
\title[On the Hardy--Littlewood--Chowla conjecture]
{On the Hardy--Littlewood--Chowla Conjecture}

\author[G. L\"u]{Guangshi L\"u}
\address{(Guangshi L\"u) School of Mathematics, Shandong University,
Jinan 250100, Shandong, China}
\email{gslv@sdu.edu.cn}

\author[Z. Wang]{Zhiwei Wang}
\address{(Zhiwei Wang) School of Mathematics, Shandong University,
Jinan 250100, Shandong, China
}
\email{zhiwei.wang@sdu.edu.cn}

\author[J. Wu]{Jie Wu}
\address{(Jie Wu) Mathematical Research Center, Shandong University, Jinan 250100, Shandong, China}
 \email{jie.wu@sdu.edu.cn}

\date{\today}

\begin{abstract}
An old conjecture dating back at least to Hildebrand's work in 1989, asserts that the M\"obius function exhibits cancellation on shifted primes; that is, 
$$\sum_{p\le x}\mu(p+2)=o(\pi(x)),\quad \text{or\ equivalently}\quad \sum_{n\le x}\Lambda(n)\mu(n+2)=o(x).$$
This is also a well-known model case for the parity problem in sieve theory.
In this paper, we prove this conjecture. As an application, we  show that each of the two sets $\{p:\, \mu(p+2)=1\}$ and $\{p:\, \mu(p+2)=-1\}$ has asymptotic density $1/2$ among the primes $p$ with $\mu^2(p+2)=1$:
\begin{equation*}
	\lim_{x\to\infty}
	\frac{\#\{p\le x:\ \mu(p+2)=\kappa\}}
	{\#\{p\le x:\ \mu^2(p+2)=1\}}
	=\frac12,
	\quad \kappa\in\{-1,\,1\}.
\end{equation*}
\end{abstract}

\maketitle

\tableofcontents

\section{\large Introduction}

For any positive integer $n$, let $\Omega(n)$ denote the total number of prime factors of $n$, counted with multiplicity.
Define the M\"obius function $\mu(n)$ and the Liouville function $\lambda(n)$ respectively by
\begin{equation*}
\mu(n) := \begin{cases}
	(-1)^{\Omega(n)} & \text{if $n$ is square free,}
	\\
	0 & \text{otherwise,}
\end{cases}
\qquad\text{and}\qquad
\lambda(n) 
:= (-1)^{\Omega(n)}.
\end{equation*}
We also define the von Mangoldt function $\Lambda(n)$  by
\begin{equation*}
\Lambda(n) := \begin{cases}
	\log p & \text{if $n=p^{\nu}$},\, \nu\geq 1,
	\\
	0 & \text{otherwise}.
\end{cases}
\end{equation*}
In analytic number theory, the M\"obius function and the Liouville function are closely related to the distribution of prime numbers. 
It is well known that for $x\to\infty$, one has
\begin{equation*}
\sum_{n\le x} \mu(n) = o(x) \;\Longleftrightarrow\;
\sum_{n\le x} \lambda(n) = o(x) \;\Longleftrightarrow\;
\sum_{n\le x} \Lambda(n) = \big(1+o(1)\big)x.
\end{equation*}
Note that the last asymptotic formula above is equivalent to the prime number theorem $\pi(x)\sim x/\log x$,
where $\pi(x)$ denotes the number of primes not exceeding $x$.
In general,  the discussion concerning the M\"obius function $\mu(n)$ has a counterpart for the Liouville function $\lambda(n)$. Hence, for simplicity, here we restrict attention to the M\"obius function.

For understanding the independence of additive and multiplicative structure, 
one often studies hybrid addition-multiplication problems.
A typical example is the twin prime conjecture.
More generally, Hardy and Littlewood proposed the following conjecture:
\textit{let $\ell\ge 0$ be an integer and let $a_1, \dots, a_{\ell}$ be distinct fixed natural numbers.
	Then for $x\to\infty$, one has
	\begin{equation*}
		\sum_{n\le x} \Lambda(n+a_1) \cdots \Lambda(n+a_{\ell}) 
		= \big(\mathfrak{S}_{\ell} + o(1)\big) x,
	\end{equation*}
	where}
\begin{equation}\label{eq:HL-singular-series}
	\mathfrak{S}_{\ell} := \prod_p \bigg(1-\frac{1}{p}\bigg)^{-\ell}
	\bigg(1-\frac{\big|\{a_1,\, \dots,\, a_{\ell}\}\,({\rm mod}\,p)\big|}{p}\bigg).
\end{equation}
Here we adopt the convention that the empty product is equal to 1.
As a M\"obius-variant of the Hardy--Littlewood conjecture,
Chowla \cite{Chowla} proposed the following well-known conjecture: 
\textit{let $k\ge 1$ be a positive integer and let $h_1, \dots, h_k$ be distinct fixed natural numbers.
	Then we have for $x\to\infty$,}
\begin{equation*}
	\sum_{n\le x} \mu(n+h_1) \cdots \mu(n+h_k) = o(x).
\end{equation*}
The following hybrid conjecture, which is called the Hardy--Littlewood--Chowla conjecture, has been studied 
by many mathematicians.
\begin{conjecture}[Hardy--Littlewood--Chowla]\label{conj-HLC}
	For any integer $k,\, \ell \geq 0$, let $h_1, \ldots, h_k, a_1, \ldots, a_{\ell}$ be distinct integers. Then we have 
	\begin{equation}\label{HLC-conj}
		\sum_{n \leq x} \mu(n+h_1) \cdots \mu(n+h_k) 
		\Lambda(n+a_1) \cdots \Lambda(n+a_{\ell})
		=\big(\mathfrak{S}_{k,\, \ell}+o(1)\big) x.
	\end{equation}
where $\mathfrak{S}_{k, \ell} = \mathfrak{S}_{\ell}$ as in \eqref{eq:HL-singular-series} if $k=0$ and $\mathfrak{S}_{k, \ell} =0$ if $k>0$. In particular, when $k=\ell=1$ and $a_1=0$, then \eqref{HLC-conj} reduces to the M\"obius function on shifted primes:
	$\sum_{n \leq x} \mu\bigl(n+h_1\bigr)  \Lambda(n) =o(x),$
or equivalently
\begin{align}\label{HLC-1-0}
	\sum_{p \leq x} \mu\bigl(p+h_1\bigr) =o\big(\pi(x)\big).
\end{align}
\end{conjecture}
Clearly, the ``pure'' cases $k=0$ and $\ell=0$ specialise to the original conjectures of Hardy--Littlewood and of Chowla, respectively. Little is known for Conjecture \ref{conj-HLC}. In 2022, Tao and Ter\"{a}v\"{a}inen \cite{TaoTeravainen2021} established the asymptotic \eqref{HLC-conj} for any $k\geq 0$ and $0\leq \ell\leq 2$, assuming that Siegel zeros exist.
Subsequently, Lichtman and Ter\"{a}v\"{a}inen \cite{LichtmanTeravainen2022} 
considered the averaged form of \eqref{HLC-conj}, and they showed that for any $k, \ell \geqslant 1$, 
\begin{align}\label{LT-average}
	\sum_{h_1 \leqslant H}\Big|\sum_{n \leqslant x} \mu\bigl(n+h_1\bigr) \cdots \mu\bigl(n+h_k\bigr) \Lambda\bigl(n+a_1\bigr) \cdots \Lambda\bigl(n+a_{\ell}\bigr)\Big|
	\ll H x \frac{\log \log H}{\log H}
\end{align}
holds for $(\log x)^{\ell+\varepsilon} \leqslant H \leqslant \exp \bigl((\log x)^a\bigr)$ with $a$ small enough.

In this paper, we are concerned with the case \eqref{HLC-1-0} of Hardy--Littlewood--Chowla conjecture.
In fact, \eqref{HLC-1-0} is a folklore conjecture and a well-known model case for the parity problem in sieve theory.
It has appeared in print at least since Hildebrand \cite{Hil} in 1989 (see also Pintz \cite{Pin} and Murty–Vatwani \cite{MurVat}). 
The following nontrivial upper bound for Liouville function, which can be compared to \eqref{HLC-1-0} 
$$
\sum_{p \leq x} \lambda\left(p+h_1\right) <(1-\varepsilon)\pi(x)
$$
is a problem of Sarnak \cite[Problem 5.2]{Sar}.

In 2022, Lichtman \cite{Lichtman} studied the sum in \eqref{HLC-1-0} on average, showing that for $H<x$ and $\log H / \log\log x \rightarrow \infty$ as $x \rightarrow \infty$, one has
\begin{equation*}
	\begin{aligned}
		\sum_{h \leqslant H}\Big|\sum_{p \leqslant x} \mu(p+h)\Big|=o\big(H \pi(x)\big) .
	\end{aligned}
\end{equation*}
As a corollary of \eqref{LT-average}, Lichtman and Ter\"{a}v\"{a}inen \cite[Corollary 1.4]{LichtmanTeravainen2022} also showed that for $(\log x)^{1+\varepsilon} \leqslant H \leqslant x$, \eqref{HLC-1-0}
holds for all but $o(H)$ values of $h \leqslant H$.

Our result is as follows.
\begin{theorem}\label{thm:main}
	As $x\to\infty$, we have
\begin{equation*}
	\sum_{p\le x}\mu(p+2)=o(\pi(x)).
\end{equation*}
\end{theorem}

\begin{remark}
	Our approach and strategy in proving Theorem \ref{thm:main} are very likely applicable to general bounded multiplicative function $f$, with $f$ satisfying the Bombieri--Vinogradov type, such as the Liouville function, etc. We do not pursue the generalization in this paper. 
	%(We shall exploit this generalization in the near future?).
\end{remark}
    
\begin{remark}
	One should be able to obtain a more refined bound $O\big(\pi(x)/(\log \log x)^A\big)$ for any $A>0$ instead of $o(\pi(x))$ 
	by removing the auxiliary weight $\mathcal W$ with more delicate arguments. Since the present manuscript is already involved, here we do not pursue this refinement as well.
\end{remark}

\begin{remark}
	For simplicity,	we here restrict the shift $h_1$ in \eqref{HLC-1-0} to the case $h_1=2$, which %corresponds to 
	is closely related to the twin prime conjecture.  One may generalize the shift to all $h_1\asymp x$ under the Quasi-RH for Dirichlet $L$-functions. We leave the details to the interested reader.
\end{remark}

As an application of Theorem \ref{thm:main}, we study the distribution of $\mu(p+2)$ taking values $1$ or $-1$. Indeed,  Sarnak's Problem 5.2 is also concerned with this question for Liouville function:

{\it{
``It would even be interesting to show that both sets $\{p:\, \lambda(p-1)=1\}$ and $\{p:\, \lambda(p-1)=-1\}$ are infinite.''
}}

By establishing the asymptotic density formula $\sum_{n\le x}\Lambda(n)\mu^2(n+2)\sim A_2x$, together with Theorem \ref{thm:main}, we could  show that each of the two patterns $\mu(p+2)=\pm1$ has ``density'' $1/2$ with respect to those primes $p,\, \mu(p+2)\neq 0$.

\begin{theorem}\label{thm2}
	For each $\kappa\in\{-1,1\}$, as $x\to\infty$, one has
	\begin{equation*}
		\#\Big\{p\le x:\ \mu(p+2)=\kappa\Big\}
		=\bigg(\,\frac{1}{2}\prod_{p>2}\Big(1-\frac{1}{p(p-1)}\Big)+o(1)\bigg)\pi(x).
	\end{equation*}
%	where $A_2$ is a positive constant defined by
%	\begin{align*}
%		A_2:=\prod_{p>2}\bigg(1-\frac{1}{p(p-1)}\bigg).
%	\end{align*}
	In particular, among the primes $p$ for which $\mu^2(p+2)=1$, each of  the values $1$ and $-1$ occur with asymptotic density $1/2$:
	\begin{equation*}
		\lim_{x\to\infty}
		\frac{\big|\big\{p\le x:\ \mu(p+2)=\kappa\big\}\big|}
		{\big|\big\{p\le x:\ \mu^2(p+2)=1\big\}\big|}
		=\frac12,
		\quad \kappa\in\{-1,\,1\}.
	\end{equation*}
\end{theorem}

\begin{remark}
	On October 6, 2026, OpenAI released 722 mathematical manuscripts, in one of which OpenAI \cite{OpenAImobius} stated that there are infinitely many primes $p$ such that $\mu(p-1)=1$\footnote{In fact, $\mu(p-1)=1$ is essentially the same as $\mu(p+2)=1$.}.
	This result may be compared with Theorem \ref{thm2}, which gives an asymptotic formula for the number of primes $p\le x$ with $\mu(p+2)=1$ or $\mu(p+2)=-1$.
\end{remark}

\vskip 5mm

\section{\large Notation and preliminaries}\label{sec:prelim}
\vskip 1.2mm

\subsection{Notation}

We use $\mathbf{1}_S$ to denote the indicator of a set $S$, that is,
$\mathbf{1}_S(n)=1$ when $n\in S$ and $\mathbf{1}_S(n)=0$ otherwise.
If $S$ is a statement, $\mathbf{1}_S$ equals $1$ when $S$ is true and $0$ otherwise. Moreover, if $S$ is a finite set, we use $|S|$ to denote its cardinality. 
For an arithmetic function $F$ on $(x,\,2x]$, define its infinity norm
$\|F\|_\infty=\max_{\substack{x<n\le2x}}|F(n)|$.
For each integer $n\geq 1$, let $P^+(n)$ denote the largest prime factor of $n$, with the convention $P^+(1)=1$; we also denote by $P^-(n)$ the smallest prime factor of $n$ with the convention $P^-(1)=\infty$.
We write $\tau_k(n)$ for the number of ordered factorizations
$n=n_1\cdots n_k$ into positive integers, and put $\tau=\tau_2$.
By convention, we write $\mathbb T=\mathbb R/\mathbb Z$, $e(t)=\exp(2\pi i t)$,
and denote by $\|t\|$ the distance from $t\in\mathbb R$ to the nearest integer. For $t\in\mathbb R$, let $\lfloor t\rfloor$ denote the greatest integer not exceeding $t$, and let $\lceil t\rceil$ denote the least integer not less than $t$.
%The notation $A\ll_{\alpha,\,\beta}B$ means that the implied constant may depend on $\alpha,\,\beta$, but not on the variables of summation.
The letter $p$, with or without a subscript, denotes a prime number.
$p^\nu\parallel n$ means that $p^\nu\mid n$ and $p^{\nu+1}\nmid n$.
We use $C_{K,\,\rho}$ and $c_{K,\,\rho}$ for generic positive
constants depending only on $K$ and $\rho$, whose values may change
between estimates. 
\vskip 1.2mm

\subsection{Sifted numbers and character--Mellin estimates}
\label{sec:sifted}\label{sec:mellin}
Note that by dyadic decomposition, and it suffices to consider the sums appearing in Theorems \ref{thm:main} and \ref{thm2} over the interval $(x,\, 2x].$

We first give a Vaughan decomposition for the von Mangoldt function in Proposition \ref{prop:vau}
and a character--Mellin estimate in Proposition \ref{prop:mel}.
The latter is used in the major-arc argument of Section \ref{sec:major}.

Put
\begin{equation}\label{eq:001}
\begin{aligned}
L=\log x,\,\quad Z=\exp(L^{0.24}),\,\quad U=V=x^{1/5},\,\qquad
R_Z(n)=\mathbf{1}_{P^-(n)>Z}.
\end{aligned}
\end{equation}
We say an integer $n$ is sifted, or $Z$-sifted, when $R_Z(n)=1$; thus $1$ is also sifted. Define
\begin{equation*}
\begin{aligned}
c_V(r):=\sum_{\substack{\ell\mid r,\,\ell>V}}\Lambda(\ell).
\end{aligned}
\end{equation*}
The identity $\sum_{\ell\mid r}\Lambda(\ell)=\log r$ immediately gives
\begin{equation*}
\begin{aligned}
0\le c_V(r)\le\log r,\,\qquad c_V(r)=0\quad(r\le V).
\end{aligned}
\end{equation*}

\begin{proposition}\label{prop:vau}
If $F$ is bounded on $(x,2x]$, then
\begin{equation}\label{eq:vaughan}
\begin{aligned}
\sum_{x<n\le2x}\Lambda(n)F(n)
=T_1(F)-T_2(F)+T_3(F)+O\big(\|F\|_\infty Z\big),
\end{aligned}
\end{equation}
where
\begin{equation*}%
\begin{aligned}
T_1(F)&=\sum_{d\le U}\mu(d)R_Z(d)
 \sum_{x<dk\le2x}R_Z(k)(\log k)F(dk),\\
T_2(F)&=\sum_{\substack{d\le U\\\ell\le V}}
 \mu(d)\Lambda(\ell)R_Z(d\ell)
 \sum_{x<d\ell k\le2x}R_Z(k)F(d\ell k),\\
T_3(F)&=\sum_{\substack{d>U,\, r>V\\x<dr\le2x}}
 \mu(d)R_Z(d)c_V(r)R_Z(r)F(dr).
\end{aligned}
\end{equation*}
We refer to $T_1(F)$ and $T_2(F)$ as the Type I sums, and to $T_3(F)$ as the Type II sum.
\end{proposition}

\begin{proof}
	For every positive integer $n$, Vaughan's identity gives
	\begin{equation}\label{eq:vauid}
		\begin{aligned}
			\Lambda(n)
			=\Lambda(n)\mathbf{1}_{n\le V}
			+\sum_{\substack{dk=n\\d\le U}}\mu(d)\log k
			-\sum_{\substack{d\ell k=n\\d\le U,\,\ \ell\le V}}
			\mu(d)\Lambda(\ell)
			+\sum_{\substack{d\ell k=n\\d>U,\,\ \ell>V}}
			\mu(d)\Lambda(\ell).
		\end{aligned}
	\end{equation}
For $x<n\le2x$, the first term of \eqref{eq:vauid}
is zero, since $V<x$. By the definition of $c_V(r)$ above, we multiply \eqref{eq:vauid} by $R_Z(n)F(n)$ and sum over $x<n\le2x$, getting
\begin{equation}\label{eq:vaughan-sifted-sum}
	\sum_{x<n\le2x}\Lambda(n)R_Z(n)F(n)
	=T_1(F)-T_2(F)+T_3(F).
\end{equation}
For a prime power $p^j$, we have $R_Z(p^j)=0$
if and only if $p\le Z$. Hence
\begin{align*}
	\bigg|\sum_{x<n\le2x}\Lambda(n)F(n)
	-\sum_{x<n\le2x}\Lambda(n)R_Z(n)F(n)\bigg|
	=\bigg|\sum_{p\le Z}
	\sum_{\substack{j\ge1\\x<p^j\le2x}}
	(\log p)F(p^j)\bigg|\ll\|F\|_\infty Z.
\end{align*}
Combining this estimate with \eqref{eq:vaughan-sifted-sum}, we conclude \eqref{eq:vaughan}.
\end{proof}

We next establish the character--Mellin estimate for the sifted M\"obius coefficient in Proposition \ref{prop:vau}.

\begin{proposition}\label{prop:mel}
Let $\delta,\,B,\,C>0$ be fixed. Suppose that
$x^\delta\le D\le2x$ and $I\subseteq[D,2D]$ is an interval.
Let $\chi$ be a Dirichlet character modulo $q$ with $1\le q\le (\log x)^B$. Then for any $A>0$
\begin{equation}\label{eq:mellin}
	\sup_{|t|\le x^C}
	\bigg|\sum_{d\in I}\mu(d)R_Z(d)\chi(d)d^{it}\bigg|
	\ll_{A,\,B,\,C,\,\delta}D L^{-A}.
\end{equation}
\end{proposition}
\begin{proof}
Let $\mathcal X\ge3$ and $1\le r\le\mathcal X$. Let $\alpha_0\in(0,1]$ be fixed.
Then for any character $\chi$ (mod $q$) with $q\leq (\log\mathcal X)^{A_0}$ and any $A_0>0$, one has (see \cite[Lemma 3.9(ii)]{MSTT})
\begin{equation}\label{eq:muinput}
\begin{aligned}
\sup_{|t|\le\mathcal X}
 \sup_{I\subset[\mathcal X^{\alpha_0},\,2\mathcal X^{\alpha_0}]}
 \biggl|\sum_{n\in I}\frac{\mu(rn)\chi(n)}{n^{1/2+it}}\biggr|
 \ll_{\alpha_0,\,A_0}
 \mathcal X^{\alpha_0/2}(\log\mathcal X)^{-A_0}.
\end{aligned}
\end{equation}

Let $H$ satisfy $x^{\delta/2}\le H\le2x$.
We first consider the sum over subintervals of $[H,\,2H]$. 
Choose $R>\max\{1,\,2(C+1)/\delta\}$ depending only on
$\delta$ and $C$, and put
\begin{equation*}
	\mathcal X=H^R,\,\qquad \alpha_0=1/R.
\end{equation*}
Then $H=\mathcal X^{\alpha_0}$ and
$\mathcal X\ge x^{R\delta/2}>x^{C+1}$, so the range
$|t|\le x^C$ is covered by that in \eqref{eq:muinput}.
%For each fixed $A_0$, the implied constant in
%\eqref{eq:muinput} is independent of $H$, because
%$\alpha_0=1/R$ is fixed. Also,
%$\log\mathcal X\asymp_{\delta,\,C}L$.
We then replace $t$ by $-t$, and apply
partial summation only to the real weight $n^{1/2}$ in \eqref{eq:muinput}, yielding
\begin{equation}\label{eq:muplain}
\begin{aligned}
\sup_{|t|\le x^C}\sup_{I\subset[H,\,2H]}
 \Bigl|\sum_{n\in I}\mu(n)\chi(n)n^{it}\Bigr|
 \ll_{A,\,B,\,C,\,\delta}H L^{-A}
\end{aligned}
\end{equation}
for sufficiently large $A>0$. The passage from \eqref{eq:muinput} to \eqref{eq:muplain} does not differentiate $n^{it}$, so it
incurs no loss in $|t|$.

The following identity
\begin{equation*}
\begin{aligned}
\mu(n)R_Z(n)=\sum_{\substack{a\mid n\\P^+(a)\le Z}}\mu(n/a),
\end{aligned}
\end{equation*}
handles the indicator $R_Z(n)$. It can be checked prime by prime:
for $p\le Z$, the right-hand side at $p^j$ is $1-1=0$ for $j\ge1$, whereas for $p>Z$ it is
$\mu(p^j)$. The sum on the left of \eqref{eq:mellin} therefore equals
\begin{equation}\label{eq:musum}
\begin{aligned}
\sum_{\substack{a\le2D\\P^+(a)\le Z}}\chi(a)a^{it}
\sum_{k\in I/a}\mu(k)\chi(k)k^{it}.%:=\mathcal{S}_D,
\end{aligned}
\end{equation}

For $a\le D^{1/2}$, put $H=D/a$. Then
$I/a\subseteq[H,\,2H]$ and $x^{\delta/2}\le H\le2x$.
Applying \eqref{eq:muplain} with $A$ replaced by $A+1$,
we obtain
\begin{equation*}
	\sup_{|t|\le x^C}
	\bigg|\sum_{k\in I/a}\mu(k)\chi(k)k^{it}\bigg|
	\ll \frac{D}{a}L^{-A-1}.
\end{equation*}
Moreover, one has
\begin{equation*}
	\sum_{P^+(a)\le Z}\frac1a
	=\prod_{p\le Z}(1-p^{-1})^{-1}
	\ll\log Z,
\end{equation*}
whence, by $\log Z=L^{0.24}$
\begin{align*}
	\sup_{|t|\le x^C}
	\bigg|
	\sum_{\substack{a\le D^{1/2}\\P^+(a)\le Z}}
	\chi(a)a^{it}\sum_{k\in I/a}\mu(k)\chi(k)k^{it}
	\bigg|
	\ll	DL^{-A-1}
	\sum_{\substack{a\le D^{1/2}\\P^+(a)\le Z}}\frac1a
	\ll DL^{-A}.
\end{align*}

For $a>D^{1/2}$, we first note that the inner sum in \eqref{eq:musum} can be bounded by $O(D/a)$ trivially. Then take $\eta=1/\log Z$. Rankin's trick gives
\begin{equation}\label{eq:smtail}
\begin{aligned}
\sum_{\substack{a>D^{1/2}\\P^+(a)\le Z}}\frac1a
\le D^{-\frac{\eta}{2}}\prod_{p\le Z}(1-p^{-1+\eta})^{-1}\ll \exp\biggl(-\frac{\log D}{2\log Z}\biggr)\log Z
\ll L^{0.24}\exp(-c_\delta L^{0.76})
\end{aligned}
\end{equation}
%For the second inequality we used
%\begin{equation*}
%\begin{aligned}
%\log\frac{\prod_{p\le Z}(1-p^{-1+\eta})^{-1}}
%              {\prod_{p\le Z}(1-p^{-1})^{-1}}
%=\sum_{p\le Z}\sum_{j\ge1}
%       \frac{p^{-j(1-\eta)}-p^{-j}}j\ll\sum_{p\le Z}\frac{p^\eta-1}p
% \ll\eta\sum_{p\le Z}\frac{\log p}p\ll1,
%\end{aligned}
%\end{equation*}
%since $\eta=1/\log Z$.
%The inner sum in \eqref{eq:musum} has at most $2D/a$ terms, 
for some $c_\delta>0$, so \eqref{eq:smtail} proves the required saving for
$a>D^{1/2}$.
\end{proof}

For any complex-valued function $w$ on an interval $I$, define the  variation by
\begin{equation}\label{eq:vardef}
\begin{aligned}
\operatorname{Var}_I(w):
 =\sup_{\substack{ t_i\in I\\ t_0<\cdots<t_j}}\,
       \sum_{i=0}^{j-1}\big|w(t_{i+1})-w(t_i)\big|.
\end{aligned}
\end{equation}
%A function is of bounded variation when \eqref{eq:vardef} is finite. 
Partial summation and Proposition \ref{prop:mel} then give
\begin{equation}\label{eq:mubv}
\begin{aligned}
\sup_{|t|\le x^C}
\Bigl|\sum_{d\in I}\mu(d)R_Z(d)\chi(d)d^{it}w(d)\Bigr|
\ll_{A,\,B,\,C,\,\delta}DL^{-A}
 \Bigl(\|w\|_\infty+\operatorname{Var}_I(w)\Bigr).
\end{aligned}
\end{equation}
Hence, we may insert some weight function $w$ into \eqref{eq:mellin} with the same upper bound, provided $\|w\|_\infty+\operatorname{Var}_I(w)\ll 1$.

\subsection{Auxiliary weights and reduction to weighted sums}
\label{sec:marks}
First, we introduce some auxiliary
weights depending on prime divisors in prescribed intervals.
Let $K\ge1$ be an integer, and let $1/10<a_1<\cdots<a_K<1/5$. 
%a number $0<\rho<1$, and exponents $0.1<a_1<\cdots<a_K<0.2$. 
For sufficiently large $x$, we define the pairwise disjoint sets of primes $\mathcal P_i$ by
\begin{equation}\label{eq:prime-sets}
	\begin{aligned}
		\mathcal P_i:=\Big\{p:\ \exp(L^{a_i})<p\le\exp(2L^{a_i})\Big\}.
	\end{aligned}
\end{equation}
Let $0<\rho<1$ be fixed. We also define the product 
\begin{equation}\label{eq:mark}
	\begin{aligned}
	\mathcal W(n) =\prod_{i=1}^K\frac{\rho^{\omega_i(n)-1}\omega_i(n)}{V_i},
	\end{aligned}
\end{equation}
where
\begin{align*}
	\omega_i(n)=\sum_{p\in\mathcal P_i}\mathbf{1}_{p\mid n}
	\quad \text{and}\quad
	V_i=\sum_{p\in\mathcal P_i}\frac1p=\log2+o(1).
\end{align*}
Note that $\mathcal W(n)$ is bounded by
\begin{equation*}
	\begin{aligned}
		0\le\mathcal W(n)
		\le\biggl(\frac{2}{\log2}\sup_{j\ge1}j\rho^{j-1}\biggr)^K
		<\infty.
	\end{aligned}
\end{equation*}

Let $P=\prod_{i=1}^K\prod_{p\in\mathcal P_i}p$.
Note that $	\omega_i(n)\rho^{\omega_i(n)-1}=\sum_{p\in\mathcal P_i}\mathbf1_{p\mid n}\prod_{\substack{q\in\mathcal P_i,\, q\ne p}}\bigl(1-(1-\rho)\mathbf1_{q\mid n}\bigr)$, thus we have by the Chinese remainder theorem
\begin{equation}\label{eq:markexp}
	\begin{aligned}
		m_{\mathcal W}
		:=\frac1{\varphi(P)}
		\sum_{\substack{a\,(\textrm{mod}\, P)\\(a,\,P)=1}}\mathcal W(a+2)=\prod_{i=1}^K\Biggl\{
		\frac1{V_i}\sum_{p\in\mathcal P_i}\frac1{p-1}
		\prod_{\substack{q\in\mathcal P_i\\q\ne p}}
		\biggl(1-\frac{1-\rho}{q-1}\biggr)\Biggr\}.
	\end{aligned}
\end{equation}
By calculation, one may find that $m_{\mathcal W}=2^{K(\rho-1)}
+o(1)$.
Since $0\le\mathcal W(n)\ll_{K,\rho}1$, we have
$\mathcal W(n)^2\ll_{K,\rho}\mathcal W(n)$. Hence there exists an constant $C_{K,\rho}$ such that $0< m_{\mathcal W}\le C_{K,\rho}$ and
\begin{equation}\label{eq:003}
	\begin{aligned}
		\frac1{\varphi(P)}
		\sum_{\substack{a\,(\textrm{mod}\, P)\\(a,\,P)=1}}
		\biggl(\frac{\mathcal W(a+2)}{m_{\mathcal W}}\biggr)^2
		\le C_{K,\rho}.
	\end{aligned}
\end{equation}

\begin{proposition}\label{prop:rem}
	Let $b\ge1$ be an integer. For each $1\le j\le b$, let $\mathcal P_{j,\,1},\,\ldots,\,\mathcal P_{j,\,K}$ be $K$ prime sets of the form \eqref{eq:prime-sets}, with  $1/10<a_{j,\,1}<\cdots<a_{j,\,K}<1/5$; and let $\mathcal W_j$ be a weight of the form	\eqref{eq:mark}, constructed from $\mathcal P_{j,\,1},\,\ldots,\,\mathcal P_{j,\,K}$. 
	Define $m_j$ by \eqref{eq:markexp} with $\mathcal W$ replaced by $\mathcal W_j$ and modulus $P=\prod_{j=1}^{b}\prod_{i=1}^{K}\prod_{p\in\mathcal P_{j,\,i}}p$. Assume that $\mathcal P_{j,\,i}\cap\mathcal P_{k,\,h}=\varnothing$ whenever $(j,\,i)\ne(k,\,h)$.
	Then as $x\to\infty$, %with $b,\,K,\,\rho$ and all exponents $a_{j,\,i}$ fixed, we have
	\begin{equation}\label{eq:004}
		\frac1x\sum_{x<p\le2x}\log p\biggl(\frac1b\sum_{j=1}^b\frac{\mathcal W_j(p+2)}{m_j}-1\biggr)^2\le\frac{C_{K,\,\rho}}b+o(1).
	\end{equation}
\end{proposition}

\begin{proof}
Let $\mathcal F$ be either $\mathcal W_j$ or $\mathcal W_j\mathcal W_k$ with $j\ne k$. There are $K$ prime sets in the first case and $2K$ in the second. Denote their number by $s$ and write them as $\mathcal P_1,\ldots,\mathcal P_s$. Thus
	\begin{equation*}
		\mathcal F(n)=\prod_{i=1}^{s}\frac{\omega_i(n)\rho^{\omega_i(n)-1}}{V_i}.
	\end{equation*}
	
	Differentiating the equality $\rho^u=\sum_{\ell=0}^{u}\binom u\ell(\rho-1)^\ell$ with respect to $\rho$, we get
	\begin{equation}\label{eq:005}
		u\rho^{u-1}=\sum_{\ell=1}^{u}\ell(\rho-1)^{\ell-1}\binom u\ell
	\end{equation}
	for any $u\in \mathbb{N}^*$, and the sum is empty when $u=0$. By \eqref{eq:005}, we immediately have
	\begin{equation*}
	\mathcal F(n)=
	\prod_{i=1}^{s}\Biggl\{\frac1{V_i}\sum_{\ell=1}^{\omega_i(n)}\ell(\rho-1)^{\ell-1}\binom{\omega_i(n)}{\ell}\Biggr\}.
	\end{equation*}
	For an integer $R\ge1$, define
	\begin{equation*}
		\mathcal F_R(n):=\prod_{i=1}^{s}\Biggl\{\frac1{V_i}\sum_{\ell=1}^{R}\ell(\rho-1)^{\ell-1}\binom{\omega_i(n)}{\ell}\Biggr\},
	\end{equation*}
	with $\binom{\omega_i(n)}{\ell}=0$ for $\ell>\omega_i(n)$. Subtracting the above two expressions  gives
	\begin{equation}\label{difference-F-F_R}
		\mathcal F(n)-\mathcal F_R(n)=
		\sum_{\substack{\ell_1,\ldots,\ell_s\ge1\\[0,4mm]\max_{i}\ell_i>R}}\prod_{i=1}^{s}
		\Biggl\{\frac{\ell_i(\rho-1)^{\ell_i-1}}{V_i}\binom{\omega_i(n)}{\ell_i}\Biggr\}.
	\end{equation}
	
	We first estimate this difference averaged over integers. 
	For each $i$, we have
	\begin{equation*}
		\sum_{\substack{\mu^2(d)=1,\,\omega(d)=\ell\\[0,4mm] p\mid d\Rightarrow p\in\mathcal P_i}}\frac1d
		\le \frac{1}{\ell!}\bigg(\sum_{p\in\mathcal P_i}\frac1p\bigg)^{\ell}=\frac{V_i^\ell}{\ell!}.
	\end{equation*}
	Next, we note that choosing $\ell$ distinct primes from $\{p\in\mathcal P_i:p\mid n\}$ is equivalent to choosing a squarefree divisor of $n$ with $\ell$ prime factors, all in $\mathcal P_i$, whence
	\begin{equation}\label{difference-2}
		\binom{\omega_i(n)}{\ell}=\sum_{\substack{d\mid n,\,\mu^2(d)=1\\p\mid d\Rightarrow p\in\mathcal P_i,\,\omega(d)=\ell}}1.
	\end{equation}
	Since the divisors supported on different $\mathcal P_i$ are coprime, we deduce from the above equality that
	\begin{equation*}
		\sum_{n\le3x}\,\prod_{i=1}^{s}\binom{\omega_i(n)}{\ell_i}\le3x\prod_{i=1}^{s}\frac{V_i^{\ell_i}}{\ell_i!}.
	\end{equation*}
	Noting that $V_i=\log2+o(1)<1$ for sufficiently large $x$, we then have
	\begin{equation*}
		\sum_{\ell\ge1}\frac{((1-\rho)V_i)^{\ell-1}}{(\ell-1)!}\le e,\qquad \sum_{\ell>R}\frac{((1-\rho)V_i)^{\ell-1}}{(\ell-1)!}\le\frac e{R!}.
	\end{equation*}
	By these two upper bounds, together with \eqref{difference-F-F_R} and \eqref{difference-2}, one finds
	\begin{equation}\label{n-diffe-estim}
		\begin{aligned}
			\sum_{n\le3x}\big|\mathcal F(n)-\mathcal F_R(n)\big|
			\le3x\sum_{\substack{\ell_1,\ldots,\ell_s\ge1
					\\[0,4mm]\max_{i}\ell_i>R}}
		\prod_{i=1}^{s}\frac{((1-\rho)V_i)^{\ell_i-1}}{(\ell_i-1)!}
			\le\frac{3se^s x}{R!}\ll_{K,\,\rho}\frac{x}{R!}.
		\end{aligned}
	\end{equation}
	
	We next estimate the difference $|\mathcal F(a+2)-\mathcal F_R(a+2)|$ over reduced residue classes $a\,(\textrm{mod}\,P),\, (a,\,P)=1$.
	Since $P$ is odd, then for any $d\mid P$, the Chinese remainder theorem gives
	\begin{equation}\label{CRT-a+2}
		\frac1{\varphi(P)}\sum_{\substack{a\,(\textrm{mod}\, P)\\[0,3mm](a,\,P)=1}}\mathbf1_{d\mid a+2}
		=\frac1{\varphi(d)}.
	\end{equation}
	The above argument for $\sum_{n}|\mathcal F(n)-\mathcal F_R(n)|$ is also applicable to $\sum_{a}|\mathcal F(a+2)-\mathcal F_R(a+2)|$, except that we need to replace $1/d$ by $1/\varphi(d)=\prod_{p\mid d}1/(p-1)$ in the divisor sums. 
	Since $V_i\gg1$ and for sufficiently large $x$
	\begin{equation*}
		\sum_{p\in\mathcal P_i}\frac1{p-1}=\log2+o(1),
	\end{equation*}
	 we obtain similarly as above
	\begin{equation}\label{a+2-diffe-estim}
		\frac1{\varphi(P)}\sum_{\substack{a\,(\textrm{mod}\,P)\\(a,\,P)=1}}|\mathcal F(a+2)-\mathcal F_R(a+2)|\ll_{K,\,\rho}\frac1{R!}.
	\end{equation}
	
	We now take $R=\lceil\log L\rceil$. We expand the binomial coefficients in $\mathcal F_R(n)$ by \eqref{difference-2} into sums over squarefree divisors $d_i\mid n$, with $p\mid d_i\Longrightarrow p\in\mathcal P_i$. The disjointness of $\mathcal P_i$ gives a unique decomposition of each divisor $d=d_1\cdots d_s$ into its factors $d_i$ supported on $\mathcal P_i$. Consequently, we may write $\mathcal F_R(n)=\sum_{d\mid n}c_d$ and
	\begin{equation*}
		d\le\exp(4KRL^{0.2})=x^{o(1)},\qquad |c_d|\ll_{K,\,\rho}R^{2K}.
	\end{equation*}
	These two ranges, together with the condition $2\nmid d$ allow us to apply the Bombieri--Vinogradov theorem, yielding 
	\begin{equation}\label{BV-p+2}
		\frac1x\sum_{x<p\le2x}\log p\,\mathcal F_R(p+2)=\sum_d\frac{c_d}{\varphi(d)}+o(1)=\frac1{\varphi(P)}\sum_{\substack{a\,(\textrm{mod}\, P)\\(a,\,P)=1}}\mathcal F_R(a+2)+o(1),
	\end{equation}
where the second equality is derived by \eqref{CRT-a+2}. Furthermore,  we have by \eqref{n-diffe-estim}
	\begin{equation*}
		\begin{aligned}
			\biggl|\frac1x\sum_{x<p\le2x}\log p\,\Bigl(\mathcal F(p+2)-\mathcal F_R(p+2)\Bigr)\biggr|
			\le\frac{\log(2x)}x\sum_{n\le3x}\big|\mathcal F(n)-\mathcal F_R(n)\big|\ll_{K,\,\rho}\frac L{R!}.
		\end{aligned}
	\end{equation*}
Combining this estimate with \eqref{a+2-diffe-estim} and \eqref{BV-p+2}, we conclude
	\begin{equation*}
		\begin{aligned}
			\frac1x\sum_{x<p\le2x}\log p\,\mathcal F(p+2)&=\frac1x\sum_{x<p\le2x}\log p\,\mathcal F_R(p+2)+O_{K,\,\rho}\biggl(\frac L{R!}\biggr)\\
			&=\frac1{\varphi(P)}\sum_{\substack{a\,(\textrm{mod}\,P)\\(a,\,P)=1}}\mathcal F_R(a+2)+o(1)+O_{K,\,\rho}\biggl(\frac L{R!}\biggr)\\
			&=\frac1{\varphi(P)}\sum_{\substack{a\,(\textrm{mod}\,P)\\(a,\,P)=1}}\mathcal F(a+2)+O_{K,\,\rho}\biggl(\frac1{R!}\biggr)+o(1)+O_{K,\,\rho}\biggl(\frac L{R!}\biggr).
		\end{aligned}
	\end{equation*}
	Since $R=\lceil\log L\rceil$, we have $L/R!\to0$. We therefore obtain
	\begin{equation}\label{eq:rem-comparison}
		\frac1x\sum_{x<p\le2x}\log p\,\mathcal F(p+2)=\frac1{\varphi(P)}\sum_{\substack{a\,(\textrm{mod}\,P)\\(a,\,P)=1}}\mathcal F(a+2)+o(1).
	\end{equation}
	
	Recall the definition of $m_j$.
%	\begin{equation*}
%		m_j=\frac1{\varphi(P)}\sum_{\substack{a\,(\textrm{mod}\, P)\\(a,\,P)=1}}\mathcal W_j(a+2).
%	\end{equation*}
	For $j\ne k$, the sets of primes $\bigcup_{i=1}^{K}\mathcal P_{j,\,i}$ and $\bigcup_{i=1}^{K}\mathcal P_{k,\,i}$ are disjoint, so
	\begin{equation*}
		m_jm_k=\frac1{\varphi(P)}\sum_{\substack{a\,(\textrm{mod}\, P)\\(a,\,P)=1}}\mathcal W_j(a+2)\mathcal W_k(a+2).
	\end{equation*}
Note that $m_j\gg 1$. Taking $\mathcal F=\mathcal W_j$ and $\mathcal F=\mathcal W_j\mathcal W_k$ in \eqref{eq:rem-comparison} respectively, we deduce that for $j\ne k$
	\begin{equation}\label{two-moment}
		\begin{aligned}
			\frac1x\sum_{x<p\le2x}\log p\,\frac{\mathcal W_j(p+2)}{m_j}=1+o(1),\quad
			\frac1x\sum_{x<p\le2x}\log p\,\frac{\mathcal W_j(p+2)\mathcal W_k(p+2)}{m_jm_k}=1+o(1).
		\end{aligned}
	\end{equation}
	Moreover, $0\le\mathcal W_j(n)\ll 1$ and $m_j\gg 1$ imply $(\mathcal W_j(n)/{m_j})^2\ll \mathcal W_j(n)/{m_j}$. Therefore, we conclude by \eqref{two-moment} that 
	\begin{equation*}
		\begin{aligned}
			\frac1x\sum_{x<p\le2x}\log p\biggl(\frac1b\sum_{j=1}^{b}\frac{\mathcal W_j(p+2)}{m_j}-1\biggr)^2
			&=\frac1{b^2x}\sum_{j=1}^{b}\sum_{x<p\le2x}\log p\biggl(\frac{\mathcal W_j(p+2)}{m_j}\biggr)^2-\frac1b+o(1)
			\\[1,2mm]
			&\le\frac{C_{K,\,\rho}}b+o(1).
		\end{aligned}
	\end{equation*}
for some constant $C_{K,\,\rho}>0$.	This completes the proof of \eqref{eq:004}.
\end{proof}

\vskip 5mm

\section{\large The Type I estimate with auxiliary weights}
\label{sec:typeI}

First, we recall a theorem of Bombieri--Vinogradov type for the M\"{o}bius function (see e.g. \cite[Proposition 3]{MurVat18}):
for any $A>0$, there exists $B=B(A)>0$ such that
\begin{equation}\label{BV-mobius}
\begin{aligned}
\sum_{q\le x^{1/2}(\log x)^{-B}}\max_{(a,\,q)=1}\max_{y\le x}
 \biggl|\sum_{\substack{n\le y\\n\equiv a\,(\mathrm{mod}\, q)}}\mu(n)\biggr|
 \ll_A \frac{x}{(\log x)^A}.
\end{aligned}
\end{equation}

\begin{proposition}\label{prop:typeI}
Fix the weight $\mathcal W$ defined by \eqref{eq:mark} and let
$F(n)=\mu(n+2)\mathcal W(n+2)$. The first two terms in
Proposition \ref{prop:vau} satisfy
\begin{equation}\label{eq:typeI}
\begin{aligned}
T_1(F),\,\ T_2(F)\ll_{A,\,K,\,\rho}\frac{x}{(\log x)^A}
\end{aligned}
\end{equation}
for any $A>0$.
\end{proposition}

\begin{proof}
First for $T_2$, let
\begin{equation*}
	\varpi_q=\sum_{\substack{d\ell=q\\[0,3mm]d\le U,\,\ell\le V}}
	\mu(d)\Lambda(\ell),
\end{equation*}
whence $|\varpi_q|\le\log q$. Since $UV=x^{2/5}$ and $R_Z(q)=R_Z(k)=1$ force $q,k$ to be odd, we have
\begin{equation*}
	T_2(F)=\sum_{\substack{q\le x^{2/5}\\R_Z(q)=1}} \varpi_q\sum_{\substack{x<qk\le2x\\2\nmid k}}R_Z(k)\mu(qk+2)\mathcal W(qk+2).
\end{equation*}
We first approximate the two weights $R_Z(k)$ and $\mathcal W(qk+2)$ by finite divisor sums. We note that the argument below uses only $|\Lambda_{U,V}(q)|\le L$, and all estimates remain valid when the summation condition $x<qk\le2x$ is replaced by $x<qk\le y$ uniformly for $x\le y\le2x$.

%Fix $A>0$ and take 
Let $R_0=\lceil C\log L\rceil$, where $C$ is sufficiently large. 
First, for $R_Z(k)$, we have
\begin{align*}
	R_Z(k)=\prod_{2<p\le Z}\bigl(1-\mathbf1_{p\mid k}\bigr)=\sum_{\substack{r\mid k,\,\mu^2(r)=1\\p\mid r\Rightarrow2<p\le Z}}\mu(r)=\sum_{\substack{r\mid k,\,\mu^2(r)=1\\p\mid r\Rightarrow2<p\le Z\\\omega(r)\le R_0}} \mu(r)+\sum_{\substack{r\mid k,\,\mu^2(r)=1\\p\mid r\Rightarrow2<p\le Z\\ \omega(r)\geq R_0+1}}\mu(r).
\end{align*}
Note that the sum with $\omega(r)\geq R_0+1$ above is equal to $
\sum_{j=R_0+1}^{s}(-1)^j\binom{s}{j}=(-1)^{R_0+1}\binom{s-1}{R_0}$, and
\begin{align}\label{equlity-omega=r}
	\binom{s-1}{R_0} \leq \binom{s}{R_0+1}=\sum_{\substack{r\mid k,\,\mu^2(r)=1\\p\mid r\Rightarrow2<p\le Z\\\omega(r)=R_0+1}}1.
\end{align}
Therefore we obtain
\begin{equation*}
	R_Z(k)=\sum_{\substack{r\mid k,\,\mu^2(r)=1\\p\mid r \Rightarrow2<p\le Z\\\omega(r)\le R_0}} \mu(r)+O\biggl(\sum_{\substack{r\mid k,\,\mu^2(r)=1\\p\mid r \Rightarrow2<p\le Z\\\omega(r)=R_0+1}}1\biggr),
\end{equation*}
and every divisor $r$ above satisfies $r\le Z^{R_0+1}=x^{o(1)}$.
We substitute this expansion into $T_2(F)$ and interchange the sums over $k$ and $r$, getting
\begin{equation}\label{replace1}
	\begin{aligned}
	T_2(F)&=\sum_{\substack{q\le x^{2/5}\\R_Z(q)=1}}\varpi_q\sum_{\substack{\mu^2(r)=1\\p\mid r \Rightarrow2<p\le Z\\\omega(r)\le R_0}}\mu(r)
	\sum_{\substack{x<qk\le2x\\2\nmid k,\,r\mid k}}
	\mu(qk+2)\mathcal W(qk+2)
	+O\biggl(\sum_{\substack{q\le x^{2/5}\\R_Z(q)=1}}|\varpi_q|
	\sum_{\substack{\mu^2(r)=1\\p\mid r\Rightarrow2<p\le Z\\ \omega(r)=R_0+1}}\frac{x}{qr}\biggr)
	\\[2mm]
	&=\sum_{\substack{q\le x^{2/5}\\R_Z(q)=1}}\varpi_q \sum_{\substack{\mu^2(r)=1\\p\mid r\Rightarrow2<p\le Z\\ \omega(r)\le R_0}}\mu(r)
	\sum_{\substack{x<qk\le2x\\2\nmid k,\,r\mid k}}
	\mu(qk+2)\mathcal W(qk+2)
	+O\biggl(\frac{xL^2}{(R_0+1)!}\biggl(\sum_{2<p\le Z} \frac1p\biggr)^{R_0+1}\biggr).
	\end{aligned}
\end{equation}

Now we treat $\mathcal W(qk+2)$. First, by \eqref{eq:005} we have
\begin{equation*}
	\mathcal W(qk+2)=\prod_{i=1}^{K}\biggl\{\frac{1}{V_i}\sum_{j=1}^{\omega_i(qk+2)}j(\rho-1)^{j-1}\binom{\omega_i(qk+2)}j\biggr\}.
\end{equation*}
We retain only the terms with $j\le R_0$ in each factor. Writing 
\begin{equation*}
	\prod_{i=1}^{K}\Biggl\{\frac1{V_i}\sum_{j=1}^{R_0}
	j(\rho-1)^{j-1}\binom{\omega_i(qk+2)}j\Biggr\}=\sum_{d\mid qk+2}c'_d,
\end{equation*}
we next estimate the difference 
%$\sum_{qk}|\mathcal W(qk+2)-\sum_{d\mid qk+2}c_d|$
\begin{equation}\label{differe-W}
	\sum_{\substack{x<qk\le y\\2\nmid k,\,r\mid k}}\biggl|\mathcal W(qk+2)-\sum_{d\mid qk+2}c'_d\biggr|.
\end{equation}
By the Taylor expansion, we have  
\begin{equation*}
	\biggl|u\rho^{u-1}-\sum_{j=1}^{R_0}j(\rho-1)^{j-1}\binom uj\biggr|\le(R_0+1)\binom u{R_0+1}
\end{equation*}
for $u\ge0$ and $0<\rho<1$. By applying this estimate with $u=\omega_i(qk+2)$ for each $1\le i\le K$, we calculate \eqref{differe-W}:
\begin{equation*}
	\begin{aligned}
		\biggl|\mathcal W(qk+2)-\sum_{d\mid qk+2}c'_d\biggr|
		\le\frac{R_0+1}{\prod_{h=1}^{K}V_h}\sum_{i=1}^{K}
		\frac{\binom{\omega_i(qk+2)}{R_0+1}}{(1-\rho)^{2(K-i)}}
		\prod_{h=1}^{i-1}\Biggl\{\sum_{j=1}^{R_0}j(1-\rho)^{j-1} \binom{\omega_h(qk+2)}j\Biggr\}.
	\end{aligned}
\end{equation*}
We then write $\binom{\omega_h(qk+2)}j$ of the form $\sum_{d_h\mid qk+2}1$, just as the equality in \eqref{equlity-omega=r}.
%\begin{equation*}
%	\binom{\omega_h(n)}j=\sum_{\substack{d_h\mid n,\,\mu^2(d_h)=1\\p\mid d_h\Rightarrow p\in\mathcal P_h\\\omega(d_h)=j}}1.
%\end{equation*}
Grouping the terms according to $d=d_1\cdots d_i$, we get
\begin{equation}\label{sum-qk+2-error}
	\sum_{\substack{x<qk\le y\\[0,3mm]2\nmid k,\,r\mid k}}\biggl|\mathcal W(qk+2)-\sum_{d\mid qk+2}c'_d\biggr|
	\le \sum_{\substack{x<qk\le y\\[0,3mm] 2\nmid k,\,r\mid k}}\,
	\sum_{d\mid qk+2}e_d
	=\sum_{d}e_d\sum_{\substack{x<qk\le y\\[0,3mm]2\nmid k,\,r\mid k\\[0,3mm] d\mid qk+2}}1
\end{equation}
for some $e_d\ge0$, with $\mu(d)\neq 0$ and $p\mid d\Longrightarrow p\in\bigcup_i\mathcal P_i$. 

Since $p\mid d\Longrightarrow p\in \bigcup_i\mathcal P_i\subset(2,\,Z]$, we have  $(q,\,d)=1$ whenever $R_Z(q)=1$. By the Chinese remainder theorem, we find that the inner sum over $k$ in \eqref{sum-qk+2-error}  is estimated by
\begin{equation*}
	\#\Bigl\{k:\ x<qk\le2x,\,2\nmid k,\,r\mid k,\,d\mid qk+2\Bigr\}\ll\frac{x}{qrd}+1,
\end{equation*}
Note that $q\le x^{2/5}$, $r\le Z^{R_0+1}=x^{o(1)}$ and $d\le\exp\bigl(O(R_0L^{0.2})\bigr)=x^{o(1)}$, thus the above term 1 can be absorbed by $x/(qrd)$.
To estimate the terms containing $x/(qrd)$, we use
\begin{equation*}
	\sum_{\substack{\mu^2(r)=1\\p\mid r\Rightarrow2<p\le Z\\ \omega(r)\le R_0}}\frac1r\le\prod_{2<p\le Z} \biggl(1+\frac1p\biggr)\ll\log Z.
\end{equation*}
Moreover, by the definitions of $c'_d$ and $e_d$, together with $V_i=\log2+o(1)$, we also obtain
\begin{equation*}
	\begin{aligned}
		\sum_d\frac{|c'_d|}{d}\ll1,\quad\
		\sum_d\frac{e_d}{d}\ll\sum_{i=1}^{K}\frac{V_i^{R_0+1}}{R_0!} \prod_{h<i}\biggl(\sum_{j=1}^{R_0}\frac{jV_h^j}{j!}\biggr)
		\ll_{K,\,\rho}\frac1{R_0!}.
	\end{aligned}
\end{equation*}
where we have used $V_i<1$ and $\sum_{j=1}^{R_0}jV_h^j/j!\le V_he^{V_h}\le e$ for the last inequality.

Hence, by the preceding estimates, we obtain uniformly for $x\le y\le2x$
\begin{equation*}
	\begin{aligned}
		\sum_{\substack{q\le x^{2/5}\\R_Z(q)=1}}|\varpi_q|
		\sum_{\substack{\mu^2(r)=1\\p\mid r\Rightarrow2<p\le Z\\ \omega(r)\le R_0}}\,
		\sum_{\substack{x<qk\le y\\2\nmid k,\,r\mid k}}
		\biggl|\mathcal W(qk+2)-\sum_{d\mid qk+2}c'_d \biggr|\le\sum_{\substack{q\le x^{2/5}\\R_Z(q)=1}}\frac{|\varpi_q|}{q}\sum_{\substack{\mu^2(r)=1\\p\mid r\Rightarrow2<p\le Z\\\omega(r)\le R_0}}
		\frac{x}{rR_0!}
		\ll\frac{xL^3}{R_0!}.
	\end{aligned}
\end{equation*}
Inserting this estimate into \eqref{replace1}, we get
\begin{equation*}
	\begin{aligned}
		T_2(F)&=\sum_{\substack{q\le x^{2/5}\\R_Z(q)=1}}\varpi_q\sum_{\substack{\mu^2(r)=1\\p\mid r\Rightarrow2<p\le Z\\\omega(r)\le R_0}} \mu(r)\sum_{\substack{d\\(d,\,r)=1}}c'_d\sum_{\substack{x<qk\le2x\\2\nmid k,\,r\mid k\\d\mid qk+2}}\mu(qk+2)
		\\[1,6mm]
		&\quad+O_{K,\,\rho}\biggl(\frac{xL^2}{(R_0+1)!}
		\biggl(\sum_{2<p\le Z}\frac1p\biggr)^{R_0+1}\biggr)
		+O_{K,\,\rho}\biggl(\frac{xL^3}{R_0!}\biggr).
	\end{aligned}
\end{equation*}

%To estimate the terms arising from the constant $1$ in the counting bound, we use $r\le Z^{R+1}=x^{o(1)}$. The definitions of $c'_d$ and $e_d$, together with \eqref{eq:prime-sets}, also give
%\begin{equation*}
%	d\le\exp\Bigl(O\big(RL^{0.2}\big)\Bigr),\qquad \sum_d\bigl(|c'_d|+e_d\bigr)\le\exp\Bigl(O\big(RL^{0.2}\big)\Bigr)=x^{o(1)},
%\end{equation*}
%where the bound for $d$ applies whenever $|c'_d|+e_d>0$.
Recall that $R_0=\lceil C\log L\rceil$. Choosing $C>e$ so that $C\log(C/e)>A+2$, one may find that the above two error terms are at most $xL^{-A}$. Therefore by the above argument, we obtain
\begin{equation*}
	\begin{aligned}
	T_2(F)=\sum_{\substack{q\le x^{2/5}\\[0,3mm]R_Z(q)=1}}\varpi_q
	\sum_{\substack{\mu^2(r)=1\\[0,3mm]p\mid r\Rightarrow2<p\le Z \\[0,3mm]\omega(r)\le R_0}}\mu(r)
	\sum_{\substack{d\\[0,3mm](d,\,r)=1}}c'_d
	\sum_{\substack{x<qk\le2x\\[0,3mm] 2\nmid k,\, r\mid k
	\\[0,3mm]d\mid qk+2}}\mu(qk+2)+O_{A,\,K,\,\rho}(xL^{-A}).
	\end{aligned}
\end{equation*}

We now estimate the remaining sum involving the M\"obius function over $x<qk\leq y$, uniformly for $y\in(x,\,2x]$. We will treat $T_2$, and the same estimates will then give $T_1$ by partial summation.

First, for any positive integers $d$ and $v$, we have
\begin{equation*}\label{eq:010}
	\mu(dv)=\mu(d)\mu(v)\mathbf1_{(v,\,d)=1}=\mu(d)\sum_{\substack{a\mid v,\, p\mid a\Rightarrow p\mid d}}\mu(v/a).
\end{equation*}
Write $k=rj$ and $qrj+2=dat$. The preceding equality therefore gives
\begin{equation}\label{mobius-sum}
	\begin{aligned}
		&\sum_{\substack{x<qrj\le y\\[0,4mm] 2\nmid j,\,d\mid qrj+2}}\mu(qrj+2)=\mu(d)\sum_{\substack{a\ge1\\[0,4mm] p\mid a\Rightarrow p\mid d}}\ \sum_{\substack{(x+2)/(da)<t\le(y+2)/(da)\\[0,4mm] dat\equiv qr+2\,(\textrm{mod}\,2qr)}}\mu(t).
	\end{aligned}
\end{equation}
Note that $(da,\,2qr)=(qr+2,\,2qr)=1$, so the inner sum runs over a reduced residue class modulo $2qr$. We estimate it separately according as $da\le x^{1/20}$ or $da>x^{1/20}$.

First suppose that $da\le x^{1/20}$ and take $Y_0=(2x+2)/(da)$. Since $qr\le x^{2/5}Z^{R_0}=x^{2/5+o(1)}$, we have
\begin{equation*}
	Y_0\ge2x^{19/20},\qquad \frac{2qr}{Y_0^{1/2}}\ll x^{-3/40+o(1)}.
\end{equation*}
Thus the modulus $2qr$ satisfies $2qr\le Y_0^{1/2}L^{-B}$ for any fixed $B>0$, as required in \eqref{BV-mobius}.
Let
\begin{equation*}
	E(Y_0;\,m):=\max_{(c,\,m)=1}\max_{ z\le Y_0}
	\biggl|\sum_{\substack{t\le z\\t\equiv c\,(\textrm{mod}\,m)}}\mu(t)\biggr|,
\end{equation*}
and hence $E(Y_0;\,m)\ll Y_0/m+1$ trivially. Then by applying \eqref{BV-mobius}, together with the Cauchy--Schwarz inequality, we have
\begin{equation*}
	\sum_{\substack{q\le x^{2/5}\\r\le Z^{R_0}}}|\varpi_q| E(Y_0;\,2qr)\le L\sum_{m\le Y_0^{1/2}L^{-B}}\tau(m)E(Y_0;\,m)\ll_A \frac{Y_0}{L^A}.
\end{equation*}
We then sum over $d$ and $a$, together with the estimate $\sum_d|c'_d|/d\ll1$, getting
\begin{equation}\label{da<}
	\begin{aligned}
	\sum_d|c'_d|\sum_{\substack{da\le x^{1/20}\\p\mid a\Rightarrow p\mid d}}\sum_{\substack{q\le x^{2/5}\\r\le Z^{R_0}}}|\varpi_q|E\biggl(\frac{2x+2}{da};\,2qr\biggr)
	\ll \frac{x}{L^A}\sum_d\frac{|c'_d|}{d}
	\prod_{p\mid d}\frac{1}{(1-p^{-1})}
	\ll_{A,\,K,\,\rho} \frac{x}{L^A}.
	\end{aligned}
\end{equation}

It remains to treat the sum in \eqref{mobius-sum} with $da>x^{1/20}$. Since $d=x^{o(1)}$, we have $a>x^{1/25}$ for sufficiently large $x$. Rankin's trick gives
\begin{equation}\label{eq:supptl}
	\begin{aligned}
		\sum_{\substack{a>x^{1/25}\\p\mid a\Rightarrow p\mid d}}\frac1a&\le x^{-1/50}\prod_{p\mid d}(1-p^{-1/2})^{-1}\ll x^{-1/50},
		\\[1.2mm]
		\sum_{\substack{a\le3x\\p\mid a\Rightarrow p\mid d}}1&\le(3x)^{1/20}\prod_{p\mid d}(1-p^{-1/20})^{-1}\ll x^{1/20},
	\end{aligned}
\end{equation}
where we have used $\prod_{p\mid d}(1-p^{-1/20})^{-1}\ll 1$, since every prime dividing $d$ exceeds $\exp(L^{0.1})$.

The arithmetic progression in \eqref{mobius-sum} contains $O(x/(daqr)+1)$ terms and is empty when $da>3x$.  Hence the contribution from $da>x^{1/20}$ is bounded by 
\begin{equation*}%\label{key}
	\begin{aligned}
	\sum_{q\le x^{2/5}}|\varpi_q| \sum_{\substack{\mu^2(r)=1\\[0,3mm]p\mid r \Rightarrow2<p\le Z\\[0,3mm]\omega(r)\le R_0}}
	\sum_d|c'_d|
	\sum_{\substack{x^{\frac{1}{20}}<da\le3x\\[0,3mm]p\mid a\Rightarrow p\mid d}}\biggl(\frac{x}{daqr}+1\biggr)
	\ll x^{\frac{49}{50}}L^2\sum_r\frac1r\sum_d\frac{|c'_d|}{d}+x^{\frac{1}{2}}\ll \frac{x}{L^A},
	\end{aligned}
\end{equation*}
where we have used \eqref{eq:supptl}.  Combining this estimate with \eqref{da<}, we obtain 
\begin{equation}\label{uniform-T2}
	\max_{x\le y\le2x}\biggl|\sum_{\substack{q\le x^{2/5}\\R_Z(q)=1}}\varpi_q\sum_{x<qk\le y}R_Z(k)F(qk)\biggr|\ll_{A,\,K,\,\rho}\frac{x}{L^A}.
\end{equation}
uniformly for any coefficients $|\varpi_q|\le L$. In particular,
we have
$$T_2(F)\ll_{A,\,K,\,\rho}xL^{-A}.$$

Finally, for $T_1$, we write $\log k=\log(qk)-\log q$. Since $
\log(qk)=\log(2x)-\int_{qk}^{2x} y^{-1} \mathrm d y$, we get
\begin{equation*}
	\begin{aligned}
		T_1(F)&=\log(2x)\sum_{q\le U}\mu(q)R_Z(q)\sum_{x<qk\le2x}R_Z(k)F(qk)\\
		-&\int_x^{2x}\biggl(\sum_{q\le U}\mu(q)R_Z(q)\sum_{x<qk\le y}R_Z(k)F(qk)\biggr)\frac{\mathrm d y}{y}-\sum_{q\le U}\mu(q)\log q\,R_Z(q)\sum_{x<qk\le2x}R_Z(k)F(qk).
	\end{aligned}
\end{equation*}
Then by \eqref{uniform-T2}, we conclude that 
\begin{equation*}
	T_1(F)\ll_{A,\,K,\,\rho}\frac{x}{L^{A+1}}\biggl(\log(2x)+\int_x^{2x}\frac{\mathrm d y}{y}+1\biggr)\ll_{A,\,K,\,\rho}\frac{x}{L^A}.
\end{equation*}
This completes the proof of Proposition \ref{prop:typeI}.
\end{proof}

\vskip 5mm

\section{\large The Type II estimate}\label{sec-type-II}

\subsection{Reduction to the major arc approximation}
\label{sec:cauchy}

Now it remains to estimate the Type II sum
\begin{equation}\label{eq:remain}
	T_3(F)=\sum_{\substack{m>U,\,n>V\\x<mn\le2x}}\mu(m)R_Z(m)c_V(n)R_Z(n)\mu(mn+2)\mathcal W(mn+2).
\end{equation}
with $U=V=x^{1/5}$. We first separate the variables $m,n$ in the condition $x<mn\le2x$. %and then expand $\mathcal W(mn+2)$ to prepare for Cauchy's inequality.

Let $B$ be a sufficiently large constant. Choose a smooth function $\psi:\mathbb R\to[0,\,1]$, analogous to that constructed in \cite[Page 242]{BomFriIwa86}, such that $|\psi^{(j)}(t)|\ll_j L^{Bj}$ for every integer $j\ge1$ and
\begin{equation*}
	\psi(t)=
	\begin{cases}
		1, & 1+L^{-B}\le t\le2-L^{-B},\\[1,6mm]
		0, & t\le1-L^{-B}\quad \text{or}\quad t\ge2+L^{-B}.
	\end{cases}
\end{equation*}
We then replace $\mathbf1_{x<mn\le2x}$ by $\psi(mn/x)$ in \eqref{eq:remain} up to an admissible error term	$O_{K,\,\rho}(xL^{2-B}),$
since $c_V(n)\ll L$, $\mathcal W\ll 1$, and $\sum_{k\in I}\tau(k)\ll |I|L+x^{1/2}$ for any interval $I\subset[x/2,\,3x]$.
With $\widehat\psi(v)=(2\pi)^{-1}\int_0^\infty\psi(u)u^{-iv}\,du/u$, Mellin inversion gives
\begin{align}\label{mellin-psi}
	\psi\Big(\frac{mn}{x}\Big)&=\int_{|v|\le L^{2B}}\widehat\psi(v)\Big(\frac{mn}{x}\Big)^{iv}\,dv+O(L^{-B}),\quad \text{and}\quad
	\int_{|v|\le L^{2B}}|\widehat\psi(v)|\,dv\ll_B\log L.
\end{align}

We then split the sum defining $T_3(F)$ into dyadic ranges $(M,\,2M]\times(N,\,2N]$ and replace $\mathbf1_{x<mn\le2x}$ by $\psi(mn/x)$, whence $MN\asymp x$.
For each such pair $M,N$ and any $|v|\le L^{2B}$, define
\begin{equation}\label{eq:011}
	T(M,N):=\sum_{\substack{M<m\le2M\\N<n\le2N}}\alpha_m\beta_n\mu(mn+2)\mathcal W(mn+2),
\end{equation}
where $\alpha_m:=\mu(m)R_Z(m)m^{iv}$, $\beta_n:=c_V(n)R_Z(n)n^{iv}$ satisfying $|\alpha_m|\le1$, $|\beta_n|\ll L$.
Then we insert \eqref{mellin-psi} into the expression \eqref{eq:remain} for $T_3(F)$, getting
\begin{equation}\label{eq:T3-T}
	\begin{aligned}
		T_3(F)&=\sum_{M,\,N}\int_{|v|\le L^{2B}} \widehat\psi(v)x^{-iv}T(M,N)\,\mathrm dv +O_{K,\,\rho}(xL^{2-B})
		\\
		&\le\sum_{M,\,N}\int_{|v|\le L^{2B}} |\widehat\psi(v)|T(M,N)\,\mathrm dv +O_{K,\,\rho}(xL^{2-B}).
	\end{aligned}
\end{equation}
where the sum is over dyadic segments $(M,\,2M]\times(N,\,2N]$ with $M,N\ge x^{1/5}$ that contain at least one pair $(m,\,n)$ with $\psi(mn/x)\ne0$.
%Hence, it suffices to estimate $T(M,N)$ uniformly over these dyadic segments and $|v|\le L^{2B}$.

Recall the definition of $\omega_i(mn+2)=\sum_{\substack{p_i\in\mathcal P_i,\,p_i\mid mn+2}}1$ appearing in the weight $\mathcal W(mn+2)$.
For each choice $p_i\in\mathcal P_i$, $1\le i\le K$, write $a=p_1\cdots p_K$, and put $\omega(h)=\sum_i\omega_i(h)$. Since the sets $\mathcal P_i$ are disjoint, $a$ is squarefree and $\mu(a)=(-1)^K$.
Writing $mn+2=ah$, we have $\mu(ah)=(-1)^K\mu(h)$, $\omega(ah)=K+\omega(h)$ when $(a,\,h)=1$; and $\mu(ah)=0$ when $(a,\,h)>1$. Hence 
\begin{equation*}
	T(M,N)=(-1)^K\biggl(\prod_{i=1}^K\frac1{V_i}\biggr)\sum_{p_1\in\mathcal P_1}\cdots\sum_{p_K\in\mathcal P_K}\, \sum_{\substack{m,\,n,\,h\\mn+2=ah\\(a,\,h)=1}}\alpha_m\beta_n\mu(h)\rho^{\omega(h)}.
\end{equation*}

We next remove the condition $(a,\,h)=1$. 
%and then apply the Cauchy-Schwarz inequality to the sum over $h$. 
To estimate the resulting error, we consider the terms satisfying $mn+2=ah$ and $(a,\,h)>1$. Since $a=\prod_{i=1}^Kp_i$, the condition $(a,\,h)>1$ implies $p_i\mid h$ for some $1\le i\le K$, whence $p_i^2\mid mn+2$ with $p_i>\exp(L^{a_i})$, $a_i>1/10$. By trivial estimate, we may see that The contribution of terms with $(a,\,h)>1$ to $T$ is $O_{K,\,\rho}(XL\exp(-L^{1/20}))$. Consequently, one has
\begin{equation*}
	\begin{aligned}
	T(M,N)=(-1)^K\biggl(\prod_{i=1}^K\frac1{V_i}\biggr)\sum_{p_1\in\mathcal P_1}\cdots\sum_{p_K\in\mathcal P_K}\, \sum_{\substack{m,\,n,\,h\\mn+2=ah}}\alpha_m\beta_n\mu(h)\rho^{\omega(h)}+O\Bigl(XL\exp(-L^{1/20})\Bigr).
	\end{aligned}
\end{equation*}

We decompose the expression for $T(M,N)$ according to the size of $a=p_1\cdots p_K$ by introducing a smooth dyadic partition of unity  $\eta$, which is supported on $[1,\,4]$, satisfying $\eta^{(j)}(t)\ll_j1$ for any $j\ge0$ and
\begin{equation*}
	\sum_{\ell\in\mathbb Z}\eta(t/2^\ell)=1.
\end{equation*}
Let $Y=2^\ell$ with $\ell\in\mathbb Z$. We define
\begin{equation*}
	T_Y=(-1)^K\biggl(\prod_{i=1}^K\frac1{V_i}\biggr)\sum_{h\ge1}\mu(h)\rho^{\omega(h)}\sum_{p_1\in\mathcal P_1}\cdots\sum_{p_K\in\mathcal P_K}\eta\Big(\frac{a}{Y}\Big)
	\sum_{\substack{m,\,n\\mn+2=ah}}\alpha_m\beta_n,
\end{equation*}
and hence
\begin{equation}\label{T-decomp}
	T(M,N)=\sum_{\ell\in\mathbb Z}T_{2^\ell} +O\Bigl(XL\exp(-L^{1/20})\Bigr).
\end{equation}
Note that $h\le(4X+2)/Y$ and $Y\le a\le4Y$ since $\eta$ is supported on $[1,\,4]$. 

Now we consider the sum $T_Y$. By the Cauchy--Schwarz inequality, together with the bound $|\mu(h)\rho^{\omega(h)}|\le1$, we obtain
\begin{equation}\label{eq:CS}
	|T_Y|^2\ll\frac XY Q_Y,
\end{equation}
where
\begin{equation}\label{eq:013}
	\begin{aligned}
		Q_Y&:=\biggl(\prod_{i=1}^K\frac1{V_i^2}\biggr)\sum_{h\ge1}\biggl|\sum_{p_1\in\mathcal P_1}\cdots\sum_{p_K\in\mathcal P_K}\eta\Big(\frac{a}{Y}\Big)\sum_{\substack{m,\,n\\mn+2=ah}}\alpha_m\beta_n\biggr|^2\\
		&=\biggl(\prod_{i=1}^K\frac1{V_i^2}\biggr)
		\sum_{\substack{a=p_1\cdots p_K\\[0,4mm]p_i\in\mathcal P_i}}
		\sum_{\substack{b=q_1\cdots q_K\\[0,4mm]q_i\in\mathcal P_i}}
		\eta\Big(\frac{a}{Y}\Big)\eta\Big(\frac{b}{Y}\Big)\sum_{\substack{m,\,n,\,r,\,s\\ \frac{mn+2}{a}=\frac{rs+2}{b}\in\mathbb N^*}}\alpha_m\overline{\alpha_r}\beta_n\overline{\beta_s}.
	\end{aligned}
\end{equation} 
Inserting the decomposition \eqref{T-decomp} and \eqref{eq:CS} into \eqref{eq:T3-T}, we obtain, for any $A>0$
\begin{equation}\label{eq:finalII}
	\begin{aligned}
	|T_3(F)|&\ll\sum_{M,\,N}\int_{|v|\le L^{2B}}|\widehat\psi(v)| \sum_Y|T_Y|\,\mathrm dv+ xL^{-A}
	\\
	&\ll\sum_{M,\,N}\int_{|v|\le L^{2B}}|\widehat\psi(v)| \sum_Y\biggl(\frac{MN}{Y}Q_Y\biggr)^{1/2}\,\mathrm dv+xL^{-A}.
\end{aligned}
\end{equation} 
Here $Y$ runs over the dyadic values for which $\eta(a/Y)\ne0$ with $a\in\mathcal A:=\{p_1\cdots p_K:p_i\in\mathcal P_i\}$. 

It remains to estimate $Q_Y$. We want to replace the condition $\frac{mn+2}{a}=\frac{rs+2}{b}\in\mathbb N^*$ by the equation $bmn-ars=2(a-b)$. To do this, we first establish the following lemma. 

\begin{lemma}\label{lem:shared}
The contribution of terms in $Q_Y$ with $(a,\,b)>1$ is
$O(XYL^{-A})$ for any $A>0$.
\end{lemma}
\begin{proof}
First, for every odd $\ell\asymp Y=x^{o(1)}$, we have
\begin{equation}\label{eq:014}
\begin{aligned}
\sum_{\substack{h\ll X/Y\\\ell h>2}}\tau(\ell h-2)^2
\le\frac XY\sum_{d\ll X^{3/4}}\frac{\tau_3(d)}d \ll\frac{XL^3}{Y}.
\end{aligned}
\end{equation}
Note that $|\alpha_m|,\,|\beta_n|\ll L^C$ for some $C>0$. Hence for fixed $a$ and $b$, by applying the Cauchy--Schwarz inequality and the above estimate \eqref{eq:014}, we get
\begin{equation}\label{sum-h,m,n}
	\begin{aligned}
		\sum_{\substack{h\ll X/Y\\[0,3mm]ah,\,bh>2}}\,
		\sum_{\substack{m,\,n,\,r,\,s\\[0,3mm]mn+2=ah\\[0,3mm]rs+2=bh}}|\alpha_m\overline{\alpha_r}\beta_n\overline{\beta_s}|
		\ll L^{4C}\sum_{\substack{h\ll X/Y\\[0,3mm]ah>2,\,bh>2}}\tau(ah-2)\tau(bh-2)
		\ll\frac XY L^{4C+3}.
	\end{aligned}
\end{equation}
If $(a,\,b)>1$, then $p\mid (a,b)$ for some $p\in\bigcup_i\mathcal P_i$. We therefore have
\begin{equation}\label{eq:abcoll}
	\begin{aligned}
		\sum_{\substack{a=p_1\cdots p_K,\,b=q_1\cdots q_K\\p_i,\,q_i\in\mathcal P_i\\Y\le a,\,b\le4Y,\,(a,\,b)>1}}1
		\le\sum_{p\in\bigcup_i\mathcal P_i}\biggl\lfloor\frac{4Y}{p}\biggr\rfloor^2\le16Y^2\sum_{p\in\bigcup_i\mathcal P_i}\frac1{p^2}\ll Y^2\exp(-L^{1/10}).
	\end{aligned}
\end{equation}

Finally, combining \eqref{sum-h,m,n} and \eqref{eq:abcoll}, together with the bounds $0\le\eta\le1$ and $\prod_{i=1}^KV_i^{-2}\ll_K1$, we find that the contribution of terms in $Q_Y$ with $(a,\,b)>1$ is at most
\begin{equation*}
	\begin{aligned}
	\ll \frac XY L^{4C+3}\cdot Y^2\exp(-L^{1/10})\ll XY\exp(-L^{1/20}),
	\end{aligned}
\end{equation*}
which completes the proof of Lemma \ref{lem:shared}.
\end{proof}

By Lemma \ref{lem:shared}, we may therefore restrict the quantity $Q_Y$ to $(a,\,b)=1$, with an adimissible error term $O(XYL^{-A})$.
For those terms in $Q_Y$ with $(a,\,b)=1$, the condition $\frac{mn+2}{a}=\frac{rs+2}{b}\in\mathbb N^*$ yields
\begin{equation}\label{eq:det2}
	\begin{aligned}
		bmn-ars=2(a-b).
	\end{aligned}
\end{equation}
Conversely, \eqref{eq:det2} gives $b(mn+2)=a(rs+2)$. Since $(a,\,b)=1$, then we have $a\mid mn+2$ and $b\mid rs+2$, which implies that $\frac{mn+2}{a}=\frac{rs+2}{b}\in\mathbb N^*$. Hence the condition $\frac{mn+2}{a}=\frac{rs+2}{b}\in\mathbb N^*$ appearing in $Q_Y$ may replaced by that in \eqref{eq:det2}.

Let $\psi_1:\mathbb R\to[0,\,1]$ be a smooth function independent of $x$,  just as \cite[Page 242]{BomFriIwa86}, such that $0\le\psi_1\le1$, $\psi_1(u)=1$ for $|u|\le4$ and $\psi_1(u)=0$ for $|u|\ge 8$. Note that $\psi_1^{(j)}(u)\ll_j1$ for any $j\ge0$.
Let $A_0>0$ be a fixed constant.  We define
\begin{equation}\label{eq:major-arcs-kernel}
	\begin{aligned}
	&\mathfrak M_1:=\bigcup_{\substack{1\le k\le L^{A_0}\\[0.3mm]
		0\le j<k,\,(j,\,k)=1}}\biggl\{\phi\in\mathbb T:\ \Bigl\|\phi-\frac jk\Bigr\|\le\frac{2L^{A_0}}Y\biggr\},
		\quad
		\mathfrak M_{-2}:=\Big\{\theta\in\mathbb T:\ -2\theta\in\mathfrak M_1\Big\},
		\\[1mm]
		&\ \psi_{-2}(u):=\psi_1\Big(-\frac{u}{2}\Big), \quad H_{-2}(t;\,a,\,b):=\psi_{-2}\Big(\frac{t}{Y}\Big)
		\int_{\mathfrak M_{-2}}
		e\bigl(\theta(t+2b-2a)\bigr)\,d\theta.
	\end{aligned}
\end{equation}
Note that $\eta(\frac{a}{Y})\eta(\frac{b}{Y})$ is supported on $Y\le a,\,b\le4Y$, so the relation $bmn-ars=-2(b-a)$ gives
\begin{equation*}
	\psi_{-2}\Bigl(\frac{bmn-ars}{Y}\Bigr)=\psi_1\Bigl(\frac{b-a}{Y}\Bigr)=1.
\end{equation*}
Therefore, we may insert the factor $\psi_{-2}(\frac{bmn-ars}{Y})$ into $Q_Y$ without changing its value.
We also define the quantity $Q_Y^{\mathrm{maj}}$ by
\begin{equation}\label{def-Q_Y^maj}
	\begin{aligned}
	Q_Y^{\mathrm{maj}}:=\biggl(\prod_{i=1}^{K}\frac{1}{V_i^2}\biggr)
	\mathop{\sum\ \sum}_{\substack{a=p_1\cdots p_K\\[0.3mm] b=q_1\cdots q_K\\[0.4mm]p_i,\,q_i\in\mathcal P_i}}
	\eta\Big(\frac{a}{Y}\Big)\eta\Big(\frac{b}{Y}\Big)
	\sum_{\substack{M<m,\,r\le2M\\[0.3mm]N<n,\,s\le2N}}
	\alpha_m\overline{\alpha_r}\beta_n\overline{\beta_s}
	H_{-2}\big(bmn-ars;a,\,b\big).
	\end{aligned}
\end{equation}
Note that the summation over $a,\,b$ with $(a,\,b)>1$ is also included.

\begin{proposition}\label{prop:D}
Let $0<\delta<1/5$ and $C>0$ be some fixed constants. Suppose $M,\,N\ge x^{\delta}$, $X=MN\asymp x$ with $x$ sufficiently large, and
\begin{equation*}
	|\alpha_m|\le L^C R_Z(m)\mathbf1_{M<m\le2M},\qquad |\beta_n|\le L^C R_Z(n)\mathbf1_{N<n\le2N}.
\end{equation*}
%For every fixed $A>0$, we can choose the major-arc parameter $A_0$ and the number $K$ of prime sets in \eqref{eq:prime-sets}, independently of the fixed exponents $a_i$, such that, for sufficiently large $x$, 
Then for any $A>0$, we can choose the major arc parameter $A_0$ in \eqref{eq:major-arcs-kernel} and the number $K$ of prime sets in \eqref{eq:prime-sets} independently of the fixed exponents $a_i$, such that
\begin{equation*}
	\big|Q_Y-Q_Y^{\mathrm{maj}}\big|\ll_A XYL^{-A}.
\end{equation*}
\end{proposition}
In the remainder of this section, we will prove Proposition \ref{prop:D}. To do this, we first introduce some operators in the next subsection. 

\vskip 3mm

\subsection{Construction of the operators}
\label{sec:001}\label{sec:input}

Retain the prime sets $\mathcal P_i$, the exponents $a_i$, the normalizing factors $V_i$, and the parameters $Z,\,\rho$ from Section \ref{sec:prelim}. In this subsection, we shall construct operators to handle the difference between the indicator function $\mathbf1_{bmn-ars=-2(b-a)}$ in $Q_Y$ and its major arc replacement in $Q_Y^{\mathrm{maj}}$:
\begin{equation*}
	\mathbf1_{bmn-ars=2(a-b)}-H_{-2}(bmn-ars;\,a,\,b),
\end{equation*}

%The coordinate change
%\begin{equation*}
%	(u,\,v)^T\longmapsto\biggl(u,\,\frac{u-v}{2}\biggr)^T\qquad(u,\,v\text{ odd})
%\end{equation*}
%leaves the first coordinate and the selected primes unchanged and multiplies determinants by $-1/2$. It therefore transforms the determinant equation with coefficient $-2$ into the corresponding equation with coefficient $1$. We shall verify that the prime weights and major-arc integrals agree under this transformation.

For any integer column vectors $P=(P_1,\,P_2)^T\in\mathbb Z^2$ and $Q=(Q_1,\,Q_2)^T\in\mathbb Z^2$, write
\begin{equation*}
\det(P,\,Q)=P_1Q_2-P_2Q_1,\qquad |P|=\max\big(|P_1|,\,|P_2|\big).
\end{equation*}
We call $P$ primitive if $(P_1,\,P_2)=1$. For a prime $p$, let $\mathbb F_p=\mathbb Z/p\mathbb Z$, and let $\mathbb P^1(\mathbb F_p)$ be the set of one-dimensional subspaces of 
$\mathbb F_p^2=\{(a,\, b)^T:\ a,\,b\in \mathbb F_p\}$:
\begin{equation}\label{def-P^1(F_p)}
	\mathbb P^1(\mathbb F_p):=\Bigl\{\{\lambda z:\lambda\in\mathbb F_p\}:z\in\mathbb F_p^2\setminus\{0\}\Bigr\}.
\end{equation}
Two nonzero vectors generate the same subspace if and only if one is a nonzero scalar multiple of the other.
There are $p+1$ such subspaces since they are generated by $(1,\,\xi)^T$ with $\xi\in\mathbb F_p$ and by $(0,\,1)^T$. For any primitive vector $z=(z_1,\,z_2)^T\in\mathbb Z^2$, define $[z]_p\in \mathbb P^1(\mathbb F_p)$ by
%denote the subspace generated by $z$ modulo $p$
\begin{equation}\label{def-[z]_p}
[z]_p:=\Bigl\{\lambda(z_1,\, z_2)^T\in \mathbb F_p^2:\ \lambda\in\mathbb F_p\Bigr\},
\end{equation}

Let $J=\lfloor L^{0.01}\rfloor$. For any choice of distinct primes $p_{i,1},\,\ldots,\,p_{i,J}\in\mathcal P_i$ for each $1\le i\le K$, there is a unique integer $\ell_0$ such that
\begin{equation}\label{ell_0-D}
	2^{\ell_0}\le\prod_{i=1}^K\prod_{j=1}^Jp_{i,j}<2^{\ell_0+1}.
\end{equation}
Let  $d_0=2^{\ell_0}$, we consider the interval $[d_0,\,2d_0)$ containing at least one of these products, and define
\begin{equation}\label{eq:physbox}
	\Omega:=[d_0YM,\,16d_0YM]\times[N,\,2N].
\end{equation}
For each primitive $P\in\Omega\cap\mathbb Z^2$, we consider ordered tuples of primes
\begin{equation*}
	\boldsymbol p_i=(p_{i,1},\,\ldots,\,p_{i,J+1})\quad(1\le i\le K),\quad\text{and}\quad \boldsymbol p=(\boldsymbol p_1,\,\ldots,\,\boldsymbol p_K).
\end{equation*}
We call each tuple $\boldsymbol p_i$ a ``\textbf{prime list}''. Next, for $1\le i\le K$, $1\le j,\,h\le J+1$ and $j\neq h$, we define the set $\mathscr S$ of pairs $(P,\,\boldsymbol p)$ by
\begin{equation}\label{mathscr-S}
	\mathscr S=\left\{(P,\,\boldsymbol p):\
	\begin{array}{l}
		P=(P_1,\,P_2)\in\Omega\cap\mathbb Z^2,\ \, \gcd(P_1,\,P_2)=1,\\[0,7mm]
		p_{i,j}\in\mathcal P_i,\ \,p_{i,j}\mid P_1,\ \,
		p_{i,j}\ne p_{i,h}
	\end{array}\right\}.
\end{equation}
We call each pair $(P,\,\boldsymbol p)\in\mathscr S$ a ``\textbf{state}''; and for a state $(P,\,\boldsymbol p)$, we call $P$ its ``\textbf{position}''.
For every $(P,\,\boldsymbol p)\in\mathscr S$, let $\sigma(P,\,\boldsymbol p)=\prod_{i=1}^{K}V_i^{-(J+1)}$, which is, in fact, independent of $P$ and selected prime tuples $\boldsymbol p$. Define the $L^2$ space of complex-valued functions on $\mathscr S$ weighted by $\sigma(P,\,\boldsymbol p)$:
\begin{equation}\label{def-H}
	\mathcal H=L^2(\mathscr S,\,\sigma):=\bigg\{f:\mathscr S\to\mathbb C:\ \sum_{(P,\,\boldsymbol p)\in\mathscr S}|f(P,\,\boldsymbol p)|^2\sigma(P,\,\boldsymbol p)<\infty\bigg\}.
\end{equation}
We then define the inner product of  $f, g\in\mathcal H$ by
\begin{equation}\label{def-innerprod}
	\langle f,\,g\rangle:=\sum_{(P,\,\boldsymbol p)\in\mathscr S}\overline{f(P,\,\boldsymbol p)}\,g(P,\,\boldsymbol p)\,\sigma(P,\,\boldsymbol p),
\end{equation}
and hence $\|f\|^2=\langle f,\,f\rangle.$
For any linear operator $\mathcal T:\mathcal H\to\mathcal H$, we denote by $\mathcal T^*$ its adjoint operator with respect to this inner product:
\begin{equation*}
	\langle\mathcal T f,\,g\rangle=\langle f,\,\mathcal T^*g\rangle\qquad(f,\,g\in\mathcal H).
\end{equation*}

For fixed $(P,\,\boldsymbol p)\in\mathscr S$, we retain $p_{i,1},\,\ldots,\,p_{i,J}$ in each tuple $\boldsymbol p_i$ and replace $p_{i,J+1}$ by a prime $p_i'\in\mathcal P_i\setminus\{p_{i,1},\,\ldots,\,p_{i,J+1}\}$ for all $1\le i\le K$. Then we define
\begin{equation}\label{Dba-definition}
	D=\prod_{i=1}^K\prod_{j=1}^J p_{i,j},\qquad b=\prod_{i=1}^K p_{i,J+1},\qquad a=\prod_{i=1}^K p_i',
\end{equation}
whence $d_0\le D<2d_0$ by \eqref{ell_0-D}.
For each $1\le i\le K$, consider the following two lists
\begin{equation*}
	\boldsymbol p_i=(p_{i,1},\,\ldots,\,p_{i,J},\,p_{i,J+1}),
	\qquad
	\boldsymbol p'_i=(p_{i,1},\,\ldots,\,p_{i,J},\,p'_i).
\end{equation*}
We call the entries $p_{i,1},\,\ldots,\,p_{i,J}$, which  occur in both lists, the ``\textbf{shared entries}'' of $\boldsymbol p_i$ and $\boldsymbol p'_i$.
We also call the entry $p_{i,J+1}$ ``\textbf{unshared entry}''
of $\boldsymbol p_i$, which occurs in $\boldsymbol p_i$ but not
in $\boldsymbol p'_i$; and hence
 $p'_i$ is the unshared entry of $\boldsymbol p'_i$.

Write $\boldsymbol p'=(\boldsymbol p_1',\,\ldots,\,\boldsymbol p_K')$. For fixed $(P,\,\boldsymbol p)\in\mathscr S$, we choose another vector $Q=(Q_1,\,Q_2)\in\Omega\cap\mathbb Z^2$ 
such that $(Q,\,\boldsymbol p')\in\mathscr S$. Then we may deduce that $Db\mid P_1$ and $Da\mid Q_1$ by the definitions of  $D,\,a,\,b$ and $\mathscr S$. Hence we define the integer $t$ by
\begin{equation*}
	t=\frac{\det(P,\,Q)}D\in \mathbb{Z}.
\end{equation*}
Recall that $\omega(n)=\sum_{i}\omega_i(n)=\sum_{i}\sum_{p\in\mathcal P_i}\mathbf1_{p\mid n}$, as in Subsection \ref{sec:cauchy}.
We define the ``\textbf{kernel}'' $K_1(P,\,Q;\,D,\,a,\,b)$ by
	\begin{equation*}
	\begin{aligned}
	K_1(P, Q;\,D, a, b):=\frac{d_0}{D}\eta\Big(\frac{a}{Y}\Big)
	\eta\Big(\frac{b}{Y}\Big)\psi_1\Big(\frac{t}{Y}\Big)
	\biggl\{\mathbf{1}_{t=b-a}-\int_{\mathfrak M_1}e\bigl(\theta(t-b+a)\bigr)\,\mathrm{d}\theta\biggr\}
	\rho^{\frac{\omega(P_1)+\omega(Q_1)-2K(J+1)}{2}}
	\end{aligned}
	\end{equation*}
	where $\psi_1$ and $\mathfrak M_1$ are defined in \eqref{eq:major-arcs-kernel}, and the factor $\psi_1\bigl(\frac{t}{Y}\bigr)$ is to restrict $t$ to $|t|\ll Y$; and the difference in braces is the contribution from $\mathbb T\setminus\mathfrak M_1$.
	For any $(P,\,\boldsymbol p)\in\mathscr S$ and $f\in\mathcal H$, we next define $(\mathcal T_1f)(P,\,\boldsymbol p)$, in which $K_1(P,\,Q;\,D,\,a,\,b)$ is regarded as the coefficient of $f(Q,\,\boldsymbol p')$ in the sum over $Q$:
\begin{equation}\label{def-T1}
	(\mathcal T_1f)(P,\,\boldsymbol p):=\mathbf1_{d_0\le D<2d_0}
	\biggl(\prod_{i=1}^K \frac{1}{V_i}\biggr)
	\sum_{\substack{p_i'\in\mathcal P_i\setminus\{p_{i,1},\ldots,p_{i,J+1}\}\\[0,4mm]1\le i\le K}}\ \sum_{\substack{Q\\[0,5mm] (Q,\,\boldsymbol p')\in\mathscr S}}K_1(P,Q;\,D, a, b)f(Q,\,\boldsymbol p').
\end{equation}

For $(P,\,\boldsymbol p)\in\mathscr S$ with $d_0\le D<2d_0$, we call the ordered pair $\big((P,\,\boldsymbol p),\,(Q,\,\boldsymbol p')\big)$ an ``\textbf{edge}'' if
\begin{equation*}
 (Q,\,\boldsymbol p')\in\mathscr S,\qquad 	p_i'\in\mathcal P_i\setminus\{p_{i,1},\,\ldots,\,p_{i,J+1}\}\quad(1\le i\le K).
\end{equation*}
Here $\boldsymbol p'$ is defined above, and we denote this edge by $(P,\,\boldsymbol p)\longrightarrow(Q,\,\boldsymbol p')$.
For an integer $k\ge1$, we call a sequence
\begin{equation}\label{path}
	\big(P^{(0)},\,\boldsymbol p^{(0)}\big)
	\longrightarrow \big(P^{(1)},\,\boldsymbol p^{(1)}\big)
	\longrightarrow\cdots\longrightarrow 
	\big(P^{(k)},\,\boldsymbol p^{(k)}\big)
\end{equation}
a ``\textbf{path}'' if every $(P^{(j)},\,\boldsymbol p^{(j)})\longrightarrow(P^{(j+1)},\,\boldsymbol p^{(j+1)})$, $0\le j<k$ is an edge; and call $P^{(0)}$ the initial position of this path. Note that we allow $P^{(j)}=P^{(h)}$ for some $0\le j<h\le k$ in \eqref{path}.

We next determine the adjoint operator $\mathcal T_1^*$ from the identity $\langle\mathcal T_1f,\,g\rangle=\langle f,\,\mathcal T_1^*g\rangle$ for any $f,\,g\in\mathcal H$. By the definitions of the inner product and $\mathcal T_1$, we have
\begin{equation*}
	\begin{aligned}
	\langle\mathcal T_1f,\,g\rangle
	&=\sum_{(P,\,\boldsymbol p)\in\mathscr S}\overline{(\mathcal T_1f)(P,\,\boldsymbol p)}g(P,\,\boldsymbol p)
	\sigma(P,\,\boldsymbol p)\\
	&=\sum_{\substack{(P,\,\boldsymbol p)\in\mathscr S\\[0,3mm]
	d_0\le D<2d_0}}\sigma(P,\,\boldsymbol p)\biggl(\prod_{i=1}^K \frac{1}{V_i}\biggr)\sum_{\substack{p_i'}}\, \sum_{\substack{Q\\[0,5mm](Q,\,\boldsymbol p')\in\mathscr S}}
\overline{K_1(P,\,Q;\,D,\,a,\,b)}\,\overline{f(Q,\,\boldsymbol p')}\,g(P,\,\boldsymbol p),
	\end{aligned}
\end{equation*}
and we need to write the preceding sum in the form $
\sum_{(P,\,\boldsymbol p)\in\mathscr S}
\sigma(P,\,\boldsymbol p)\overline{f(P,\,\boldsymbol p)}
\times \Xi$, whence $(\mathcal T_1^*g)(P,\,\boldsymbol p)=\Xi$.
Note that $\sigma(P,\,\boldsymbol p)=\sigma(Q,\,\boldsymbol p')=\prod_{i=1}^K V_i^{-(J+1)}$, and for the edge $(Q,\,\boldsymbol p')\longrightarrow(P,\,\boldsymbol p)$, the products $D,\,a,\,b$ become $D,\,b,\,a$, respectively. By calculation, one may find
\begin{equation*}
	(\mathcal T_1^*f)(P,\,\boldsymbol p)=\mathbf1_{d_0\le D<2d_0}\biggl(\prod_{i=1}^K \frac{1}{V_i}\biggr)
	\sum_{\substack{p_i'\in\mathcal P_i\setminus\{p_{i,1},\,\ldots,\,p_{i,J+1}\}\\[0,3mm]1\le i\le K}}\ \sum_{\substack{Q\\[0,5mm](Q,\,\boldsymbol p')\in\mathscr S}}\overline{K_1(Q,\,P;\,D,\,b,\,a)}f(Q,\,\boldsymbol p').
\end{equation*}

We next introduce two arithmetic conditions on states in $\mathscr S$, and we shall require that both $(P,\,\boldsymbol p)$ and $(Q,\,\boldsymbol p')$ in the expression for $\mathcal T_1f$ satisfy them in the subsequent operator estimates. 

For each primitive vector $P=(P_1,\,P_2)^T\in\mathbb Z^2$ with $P_1>0$, there exist integers $c,\,d$ such that $P_1d-P_2c=1$. We therefore define $r_P\in\mathbb T$ by
\begin{equation}\label{def-r_P}
	r_{_P}\equiv\frac{c}{P_1}\,(\textrm{mod}\,1).
\end{equation}
We call $(c,\,d)^T$ a ``\textbf{complement}'' to $P$. Note that for any other complement $(c',\,d')^T$, we may deduce $c'/P_1\equiv c/P_1\,(\textrm{mod}\,1)$. Thus $r_{_P}$ does not depend on the choice of the complement.
For any state $(P,\,\boldsymbol p)\in\mathscr S$ with $\boldsymbol p_i=(p_{i,1},\,\ldots,\,p_{i,J+1})$, $p_{i,j}\in\mathcal P_i$, $p_{i,j}\mid P_1$, $1\le i\le K$, we define
\begin{equation*}
	\widetilde D=\prod_{i=1}^{K}\prod_{\substack{1\le j\le J+1\\j\ne h_i}}p_{i,j},\qquad \widetilde D'=\prod_{i=1}^{K}\prod_{\substack{1\le j\le J+1\\j\ne h_i'}}p_{i,j}
\end{equation*}
for any $h_i,\,h_i'\in\{1,\,\ldots,\,J+1\}$.
Then we call a state $(P,\,\boldsymbol p)\in\mathscr S$ ``\textbf{good}'' if the following two conditions hold.
\begin{enumerate}
	\item 
	For any choice of $h_1,\,\ldots,\,h_K$ satisfying $\widetilde D\in[d_0,\,2d_0)$, and any $\ell\in\mathbb Z,\ 1\le\ell\le Y^{0.2}$, we have
	\begin{equation}\label{eq:goodone}
		\|\ell \widetilde Dr_{_P}\|>Y^{-0.7}.
	\end{equation}
	
	\item 
	Let $I\neq \varnothing$ be a subset of index set $\{1,\,\ldots,\,K\}$, and denote its largest element by $i_*$. Let $a_{i_*}$ be the exponent as in \eqref{eq:prime-sets}. For any primes $q_i,\,q_i'\in\mathcal P_i,\, 1\le i\le K$, we define
	\begin{equation*}
		Z_I=\prod_{i\in I\setminus\{i_*\}}q_i,\quad Z'_I=\prod_{i\in I\setminus\{i_*\}}q_i',\quad \text{and}\quad  T_f=\prod_{i\notin I}q_i.
	\end{equation*}
	Then we require the following bound for every nonempty $I$:
	\begin{equation}\label{eq:goodtwo}
		\begin{aligned}
		\sum_{\substack{q_i\in\mathcal P_i\ (i\notin I)\\[0,4mm] \exists\,h_i,\,h_i'\in\{1,\,\ldots,\,J+1\},\
		q_j,\,q_j'\in\mathcal P_j\ (j\in I\setminus\{i_*\}),
		\\[0,4mm] \text{such that }\widetilde D Z_I\ne\widetilde D'Z'_I \text{ and }
		\\[0,4mm] \big\|\big(\widetilde D Z_I-\widetilde D'Z'_I\big)T_fr_{_P}\big\|<100\exp(-L^{a_{i_*}})}}
		\prod_{i\notin I}\frac1{q_iV_i}\le\exp(-L^{a_{i_*}}/4).
	\end{aligned}
\end{equation}
\end{enumerate}

We fix the prime tuples $\boldsymbol p_i=(p_{i,1},\,\ldots,\,p_{i,J+1})$ for $1\le i\le K$ and estimate the measure of  values $r\in\mathbb T$ that do not satisfy \eqref{eq:goodone} when $r_{_P}$ is replaced by $r$. We write ``$\operatorname{meas}$'' for Lebesgue measure on $\mathbb T=\mathbb R/\mathbb Z$.
Note that there are at most $(J+1)^K$ choices of $\widetilde D$ in condition 1. For each such $\widetilde D$ and positive integer $1\le\ell\le Y^{0.2}$, we have
\begin{equation*}
	\operatorname{meas}\big\{r\in\mathbb T:\ \|\ell \widetilde Dr
	\|\le Y^{-0.7}\big\}=\frac1{\ell \widetilde D}\int_0^{\ell \widetilde D}\mathbf1_{\|u\|\le Y^{-0.7}}\,\mathrm du=2Y^{-0.7},
\end{equation*}
whence, the measure of $r\in\mathbb T$ that the first condition fails is 
\begin{equation}\label{eq:bad-root-measure}
	\operatorname{meas}\biggl(\bigcup_{\widetilde D}\ \bigcup_{\ell=1}^{\lfloor Y^{0.2}\rfloor}\big\{r\in\mathbb T:\ \|\ell \widetilde Dr\|\le Y^{-0.7}\big\}\biggr)\le2(J+1)^K  Y^{0.2} Y^{-0.7}\le2(J+1)^KY^{-1/2}.
\end{equation}

We next estimate the measure of $r\in\mathbb T$ for which \eqref{eq:goodtwo} does not hold. For each nonempty $I\subseteq\{1,\,\ldots,\,K\}$, we denote by $F_I(r)$ the sum on the left-hand side of \eqref{eq:goodtwo}, with $r_{_P}$ replaced by $r$. 
%Thus the condition for this $I$ is
%\begin{equation*}	F_I(r)\le\exp(-L^{a_{i_*}}/4).\end{equation*}
As $h_1,\,\ldots,\,h_K$ and the primes $q_i\in\mathcal P_i$, $i\in I\setminus\{i_*\}$, vary, the number of distinct products $\widetilde D Z_I$ satisfies
\begin{equation}\label{DZ_I}
	\#\big\{\widetilde D Z_I\big\}\le(J+1)^K\prod_{i\in I\setminus\{i_*\}}|\mathcal P_i|=\exp\bigl(o(L^{a_{i_*}})\bigr),
\end{equation}
where we have used $J\le L^{0.01}$, $|\mathcal P_i|\le\exp(2L^{a_i})$, and $a_i<a_{i_*}$ for every $i\in I\setminus\{i_*\}$.
Moreover, one has $\operatorname{meas}\{r\in\mathbb T:\ \|mr\|<\delta\}
=|m|^{-1}\int_0^{|m|}\mathbf1_{\|u\|<\delta}\,\mathrm du=2\delta$  for any integer $m\neq0$ and $0<\delta<1/2$, whence
\begin{equation*}
	\operatorname{meas}\Bigl\{r\in\mathbb T:\ \bigl\|(\widetilde D Z_I-\widetilde D'Z'_I)T_fr\bigr\|<100\exp(-L^{a_{i_*}})\Bigr\}=200\exp(-L^{a_{i_*}}).
\end{equation*}
The number of ordered pairs of distinct integers $(\widetilde D Z_I,\,\widetilde D'Z'_I)$ is at most $\bigl(\#\{\widetilde D Z_I\}\bigr)^2$.
Moreover, the weights in the sum $F_I(r)$ satisfy
\begin{equation*}
	\sum_{\substack{q_i\in\mathcal P_i,\,i\notin I}}\, \prod_{i\notin I}\frac1{q_iV_i}=\prod_{i\notin I}\biggl(\frac1{V_i}\sum_{q_i\in\mathcal P_i}\frac1{q_i}\biggr)=1.
\end{equation*}
Therefore, summing the preceding measure estimate over the pairs $(\widetilde D Z_I,\,\widetilde D'Z'_I)$ and over these prime choices $q_i\in\mathcal P_i,\,i\notin I$, we get
\begin{equation*}
	\int_{\mathbb T}F_I(r)\,\mathrm dr\le200\exp(-L^{a_{i_*}})\,\#\{\widetilde D  Z_I\}^2\le\exp\bigl(-L^{a_{i_*}}+o(L^{a_{i_*}})\bigr).
\end{equation*}
Since $F_I(r)\ge0$, Markov's inequality yields
\begin{equation}\label{estimate-F_I>}
	\operatorname{meas}\Big\{r\in\mathbb T:\ F_I(r)>\exp(-L^{a_{i_*}}/4)\Big\}
	\le\exp(L^{a_{i_*}}/4)\int_{\mathbb T}F_I(r)\,\mathrm dr\le\exp\bigl(-L^{a_{i_*}}/2\bigr).
\end{equation}

Note that there are $2^K-1$ nonempty subsets $I$. Thus, combining \eqref{estimate-F_I>} and \eqref{eq:bad-root-measure}, together with the bounds $a_{i_*}>0.1$ and $\log Y\gg L^{0.1}$, we conclude that, for some constant $c>0$
\begin{equation}\label{meas-upperbound}
	\operatorname{meas}\Big\{r\in\mathbb T:\text{at least one of \eqref{eq:goodone} and \eqref{eq:goodtwo} fails with }r_{_P}=r\Big\}\ll_K\exp(-cL^{0.1}).
\end{equation}

Next we introduce two maps $\mathcal S$ and $\mathcal G_1$ . Let $\mathfrak S_{J+1}$ be the set of all permutations of $\{1,\,\ldots,\,J+1\}$. With the $L^2$ space $\mathcal H$ defined by \eqref{def-H}, we define $\mathcal S:\mathcal H\to\mathcal H$ by
\begin{equation}\label{def-S}
	(\mathcal S f)(P,\,\boldsymbol p) :=\frac1{\big((J+1)!\big)^K}\sum_{\pi_1\in\mathfrak S_{J+1}}
	\cdots\sum_{\pi_K\in\mathfrak S_{J+1}}f\Bigl(P,\,\bigl((p_{1,\pi_1(j)})_{j=1}^{J+1},\,\ldots,\,(p_{K,\pi_K(j)})_{j=1}^{J+1}\bigr)\Bigr).
\end{equation}
Note that for every $\pi_1,\,\ldots,\,\pi_K\in\mathfrak S_{J+1}$, we have
\begin{equation*}
	(\mathcal S f)\Bigl(P,\,\bigl((p_{1,\pi_1(j)})_{j=1}^{J+1},\,\ldots,\,(p_{K,\pi_K(j)})_{j=1}^{J+1}\bigr)\Bigr)=(\mathcal S f)(P,\,\boldsymbol p),
\end{equation*}
whence $\mathcal S^2f=\mathcal S(\mathcal S f)=\mathcal S f.$
Moreover, both $(\pi_1,\,\ldots,\,\pi_K)$ and $(\pi_1^{-1},\,\ldots,\,\pi_K^{-1})$ occur in $\mathcal S$, and $\sigma(P,\,\boldsymbol p)$ does not depend on the order of the primes $p_{i,1},\,\ldots,\,p_{i,J+1}\in\mathcal P_i$ within each tuple $\boldsymbol p_i$, we therefore obtain $\langle\mathcal S f,\,g\rangle=\langle f,\,\mathcal S g\rangle$ for any $f,\,g\in\mathcal H$.
Hence $\mathcal S^*=\mathcal S$, and therefore 
\begin{equation}\label{eq:S-properties}
	\mathcal S^2=\mathcal S=\mathcal S^*.
\end{equation}

We next define the map $\mathcal G_1:\mathcal H\to\mathcal H$ by
\begin{equation}\label{def-G1}
	(\mathcal G_1f)(P,\,\boldsymbol p) :=
	\begin{cases}f(P,\,\boldsymbol p)  &\text{if both \eqref{eq:goodone} and \eqref{eq:goodtwo} hold},
		\\[1,2mm]
	0 &\text{otherwise},\end{cases}.
\end{equation}
that is,  $(\mathcal G_1f)(P,\,\boldsymbol p)=0$ whenever $(P,\,\boldsymbol p)$ is not ``good''. Moreover, it is easy to see $\mathcal G_1^*=\mathcal G_1$.

We claim that, for a fixed state $(P,\,\boldsymbol p)\in\mathscr S$, permuting the primes $p_{i,j}$ within each tuple $\boldsymbol p_i$ does not affect the validity of conditions \eqref{eq:goodone} and \eqref{eq:goodtwo}.
To see this, we first note that these permutations preserve the sets of possible values of $\widetilde D$ and $\widetilde D'$, since
for each $1\le i\le K$ and $\pi_i\in\mathfrak S_{J+1}$, 
\begin{equation*}
	\prod_{\substack{1\le j\le J+1\\j\ne h_i}}p_{i,\pi_i(j)}=\prod_{\substack{1\le j\le J+1\\j\ne\pi_i(h_i)}}p_{i,j}\qquad(1\le h_i\le J+1).
\end{equation*}
For \eqref{eq:goodone}, $P$ and $r_{_P}$ remain fixed, so the inequalities $\|\ell\widetilde D r_{_P}\|>Y^{-0.7}$ remain unchanged under these permutations.
For \eqref{eq:goodtwo}, the possible pairs $(\widetilde D Z_I,\,\widetilde D'Z'_I)$ are also unchanged, and hence the same tuples $(q_i)_{i\notin I}$ are counted in $F_I(r_{_P})$ with the same weights $\prod_{i\notin I}(q_iV_i)^{-1}$, which implies that $F_I(r_{_P})$ has the same value before and after each permutation.
By the definitions of $\mathcal S$ and $\mathcal G_1$ above, we therefore obtain, for any $f\in\mathcal H$,
\begin{equation*}
	(\mathcal S\mathcal G_1f)(P,\,\boldsymbol p)=\begin{cases}(\mathcal Sf)(P,\,\boldsymbol p) &\text{if both \eqref{eq:goodone} and \eqref{eq:goodtwo} hold},
		\\[1,2mm]
	0 &\text{otherwise},\end{cases}=(\mathcal G_1\mathcal Sf)(P,\,\boldsymbol p),
\end{equation*}
and consequently $\mathcal G_1\mathcal S=\mathcal S\mathcal G_1$.

We now construct the operator used to estimate $Q_Y-Q_Y^{\mathrm{maj}}$. For the variables $a,\,b,\,m,\,n,\,r,\,s$ in the sum $Q_Y$ in \eqref{eq:013}, we let
\begin{equation}\label{P-Q-t}
	P=\begin{pmatrix}Dbm\\s\end{pmatrix},\qquad Q=\begin{pmatrix}Dar\\n\end{pmatrix},\qquad t=\frac{\det(P,\,Q)}D=bmn-ars.
\end{equation}
The condition \eqref{eq:det2} therefore gives $t=2(a-b)$. 
Since $D,\,a,\,b$ are odd, the condition $\alpha_m\overline{\alpha_r}\beta_n\overline{\beta_s}\ne0$ implies that both coordinates of $P$ and $Q$ are odd. Thus we may restrict the states $(P,\,\boldsymbol p)\in\mathscr S$ to those  $P\in\mathcal C$, where
\begin{equation}\label{eq:017}
	\mathcal C:=\Big\{(u,\,v)^T\in\mathbb Z^2:\ \gcd(u,\,v)=1,\,u\equiv v\equiv1\,(\textrm{mod}\,2)\Big\}.
\end{equation}
For any $(P,\,\boldsymbol p)\in\mathscr S$, we consider functions $f(P,\,\boldsymbol p)\in\mathcal H$ satisfying
\begin{equation}\label{def-f-extend}
	\begin{cases}
	f(P,\,\boldsymbol p)\in\mathbb C \  &\text{if }\ P\in\mathcal C,
		\\[1mm]
	f(P,\,\boldsymbol p)=0 \ &\text{if }\ P\notin\mathcal C,
	\end{cases}
\end{equation}
and we use the inner product on $\mathcal H$ as before.

For any $f$  as above, we define 
$(\mathcal T_{-2}^{\mathcal C}f)(P,\,\boldsymbol p)$ on $\mathscr S$ as follows: set $(\mathcal T_{-2}^{\mathcal C}f)(P,\,\boldsymbol p)=0$ if $P\notin\mathcal C$; and when $P\in\mathcal C$, we define
%With the measure $\sigma$ as before, 
%For any $f:\{(P,\,\boldsymbol p)\in\mathscr S:\ P\in\mathcal C\}\to\mathbb C$,  we define  
\begin{equation}\label{def-T^C_-2}
	\begin{aligned}
		(\mathcal T_{-2}^{\mathcal C}f)(P,\,\boldsymbol p) &:=\mathbf1_{d_0\le D<2d_0}\biggl(\prod_{i=1}^K\frac1{V_i}\biggr)\sum_{\substack{p_i'\in\mathcal P_i\setminus\{p_{i,1},\,\ldots,\,p_{i,J+1}\}\\1\le i\le K}}\ \sum_{\substack{Q\in\mathcal C\\(Q,\,\boldsymbol p')\in\mathscr S}}
		\eta\Big(\frac{a}{Y}\Big)\eta\Big(\frac{b}{Y}\Big)\psi_{-2}\Big(\frac{t}{Y}\Big)
		\\[1,1mm]
		&\times\frac{d_0}{D}\rho^{\frac{\omega(P_1)+\omega(Q_1)-2K(J+1)}{2}}\biggl\{\mathbf1_{t=2(a-b)}-\int_{\mathfrak M_{-2}}e\bigl(\theta(t-2(a-b))\bigr)\,\mathrm d\theta\biggr\}f(Q,\,\boldsymbol p').
	\end{aligned}
\end{equation}
where $t=\det(P,\,Q)/D$, the products $a,\,b$ are defined in \eqref{Dba-definition}, and $\psi_{-2}$, $\mathfrak M_{-2}$ are defined in \eqref{eq:major-arcs-kernel}. We note that this definition is analogous to that of $\mathcal T_1$ in \eqref{def-T1}. 
%We extend the function $f:\{(P,\,\boldsymbol p)\in\mathscr S:\ P\in\mathcal C\}\to\mathbb C$ by zero to $\mathscr S$ and use the inner product on $\mathcal H$ as before.

We also define $\mathcal G_{-2}$ by
\begin{equation}\label{def-G_-2}
	(\mathcal G_{-2}f)(P,\,\boldsymbol p)=\begin{cases}
		f(P,\,\boldsymbol p) &\text{if both \eqref{eq:goodone} and \eqref{eq:goodtwo} hold with }r_{_P}\text{ replaced by }-2r_{_P},\\
		0 &\text{otherwise}.
	\end{cases}
\end{equation}
 In particular, the first condition \eqref{eq:goodone} after replacement is
\begin{equation}\label{eq:good2}
	\|{-2\ell \widetilde{D}r_{_P}}\|>Y^{-0.7}\qquad(1\le\ell\le Y^{0.2}).
\end{equation}
%\begingroup\color{red}
We also define
\begin{equation}\label{eq:aone}
	\mathcal A_1:=\mathcal G_1\mathcal S\mathcal T_1\mathcal S\mathcal G_1,\quad
	\mathcal A_{-2}^{\mathcal C}:=\mathcal G_{-2}\mathcal S\mathcal T_{-2}^{\mathcal C}\mathcal S\mathcal G_{-2}.
\end{equation}
Note that $\mathcal A_1^*=\mathcal G_1\mathcal S\mathcal T_1^*\mathcal S\mathcal G_1$, since $\mathcal G_1^*=\mathcal G_1$ and $\mathcal S^*=\mathcal S$. 

\vskip 3mm
\subsection{Reduction to operator estimates}\label{sec:ends}

In this subsection, we first express $Q_Y-Q_Y^{\mathrm{maj}}$ as a sum of inner products involving $\mathcal T_{-2}^{\mathcal C}$. We then separate the error from imposing the two conditions defining $\mathcal G_{-2}$ and show how a moment bound for $\mathcal A_{-2}^{\mathcal C}$ controls the remaining inner products.

For any position $P=(P_1,\,P_2)^T$, let
\begin{equation*}
	G(P_1):=\prod_{\substack{p^\nu\parallel P_1,\ p\in\bigcup_{i=1}^{K}\mathcal P_i}}p^\nu,
\end{equation*}
Then we define the function $f:\mathscr S\to\mathbb C$, depending only on the vector $P$, by
\begin{equation}\label{def-f}
	f(P,\,\boldsymbol p)=f(P):=
	\begin{cases}
		\overline{\alpha_{P_1/G(P_1)}}\,\beta_{P_2} & \text{if }\ \mu^2\big(G(P_1)\big)=1,\ \ \omega_i(P_1)=J+1\ \ (1\le i\le K),\\[1mm]
		0 & \text{otherwise},
	\end{cases}
\end{equation}
where $\alpha,\, \beta$ defined by \eqref{eq:011}, are zero outside their prescribed supports.

Recall that $P=\bigl(Dbm,\, s \bigr)^T$, $Q=\bigl(Dar,\, n \bigr)^T$  as in \eqref{P-Q-t}. For terms satisfying $\alpha_m\overline{\alpha_r}\beta_n\overline{\beta_s}\ne0$, we have	$P^-(mnrs)>Z.$
In addition, by \eqref{Dba-definition}, the products $Db$ and $Da$ are squarefree and $p\mid Dab\Longrightarrow p\in \bigcup_{i=1}^{K}\mathcal P_i\subset(2,\,Z)$, hence 
\begin{equation*}
	G(P_1)=G(Dbm)=Db,\qquad G(Q_1)=G(Dar)=Da.
\end{equation*}
The definition of $f$ therefore gives
\begin{equation*}
	\overline{f(P)}f(Q)=\alpha_m\overline{\alpha_r}\beta_n\overline{\beta_s}.
\end{equation*}
Moreover, $\omega_i(P_1)=\omega_i(Q_1)=J+1$ for every $1\le i\le K$, so we have $\rho^{\frac{\omega(P_1)+\omega(Q_1)-2K(J+1)}2}=1.$
%\begin{equation*}
%	\rho^{\frac{\omega(P_1)+\omega(Q_1)-2K(J+1)}2}=1.
%\end{equation*}
%Since $D,\,a,\,b$ are odd and the coefficient support conditions imply that $m,\,n,\,r,\,s$ are odd, both coordinates of $P$ and $Q$ are odd.

We next verify that $P$ and $Q$ are primitive. Since $\gcd(a,\,b)=1$ and $P^-(rs)>Z$, one has $\gcd(b,\,ars)=1$, by which we deduce that $t=bmn-ars\ne0$. If there exists a prime $p$ such that $p\mid Dbm$, $p\mid s$, then $p>Z$ and hence $p\nmid Db$, $p\mid m$. This shows that
\begin{equation*}
	 p\mid bmn-ars=t,
\end{equation*}
which contradicts $0<|t|\ll Y=o(Z)$, since $Y\le a=\prod_{i=1}^{K}p_i\le\exp(2KL^{0.2})$, thereby $\gcd(Dbm,\, s)=1$. The same argument gives $\gcd(Dar,\,n)=1$. For the terms with $\gcd(a,\,b)=1$ in $Q_Y$ and $Q_Y^{\mathrm{maj}}$, we may therefore impose $\gcd(P_1,\,P_2)=\gcd(Q_1,\,Q_2)=1$, as required in $\mathscr S$, without changing either sum.

We now estimate $\sum_{(P,\,\boldsymbol p)\in\mathscr S}\sigma(P,\,\boldsymbol p)$ to bound $\|f\|$. For fixed prime tuples $\boldsymbol p_i=(p_{i,1},\,\ldots,\,p_{i,J+1})$, $1\le i\le K$, the definition of $\mathscr S$ gives $\prod_{i=1}^{K}\prod_{j=1}^{J+1}p_{i,j}\mid P_1.$
By \eqref{eq:physbox}, we therefore have
\begin{equation*}
	\#\{P:\ (P,\,\boldsymbol p)\in\mathscr S\}\le(N+1)\biggl\lfloor\frac{16d_0YM}{\prod_{i=1}^{K}\prod_{j=1}^{J+1}p_{i,j}}\biggr\rfloor\ll d_0YMN\prod_{i=1}^{K}\prod_{j=1}^{J+1}\frac1{p_{i,j}}.
\end{equation*}
We then sum over the prime tuples. Using $\sigma(P,\,\boldsymbol p)=\prod_{i=1}^{K}V_i^{-(J+1)}$, and then removing the restrictions that the primes within each tuple are distinct, we obtain
\begin{align*}
	\sum_{(P,\,\boldsymbol p)\in\mathscr S}\sigma(P,\,\boldsymbol p)
	&=\biggl(\prod_{i=1}^{K}V_i^{-(J+1)}\biggr)\sum_{\substack{p_{i,j}\in\mathcal P_i\ (1\le i\le K,\,1\le j\le J+1)\\p_{i,j}\ne p_{i,k}\ (1\le i\le K,\,1\le j<k\le J+1)}}\#\{P:\ (P,\,\boldsymbol p)\in\mathscr S\}\\
	&\ll d_0YMN\prod_{i=1}^{K}\biggl\{\frac1{V_i^{J+1}}\sum_{\substack{p_{i,1},\,\ldots,\,p_{i,J+1}\in\mathcal P_i\\p_{i,j}\ne p_{i,k}\ (1\le j<k\le J+1)}}\prod_{j=1}^{J+1}\frac1{p_{i,j}}\biggr\}\\
	&\le d_0YMN\prod_{i=1}^{K}\biggl\{\frac1{V_i^{J+1}}\biggl(\sum_{p\in\mathcal P_i}\frac1p\biggr)^{J+1}\biggr\}=d_0YMN.
\end{align*}
Under the assumption in Proposition \ref{prop:D}, we have the bounds $|\alpha_m|,\,|\beta_n|\le L^C$, which implies that $\sup|f|\le L^{2C}$ and
\begin{equation}\label{eq:endnorm}
%\sup|f|\le L^{2C}\quad\text{and}\quad  
\|f\|\le L^{2C}\biggl(\sum_{(P,\,\boldsymbol p)\in\mathscr S}\sigma(P,\,\boldsymbol p)\biggr)^{1/2}\ll(d_0YMN)^{1/2}L^{2C}.
\end{equation}

For each interval $[d_0,\,2d_0)$ containing at least one value of $D=\prod_{i=1}^{K}\prod_{j=1}^{J}p_{i,j}$, we define
\begin{equation*}%\label{def-B_d_0}
\mathcal B_{d_0}:=\langle f,\,\mathcal T_{-2}^{\mathcal C}f\rangle
=\sum_{\substack{(P,\,\boldsymbol p)\in\mathscr S\\[0.4mm] P\in\mathcal C}}\overline{f(P,\,\boldsymbol p)}(\mathcal T_{-2}^{\mathcal C}f)(P,\,\boldsymbol p)\sigma(P,\,\boldsymbol p).
\end{equation*}
We then expand $\mathcal B_{d_0}/d_0$ to express it in terms of the variables $a,\,b,\,m,\,n,\,r,\,s$ as in $Q_Y$ in \eqref{eq:013}, by substituting $\mathcal T_{-2}^{\mathcal C}f$ defined by \eqref{def-T^C_-2} into the above inner product. With $a,\,b$ defined by \eqref{Dba-definition} and $P=(Dbm,\,s)^T$, $Q=(Dar,\,n)^T$, we obtain
\begin{equation}\label{B_d_0-expand}
	\begin{aligned}
	\frac{\mathcal B_{d_0}}{d_0}
	&=\biggl(\prod_{i=1}^{K}\frac{1}{V_i^2}\biggr)
	\sum_{\substack{p_{i,j},\,p_i'\in\mathcal P_i\ (1\le i\le K,\,1\le j\le J+1) \\[0,3mm]
	p_{i,j}\ne p_{i,k}\ (1\le i\le K,\,1\le j<k\le J+1) \\[0,3mm]
	p_i'\notin\{p_{i,1},\,\ldots,\,p_{i,J+1}\}\ (1\le i\le K) \\[0,3mm] d_0\le D<2d_0}}
	\biggl(\prod_{i=1}^{K}\prod_{j=1}^{J}\frac1{p_{i,j}V_i}\biggr)
	\\[1,3mm]
	&\quad\times\sum_{m,\,n,\,r,\,s}\alpha_m\overline{\alpha_r}\beta_n\overline{\beta_s}
	\Bigl\{\mathbf1_{bmn-ars=2(a-b)}-H_{-2}(bmn-ars;\,a,\,b)\Bigr\} \eta\Bigl(\frac aY\Bigr)\eta\Bigl(\frac bY\Bigr),
	\end{aligned}
\end{equation}
where the coefficient follows from
\begin{equation}\label{eq:padmass}
	\frac{\sigma(P,\,\boldsymbol p)}{d_0}\frac{d_0}{D}\prod_{i=1}^{K}V_i^{-1}
	=\frac1D\prod_{i=1}^{K}V_i^{-(J+2)}
	=\biggl(\prod_{i=1}^{K}\prod_{j=1}^{J}\frac1{p_{i,j}V_i}\biggr)\prod_{i=1}^{K}V_i^{-2}.
\end{equation}
We first sum $\mathcal B_{d_0}/d_0$ over the dyadic values $d_0=2^{\ell_0}$. Note that $\sum_{\ell_0\in\mathbb Z}\mathbf1_{2^{\ell_0}\le D<2^{\ell_0+1}}=1$ for every fixed $D$.
Thus, after summing over $d_0$, we no longer have the restriction $d_0\le D<2d_0$.

For fixed $a=\prod_{i=1}^{K}p_i'$ and $b=\prod_{i=1}^{K}p_{i,J+1}$, the sum over $m,\,n,\,r,\,s$ in \eqref{B_d_0-expand}, say $\mathcal D(a,\,b)$, does not involve the primes $p_{i,j}$, $1\le i\le K$, $1\le j\le J$. We may therefore take it outside the sums over $p_{i,1},\,\ldots,\,p_{i,J}$. Define the product
\begin{equation*}
	w(a,\,b):=\prod_{i=1}^{K}\Biggl\{\sum_{\substack{p_{i,1},\,\ldots,\,p_{i,J}\in\mathcal P_i\setminus \{p_i',\,p_{i,J+1}\} \\[0,4mm]p_{i,j}\ne p_{i,k}\ (1\le j<k\le J)}}\, \prod_{j=1}^{J}\frac1{p_{i,j}V_i}\Biggr\}.
\end{equation*}
Summing \eqref{B_d_0-expand} over $d_0$ and collecting the terms with the same $a$ and $b$, we obtain
\begin{equation}\label{eq:allpads}
	\sum_{d_0}\frac{\mathcal B_{d_0}}{d_0}= \biggl(\prod_{i=1}^{K}\frac{1}{V_i^2}\biggr)\sum_{\substack{a,\,b\\\gcd(a,\,b)=1}}w(a,\,b)\mathcal D(a,\,b).
\end{equation}

Our aim is to prove, for any $A>0$,
\begin{equation}\label{aim-Q_Y-Q^maj-1}
	Q_Y-Q_Y^{\mathrm{maj}}=\sum_{d_0}\frac{\mathcal B_{d_0}}{d_0}+O_{A,\,K,\,\rho}(XYL^{-A}).
\end{equation}
Note that by Lemma \ref{lem:shared}, we may omit the terms with $\gcd(a,\,b)>1$ from $Q_Y$ with a negligible error term $O(XYL^{-A})$. The same conclusion holds for $Q_Y^{\mathrm{maj}}$, since there will be a factor $1/p^2$ with $p\geq \exp(L^{1/10}),\, p\mid a,\, p\mid b$ arising from the condition $\gcd(a,\,b)>1$, which is enough to control the contribution of terms  with $\gcd(a,\,b)>1$ from $Q_Y^{\mathrm{maj}}$.
Therefore, by the definitions of $Q_Y$, $Q_Y^{\mathrm{maj}}$ in \eqref{eq:013}, \eqref{def-Q_Y^maj} and $\mathcal D(a,\,b)$, we have
\begin{equation*}
	\begin{aligned}
	Q_Y-Q_Y^{\mathrm{maj}}=\biggl(\prod_{i=1}^{K}\frac{1}{V_i^2} \biggr) \sum_{\substack{a,\,b\\\gcd(a,\,b)=1}}\mathcal D(a,\,b)
		+O_{A,\,K,\,\rho}(XYL^{-A}).
	\end{aligned}
\end{equation*}
Subtract this expression from \eqref{eq:allpads}, we obtain
\begin{equation*}
	\begin{aligned}
	Q_Y-Q_Y^{\mathrm{maj}}-\sum_{d_0}\frac{\mathcal B_{d_0}}{d_0}
	=\biggl(\prod_{i=1}^{K}\frac{1}{V_i^2}\biggr) \sum_{\substack{a,\,b\\\gcd(a,\,b)=1}} \Bigl(1-w(a,\,b)\Bigr)\mathcal D(a,\,b) +O_{A,\,K,\,\rho}(XYL^{-A}).
	\end{aligned}
\end{equation*}
Hence, to get \eqref{aim-Q_Y-Q^maj-1}, it remains to show
\begin{equation}\label{aim-Q_Y-Q^maj-2}
	\begin{aligned}
		\biggl(\prod_{i=1}^{K}\frac{1}{V_i^2}\biggr) \sum_{\substack{a,\,b\\\gcd(a,\,b)=1}} \Bigl(1-w(a,\,b)\Bigr)\mathcal D(a,\,b) \ll_{A,\,K,\,\rho}(XYL^{-A}).
	\end{aligned}
\end{equation}

We first treat the factor $1-w(a,\,b)$. By the definition of $V_i$, we have
\begin{equation*}
	w(a,\,b)\leq \prod_{i=1}^{K}\Biggl\{ \sum_{p_{i,1}\in\mathcal P_i} \cdots\sum_{p_{i,J}\in\mathcal P_i} \,\prod_{j=1}^{J}\frac1{p_{i,j}V_i} \Biggr\}
	=\prod_{i=1}^{K}\biggl(\sum_{p\in\mathcal P_i} \frac1{pV_i}\biggr)^J=1.
\end{equation*}
The prime tuples excluded from $w(a,\,b)$ satisfy, for some $i,\, j$, either $p_{i,j}=p_{i,k}$ for some $j<k$, or $p_{i,j}\in\{p_i',\,p_{i,J+1}\}$. We sum the nonnegative weights over each of these possibilities, yielding
\begin{equation*}%\label{eq:030}
	\begin{aligned}
		0\le1-w(a,\,b)&\le\sum_{i=1}^{K}\biggl\{\binom J2 \sum_{p\in\mathcal P_i}\frac1{(pV_i)^2} +\frac{J}{p_i'V_i}+\frac{J}{p_{i,J+1}V_i}\biggr\}
		\ll_K J^2\exp(-cL^{0.1}).
	\end{aligned}
\end{equation*}
This upper bound shows that to obtain \eqref{aim-Q_Y-Q^maj-2}, it suffices to prove
\begin{equation}\label{D(a,b)-estima}
	\biggl(\prod_{i=1}^{K}V_i^{-2}\biggr)\sum_{\substack{a,\,b\\\gcd(a,\,b)=1}}|\mathcal D(a,\,b)|\ll_{K,\,\rho}XYL^{O(1)}.
\end{equation}

We shall estimate the terms involving $\mathbf1_{bmn-ars=2(a-b)}$ and $H_{-2}(bmn-ars;\,a,\,b)$ in $\mathcal D(a,\,b)$ separately.
For the first one, by applying the Cauchy--Schwarz inequality and \eqref{eq:014}, we get
\begin{equation}\label{eq:031}
	\biggl(\prod_{i=1}^{K}V_i^{-2}\biggr)\sum_{\substack{a,\,b\\\gcd(a,\,b)=1}}\biggl|\eta\Bigl(\frac aY\Bigr)\eta\Bigl(\frac bY\Bigr)\biggr|\sum_{\substack{m,\,n,\,r,\,s\\bmn-ars=2(a-b)}}|\alpha_m\overline{\alpha_r}\beta_n\overline{\beta_s}|\ll_{K,\,\rho}XYL^{O(1)}.
\end{equation}

We next estimate the terms containing $H_{-2}(bmn-ars;\,a,\,b)$. By \eqref{eq:major-arcs-kernel} and the support of $\psi_1$, we have
$|H_{-2}(bmn-ars;\,a,\,b)|\ll L^{O(1)}Y^{-1} \mathbf1_{|bmn-ars|<16Y}.$ Recalling that $MN=X$ and $|\alpha|,\,|\beta|\ll L^C$ for some $C>0$, we derive that
\begin{align*}
\sum_{m,\,n,\,r,\,s}\Big|\alpha_m\overline{\alpha_r}\beta_n\overline{\beta_s}H_{-2}(bmn-ars;\,a,\,b)\Big| &\ll \frac{L^{O(1)}}{Y}\sum_{\substack{M<m,\,r\le2M\\N<n,\,s\le2N\\|bmn-ars|<16Y}}
\big|\alpha_m\overline{\alpha_r}\beta_n\overline{\beta_s}\big|
\\[1mm]
&\ll\frac{L^{O(1)}}{Y}\sum_{\substack{k,\,l\le4X\\|bk-al|<16Y}}\tau(k)\tau(l) \ll  \frac{XL^{O(1)}}{Y},
\end{align*}
where the last estimate follows from $\tau(k)\tau(l)\le\tau(k)^2+\tau(l)^2$ and the fact that there are at most $O(1)$ values of $l$ satisfying $|bk-al|<16Y$ for each fixed $k$.
Then summing over the $O(Y^2)$ pairs of integers $a,\,b$ with $\eta(a/Y)\eta(b/Y)\ne0$, we arrive at
\begin{equation}\label{eq:032}
\biggl(\prod_{i=1}^{K}\frac{1}{V_i^2}\biggr)
\sum_{\substack{a,\,b\\\gcd(a,\,b)=1}}
\biggl|\eta\Bigl(\frac aY\Bigr)\eta\Bigl(\frac bY\Bigr)\biggr|
\sum_{m,\,n,\,r,\,s}\Big|\alpha_m\overline{\alpha_r}\beta_n\overline{\beta_s}H_{-2}(bmn-ars;\,a,\,b)\Big|\ll XYL^{O(1)}.
\end{equation}

Finally, combining \eqref{eq:031} and \eqref{eq:032}, we obtain \eqref{D(a,b)-estima}, and thereby proving the desired expression \eqref{aim-Q_Y-Q^maj-1}. Recall that $\mathcal B_{d_0}=\langle f,\,\mathcal T_{-2}^{\mathcal C}f\rangle$. Applying the triangle inequality to \eqref{aim-Q_Y-Q^maj-1}, we obtain the upper bound
\begin{equation}\label{eq:operator-reduction}
\begin{aligned}
	\big|Q_Y-Q_Y^{\mathrm{maj}}\big|
	\le\sum_{d_0}\frac{\big|\langle f,\,\mathcal A_{-2}^{\mathcal C}f\rangle\big|}{d_0}
	+\sum_{d_0}\frac{\big|\langle f,\,\mathcal T_{-2}^{\mathcal C} f\rangle-\langle f,\,\mathcal A_{-2}^{\mathcal C}f\rangle \big|}{d_0}
	+O_{A,\,K,\,\rho}\Big(\frac{XY}{L^A}\Big).
\end{aligned}
\end{equation}
We shall estimate the second sum above in Subsection \ref{sec:bad}.

Helfgott and Radziwi{\l}{\l} \cite{HelRad21} studied averages over primes using a prime divisibility graph and obtained eigenvalue bounds from estimates for traces of even powers of the operator.
Pilatte \cite{Pil26} succeeded in developing this approach by averaging over products of primes instead of primes, a direction proposed in \cite{HelRad21}. 
Recently, these methods were adapted by OpenAI \cite{PD}. We now reduce the first sum to a moment estimate for $\mathcal A_{-2}^{\mathcal C}$ as in \cite{PD}.
To do this, for any fixed primitive vector $P\in\Omega\cap\mathbb Z^2$, we first define the function $u_{_P}\in\mathcal H$ by
\begin{equation}\label{def-u_P}
	u_{_P}(Q,\,\boldsymbol p)=\mathbf1_{Q=P}\qquad
	\big((Q,\,\boldsymbol p)\in\mathscr S\big).
\end{equation}
with $\mathcal H$, $\mathscr S$ defined by \eqref{def-H} and \eqref{mathscr-S}.
In addition, we note that we shall take $R=\lceil L^{1/2}/2\rceil$ throughout the remainder of Section \ref{sec-type-II}. 

\begin{proposition}\label{prop:moment-reduction}
For any $A>0$, let $E_0>A+4C+2$ be fixed constant, where $C>0$ is as in Proposition \ref{prop:D}. Suppose that
\begin{equation}\label{eq:moment-assumption}
\frac1{d_0YMN}\sum_{P\in\Omega\cap\mathcal C}\bigl\langle u_{_P},\, \bigl(\mathcal A_{-2}^{\mathcal C}(\mathcal A_{-2}^{\mathcal C})^*\bigr)^Ru_{_P}\bigr\rangle
\footnotetext{We also call inner products of the form $\langle f,\,(\mathcal A\mathcal A^*)^k f\rangle$ operator moments, for any integer $k\ge1$.}
\ll L^{-2RE_0}
\end{equation}
holds uniformly for all $d_0$ occurring in  \eqref{eq:operator-reduction},  then with $f$ defined in \eqref{def-f}, we have
\begin{equation}\label{eq:first-sum-bound}
	\sum_{d_0}\frac{|\langle f,\,\mathcal A_{-2}^{\mathcal C}f\rangle|}{d_0}\ll_{A,\,K,\,\rho}XYL^{-A}.
\end{equation}
\end{proposition}

\begin{proof}
We first estimate $\langle f,\,\mathcal A_{-2}^{\mathcal C}g\rangle$ for some function $g\in\mathcal H$ satisfying
\begin{equation}\label{g-condition}
	g(P,\,\boldsymbol p)=0\quad(P\notin\mathcal C),\qquad \|g\|\ll(d_0YMN)^{1/2}L^{2C},
\end{equation}
and then we take $g=f$, since $f$ satisfies the above two conditions by its definition and \eqref{eq:endnorm}. 
%To simplify notation, we write
%\begin{equation*}
%	\mathcal A=\mathcal A_{-2}^{\mathcal C},\qquad B_0=\mathcal A\mathcal A^*.
%\end{equation*}

Define the intervals $I_j=\bigl[\frac{j}{d_0YM^2},\,\frac{j+1}{d_0YM^2}\bigr)$ with $j\in\mathbb Z$.
For each fixed $j$, we define $g_j\in\mathcal H$ by
\begin{equation*}
	g_j(Q,\,\boldsymbol p')=
	\begin{cases}
		g(Q,\,\boldsymbol p') &
		\displaystyle\frac{Q_2}{Q_1}\in\bigcup_{|k-j|\leq 32} I_k,\\[1,3mm]
		0 &\text{otherwise}.
	\end{cases}
\end{equation*}
We claim that, for every $(P,\,\boldsymbol p)\in\mathscr S$ with $P\in\Omega\cap\mathcal C$ and $P_2/P_1\in I_j$,
\begin{equation}\label{claim}
	(\mathcal A_{-2}^{\mathcal C}g)(P,\,\boldsymbol p)
	=(\mathcal A_{-2}^{\mathcal C}g_j)(P,\,\boldsymbol p).
\end{equation}

Recall that $\mathcal A_{-2}^{\mathcal C}
=(\mathcal G_{-2}\mathcal S)\mathcal T_{-2}^{\mathcal C}
(\mathcal S\mathcal G_{-2}).$ We first treat the operator $\mathcal T_{-2}^{\mathcal C}$ by analyzing the sum over $Q$ in \eqref{def-T^C_-2}.
Note that $\psi_{-2}(u)=0$ for $|u|\ge16$ and $D<2d_0$, whence
 $\psi_{-2}\bigl(\det(P,\,Q)/(DY)\bigr)\neq 0$ implies that $|\det(P,\,Q)|<16DY<32d_0Y$.
Moreover, since $P,\,Q\in\Omega=[d_0YM,\,16d_0YM]\times[N,\,2N]$, we have $P_1,\,Q_1\ge d_0YM$, and hence
\begin{equation*}
	\biggl|\frac{P_2}{P_1}-\frac{Q_2}{Q_1}\biggr|
	=\frac{|\det(P,\,Q)|}{P_1Q_1}
	<\frac{32}{d_0YM^2}.
\end{equation*}
If $Q_2/Q_1\in I_k$, it follows by the triangle inequality that
\begin{equation*}
	|j-k|\leq \biggl|\Bigl(j-d_0YM^2\frac{P_2}{P_1}\Bigr)-\Bigl(k-d_0YM^2
	\frac{Q_2}{Q_1}\Bigr)\biggr| +d_0YM^2\biggl|\frac{P_2}{P_1}-\frac{Q_2}{Q_1}\biggr| <1+32,
\end{equation*}
whence $|j-k|\le32$.
This implies that $g(Q,\,\boldsymbol p')=g_j(Q,\,\boldsymbol p')$
for any $\boldsymbol p'$ with $(Q,\,\boldsymbol p')\in\mathscr S$. Consequently, by the definition of $\mathcal T_{-2}^{\mathcal C}$ in \eqref{def-T^C_-2}, we get
%(g-g_j)(Q,\,\boldsymbol p')=0.
\begin{equation}\label{T_2g}
	 (\mathcal T_{-2}^{\mathcal C}g)(P,\,\boldsymbol p)
	=(\mathcal T_{-2}^{\mathcal C}g_j)(P,\,\boldsymbol p).
\end{equation}
Next for  $\mathcal G_{-2}$ and $\mathcal S$, by their definitions in \eqref{def-G_-2} and  in \eqref{def-S}, we may also deduce that
\begin{equation}\label{SG_-2g}
\bigl(\mathcal S\mathcal G_{-2} g\bigr)(Q,\,\boldsymbol p') =\bigl(\mathcal S\mathcal G_{-2} g_j\bigr)(Q,\,\boldsymbol p')
\qquad\biggl(\frac{Q_2}{Q_1}\in\bigcup_{k=j-32}^{j+32}I_k\biggr).
\end{equation}
Note that the terms in \eqref{def-T^C_-2} with
$Q_2/Q_1\notin\bigcup_{k=j-32}^{j+32}I_k$
vanish because $\psi_{-2}\bigl(\det(P,\,Q)/(DY)\bigr)=0$.
Hence, for each fixed $P\in\Omega\cap\mathcal C$ with $P_2/P_1\in I_j$, \eqref{T_2g} and \eqref{SG_-2g} yield
\begin{equation}\label{T-S-G}
	\bigl(\mathcal T_{-2}^{\mathcal C}\mathcal S\mathcal G_{-2}g\bigr)(P,\,\boldsymbol p)
	=\bigl(\mathcal T_{-2}^{\mathcal C}\mathcal S\mathcal G_{-2}g_j\bigr)(P,\,\boldsymbol p)
\end{equation}
for every $\boldsymbol p$ with $(P,\,\boldsymbol p)\in\mathscr S$.
Finally, with the help of the definitions \eqref{def-S} and \eqref{def-G_-2} again, we apply $\mathcal S$ and then $\mathcal G_{-2}$ to both sides of \eqref{T-S-G}, yielding the claim:
\begin{equation*}
	(\mathcal G_{-2}\mathcal S\mathcal T_{-2}^{\mathcal C}
	\mathcal S\mathcal G_{-2} g)(P,\,\boldsymbol p)
	=(\mathcal G_{-2}\mathcal S\mathcal T_{-2}^{\mathcal C}
	\mathcal S\mathcal G_{-2} g_j)(P,\,\boldsymbol p).
\end{equation*}

We next show that, for each $j\in\mathbb Z$, there are $O(d_0Y)$
primitive vectors $P=(P_1,\,P_2)^T\in\Omega\cap\mathbb Z^2$
satisfying $P_2/P_1\in I_j$.
If no such vector exists, the assertion is trivial.
Otherwise, fix one of these primitive vectors, say
$P^{(0)}=(P^{(0)}_1,\,P^{(0)}_2)^T$.
Then for any primitive  $P\in\Omega\cap\mathbb Z^2$
with $P_2/P_1\in I_j$, we have $|\det(P^{(0)},\,P)|\leq 256d_0Y$,
whence $\det(P^{(0)},\,P)\in\mathbb{Z}$ has $O(d_0Y)$ possible values.
Thus, for each fixed integer $|h|\leq 256d_0Y$, it remains to count the vectors $P=(P_1,\,P_2)$ satisfying
\begin{equation*}
	\det(P^{(0)},\,P)=P^{(0)}_1P_2-P^{(0)}_2P_1=h.
\end{equation*}
By calculation, one finds that for each $h$, there are at most $16$ solutions $(P_1,\,P_2)$ to the above equation. Then, summing over the $O(d_0Y)$ possible values of $h$, we prove the assertion.

For each $j\in\mathbb Z$, we define $f_j(P,\,\boldsymbol p):= f(P,\,\boldsymbol p)\mathbf1_{P_2/P_1\in I_j}\in \mathcal H.$
Note that $u_{_P}(Q,\,\boldsymbol p)=\mathbf1_{Q=P}$ and  we write $f(P,\,\boldsymbol p)=f(P)$ in \eqref{def-f}. Hence, for every $(Q,\,\boldsymbol p)\in\mathscr S$, we have
\begin{equation*}
	f_j(Q,\,\boldsymbol p)=
   \sum_{\substack{P\in\Omega\cap\mathcal C\\[0,4mm]P_2/P_1\in I_j}}
	f(P)u_{_P}(Q,\,\boldsymbol p).
\end{equation*}
The estimates below are immediate when $f_j=0$, so we may assume
$f_j\ne0$.

We consider the finite matrix
\begin{equation*}
\mathbf M_j:=
\biggl(\Bigl\langle u_{_P},\,\bigl(\mathcal A_{-2}^{\mathcal C} (\mathcal A_{-2}^{\mathcal C})^*\bigr)^Ru_{_Q}\Bigr\rangle
\biggr)_{\substack{P,\,Q \in \Omega\cap\mathcal C \\P_2/P_1,\,Q_2/Q_1\in I_j}},
\end{equation*}
where $P$ indexes the rows and $Q$ indexes the columns,  and both indices range over the same set.
Note that $\bigl\langle h,\,\mathcal A_{-2}^{\mathcal C} (\mathcal A_{-2}^{\mathcal C})^*h \bigr\rangle=\bigl\|(\mathcal A_{-2}^{\mathcal C})^*h\bigr\|^2\ge0$, thus the operator
$\mathcal A_{-2}^{\mathcal C}(\mathcal A_{-2}^{\mathcal C})^*$
is self-adjoint and positive semidefinite, and so is
$\bigl(\mathcal A_{-2}^{\mathcal C}
(\mathcal A_{-2}^{\mathcal C})^*\bigr)^R$. 

For each $P\in\Omega\cap\mathcal C$ with $P_2/P_1\in I_j$,
and for any complex coefficient $\xi_{_P}$, we have
\begin{equation*}
	\begin{aligned}
	\sum_{\substack{P,Q\in\Omega\cap\mathcal C\\[0,4mm] P_2/P_1,\,Q_2/Q_1\in I_j}}
	\overline{\xi_{_P}}\cdot(\mathbf M_j)_{P,Q}\cdot\xi_{_Q}
	=\biggl\langle \sum_{\substack{P\in\Omega\cap\mathcal C \\[0,4mm]  P_2/P_1\in I_j}}\xi_{_P}u_{_P},\
	\bigl(\mathcal A_{-2}^{\mathcal C}
	(\mathcal A_{-2}^{\mathcal C})^*\bigr)^R
	\sum_{\substack{Q\in\Omega\cap\mathcal C\\[0,4mm] Q_2/Q_1\in I_j}}
	\xi_{_Q}u_{_Q}\biggr\rangle \ge0,
	\end{aligned}
\end{equation*}
which implies that the matrix $\mathbf M_j$ is positive semidefinite.
Let $d_j=$Tr$(\mathbf M_j)$ be the trace of $\mathbf M_j$, and denote its largest eigenvalue by
$\lambda_{\max}(\mathbf M_j)$. Since all eigenvalues of
$\mathbf M_j$ are nonnegative, we obtain
\begin{equation}\label{def-d_j}
	\lambda_{\max}(\mathbf M_j)\le d_j
	=\sum_{\substack{P\in\Omega\cap\mathcal C\\[0,4mm]P_2/P_1\in I_j}}
	\bigl\langle u_{_P},\,
	\bigl(\mathcal A_{-2}^{\mathcal C}
	(\mathcal A_{-2}^{\mathcal C})^*\bigr)^Ru_{_P}\bigr\rangle.
\end{equation}
Note that the matrix $\lambda_{\max}(\mathbf M_j)\mathbf I-\mathbf M_j$ is also positive semidefinite. By this fact, together with the expansion of $f_j$ above, we obtain
\begin{equation}\label{eq:psd}
	\begin{aligned}
		\Bigl\langle f_j,\, \bigl(\mathcal A_{-2}^{\mathcal C} (\mathcal A_{-2}^{\mathcal C})^*\bigr)^R f_j\Bigr\rangle
		&=\sum_{\substack{P,Q\in\Omega\cap\mathcal C\\
				P_2/P_1,\,Q_2/Q_1\in I_j}}
		\overline{f(P)}f(Q)(\mathbf M_j)_{P,Q}\\
		&\le\lambda_{\max}(\mathbf M_j)
		\sum_{\substack{P\in\Omega\cap\mathcal C\\P_2/P_1\in I_j}}|f(P)|^2	\ll d_j d_0Y(\sup|f|)^2.
	\end{aligned}
\end{equation}
For the last estimate, we have used the bound $\#\big\{P\in\Omega\cap\mathcal C:\, P_2/P_1\in I_j\big\}\ll d_0Y$ as above.

For each given $(P,\,\boldsymbol p)\in\mathscr S$, denote by $k$
the unique integer satisfying $P_2/P_1\in I_k$.
By the definitions of $f_j$ and $g_j$, we have
$f_k(P,\,\boldsymbol p)=f(P,\,\boldsymbol p)$ and
$f_j(P,\,\boldsymbol p)=0$ for $j\ne k$;
and $g_j(P,\,\boldsymbol p)=g(P,\,\boldsymbol p)$
if $|j-k|\le32$ and $g_j(P,\,\boldsymbol p)=0$ otherwise.
Therefore,
\begin{equation*}
	\sum\nolimits_j|f_j(P,\,\boldsymbol p)|^2
	=|f(P,\,\boldsymbol p)|^2,\qquad
	\sum\nolimits_j|g_j(P,\,\boldsymbol p)|^2
	\le65|g(P,\,\boldsymbol p)|^2.
\end{equation*}
Multiplying the above two relations by $\sigma(P,\,\boldsymbol p)$
and summing over $(P,\,\boldsymbol p)\in\mathscr S$, we get
\begin{equation*}
	\sum\nolimits_j\|f_j\|^2=\|f\|^2,\qquad
	\sum\nolimits_j\|g_j\|^2\le65\|g\|^2.
\end{equation*}
Note that $f=\sum_j f_j$ and
$(\mathcal A_{-2}^{\mathcal C}g)(P,\,\boldsymbol p)
=(\mathcal A_{-2}^{\mathcal C}g_j)(P,\,\boldsymbol p)$
whenever $P_2/P_1\in I_j$, hence we  have

\begin{equation}\label{eq:027}
	\begin{aligned}
	\bigl|\langle f,\,\mathcal A_{-2}^{\mathcal C}g\rangle\bigr|^2
	&=\Bigl|\sum\nolimits_j\big\langle f_j,\,\mathcal A_{-2}^{\mathcal C}g_j\big\rangle\Bigr|^2
	=\Bigl|\sum\nolimits_j \bigl\langle(\mathcal A_{-2}^{\mathcal C})^*f_j,\,g_j\bigr\rangle\Bigr|^2
	\\[1mm]
		&\le\biggl(\sum\nolimits_j\|g_j\|^2\biggr)
		\biggl(\sum\nolimits_j\bigl\|(\mathcal A_{-2}^{\mathcal C})^*f_j\bigr\|^2\biggr)
		\le65\|g\|^2\sum\nolimits_j\bigl\|(\mathcal A_{-2}^{\mathcal C})^*f_j\bigr\|^2.
	\end{aligned}
\end{equation}

By the spectral theorem, one may choose an
orthonormal basis of eigenfunctions $e_\nu$ of
$\mathcal A_{-2}^{\mathcal C}(\mathcal A_{-2}^{\mathcal C})^*$
with respect to the inner product \eqref{def-innerprod}, that is,
$\|e_\nu\|=1$ and $\langle e_\nu,\,e_{\nu'}\rangle=0$
for $\nu\ne\nu'$.
Therefore, for any $(Q,\,\boldsymbol p)\in\mathscr S$, one may express the function $f_j$ by
\begin{equation*}
	f_j(Q,\,\boldsymbol p)
	=\sum\nolimits_\nu c_\nu\cdot e_\nu(Q,\,\boldsymbol p),
	\qquad c_\nu=\langle e_\nu,\,f_j\rangle,
\end{equation*}
where $e_\nu(Q,\,\boldsymbol p)=0$ if $Q\notin\mathcal C$,
as in \eqref{def-f-extend}. Let $\lambda_\nu$ be the corresponding non-negative eigenvalues, i.e. $\mathcal A_{-2}^{\mathcal C}
(\mathcal A_{-2}^{\mathcal C})^*e_\nu=\lambda_\nu e_\nu.$
Hence, by H\"older's inequality with exponents $R>1$ and $R/(R-1)$, 
the sum $\sum_j\|(\mathcal A_{-2}^{\mathcal C})^*f_j\|^2$ in \eqref{eq:027} can be  estimated as follows:
\begin{equation}\label{eq:sphold}
	\begin{aligned}
		\sum_j\bigl\|(\mathcal A_{-2}^{\mathcal C})^*f_j\bigr\|^2
		&=\sum_j\Bigl\langle f_j,\,\mathcal A_{-2}^{\mathcal C}
		(\mathcal A_{-2}^{\mathcal C})^*f_j\Bigr\rangle
		=\sum_{j,\,\nu}\lambda_\nu|\langle e_\nu,\,f_j\rangle|^2
		\\[1mm]
		&\le\biggl(\sum_{j,\,\nu}|\langle e_\nu,\,f_j\rangle|^2\biggr)^{1-1/R}
		\biggl(\sum_{j,\,\nu}\lambda_\nu^R
		|\langle e_\nu,\,f_j\rangle|^2\biggr)^{1/R}
		\\[1mm]
		&=\|f\|^{2-2/R}
		\biggl(\sum_j\Bigl\langle f_j,\,
		\bigl(\mathcal A_{-2}^{\mathcal C}
		(\mathcal A_{-2}^{\mathcal C})^*\bigr)^R f_j
		\Bigr\rangle\biggr)^{1/R}
		\\[1mm]
		&\ll\|f\|^{2-2/R}
		\biggl(d_0Y\big(\sup|f|\big)^2\sum_j d_j\biggr)^{1/R},
	\end{aligned}
\end{equation}
where the last estimate follows by summing \eqref{eq:psd} over $j$.
Recalling the definition of $d_j$ in \eqref{def-d_j} and the assumption \eqref{eq:moment-assumption}, one finds easily that
\begin{equation}\label{sum-d_j}
\sum_jd_j=\sum_{P\in\Omega\cap\mathcal C}\bigl\langle u_{_P},\,\bigl(\mathcal A_{-2}^{\mathcal C}(\mathcal A_{-2}^{\mathcal C})^*\bigr)^Ru_{_P}\bigr\rangle
\ll d_0YMN L^{-2RE_0}
\end{equation}

Collecting \eqref{eq:027}, \eqref{eq:sphold} and \eqref{sum-d_j}, we conclude that
\begin{equation}\label{eq:pair}
	\begin{aligned}
		\big|\langle f,\,\mathcal A_{-2}^{\mathcal C} g\rangle\big| 
		\ll d_0YMN L^{-E_0+4C}(d_0Y)^{1/(2R)}
		\ll_K d_0YMN L^{-E_0+4C},
	\end{aligned}
\end{equation}
since $\log(d_0Y)=O_K(L^{0.21})$, $2R\asymp L^{1/2}$, and $\|g\|, \|f\|\ll(d_0YMN)^{1/2}L^{2C}$ by \eqref{g-condition}.

Finally, note that $MN=X$ and there are $O_K(L)$ intervals $[d_0,\,2d_0)$.  Thus, we obtain from \eqref{eq:pair} with $g=f$ that
\begin{equation*}
	\sum_{d_0}\frac{|\langle f,\,\mathcal A_{-2}^{\mathcal C} f\rangle|}{d_0}\ll_K XYL^{-E_0+4C+1}\ll_{A,\,K,\,\rho}XYL^{-A}
\end{equation*}
for sufficiently large $A>0$, which completes the proof of Proposition \ref{prop:moment-reduction}.

\end{proof}

Now we are left to establish the assumption \eqref{eq:moment-assumption}, and estimate the second sum in \eqref{eq:operator-reduction}.

\vskip 3mm

\subsection{Transformation of operator inner products}

To prove the assumption \eqref{eq:moment-assumption} for
$\mathcal A_{-2}^{\mathcal C}$ in Subsection \ref{sec:ends}, we first compare the indicator
$\mathbf1_{t=2(a-b)}$ in \eqref{def-T^C_-2} defining $\mathcal T_{-2}^{\mathcal C}$ with the indicator
$\mathbf1_{t=b-a}$ in the kernel $K_1$ of $\mathcal T_1$, where  $t=\det(P,\,Q)/D$.
We introduce a linear transformation $T$ such that
\begin{equation}\label{transform-T}
	\frac{\det(P,\,Q)}D=2(a-b)
	\quad\Longleftrightarrow\quad
	\frac{\det(TP,\,TQ)}D=b-a.
\end{equation}

We denote the operator obtained after the transformation $T$ by $\mathcal A_1^{\Omega',\,\mathrm{odd}}$ defined below, and we shall estimate its moments in Subsections \ref{sec:roots}--\ref{sec:memory}. In the following Proposition \ref{prop:parity-conjugacy}, we first prove that the corresponding moments of $\mathcal A_{-2}^{\mathcal C}$ and $\mathcal A_1^{\Omega',\,\mathrm{odd}}$ are equal. The estimates for $\mathcal A_1^{\Omega',\,\mathrm{odd}}$ therefore give the required bound for $\mathcal A_{-2}^{\mathcal C}$.
%\endgroup

We take the linear transformation $T=\big(\begin{smallmatrix}1&0\\ 1/2&-1/2\end{smallmatrix}\big)$, whence $\det T=-1/2$ and $T^{-1}=\big(\begin{smallmatrix}1&0\\ 1&-2\end{smallmatrix}\big)$. Then for any vectors $P=(P_1,\,P_2)^T$ and $Q=(Q_1,\,Q_2)^T$, 
\begin{equation*}
TP=\begin{pmatrix}P_1\\[0,7mm]\frac{P_1-P_2}{2}\end{pmatrix},\qquad TQ=\begin{pmatrix}Q_1\\[0,7mm]\frac{Q_1-Q_2}{2}\end{pmatrix},
\end{equation*}
which shows that for $(P,\,\boldsymbol p),\,(Q,\,\boldsymbol p')\in\mathscr S$, the prime tuples $\boldsymbol p$ and $\boldsymbol p'$ also satisfy the divisibility conditions at $TP$ and $TQ$.
Moreover, if $P,\,Q\in\mathcal C$, the vectors $TP$ and $TQ$ have integer coordinates, and $\det(TP,\,TQ)/D=-t/2$, whence \eqref{transform-T} holds.

By the definition of $\mathscr S$, the position vectors $P$ and $Q$ lie in $\Omega=[d_0YM,\,16d_0YM]\times[N,\,2N]$. Under the transformation $T$, this restriction becomes $TP,\,TQ\in\Omega'$, where
\begin{equation}\label{eq:021}
	\Omega'=T\Omega=\Big\{(u,\,v)\in\mathbb R^2:\ d_0YM\le u\le16d_0YM,\ N\le u-2v\le2N\Big\}.
\end{equation}
Then, for functions $f:\bigl\{(TP,\,\boldsymbol p):\ (P,\,\boldsymbol p)\in\mathscr S,\,P\in\mathcal C\bigr\}\to\mathbb C$, we define the following operators:
\begin{itemize}
	\item Define $\mathcal G_1^{\Omega',\,\mathrm{odd}}$ and $\mathcal S^{\Omega',\,\mathrm{odd}}$ by the expressions for $\mathcal G_1$ and $\mathcal S$, respectively, with $P$ replaced by $TP$;
	\\[-3,5mm]
	\item Define $\mathcal T_1^{\Omega',\,\mathrm{odd}}$ by expression \eqref{def-T1} for $\mathcal T_1$, with $P,\,Q$ replaced by $TP,\,TQ$, respectively, and with $P,\,Q\in\Omega\cap\mathcal C$.
\end{itemize}
In particular, the conditions defining $\mathcal G_1^{\Omega',\,\mathrm{odd}}$ involve $r_{_{TP}}$, and the kernel in $\mathcal T_1^{\Omega',\,\mathrm{odd}}$ is $K_1(TP,\,TQ;\,D,\,a,\,b)$, with $t=\det(TP,\,TQ)/D$. Then following \eqref{eq:aone}, we define
\begin{equation}\label{def-A_1^Omega'-odd}
	\mathcal A_1^{\Omega',\,\mathrm{odd}}:=\mathcal G_1^{\Omega',\,\mathrm{odd}}\mathcal S^{\Omega',\,\mathrm{odd}}\mathcal T_1^{\Omega',\,\mathrm{odd}}\mathcal S^{\Omega',\,\mathrm{odd}}\mathcal G_1^{\Omega',\,\mathrm{odd}}.
\end{equation}
%Now we define $\mathcal A_1^{\Omega',\,\mathrm{odd}}$ on functions $f:\bigl\{(TP,\,\boldsymbol p):\ (P,\,\boldsymbol p)\in\mathscr S,\,P\in\mathcal C\bigr\}\longrightarrow\mathbb C$, by the construction $\mathcal A_1=\mathcal G_1\mathcal S\mathcal T_1\mathcal S\mathcal G_1$ in \eqref{eq:aone}  with the following substitutions for $P,\,Q\in\Omega\cap\mathcal C$:
%\begin{itemize}
%	\item For $\mathcal G_1$ and $\mathcal S$, we replace $P$ by $TP$ in their expressions;
%	\\[-3,7mm]
%	\item For $\mathcal T_1$, we replace $P,\,Q$ by $TP,\,TQ$ respectively.
%\end{itemize}
%The transformed vectors $TP,\,TQ$ lie in $\Omega'=T\Omega$ and have odd first coordinates. In particular, we use $r_{TP}$ in the conditions defining $\mathcal G_1$ and $\det(TP,\,TQ)/D$ in place of $t$ in $K_1(TP,\,TQ;\,D,\,a,\,b)$.

To compare $\mathcal A_{-2}^{\mathcal C}$ and $\mathcal A_1^{\Omega',\,\mathrm{odd}}$, 
we define $\mathcal U_Tf$ for any function $f:\{(P,\,\boldsymbol p)\in\mathscr S:\ P\in\mathcal C\}\to\mathbb C$ by
\begin{equation}\label{eq:utdef}
	(\mathcal U_Tf)(TP,\,\boldsymbol p)=f(P,\,\boldsymbol p).
\end{equation}
Thus $\mathcal U_Tf$ is defined on the set $\{(TP,\,\boldsymbol p):\ (P,\,\boldsymbol p)\in\mathscr S,\,P\in\mathcal C\}$. Recall the definition of $u_{_P}$ in \eqref{def-u_P}.
 For any $P\in\Omega\cap\mathcal C$, we define $\widetilde u_{_{TP}}=\mathcal U_T(u_{_P})$, whence for any $(Q,\,\boldsymbol p)\in\mathscr S,\,Q\in\mathcal C$,
\begin{equation}\label{def-tilde-u_TP}
	\widetilde u_{_{TP}}(TQ,\,\boldsymbol p)=u_{_P}(Q,\,\boldsymbol p)=\mathbf1_{Q=P}=\mathbf1_{TQ=TP}.
\end{equation}

\begin{proposition}%[Parity conjugacy]
	\label{prop:parity-conjugacy}
	The operator $\mathcal U_T$ defined by \eqref{eq:utdef} is unitary:
	\begin{equation*}
		\langle\mathcal U_Tf,\,\mathcal U_Tg\rangle=\langle f,\,g\rangle.
	\end{equation*}
	Moreover, we have
	\begin{equation}\label{eq:conj}
		\mathcal U_T\mathcal A_{-2}^{\mathcal C}\mathcal U_T^{-1}=\mathcal U_T\mathcal G_{-2}\mathcal S\mathcal T_{-2}^{\mathcal C}\mathcal S\mathcal G_{-2}\mathcal U_T^{-1}=\mathcal A_1^{\Omega',\,\mathrm{odd}},
	\end{equation}
and, for any integer $R\ge1$ and any $P\in\Omega\cap\mathcal C$, 
	\begin{equation}\label{oper-moment-iden}
	\big\langle u_{_P},\,(\mathcal A_{-2}^{\mathcal C}(\mathcal A_{-2}^{\mathcal C})^*)^Ru_{_P}\big\rangle=\big\langle \widetilde u_{_{TP}},\,(\mathcal A_1^{\Omega',\,\mathrm{odd}}
	\big(\mathcal A_1^{\Omega',\,\mathrm{odd}})^*\big)^R
	\widetilde u_{_{TP}}\big\rangle.
	\end{equation}
\end{proposition}

\begin{proof}
	First we recall the definition of inner product by \eqref{def-innerprod}. Obviously, the map $(P,\,\boldsymbol p)\mapsto(TP,\,\boldsymbol p)$ preserves the prime tuples $\boldsymbol p$ and the weights $\sigma(P,\,\boldsymbol p)=\prod_{i=1}^{K}V_i^{-(J+1)}$. For odd integers $u,\,v$, we have $\gcd(u,\,(u-v)/2)=\gcd(u,\,v)$; conversely, if $u$ is odd and $v\in\mathbb Z$, then both coordinates of $T^{-1}(u,\,v)^T=(u,\,u-2v)^T$ are odd, and $\gcd(u,\,u-2v)=\gcd(u,\,v)$. Since $\Omega'=T\Omega$, we conclude that $T$ gives a bijection between $\Omega\cap\mathcal C$ and the primitive integer vectors in $\Omega'$ with odd first coordinate.
	Hence
	\begin{align*}
			\langle\mathcal U_Tf,\,\mathcal U_Tg\rangle
			&=\sum_{\substack{(P,\,\boldsymbol p)\in\mathscr S,\, P\in\mathcal C}}
			\overline{(\mathcal U_Tf)(TP,\,\boldsymbol p)}
			(\mathcal U_Tg)(TP,\,\boldsymbol p)\sigma(TP,\,\boldsymbol p)\\
			&=\sum_{\substack{(P,\,\boldsymbol p)\in\mathscr S, \,P\in\mathcal C}}
			\overline{f(P,\,\boldsymbol p)}g(P,\,\boldsymbol p)\sigma(P,\,\boldsymbol p)=\langle f,\,g\rangle,
	\end{align*}
since we have extended $f$ and $g$ by zero
 to all $(P,\,\boldsymbol p)\in\mathscr S$ with $P\notin\mathcal C$.
In addition, $\mathcal U_T$ has the inverse $(\mathcal U_T^{-1}g)(P,\,\boldsymbol p)=g(TP,\,\boldsymbol p)$ and is therefore surjective. Thus, $\mathcal U_T$ is unitary.

To prove \eqref{eq:conj}, we first show that
\begin{equation}\label{operator=1}
	\mathcal U_T\mathcal T_{-2}^{\mathcal C}\mathcal U_T^{-1}=\mathcal T_1^{\Omega',\,\mathrm{odd}}.
\end{equation}
 For each summand in \eqref{def-T^C_-2} for $\mathcal T_{-2}^{\mathcal C}f$, we have $D\mid P_1$, $D\mid Q_1$, and $P,\,Q\in\mathcal C$. Since $P_1,\,Q_1,\,P_2,\,Q_2$ are all odd, we have $t=(P_1Q_2-P_2Q_1)/D\in2\mathbb Z$.
Consequently, one finds that $t'=\det(TP,\,TQ)/D=-t/2\in\mathbb Z$ because $\det T=-1/2$. Therefore, we obtain
\begin{equation}\label{eq:detconj}
	\mathbf1_{t=2(a-b)}=\mathbf1_{t'=b-a},\qquad \psi_{-2}\Bigl(\frac tY\Bigr)=\psi_1\Bigl(\frac{t'}Y\Bigr).
\end{equation}
Next, for the two integrals over $\mathfrak M_{-2}$ and $\mathfrak M_1$, we have $\mathbf1_{\mathfrak M_{-2}}(\theta)=\mathbf1_{\mathfrak M_1}(-2\theta)$ by \eqref{eq:major-arcs-kernel}. On each interval $[j/2,\,(j+1)/2]$,  $j=0,\,1$, we substitute $\phi=j+1-2\theta$, getting
\begin{equation*}
	\int_{j/2}^{(j+1)/2}\mathbf1_{\mathfrak M_1}(-2\theta)e\bigl(\theta(t-2(a-b))\bigr)\,\mathrm d\theta=\frac12\int_{\mathfrak M_1}e\bigl(\phi(t'-b+a)\bigr)\,\mathrm d\phi
\end{equation*}
for $j=0,\,1$. Adding these two identities, we obtain
\begin{equation}\label{eq:019}
	\int_{\mathfrak M_{-2}}e\bigl(\theta(t-2(a-b))\bigr)\,\mathrm d\theta=\int_{\mathfrak M_1}e\bigl(\phi(t'-b+a)\bigr)\,\mathrm d\phi.
\end{equation}
Finally, under the transformation $T$, we find that the first coordinates $P_1,\,Q_1$, the prime tuples $\boldsymbol p,\,\boldsymbol p'$, and hence $D,\,a,\,b$, $\frac{d_0}{D}\eta(\frac aY)\eta(\frac bY)\rho^{\frac{\omega(P_1)+\omega(Q_1)-2K(J+1)}2}$ all remain unchanged, as well as for the sets $\mathcal P_i\setminus\{p_{i,1},\,\ldots,\,p_{i,J+1}\}$ and the coefficient $\prod_{i=1}^K V_i^{-1}$.
Combining \eqref{eq:detconj} and \eqref{eq:019}, together with the bijection $Q\mapsto TQ$ and $(\mathcal U_Tf)(TQ,\,\boldsymbol p') =f(Q,\,\boldsymbol p')$, we obtain
\begin{equation*}
	\bigl(\mathcal T_1^{\Omega',\,\mathrm{odd}}(\mathcal U_Tf)\bigr)(TP,\,\boldsymbol p)=(\mathcal T_{-2}^{\mathcal C}f)(P,\,\boldsymbol p)=\bigl(\mathcal U_T(\mathcal T_{-2}^{\mathcal C}f)\bigr)(TP,\,\boldsymbol p).
\end{equation*}
This shows that \eqref{operator=1} holds since $f$ is arbitrary. 

We next compare $\mathcal G_{-2}$ with $\mathcal G_1^{\Omega',\,\mathrm{odd}}$. For $P=(P_1,\,P_2)^T\in\Omega\cap\mathcal C$, one may choose $c,\,d\in\mathbb Z$ such that $P_1d-P_2c=1$. Furthermore, one may verify that $(-2c,\,d-c)^T$ is a complement to $TP$.
By the definitions of $r_{_{TP}}$ and $r_{_P}$ in \eqref{def-r_P}, we obtain
\begin{equation*}%\label{eq:022}
r_{_{TP}}\equiv-\frac{2c}{P_1}\equiv-2r_{_P}\,(\textrm{mod}\,1).
\end{equation*}
Thus conditions \eqref{eq:goodone} and \eqref{eq:goodtwo} at $(TP,\,\boldsymbol p)$ are equivalent to that at $(P,\,\boldsymbol p)$ with $r_{_P}$ replaced by $-2r_{_P}$. By the definitions of the two operators, this gives
\begin{equation}\label{operator=2}
	\mathcal U_T\mathcal G_{-2}\mathcal U_T^{-1}=\mathcal G_1^{\Omega',\,\mathrm{odd}}.
\end{equation}

We also have $\mathcal U_T\mathcal S\mathcal U_T^{-1}=\mathcal S^{\Omega',\,\mathrm{odd}},$
because $\mathcal S$ averages over permutations within each tuple $\boldsymbol p_i$, and we use $(\mathcal U_Tf)(TP,\,\boldsymbol p)=f(P,\,\boldsymbol p)$ in each term of the expression for $\mathcal S^{\Omega',\,\mathrm{odd}}(\mathcal U_Tf)$.
Applying this identity, together with \eqref{operator=1} and \eqref{operator=2}, to $\mathcal A_{-2}^{\mathcal C}=\mathcal G_{-2}\mathcal S\mathcal T_{-2}^{\mathcal C}\mathcal S\mathcal G_{-2}$, then \eqref{eq:conj} is true:
\begin{align*}
	\mathcal U_T\mathcal A_{-2}^{\mathcal C}\mathcal U_T^{-1}
	=\mathcal U_T\bigl(\mathcal G_{-2}\mathcal S\mathcal T_{-2}^{\mathcal C}\mathcal S\mathcal G_{-2}\bigr)\mathcal U_T^{-1}=\mathcal G_1^{\Omega',\,\mathrm{odd}}\mathcal S^{\Omega',\,\mathrm{odd}}\mathcal T_1^{\Omega',\,\mathrm{odd}}\mathcal S^{\Omega',\,\mathrm{odd}}\mathcal G_1^{\Omega',\,\mathrm{odd}}=\mathcal A_1^{\Omega',\,\mathrm{odd}}.
\end{align*}
%which proves \eqref{eq:conj}.

Finally, since $\mathcal U_T$ is unitary, \eqref{eq:conj} gives
$\bigl(\mathcal A_1^{\Omega',\,\mathrm{odd}}\bigr)^*=\mathcal U_T(\mathcal A_{-2}^{\mathcal C})^*\mathcal U_T^{-1}.$
Consequently, for every integer $R\ge1$,
\begin{equation*}
	\bigl(\mathcal A_1^{\Omega',\,\mathrm{odd}}(\mathcal A_1^{\Omega',\,\mathrm{odd}})^*\bigr)^R=\mathcal U_T\bigl(\mathcal A_{-2}^{\mathcal C}(\mathcal A_{-2}^{\mathcal C})^*\bigr)^R\mathcal U_T^{-1}.
\end{equation*}
For $P\in\Omega\cap\mathcal C$, we use $\widetilde u_{_{TP}}=\mathcal U_Tu_{_P}$, together with the identity $\langle\mathcal U_Tf,\,\mathcal U_Tg\rangle=\langle f,\,g\rangle$, getting
\begin{align*}
	\bigl\langle \widetilde u_{_{TP}},\,\bigl(\mathcal A_1^{\Omega',\,\mathrm{odd}}(\mathcal A_1^{\Omega',\,\mathrm{odd}})^*\bigr)^R\widetilde u_{_{TP}} \bigr\rangle
	=\bigl\langle\mathcal U_Tu_{_P},\,\mathcal U_T\bigl(\mathcal A_{-2}^{\mathcal C}(\mathcal A_{-2}^{\mathcal C})^*\bigr)^Ru_{_P}\bigr\rangle 
	=\bigl\langle u_{_P},\,\bigl(\mathcal A_{-2}^{\mathcal C}(\mathcal A_{-2}^{\mathcal C})^*\bigr)^Ru_{_P}\bigr\rangle,
\end{align*}
which is just \eqref{oper-moment-iden}. This completes the proof of Proposition \ref{prop:parity-conjugacy}.
\end{proof}

Summing the moment identity \eqref{oper-moment-iden}  in Proposition \ref{prop:parity-conjugacy}
over $P\in\Omega\cap\mathcal C$, one may find that the assumption \eqref{eq:moment-assumption} required in
Proposition \ref{prop:moment-reduction} is equivalent to
\begin{equation}\label{root-aver}
	\frac1{d_0YMN}
	\sum_{P\in\Omega\cap\mathcal C}
	\Bigl\langle \widetilde u_{_{TP}},\,
	\bigl(\mathcal A_1^{\Omega',\,\mathrm{odd}}
	(\mathcal A_1^{\Omega',\,\mathrm{odd}})^*\bigr)^R
	\widetilde u_{_{TP}}\Bigr\rangle
	\ll L^{-2RE_0}.
\end{equation}

\vskip 3mm

\subsection{An integral formula for the sum over positions} \label{sec:roots}
% Review markup for the revisions beginning in this subsection.
% Change the following switch to \MSreviewfalse for the revised reading text.
\newif\ifMSreview
\MSreviewtrue
\DeclareRobustCommand{\MSword}[2]{%
  {\ifMSreview\color{red}\hypersetup{linkcolor=red,citecolor=red}\fi#1}%
  \ifMSreview\space{\color{black}\hypersetup{linkcolor=black,citecolor=black}(#2)}\fi}
\DeclareRobustCommand{\MSnote}[1]{%
  \ifMSreview\par\noindent{\small\color{red}\hypersetup{linkcolor=red,citecolor=red}(#1)}\par\fi}

In this subsection, we shall express the sum over $P$ in \eqref{root-aver}
as a triple integral with a negligible error term.

We first use Lemma \ref{lem:app-root-count} to count the vectors $P$ restricted by some congruence conditions.
Following the notation of Lemma \ref{lem:app-root-count}, we write $P=(P_1,\,P_2)^T =(u,\,v)^T\in\Omega$, whence
\begin{equation}\label{u-v-range}
	d_0YM\le u\le16d_0YM,\qquad N\le v\le2N,
\end{equation}
by the definition of $\Omega$ in \eqref{eq:physbox}. Note that $TP$ lies in the parallelogram $T\Omega=\Omega'$  by \eqref{eq:021}.

Recall the definition of $\mathcal C$ in \eqref{eq:017}.
For each $P=(u,\,v)^T\in\Omega\cap\mathcal C$, choose $c,d\in\mathbb{Z}$ such that $ud-vc=1$ and $0\le c<u$. Put $r=c/u$ whence $0\le r<1$. By \eqref{def-r_P}, we have $r\equiv r_{_P}\,(\textrm{mod}\,1)$. Recall that $T=\Big(\begin{smallmatrix}1&0\\[0,7mm]1/2&-1/2\end{smallmatrix}\Big)$, $T^{-1}=\Big(\begin{smallmatrix}1&0\\[0,7mm]1&-2\end{smallmatrix}\Big)$, and define
\begin{equation*}
	g=\begin{pmatrix}u&c\\v&d\end{pmatrix}\in\mathrm{SL}_2(\mathbb Z),\qquad E=\begin{pmatrix}1&0\\0&-2\end{pmatrix}.
\end{equation*}
Since $u,v$ are odd and $ud-vc=1$, we obtain
\begin{equation}\label{eq:gprime}
	g':=TgE=\begin{pmatrix}u&-2c\\(u-v)/2&d-c\end{pmatrix}\in\mathrm{SL}_2(\mathbb Z),
\end{equation}
whence $(g')^{-1}\in\mathrm{SL}_2(\mathbb Z)$.
For any $Q=(q_1,\,q_2)\in\Omega\cap\mathcal C$, one has $TQ\in\mathbb Z^2\cap \Omega'$ is  primitive. Thus
\begin{equation*}
	z=(z_1,\,z_2)^T:=(g')^{-1}TQ\in\mathbb Z^2
\end{equation*}
is also primitive.
We call $z$ the \textbf{relative position} of $TQ$ with respect to $g'$. Note that, now the original position $Q$ of $z$ is  $Q=T^{-1}g'z=gEz$. 
Since $g\in\mathrm{SL}_2(\mathbb Z)$ and $\gcd(z_1,\,z_2)=1$, we have
\begin{equation*}
\gcd(q_1,\,q_2)=\gcd\bigl(uz_1-2cz_2,\,vz_1-2dz_2\bigr)
=\gcd(z_1,\,2z_2)=\gcd(z_1,\,2).
\end{equation*}
Hence $Q=gEz$ is primitive is equivalent to $2\nmid z_1$.
Moreover, since $2\nmid u$, \eqref{eq:gprime} gives
\begin{equation}\label{q1z1mod2}
	q_1=uz_1-2cz_2\equiv z_1\,(\textrm{mod}\,2).
\end{equation}
Recall that in \eqref{def-A_1^Omega'-odd} defining $\mathcal A_1^{\Omega',\,\mathrm{odd}}$, the first coordinate $q_1$ of $g'z=TQ$ should be odd, which is equivalent to $2\nmid z_1$ by \eqref{q1z1mod2}.
Recall the notation $\mathbb P^1(\mathbb F_p)$ in \eqref{def-P^1(F_p)}. For any $p>2$, let
\begin{equation*}
	\ell_p(g'):=\Bigl\{(\xi_1,\,\xi_2)^T\in\mathbb F_p^2:\ u\xi_1-2c\xi_2=0\Bigr\}.
\end{equation*}
Since $\det g'=1$, we have $p\nmid(u,\,-2c)$. Hence $\ell_p(g')$ is one-dimensional and $\ell_p(g')\in\mathbb P^1(\mathbb F_p)$.

By the definition of $\mathscr S$ in \eqref{mathscr-S}, for any $(Q,\,\boldsymbol p)\in\mathscr S$ with $Q=(q_1,\,q_2)^T$, we have
\begin{equation}\label{cond-p|q_1}
	p_{i,j}\mid q_1\qquad\text{for }1\le i\le K,\ 1\le j\le J+1.
\end{equation}
Since $T$ preserves the first coordinate, these divisibility conditions also occur in $\mathcal A_1^{\Omega',\,\mathrm{odd}}$ at $(TQ,\,\boldsymbol p)$.
We next express the condition \eqref{cond-p|q_1} in terms of $z=(z_1,\,z_2)$, where $TQ=g'z$. Recall the notation $\mathbb P^1(\mathbb F_p)$ and $[z]_p$ in \eqref{def-P^1(F_p)} and \eqref{def-[z]_p}. For any $p>2$, define
\begin{equation*}
	\ell_p(g'):=\Bigl\{(\xi_1,\,\xi_2)^T\in\mathbb F_p^2:\ u\xi_1-2c\xi_2=0\Bigr\}.
\end{equation*}
Since $\det g'=1$, then $p\nmid \gcd(u,\,-2c)$. Thus $\ell_p(g')$ is  one-dimensional and belongs to $\mathbb P^1(\mathbb F_p)$.
Also, $z$ is primitive, whence $p\nmid\gcd(z_1,\,z_2)$. Thus, $z$ satisfies the defining equation of $\ell_p(g')$ modulo $p$ if and only if $[z]_p=\ell_p(g')$. Recalling that $q_1=uz_1-2cz_2$, we obtain
\begin{equation}\label{equiv-relation}
	p\mid q_1\ \Longleftrightarrow\  uz_1-2cz_2\equiv0\,(\textrm{mod}\,p)\ \Longleftrightarrow\quad[z]_p=\ell_p(g').
\end{equation}

Let $S$ be an integer with $2\nmid S,\, \mu(S)\neq 0,\,S\le\exp\big(O(L^{0.98})\big)$. For any prime $p\mid S$, fix a line
\begin{equation}\label{def-ell_p}
	\ell_p=\Bigl\{\lambda(\xi_p,\,\eta_p)^T:\ \lambda\in\mathbb F_p\Bigr\}\in\mathbb P^1(\mathbb F_p).
\end{equation}
We shall apply Lemma \ref{lem:app-root-count} to count vectors $P=(u,\,v)^T\in\Omega\cap\mathcal C$ satisfying 
\begin{equation}\label{l_g'=l_p}
	\ell_p(g')=\ell_p\quad\ \text{for any }p\mid S,
\end{equation}
where $g'$ varies with $P$ as in \eqref{eq:gprime}, while the lines $\ell_p$ remain fixed. 
To do this, we split the sum over $P=(u,\,v)^T$ according to the value of the associated matrix $g$ modulo $2S$, that is, we write $g\equiv g_0\,(\textrm{mod}\,2S)$, where
$g_0=\Big(\begin{smallmatrix}u_0&c_0\\[0,7mm]v_0&d_0'\end{smallmatrix}\Big)\in\mathrm{SL}_2\bigl(\mathbb Z/2S\mathbb Z\bigr)$ is a fixed residue class. Then the condition \eqref{l_g'=l_p} is equivalent to $u_0\xi_p-2c_0\eta_p\equiv0\,(\textrm{mod}\,p)$ 
for any $p\mid S$.
Besides, we let $I_1\subseteq[1,\,16]$, $I_2\subseteq[1,\,2]$, $I_3\subseteq[0,\,1)$ be three intervals, and assume	$\frac{u}{d_0YM}\in I_1,\, \frac vN\in I_2,\, \frac cu\in I_3.$
By the above arguments, we can write the sum over $P$ as
\begin{equation}\label{eq:position-decomp}
	\begin{aligned}
	\sum_{\substack{P=(u,\,v)^T\in\Omega\cap\mathcal C\\[0,7mm]
	\frac{u}{d_0YM}\in I_1,\,\, \frac{v}{N}\in I_2,\,\, \frac{c}{u}\in I_3\\[0,5mm]
	\ell_p(g')=\ell_p\ \text{for any }p\mid S}}1 =
\sum_{\substack{g_0\in\mathrm{SL}_2(\mathbb Z/2S\mathbb Z)\\[0,5mm]
	u_0\equiv v_0\equiv1\,(\textrm{mod}\,2)\\[0,5mm]
	u_0\xi_p-2c_0\eta_p\equiv0\,(\textrm{mod}\,p)
	\text{ for any }p\mid S}}\
	\sum_{\substack{u,\,v\ge1,\ \gcd(u,\,v)=1\\[0,7mm]
	\frac{u}{d_0YM}\in I_1,\,\, \frac{v}{N}\in I_2,\,\, \frac{c}{u}\in I_3\\[0,2mm]
	g\equiv g_0\,(\textrm{mod}\,2S)}}1.
	\end{aligned}
\end{equation}
For each fixed $g_0\in\mathrm{SL}_2(\mathbb Z/2S\mathbb Z)$, we apply Lemma \ref{lem:app-root-count} to the inner sum above, getting
\begin{equation}\label{appendix-appli}
	\frac1{d_0YMN}\sum_{\substack{u,\,v\ge1,\ \gcd(u,\,v)=1\\[0,7mm]
			\frac{u}{d_0YM}\in I_1,\,\, \frac{v}{N}\in I_2,\,\, \frac{c}{u}\in I_3\\[0,2mm]
			g\equiv g_0\,(\textrm{mod}\,2S)}}1
	=\frac{|I_1||I_2||I_3|}{\zeta(2)\big|\mathrm{SL}_2(\mathbb Z/2S\mathbb Z)\big|}
	+O\Bigl(S^2\min\{d_0YM,\,N\}^{-1/100}\Bigr).
\end{equation}
Note that the main term is independent of $g_0$ and the error estimate is uniform in $g_0$. Thus we next calculate the number of $g_0$ in the outer sum in \eqref{eq:position-decomp}.

First, there are exactly two matrices in $\mathrm{SL}_2(\mathbb F_2)$ whose first column is $(1,\,1)^T$.
For $p>2,\,p\mid S$, there are exactly $p-1$ nonzero pairs $(u_0,\,c_0)\in\mathbb F_p^2$ such that $u_0\xi_p-2c_0\eta_p=0\,(\textrm{mod}\,p)$. For each $(u_0,\,c_0)$, there are exactly $p$ pairs $(v_0,\,d'_0)\in\mathbb F_p^2$ satisfying $u_0d'_0-c_0v_0=1$, i.e. $g_0\in\mathrm{SL}_2(\mathbb Z/2S\mathbb Z)$. Thus we have $2\prod_{p\mid S}p(p-1)$ choices of $g_0$ by the Chinese remainder theorem. 

By the above argument, we see that on the one hand, the main term in \eqref{eq:position-decomp} from \eqref{appendix-appli} is
\begin{equation*}
\frac{|I_1||I_2||I_3|}{\zeta(2)\big|\mathrm{SL}_2(\mathbb Z/2S\mathbb Z)\big|}\cdot 2\prod_{p\mid S}p(p-1)
	=\frac{|I_1||I_2||I_3|}{\zeta(2)}\cdot
	\frac{2\prod_{p\mid S}p(p-1)}{6\prod_{p\mid S}p(p^2-1)}
	=\frac{|I_1||I_2||I_3|}{3\zeta(2)\prod_{p\mid S}(p+1)},
\end{equation*}
since 
\begin{equation*}
\big|\mathrm{SL}_2(\mathbb Z/2S\mathbb Z)\big| =\big|\mathrm{SL}_2(\mathbb F_2)\big|
\prod\nolimits_{p\mid S}\big|\mathrm{SL}_2(\mathbb F_p)\big|
=6\cdot\prod\nolimits_{p\mid S}p(p^2-1).
\end{equation*}
On the other hand, noting that $S\le\exp\big(O(L^{0.98})\big)$ and $2\prod_{p\mid S}p(p-1)\ll S^2$, hence the total error term after inserting \eqref{appendix-appli} into \eqref{eq:position-decomp} is estimated by 
$$\ll S^4\min\{d_0YM,\,N\}^{-1/100}\ll x^{-c_\delta}$$ 
for some $c_\delta>0$, since $M,\,N\gg x^{\delta}$. Consequently, we conclude that
\begin{equation}\label{eq:position-count}
	\frac1{d_0YMN}
	\sum_{\substack{P=(u,\,v)^T\in\Omega\cap\mathcal C\\[0,7mm]
	\frac{u}{d_0YM}\in I_1,\,\, \frac{v}{N}\in I_2,\,\, \frac{c}{u}\in I_3\\[0,5mm]
	\ell_p(g')=\ell_p\ \text{for any }p\mid S}}1 
	=\frac{|I_1||I_2||I_3|}{3\zeta(2)\prod_{p\mid S}(p+1)}+O\bigl(x^{-c_\delta}\bigr).
\end{equation}

We shall use \eqref{eq:position-count} to prove the integral formula
\eqref{eq:moment-integral} in Proposition \ref{prop:position-integrals} below. Before stating Proposition \ref{prop:position-integrals}, we first define the operator $\mathcal A_{1,\boldsymbol\ell}$ and the quantity $\mathcal M_R$ in the sequel.

We extend the matrix $g'\in \mathrm{SL}_2(\mathbb Z)$ in \eqref{eq:gprime} to $\widetilde{g}(u,\,v,\,r)\in \mathrm{SL}_2(\mathbb R)$ defined by
\begin{equation}\label{def-g'(uvr)}
	\widetilde g(u,\,v,\,r):= \begin{pmatrix}u&-2ur\\(u-v)/2&r(v-u)+1/u\end{pmatrix}\in\mathrm{SL}_2(\mathbb R)
\end{equation}
where $u,\,v\in\mathbb{R}$ satisfying \eqref{u-v-range} and $0\le r<1$.
Note that for $g'$ in \eqref{eq:gprime}, one has $g'=TgE=\widetilde g(u,\,v,\,c/u).$ In the sequel, we will use abbreviation $\widetilde g=\widetilde g(u,\,v,\,r)$ for fixed $u,\,v,\,r\in \mathbb{R}$.

For each prime $p\in\bigcup_{i=1}^K\mathcal P_i$, there are $p+1$ lines of the form \eqref{def-ell_p} in $\mathbb P^1(\mathbb F_p)$ and we choose one of them, say $\ell_p$. We denote the collection of all these choices by $\boldsymbol\ell=(\ell_p)_{p\in\bigcup_{i=1}^K\mathcal P_i}$.
%and define
%\begin{equation*}
%	\mathcal A_{1,\boldsymbol\ell}:=\mathcal G_{1,\boldsymbol\ell}\mathcal S\mathcal T_{1,\boldsymbol\ell}\mathcal S\mathcal G_{1,\boldsymbol\ell}.
%\end{equation*}
%where $\mathcal S$ is as in \eqref{def-S}, 
%and we shall define the operators $\mathcal T_{1,\boldsymbol\ell}$ and $\mathcal G_{1,\boldsymbol\ell}$  as follows.
Then, for $1\le i\le K$, $1\le j,\,h\le J+1$ and $j\neq h$, as $\mathscr S$ in \eqref{mathscr-S}, we define the set $\mathscr S_{\boldsymbol\ell}$ by
\begin{equation}\label{def-mathscrS-ell}
	\mathscr S_{\boldsymbol\ell}:=\left\{(z,\,\boldsymbol p):\
	\begin{aligned}
		&z=(z_1,\,z_2)^T\in\mathbb Z^2,\ \, \gcd(z_1,\,z_2)=1,\ \, 2\nmid z_1,\quad \widetilde gz\in\Omega',\ \\[0,1mm]
		&\boldsymbol p=(\boldsymbol p_1,\,\ldots,\,\boldsymbol p_K),\ \, \boldsymbol p_i=(p_{i,1},\,\ldots,\,p_{i,J+1}) \in\mathcal P_i^{J+1},\\[0,1mm]
		&[z]_{p_{i,j}}=\ell_{p_{i,j}},\quad p_{i,j}\ne p_{i,h} 
	\end{aligned}
	\right\}.
\end{equation}
Note that by \eqref{equiv-relation}, the equality $[z]_{p_{i,j}}=\ell_{p_{i,j}}$ in $\mathscr S_{\boldsymbol\ell}$ arises from  the condition  $p_{i,j}\mid q_1$ in $\mathscr S$ at $Q=(q_1,\,q_2)^T$ when $\ell_{p_{i,j}}=\ell_{p_{i,j}}(g')$. 
We then define the operator $\mathcal T_{1,\boldsymbol\ell}$, just as $\mathcal T_1$ in \eqref{def-T1}, by
\begin{equation*}
	(\mathcal T_{1,\boldsymbol\ell}f)(z,\,\boldsymbol p): =\mathbf1_{d_0\le D<2d_0}\biggl(\prod_{i=1}^K\frac1{V_i}\biggr) \sum_{\substack{p_i'\in\mathcal P_i \setminus\{p_{i,1}, \,\ldots,\,p_{i,J+1}\}\\[0,2mm]1\le i\le K}}\ \sum_{\substack{z'\\[0,3mm](z',\,\boldsymbol p')\in 
	\mathscr S_{\boldsymbol\ell}}} K_{1,\boldsymbol\ell}(z,\,z';\,D,\,a,\,b)f(z',\,\boldsymbol p'),
\end{equation*}
where $\boldsymbol p'=(\boldsymbol p'_1,\,\ldots,\,\boldsymbol p'_K)$ with $\boldsymbol p'_i=(p_{i,1},\,\ldots,\,p_{i,J},\,p_i')$,  and $D,\,a,\,b$ are defined by \eqref{Dba-definition}. 
The kernel $K_{1,\boldsymbol\ell}=K_{1,\boldsymbol\ell}(z,\,z';\,D,\,a,\,b)$ is defined, just as $K_1$ in $\mathcal T_1$, by
\begin{equation*}
	\begin{aligned}
		K_{1,\boldsymbol\ell}:=\frac{d_0}{D}\eta\biggl(\frac aY\biggr)\eta\biggl(\frac bY\biggr)\psi_1\biggl(\frac{t'}Y\biggr)\rho^{\frac{\omega_{\boldsymbol\ell}(z)+\omega_{\boldsymbol\ell}(z')}{2}-K(J+1)}\biggl\{\mathbf1_{t'=b-a}-\int_{\mathfrak M_1}e\bigl(\theta(t'-b+a)\bigr)\,\mathrm d\theta\biggr\},
	\end{aligned}
\end{equation*}
where $t'$ and $\omega_{\boldsymbol\ell}(z)$ are defined, also as $\omega(q_1)=\sum_{i=1}^K\sum_{p\in\mathcal P_i}\mathbf1_{p\mid q_1}$ in $\mathcal T_1$, by
\begin{equation*}
	t':=\frac{\det(z,\,z')}{D},\qquad \omega_{\boldsymbol\ell}(z):=\sum_{i=1}^K\sum_{p\in\mathcal P_i}\mathbf1_{[z]_p=\ell_p}.
\end{equation*}
Here $\omega_{\boldsymbol\ell}(z)$ is obtained from $\omega(q_1)$ with $\mathbf1_{p\mid q_1}$ replaced  by $\mathbf1_{[z]_p=\ell_p}$; and one has $\mathbf1_{p\mid q_1}=\mathbf1_{[z]_p=\ell_p}$ when $\ell_p=\ell_p(g')$ by \eqref{equiv-relation}.
In addition, for any $p\mid D=\prod_{i=1}^K\prod_{j=1}^J p_{i,j}$, we have $[z]_p=[z']_p=\ell_p$ by the conditions in $\mathscr S_{\boldsymbol\ell}$. This implies that $D\mid\det(z,\,z')$, thereby $t'\in\mathbb Z$. 

Next, we define $\mathcal G_{1,\boldsymbol\ell}$, similar to $\mathcal G_1$ in \eqref{def-G1}, by
\begin{equation*}
	(\mathcal G_{1,\boldsymbol\ell}f)(z,\,\boldsymbol p):=
	\begin{cases}
	f(z,\,\boldsymbol p)\ &\text{if \eqref{eq:goodone} and \eqref{eq:goodtwo} hold at }(\widetilde gz,\,\boldsymbol p), \\[1mm]
		0\ &\text{otherwise}.
	\end{cases}
\end{equation*}
In conditions  \eqref{eq:goodone} and \eqref{eq:goodtwo}, here the parameter $r_{\tilde gz}\in[0,\,1)$, instead of $r_{_P}$, is defined by 
\begin{equation*}
	 r_{\tilde gz}\equiv\frac{\tau(w)}{\tau(z)} \,(\textrm{mod}\,1),\quad
	\text{where\ \ } \tau(z):=z_1-2rz_2,
\end{equation*}
for some $w\in\mathbb Z^2$ with $\det(z,\,w)=1$, and $r\in[0,\,1)$ is as in \eqref{def-g'(uvr)} defining $\widetilde g$. Note that by calculation, the value of $r_{\tilde gz}$ is independent of the choice of $w$.
Moreover, for $(z,\,\boldsymbol p)\in\mathscr S_{\boldsymbol\ell}$, we have $\widetilde gz\in\Omega'$. This shows that the first coordinate $u\tau(z)$ of $\widetilde gz$ satisfies $u\tau(z)\in[d_0YM\,\,16d_0YM]$, thereby $\tau(z)\neq 0$.
Now, with $T_{1,\boldsymbol\ell}$ and $G_{1,\boldsymbol\ell}$ as above and $\mathcal S$ as in \eqref{def-S}, we define $\mathcal A_{1,\boldsymbol\ell}$ by
%the operator $\mathcal A_{1,\boldsymbol\ell}$ by
\begin{equation}\label{def-A_1,ell}
	\mathcal A_{1,\boldsymbol\ell}:=\mathcal G_{1,\boldsymbol\ell}\mathcal S\mathcal T_{1,\boldsymbol\ell}\mathcal S\mathcal G_{1,\boldsymbol\ell}.
\end{equation}

Put $z_0=(1,\,0)^T$. Similar to $u_{_P}$ defined by \eqref{def-u_P} in the moment estimate \eqref{eq:moment-assumption}, we define
\begin{equation*}
	u_{z_0}(z,\,\boldsymbol p):=\mathbf1_{z=z_0}\qquad\bigl((z,\,\boldsymbol p)\in\mathscr S_{\boldsymbol\ell}\bigr).
\end{equation*}
For $g'$, $\widetilde u$ defined in \eqref{eq:gprime}, \eqref{def-tilde-u_TP} respectively, and $P=(u,\,v)^T$, we have $TP=g'z_0$, whence 
\begin{equation}\label{uTP=uz0}
	\widetilde u_{_{TP}}(g'z,\,\boldsymbol p)=\mathbf1_{g'z=TP}=\mathbf1_{z=z_0}=u_{z_0}(z,\,\boldsymbol p).
\end{equation}
Thus $u_{z_0}$ corresponds to the function $\widetilde u_{_{TP}}$ appearing in \eqref{root-aver}. We define $\mathcal M_R(u,\,v,\,r)$ by
\begin{equation}\label{def-M_R}
\mathcal M_R(u,\,v,\,r):= 
\biggl(\prod_{p\in\bigcup_{i=1}^K\mathcal P_i}\frac1{p+1}\biggr) \sum_{\boldsymbol\ell}\bigl\langle u_{z_0},\,
(\mathcal A_{1,\boldsymbol\ell}\mathcal A_{1,\boldsymbol\ell}^*)^R\,
 u_{z_0}\bigr\rangle_{\boldsymbol\ell},
\end{equation}
where the inner product $\langle f_1,\,f_2\rangle_{\boldsymbol\ell}$, which can be regarded as $\langle \cdot,\,\cdot\rangle$ in \eqref{def-innerprod} with $\mathscr S$ replaced by $\mathscr S_{\boldsymbol\ell}$, is defined for $f_1,\,f_2:\mathscr S_{\boldsymbol\ell}\to\mathbb C$ as follows
\begin{equation*}
	\langle f_1,\,f_2\rangle_{\boldsymbol\ell}:= \biggl(\prod_{i=1}^K\frac1{V_i^{J+1}}\biggr)\sum_{(z,\,\boldsymbol p)\in\mathscr S_{\boldsymbol\ell}}\overline{f_1(z,\,\boldsymbol p)}f_2(z,\,\boldsymbol p).
\end{equation*}
%The adjoint $\mathcal A_{1,\boldsymbol\ell}^*$ is taken with respect to this inner product.

Put $z_1=(h,\,k)^T$ for some $h,\, k\in \mathbb{Z}$, whence $t'=\det(z_0,\,z_1)/D=k/D.$
To estimate the two terms between curly brackets in $K_{1,\boldsymbol\ell}$ separately, we let
\begin{equation}\label{W-choices}
	W(w)=\mathbf1_{w=0}\quad\text{or}\quad W(w)=\Bigl|\int_{\mathfrak M_1}e(\theta w)\,\mathrm d\theta\Bigr|.
\end{equation}
%For $w=t'-b+a$, these give the absolute values of the two terms in braces in $K_{1,\boldsymbol\ell}$.
Recall that $d_0=2^{\ell_0}$ is fixed. 
%and $D\in[d_0, 2d_0)$ as in \eqref{ell_0-D}.
For either choice of $W$ above, define
\begin{equation*}
	\begin{aligned}
		\Phi(u,\,v,\,r;\,\boldsymbol\ell):=
		\bigg(\prod_{i=1}^K\frac1{V_i^{J+2}}\bigg)
		\mathop{{\sum}^*}_{\substack{z_1,\,\, \boldsymbol p\\ p'_1,\,\ldots,\,p'_K}}\frac{d_0}{D}\Bigl|
		\eta\Bigl(\frac aY\Bigr)\eta\Bigl(\frac bY\Bigr) \psi_1\Bigl(\frac{t'}Y\Bigr)\Bigr|\rho^{\frac{\omega_{\boldsymbol\ell}(z_0)+\omega_{\boldsymbol\ell}(z_1)}2-K(J+1)}
		W(t'-b+a),
	\end{aligned}
\end{equation*}
where $\boldsymbol p'=(\boldsymbol p'_1,\,\ldots,\,\boldsymbol p'_K)$ with $\boldsymbol p'_i=(p_{i,1},\,\ldots,\,p_{i,J},\,p'_i)$, $p'_i\in\mathcal P_i\setminus\{p_{i,1},\,\ldots,\,p_{i,J+1}\}$ for any $1\le i\le K$; and $D,\,a,\,b$ are as in \eqref{Dba-definition}. The sum $\sum^*$ is restricted to tuples satisfying
\begin{equation*}
	\begin{aligned}
		&(z_0,\,\boldsymbol p)\in\mathscr S_{\boldsymbol\ell},\quad (z_1,\,\boldsymbol p')\in\mathscr S_{\boldsymbol\ell},\quad d_0\le D<2d_0.
	\end{aligned}
\end{equation*}

With the notation above, we now state the following proposition.

\begin{proposition}\label{prop:position-integrals}
	\textup{(i)} For some constant $c>0$, we have
	\begin{equation}\label{eq:moment-integral}
		\begin{aligned}
			&\frac1{d_0YMN}\sum_{P\in\Omega\cap\mathcal C} \bigl\langle\widetilde u_{_{TP}},\,\bigl(\mathcal A_1^{\Omega',\,\mathrm{odd}}(\mathcal A_1^{\Omega',\,\mathrm{odd}})^*\bigr)^R\,\widetilde u_{_{TP}}\bigr\rangle\\[1mm]
			&\qquad=\frac1{3\zeta(2)}\int_1^{16}\int_1^2\int_0^1\mathcal M_R(d_0YM s,\,Nt_0,\,r)\,\mathrm dr\,\mathrm dt_0\,\mathrm ds+O\big(\exp(-cL^{3/4})\big).
		\end{aligned}
	\end{equation}
	
	\textup{(ii)} For either choice of $W$ in \eqref{W-choices} and some constant $c>0$, we have
	\begin{equation}\label{eq:one-edge-integral}
		\begin{aligned}
			&\frac1{d_0YMN}\sum_{\substack{P=(u,\,v)^T\\P\in\Omega\cap\mathcal C}}\Phi\Bigl(u,\,v,\,\frac cu;\,(\ell_p(g'))_p\Bigr)\\[1mm]
			&\qquad=\frac1{3\zeta(2)}\int_1^{16}\int_1^2\int_0^1\frac{\sum_{\boldsymbol\ell}\Phi(d_0YM s,\,Nt_0,\,r;\,\boldsymbol\ell)}{\prod_{p\in\bigcup_{i=1}^K\mathcal P_i}(p+1)}\,\mathrm dr\,\mathrm dt_0\,\mathrm ds+O\big(\exp(-cL^{3/4})\big).
		\end{aligned}
	\end{equation}
The formula \eqref{eq:one-edge-integral} also holds if, on both sides, $\sum^*$ in $\Phi$ is restricted to terms for which at least one of \eqref{eq:goodone} and \eqref{eq:goodtwo} fails at $(\widetilde gz_0,\,\boldsymbol p)$ (or, separately, at $(\widetilde gz_1,\,\boldsymbol p')$).

\end{proposition}

\begin{proof}
Note that (i) and (ii) both use the change of variables $TQ=g'z$ from pairs $(TQ,\,\boldsymbol p)$ to $(z,\,\boldsymbol p)\in\mathscr S_{\boldsymbol\ell}$. We first show that the change  $TQ=g'z$ converts the formula defining $\mathcal A_1^{\Omega',\,\mathrm{odd}}$ into that for $\mathcal A_{1,\boldsymbol\ell}$.
For any fixed $P=(u,\,v)^T\in\Omega\cap\mathcal C$, let $c,\,d$ be as in \eqref{eq:gprime}, $r=c/u$, and take the line $\ell_p=\ell_p(g')$ for any $p\in\bigcup_{i=1}^K\mathcal P_i$,  whence $\widetilde g=g'$. For $TQ=g'z$ with $Q=(q_1,\,q_2)^T$, we derive by \eqref{equiv-relation} and $\det(g')=1$ that
\begin{equation*}
p\mid q_1\, \Longleftrightarrow\, [z]_p=\ell_p(g')=\ell_p,\quad\ \omega(q_1)=\omega_{\boldsymbol\ell}(z),\quad\ \frac{\det(g'z,\,g'z')}{D}=\frac{\det(z,\,z')}{D}=t'.
\end{equation*}
For $(z,\,\boldsymbol p)\in\mathscr S_{\boldsymbol\ell}$, choose $w\in\mathbb Z^2$ with $\det(z,\,w)=1$. Since $g'\in\mathrm{SL}_2(\mathbb Z)$, we have $\det(g'z,\,g'w)=1$. The first coordinates of $g'z$ and $g'w$ are $u\tau(z)$ and $u\tau(w)$, respectively, whence \eqref{def-r_P} yields
\begin{equation*}
	r_{g'z}\equiv\frac{u\tau(w)}{u\tau(z)}=
	\frac{\tau(w)}{\tau(z)}\equiv r_{\tilde gz}\,(\textrm{mod}\,1).
\end{equation*}
By the identities above, one may find that under  $TQ=g'z$, the kernels $K_1$ and $K_{1,\boldsymbol\ell}$ satisfy 
$K_1(g'z,\,g'z';\,D,\,a,\,b)=K_{1,\boldsymbol\ell}(z,\,z';\,D,\,a,\,b);$
the condition $p_{i,j}\mid q_1$ are equivalent to $[z]_{p_{i,j}}=\ell_{p_{i,j}}$ by \eqref{equiv-relation}. Moreover, conditions \eqref{eq:goodone} and \eqref{eq:goodtwo} are unchanged when $r_{g'z}$ is replaced by $r_{\tilde gz}$, since $r_{g'z}\equiv r_{\tilde gz}\,(\textrm{mod}\,1)$. 
Moreover, the permutations of the prime lists and the replacement of $p_{i,J+1}$ by $p'_i$ follow the same rules in the definitions of $\mathcal A_1^{\Omega',\,\mathrm{odd}}$ and $\mathcal A_{1,\boldsymbol\ell}$.
Hence, the summation formula for $\mathcal A_1^{\Omega',\,\mathrm{odd}}$ under $TQ=g'z$ becomes the formula defining $\mathcal A_{1,\boldsymbol\ell}$, that is 
\begin{equation}\label{equa-A_1,f}
	(\mathcal A_{1,\boldsymbol\ell}f)(z,\,\boldsymbol p)=(\mathcal A_1^{\Omega',\,\mathrm{odd}}F)(g'z,\,\boldsymbol p),
\end{equation}
where $F$ is any function on $\mathcal S= \bigl\{(TQ,\ \boldsymbol p):\, (Q,\,\boldsymbol p)\in\mathscr S,\ Q\in\mathcal C\bigr\}$ and $f(z,\,\boldsymbol p):=F(g'z,\,\boldsymbol p)$. 

We now prove \textup{(i)}. First, by \eqref{q1z1mod2} and \eqref{equiv-relation}, one may deduce  that $(z,\,\boldsymbol p)\mapsto(g'z,\,\boldsymbol p)$ is a bijection from $\mathscr S_{\boldsymbol\ell}$ onto $\mathcal S$.

For any functions $F_1,F_2$ on $\mathcal S$ and $f_j(z,\,\boldsymbol p):=F_j(g'z,\,\boldsymbol p)$, $j=1,2$. Under the change $TQ=g'z$, one finds $\langle f_1,\,f_2\rangle_{\boldsymbol\ell}=\langle F_1,\,F_2\rangle$ by the definition of $\langle\cdot,\,\cdot\rangle$ in \eqref{def-innerprod} and $\langle\cdot,\,\cdot\rangle_{\boldsymbol\ell}$.
This equality, together with \eqref{equa-A_1,f} implies that $(\mathcal A_{1,\boldsymbol\ell}^*f)(z,\,\boldsymbol p)=\bigl((\mathcal A_1^{\Omega',\, \mathrm{odd}})^*F\bigr)(g'z,\,\boldsymbol p).$
%\begin{equation*}
%	(\mathcal A_{1,\boldsymbol\ell}^*f)(z,\,\boldsymbol p)=\bigl((\mathcal A_1^{\Omega',\, \mathrm{odd}})^*F\bigr)(g'z,\,\boldsymbol p).
%\end{equation*}
Since $\widetilde u_{_{TP}}(g'z,\,\boldsymbol p)=u_{z_0}(z,\,\boldsymbol p)$ by \eqref{uTP=uz0}, we obtain
\begin{equation}\label{operator-transf}
	\bigl\langle\widetilde u_{_{TP}},\,(\mathcal A_1^{\Omega',\,\mathrm{odd}}(\mathcal A_1^{\Omega',\,\mathrm{odd}})^*)^R\widetilde u_{_{TP}}\bigr\rangle=\bigl\langle u_{z_0},\,(\mathcal A_{1,\boldsymbol\ell}\mathcal A_{1,\boldsymbol\ell}^*)^R \,u_{z_0}\bigr\rangle_{\boldsymbol\ell}.
\end{equation}
%We shall expand $\langle\cdot,\,\cdot\rangle_{\boldsymbol\ell}$ above and apply \eqref{eq:position-count} to the sum over $P$ for each term of the expansion.

To expand 
$\langle u_{z_0},\,(\mathcal A_{1,\boldsymbol\ell}\mathcal A_{1,\boldsymbol\ell}^*)^R u_{z_0}\rangle_{\boldsymbol\ell}$, for fixed $u,\,v$ satisfying \eqref{u-v-range}, $r\in[0,\,1)$, and $\boldsymbol\ell$,  we write
\begin{equation}\label{A_1,ell-expansion}
	(\mathcal A_{1,\boldsymbol\ell}f)(z,\,\boldsymbol p)
	=\sum_{(z',\,\boldsymbol p')\in\mathscr S_{\boldsymbol\ell}}
	a_{\boldsymbol\ell}(z,\,\boldsymbol p;\,z',\,\boldsymbol p')f(z',\,\boldsymbol p'),
\end{equation}
where $a_{\boldsymbol\ell}(z,\,\boldsymbol p;\,z',\,\boldsymbol p')$
is the coefficient obtained by expanding
$\mathcal A_{1,\boldsymbol\ell}=
\mathcal G_{1,\boldsymbol\ell}\mathcal S\mathcal T_{1,\boldsymbol\ell} \mathcal S\mathcal G_{1,\boldsymbol\ell}$:
\begin{equation}\label{def-a_ell}
	a_{\boldsymbol\ell}(z,\,\boldsymbol p;\,z',\,\boldsymbol p') :=\frac{1}{\big((J+1)!\big)^{2K}}
	\biggl(\prod_{i=1}^K\frac1{V_i}\biggr)
	\mathop{\sum\nolimits^*}_{\boldsymbol\pi,\,\boldsymbol\sigma\in\mathfrak S_{J+1}^K} K_{1,\boldsymbol\ell}(z,\,z';\,D_{\boldsymbol\pi},\,
	a_{\boldsymbol\sigma},\,b_{\boldsymbol\pi}).
\end{equation}
Here $\boldsymbol\pi=(\pi_1,\,\ldots,\,\pi_K),\ \boldsymbol\sigma=(\sigma_1,\,\ldots,\,\sigma_K)$ specify the permutations of $\boldsymbol p_i=(p_{i,1},\,\ldots,\,p_{i,J+1})$ and $\boldsymbol p'_i=(p'_{i,1},\,\ldots,\,p'_{i,J+1})$, $1\leq i\leq K$ respectively; and $D_{\boldsymbol\pi},\,
a_{\boldsymbol\sigma},\,b_{\boldsymbol\pi}$ are defined by
\begin{equation*}
	D_{\boldsymbol\pi}:=\prod_{i=1}^K\prod_{h=1}^J p_{i,\pi_i(h)},\qquad b_{\boldsymbol\pi}:=\prod_{i=1}^K p_{i,\pi_i(J+1)},\qquad a_{\boldsymbol\sigma}:=\prod_{i=1}^K p'_{i,\sigma_i(J+1)}.
\end{equation*}
The symbol $\sum\nolimits^*$ restricts the sum by the following conditions, for every $1\le i\le K$
\begin{equation*}
	d_0\le D_{\boldsymbol\pi}<2d_0,\quad p'_{i,\sigma_i(J+1)}\notin\{p_{i,1},\,\ldots,\,p_{i,J+1}\},\quad p_{i,\pi_i(h)}=p'_{i,\sigma_i(h)}\quad\text{for }1\le h\le J,
\end{equation*}
since after permutation, $\mathcal T_{1,\boldsymbol\ell}$ retains the first $J$ primes $p_{i,\pi_i(1)},\,\ldots,\,p_{i,\pi_i(J)}$ and replaces $p_{i,\pi_i(J+1)}$ by $p'_{i,\sigma_i(J+1)}\notin\{p_{i,1},\,\ldots,\,p_{i,J+1}\}$;
$\sum\nolimits^*$ also requires that \eqref{eq:goodone} and \eqref{eq:goodtwo} both hold at each of $(\widetilde gz,\,\boldsymbol p)$ and $(\widetilde gz',\,\boldsymbol p')$, since the operator $\mathcal G_{1,\boldsymbol\ell}$ appears in $\mathcal A_{1,\boldsymbol\ell}$ twice.
With the expansion \eqref{A_1,ell-expansion}, we also have $(\mathcal A_{1,\boldsymbol\ell}^*f)(z,\,\boldsymbol p)
=\sum_{(z',\,\boldsymbol p')\in\mathscr S_{\boldsymbol\ell}}
\overline{a_{\boldsymbol\ell}(z',\,\boldsymbol p';\,z,\,\boldsymbol p)}f(z',\,\boldsymbol p').$

Thus, with $u_{z_0}(z,\,\boldsymbol p)=\mathbf1_{z=z_0}$, one may deduce that 
\begin{equation}\label{moment-expansion}
	\begin{aligned}
		\bigl\langle u_{z_0},\,(\mathcal A_{1,\boldsymbol\ell} \mathcal A_{1,\boldsymbol\ell}^*)^R\, u_{z_0}\bigr\rangle_{\boldsymbol\ell}
		=&\prod_{i=1}^K\frac1{V_i^{J+1}}
		\sum_{\substack{z_{_{2R}}=z_{_0}\\[0,6mm]
		(z_j,\,\boldsymbol p^{(j)})\in
		\mathscr S_{\boldsymbol\ell}\\[0,6mm] 0\le j\le2R }} \prod_{j=0}^{R-1}\Bigl(
		a_{\boldsymbol\ell}\big(z_{2j},\,\boldsymbol p^{(2j)}; \,z_{2j+1},\,\boldsymbol p^{(2j+1)}\big)\\[-1,2mm]
		&\qquad\qquad\qquad\qquad\times \overline{a_{\boldsymbol\ell}\big(z_{2j+2},\, \boldsymbol p^{(2j+2)};\,z_{2j+1},\,\boldsymbol p^{(2j+1)} \big)}\Bigr).
	\end{aligned}
\end{equation}
For $0\le j\le2R$, the prime lists $\boldsymbol p^{(j)}=(\boldsymbol p_1^{(j)},\,\ldots,\,\boldsymbol p_K^{(j)})$ with $\boldsymbol p_i^{(j)}=(p_{i,1}^{(j)},\,\ldots,\, p_{i,J+1}^{(j)}), 1\le i\le K$ is paired with $z_j$ in the summation variable $(z_j,\,\boldsymbol p^{(j)})$. The superscript $(j)$ only identifies the pair, and the prime lists with different  $j$ need not be distinct. The condition $z_{2R}=z_0=(1,\,0)^T$ comes from the two occurrences of $u_{z_0}$ in $\langle\cdot,\,\cdot\rangle_{\boldsymbol\ell}$ since $\overline{u_{z_0}(z,\,\boldsymbol p^{(0)})}u_{z_0}(z_{2R},\,\boldsymbol p^{(2R)})=\mathbf1_{z=z_0}\mathbf1_{z_{_{2R}}=z_0}$.
This equality, together with \eqref{operator-transf}, shows that to get \eqref{eq:moment-integral}, we are left to treat the sum
\begin{equation}\label{def-S(K,R)}
	\begin{aligned}
	\prod_{i=1}^K\frac1{V_i^{J+1}}
	\sum_{\substack{\boldsymbol p_i^{(j)}\in\mathcal P_i^{J+1} \\[0,6mm]
	z_j\in\mathbb Z^2,\, z_{_{2R}}=z_0 
	\\[0,8mm]  0\le j\le2R }} \,
	\sum_{\substack{P\in\Omega\cap\mathcal{C}\\[0,8mm]
		(z_j,\boldsymbol p^{(j)})\in\mathscr S_{\boldsymbol\ell}
		\\[0,8mm] \text{}0\le j\le2R}} \, \prod_{j=0}^{R-1}a_{\boldsymbol\ell}\big(z_{2j},
    \boldsymbol p^{(2j)}; \,z_{2j+1}, \boldsymbol p^{(2j+1)}\big) \overline{a_{\boldsymbol\ell}\big(z_{2j+2}, \boldsymbol p^{(2j+2)};\,z_{2j+1}, \boldsymbol p^{(2j+1)} \big)},
	\end{aligned}
\end{equation}
say $S(K,R)$, where we have exchanged the order of the above two summations.

We first fix the pairs $(z_j,\,\boldsymbol p^{(j)})$, $0\le j\le2R$, with $z_0=z_{_{2R}}=(1,\,0)^T$. Substituting the expression \eqref{def-a_ell} for $a_{\boldsymbol\ell}$ into the product above, we get $2R$ sums and for each of them, we fix a choice of $\boldsymbol\pi,\,\boldsymbol\sigma\in\mathfrak S_{J+1}^K$.
The sum over $P$ is still restricted by $\widetilde gz_j\in\Omega'$ and by \eqref{eq:goodone} and \eqref{eq:goodtwo} at $(\widetilde gz_j,\,\boldsymbol p^{(j)})$ for every $0\le j\le2R$. 
To apply \eqref{eq:position-count} to the inner sum over $P$ subject to these conditions, we first bound the coordinates of $z_j=(h_j,\,k_j)^T$.

For $0\le j\le2R$, we put $P_j=(u_j,\,v_j)^T:=T^{-1}\widetilde gz_j$, and more precisely
\begin{equation*}
\begin{aligned}
P_j=\begin{pmatrix}1&0\\[0,3mm]1&-2\end{pmatrix}
\begin{pmatrix}u&-2ur\\[0,3mm]\frac{u-v}{2}&r(v-u)+\frac{1}{u}\end{pmatrix}\begin{pmatrix}h_j\\[0,3mm]k_j\end{pmatrix}
=\begin{pmatrix}u(h_j-2rk_j)\\[0,3mm]v(h_j-2rk_j)-\frac{2k_j}{u}
\end{pmatrix}
=\begin{pmatrix}u\tau(z_j)\\[0,3mm]v\tau(z_j)-\frac{2k_j}{u}
\end{pmatrix}\in\Omega.
\end{aligned}
\end{equation*}
Here $P_j\in\Omega$ follows from $\widetilde gz_j\in\Omega'$ and $\Omega'=T\Omega$. The support of $\psi_1(t'/Y)$ in $K_{1,\boldsymbol\ell}$ implies that $ \big|\det(z_j,\,z_{j+1})\big|=D|t'|\ll d_0Y$ for any $0\le j<2R$.
Since $\det(\widetilde g)=1$, $\det(T^{-1})=-2$, we have
\begin{equation*}
	\big|\det(P_j,\,P_{j+1})\big|=\big|-2\det(z_j,\,z_{j+1})\big|
	\ll d_0Y.
\end{equation*}
Moreover, $P_j\in\Omega$ gives $d_0YM\le u_j\le16d_0YM$. Hence, with $P_0=T^{-1}\widetilde gz_0=(u,\,v)^T$, we get
\begin{equation*}
	\biggl|\frac{v_j}{u_j}-\frac vu\biggr|
	\le \sum_{\nu=0}^{j-1}\bigg|\frac{v_{\nu+1}}{u_{\nu+1}}
	-\frac{v_\nu}{u_\nu}\bigg|
	=\sum_{\nu=0}^{j-1}
	\frac{\big|\det(P_{\nu},\,P_{\nu+1})\big|}{u_{\nu} u_{\nu+1}}
	\ll\frac{R}{d_0YM^2}
\end{equation*}
for any $1\le j\le2R.$
The explicit coordinate expression for $P_j=(u_j,\,v_j)^T$ above gives
\begin{equation*}
	\tau(z_j)=\frac{u_j}{u}\in[1/16,\,16],\quad\ \frac{v_j}{u_j}-\frac vu=\frac{v_j}{u_j}-\frac{v\tau(z_j)}{u_j} =-\frac{2k_j}{uu_j}.
\end{equation*}
Then, using the bounds $u,u_j\le16d_0YM$ and $0\le r<1$, we conclude that
\begin{equation*}
	|h_j|=|\tau(z_j)+2rk_j|\ll Rd_0Y,\quad\ |k_j|=\frac{uu_j}{2}\biggl|\frac{v_j}{u_j}-\frac vu\biggr|\ll Rd_0Y.
\end{equation*}

We now estimate the inner sum over $P$ for the fixed pairs $(z_j,\,\boldsymbol p^{(j)})$ and permutations $\boldsymbol\pi,\,\boldsymbol\sigma$. 
For each $P=P_0=P_{2R}=(u,\,v)^T\in\Omega\cap\mathcal C$, let $g'$ be as in \eqref{eq:gprime}, and define the tuple of lines
\begin{align}\label{def-ell(P)}
	\boldsymbol\ell(P):=\bigl(\ell_{p_1}(g'),\,\ldots,\,\ell_{p_m}(g')\bigr), \quad \text{with}\quad p_1<\cdots<p_m\in \bigcup\nolimits_{1\leq i\leq K}\mathcal P_i
\end{align}

Let $s=u/(d_0YM)$ and $t_0=v/N$, and for $\widetilde g=\widetilde g(d_0YM s,\,Nt_0,\,r)$, we define the set
\begin{equation}\label{def-mathscrB}
	\mathscr B:=\left\{(s,\,t_0,\,r)\in[1,\,16]\times[1,\,2]\times[0,\,1):\
	\begin{aligned}
		&\text{\eqref{eq:goodone}, \eqref{eq:goodtwo} both hold at }
		(\widetilde gz_j,\,\boldsymbol p^{(j)})\\
		&\text{for any }0\le j\le2R,\quad \widetilde gz_j\in\Omega'
	\end{aligned}
	\right\}.
\end{equation}
With $S$ and the lines $\ell_p$ as in \eqref{def-ell_p}, we divide $[1,\,16]\times[1,\,2]\times[0,\,1)$ into small rectangular boxes $\Box$, each with side lengths $\asymp x^{-c_\delta/8}$. 
Let $\mathscr Q$ be the set of these boxes, and define
\begin{equation*}
	\mathscr Q_{\mathrm{in}}:=\{\Box\in\mathscr Q:\ \Box\subseteq\mathscr B\},\qquad
	\mathscr Q_{\partial}:=\{\Box\in\mathscr Q:\ \Box\cap\mathscr B\ne\varnothing,\ \Box\not\subseteq\mathscr B\}.
\end{equation*}
For $\Box\in \mathscr Q_{\mathrm{in}}$, we shall apply \eqref{eq:position-count} directly later. For any $\Box\in \mathscr Q_{\partial}$, applying \eqref{eq:position-count}, we get
\begin{equation*}
	\begin{aligned}
	&\frac1{d_0YMN}
	\sum_{\substack{P=(u,\,v)^T\in\Omega\cap\mathcal C\\[0,4mm]
	(\frac{u}{d_0YM},\,\frac{v}{N},\,\frac{c}{u})\in \Box\cap
	\mathscr B\\[0,4mm] \ell_p(g')=\ell_p\ \text{for any }p\mid S}}1
	\le\frac1{d_0YMN}
	\sum_{\substack{P=(u,\,v)^T\in\Omega\cap\mathcal C\\[0,4mm]
	(\frac{u}{d_0YM},\,\frac{v}{N},\,\frac{c}{u})\in \Box \\[0,4mm]
	\ell_p(g')=\ell_p\ \text{for any }p\mid S}}1
	=\frac{\operatorname{vol}(\Box)}{3\zeta(2)\prod_{p\mid S}(p+1)} +O(x^{-c_\delta}).
	\end{aligned}
\end{equation*}
Summing this upper bound over all $\Box\in\mathscr Q_{\partial}$, we obtain
\begin{equation*}
	\begin{aligned}
	\frac1{d_0YMN}\sum_{\Box\in\mathscr Q_{\partial}}
	\sum_{\substack{P=(u,\,v)^T\in\Omega\cap\mathcal C\\[0,4mm]
	(\frac{u}{d_0YM},\,\frac{v}{N},\,\frac{c}{u})\in \Box\cap
	\mathscr B\\[0,4mm] \ell_p(g')=\ell_p\ 
	\text{for any }p\mid S}}1	
	\ll\sum_{\Box\in\mathscr Q_{\partial}}\operatorname{vol}(\Box)
	+\#\mathscr Q_{\partial}\,x^{-c_\delta}
	\ll\#\mathscr Q_{\partial} \bigl(x^{-\frac{3}{8}c_\delta}+x^{-c_\delta}\bigr).
	\end{aligned}
\end{equation*}

We are left to bound $\#\mathscr Q_{\partial}$. We write $\Box=I_1\times I_2\times I_3$, and we first consider the subset
\begin{equation*}
	S_{\Box}:=\Big\{\Box\in\mathscr Q_{\partial}:\ \exists(s,\,t_0,\,r)\in \Box,\, \exists j\in\{0,\,\ldots,\,2R\} \text{ such that } \widetilde g(d_0YM s,\,Nt_0,\,r)z_j\notin\Omega'\Big\}.
\end{equation*}
Since $\Omega'=T\Omega$, the condition $\widetilde gz_j\in\Omega'$ is equivalent to $P_j\in\Omega$, whence $1/16\le\tau(z_j)\le16$ and
\begin{equation}\label{faces}
	\begin{aligned}
		\frac1{\tau(z_j)}&\le s\le\frac{16}{\tau(z_j)},\quad
		\frac1{\tau(z_j)}+\frac{2k_j}{d_0YMN s\tau(z_j)}
		\le t_0\le\frac2{\tau(z_j)}+\frac{2k_j}{d_0YMN s\tau(z_j)}.
	\end{aligned}
\end{equation}
For $c_1\in\{1,\,16\}$ and $c_2\in\{1,\,2\}$, differentiating the bound functions in \eqref{faces} gives
\begin{equation}\label{different}
	\begin{aligned}
	\frac{\partial}{\partial r}\biggl(\frac{c_1}{\tau(z_j)}\biggr)
	&=\frac{2c_1k_j}{\tau(z_j)^2},\quad\ 
	\frac{\partial}{\partial s}\biggl(\frac{c_2}{\tau(z_j)} +\frac{2k_j}{d_0YMN s\tau(z_j)}\biggr)
	=-\frac{2k_j}{d_0YMN s^2\tau(z_j)},
\\[1mm]
	&\frac{\partial}{\partial r} \biggl(\frac{c_2}{\tau(z_j)} +\frac{2k_j}{d_0YMN s \tau(z_j)}\biggr)
	=\frac{2c_2k_j}{\tau(z_j)^2}+\frac{4k_j^2}{d_0YMN s\tau(z_j)^2}.
	\end{aligned}
\end{equation}
Recall that $R=\lceil L^{1/2}/2\rceil$, $J\le L^{0.01}$ and every selected prime is at most $\exp(2L^{0.2})$. By \eqref{ell_0-D} and the support of $\eta(a/Y)$, one may deduce that $\log d_0\le2KJ L^{0.2}$, $\log Y\le2K L^{0.2},$ whence
\begin{equation*}
	\log(Rd_0Y)\le O(\log L)+2K(J+1)L^{0.2}=O_K(L^{0.21}).
\end{equation*}
These estimates, together with $|k_j|\ll Rd_0Y$, $s\ge1$ and $\tau(z_j)\ge1/16$, imply that each derivative in \eqref{different} is bounded by $\exp\big(O_K(L^{0.21})\big)$. One can verify that the bound $\exp\big(O_K(L^{0.21})\big)$ also holds for those $\Box\in\mathscr Q_{\partial}$, and thereby for every $\Box\in\mathscr Q_{\mathrm{in}}\cup\mathscr Q_{\partial}$.

Fix $j$ and the intervals $I_2,I_3$. Consider the boxes $\Box=I_1\times I_2\times I_3\in\mathscr Q_{\partial}$ containing $(s,\,t_0,\,r)$ with $s<1/\tau(z_j)$ or $s>16/\tau(z_j)$. Since $\Box\cap\mathscr B\ne\varnothing$, at least one of the continuous functions $(s,r)\mapsto s-c_1/\tau(z_j)=s-c_1/(h_j-2rk_j)$ vanishes in $I_1\times I_3$.
For each fixed $c_1$, the mean value theorem and the preceding derivative bound imply that the values of $c_1/\tau(z_j)$ with $r\in I_3$ lie in an interval of length at most $x^{-c_\delta/8}\exp(O_K(L^{0.21}))$. Since the intervals $I_1$ are disjoint and have length $\asymp x^{-c_\delta/8}$, thus there are at most $\exp(O_K(L^{0.21}))$ choices of $I_1$ such that $c_1/\tau(z_j)\in I_1$ for some $r\in I_3$.
Therefore, at most $\exp(O_K(L^{0.21}))$ boxes $\Box\in\mathscr Q_{\partial}$ contain $(s,\,t_0,\,r)$ with $s\notin[1/\tau(z_j),\,16/\tau(z_j)]$.
For the case when $I_1,I_3$ are fixed, applying the same argument to the inequalities for $t_0$ in \eqref{faces}, we get the same bound for the number of choices of $I_2$.

There are $O(x^{c_\delta/4})$ choices of $(I_2,I_3)$ and also of $(I_1,I_3)$. Summing the preceding bounds over these interval choices and over $0\le j\le2R$, we obtain
\begin{equation}\label{bound-S-Box}
	\#S_\Box\ll(2R+1)x^{c_\delta/4}\exp(O_K(L^{0.21}))=x^{c_\delta/4+o(1)}.
\end{equation}

We next bound $\#(\mathscr Q_{\partial}\setminus S_\Box)$. By definition, for any $\Box\in\mathscr Q_{\partial}\setminus S_\Box$, we have $\widetilde gz_j\in\Omega'$ for all $(s,\,t_0,\,r)\in \Box$ and all $0\le j\le2R$. Since $\Box\not\subseteq\mathscr B$, there exists $(s,\,t_0,\,r)\in \Box\setminus\mathscr B$. For this choice of $(s,\,t_0,\,r)$, $\widetilde gz_j\in\Omega'$ holds for any $0\le j\le2R$. Thus, \eqref{eq:goodone} or \eqref{eq:goodtwo} fails at $(\widetilde gz_j,\,\boldsymbol p^{(j)})$ for at least one $j$ by  the definition of $\mathscr B$.
For each $z_j$, choose $w_j\in\mathbb Z^2$ with $\det(z_j,\,w_j)=1$ and $|w_j|\ll1+|z_j|$. Then
\begin{equation}\label{eq:rderiv}
	r_{\tilde gz_j}\equiv\frac{\tau(w_j)}{\tau(z_j)}\,(\textrm{mod}\,1),\qquad
	\frac{\partial}{\partial r}\frac{\tau(w_j)}{\tau(z_j)}=-\frac2{\tau(z_j)^2}.
\end{equation}

Recall that \eqref{eq:goodone} is of the form $\|a\tau(w_j)/\tau(z_j)\|>\vartheta$, with $(a,\,\vartheta)=(\ell\widetilde D,\,Y^{-0.7})$; and
in \eqref{eq:goodtwo}, a tuple $(q_i)_{i\notin I}$ is counted  if at least one choice of $h_i,h_i',q_j,q_j'$ with $\widetilde D Z_I\ne\widetilde D'Z'_I$ also satisfies the inequality of the form $\|a\tau(w_j)/\tau(z_j)\|<\vartheta$, with
$(a,\,\vartheta)=\bigl((\widetilde D Z_I-\widetilde D'Z'_I)T_f,\,100\exp(-L^{a_{i_*}})\bigr).$
Moreover, for fixed $(a,\,\vartheta)$, the function $r\mapsto\|a\tau(w_j)/\tau(z_j)\|$ is continuous where $\tau(z_j)>0$. 
Hence, the validity of either \eqref{eq:goodone} or \eqref{eq:goodtwo} will change only when $\bigl\|a\frac{\tau(w_j)}{\tau(z_j)}\bigr\|=\vartheta$, that is,
\begin{equation*}
	a\frac{\tau(w_j)}{\tau(z_j)}=n+\vartheta
	\quad\text{or}\quad
	a\frac{\tau(w_j)}{\tau(z_j)}=n-\vartheta,
	\qquad n\in\mathbb Z.
\end{equation*}

Recall $\tau(w_j)=w_{j,1}-2rw_{j,2}$ with $w_j=(w_{j,1},\,w_{j,2})$. Define the set $\mathscr R$ of $r$ by
\begin{equation*}
	\mathscr R:=\bigcup_{j=0}^{2R}\,
	\bigcup_{\substack{(a,\,\vartheta)}}
	\biggl\{r\in[0,\,1):\ \tau(z_j)\ge1/16,\ \exists\, n\in\mathbb Z\ \text{such that }a\frac{\tau(w_j)}{\tau(z_j)}=n\pm\vartheta\biggr\}.
\end{equation*}
We next bound $\#\mathscr R$.
For each fixed $j$, the pairs $(a,\,\vartheta)$ defined above satisfy
\begin{equation*}
	\#\{(a,\,\vartheta)\}\le\exp(O_K(L^{0.21})),\qquad
	0<|a|\le\exp(O_K(L^{0.21})),\qquad 0<\vartheta\ll1.
\end{equation*}
To bound the number of choices of $n$, we use $0\le r<1$ and $\tau(z_j)\ge1/16$ to obtain
\begin{equation*}
	\bigl|\tau(w_j)/\tau(z_j)\bigr|
	\le16\bigl(|w_{j,1}|+2|w_{j,2}|\bigr)
	\ll1+|z_j|\ll1+Rd_0Y.
\end{equation*}
This, together with the bound $\log(Rd_0Y)\ll L^{0.21}$, shows that the number of $n$ counted in $\mathscr R$ is 
\begin{equation*}
	\ll  |a|\bigl|\tau(w_j)/\tau(z_j)\bigr|+\vartheta
	\ll |a|(1+Rd_0Y)+1
	\le\exp(O_K(L^{0.21})).
\end{equation*}
Moreover, for fixed $j,a,\vartheta,n$, by \eqref{eq:rderiv} one finds that $	\frac{\partial}{\partial r}\bigl(a\frac{\tau(w_j)}{\tau(z_j)}\bigr)
=-2a/\tau(z_j)^2\ne0$, whence  the equation in $\mathscr R$ has at most one solution $r$ with $\tau(z_j)=z_{j,1}-2rz_{j,2}>0$.
Finally, counting the choices of $j$, $(a,\,\vartheta)$, together with the number of $n$ estimated above,  we obtain
\begin{equation*}
	\#\mathscr R
	\le(2R+1)\exp(O_K(L^{0.21}))\exp(O_K(L^{0.21}))\cdot2
	=\exp(O_K(L^{0.21})).
\end{equation*}

We shall use this bound for $\#\mathscr R$ to estimate $\#(\mathscr Q_{\partial}\setminus S_{\Box})$. For $\Box=I_1\times I_2\times I_3\in\mathscr Q_{\partial}\setminus S_{\Box}$, we first claim that $I_3\cap\mathscr R\ne\varnothing$.
Otherwise, if $I_3\cap\mathscr R=\varnothing$, then, for each $j$, the validity of conditions \eqref{eq:goodone} and \eqref{eq:goodtwo} is independent of $r\in I_3$. This, together with $\Box\cap\mathscr B\ne\varnothing$, shows that \eqref{eq:goodone} and \eqref{eq:goodtwo} both hold for all $(s,\,t_0,\,r)\in\Box$. 
Besides, for any $0\le j\le2R$, by the definition of $S_{\Box}$, one has $\widetilde gz_j\in\Omega'$ for all $(s,\,t_0,\,r)\in\Box$ since $\Box\notin S_{\Box}$. Thus all the conditions defining $\mathscr B$ hold for all $(s,\,t_0,\,r)\in\Box$, whence $\Box\subseteq\mathscr B$, which contradicts to $\Box\in\mathscr Q_{\partial}$.
Since the intervals $I_3$ make a disjoint partition of $[0,\,1)$, each $r\in\mathscr R$ lies in exactly one $I_3$. Hence the condition $I_3\cap\mathscr R\ne\varnothing$ as claimed, allows at most $\#\mathscr R$ choices of $I_3$.
Recall that there are also $O(x^{c_\delta/4})$ choices of $(I_1,\,I_2)$. Hence
\begin{equation*}
	\#(\mathscr Q_{\partial}\setminus S_{\Box})\ll x^{c_\delta/4}\#\mathscr R
	\le x^{c_\delta/4}\exp(O_K(L^{0.21}))
	=x^{c_\delta/4+o(1)}.
\end{equation*}

Combining this with the bound \eqref{bound-S-Box}, we have $\#\mathscr Q_{\partial}
=\#S_{\Box}+\#(\mathscr Q_{\partial}\setminus S_{\Box})
\ll x^{c_\delta/4+o(1)}.$
Since each box has volume $\asymp x^{-3c_\delta/8}$, it follows that
\begin{equation*}
	\sum_{\Box\in\mathscr Q_{\partial}}\operatorname{vol}(\Box)
	\ll\#\mathscr Q_{\partial}\,x^{-3c_\delta/8}
	\ll x^{-c_\delta/8+o(1)}.
\end{equation*}

Now we can estimate the sum over $P$ with
$(s,\,t_0,\,r)\in\mathscr B$.
By the following inclusions
\begin{equation*}
	\bigcup_{\Box\in\mathscr Q_{\mathrm{in}}}\Box
	\subseteq\mathscr B
	\subseteq
	\bigcup_{\Box\in\mathscr Q_{\mathrm{in}}\cup\mathscr Q_{\partial}}\Box,
\end{equation*}
we deduce that
\begin{align}\label{sum-P}
	\sum_{\Box\in\mathscr Q_{\mathrm{in}}}
	\sum_{\substack{P\\ (s,\,t_0,\,r)\in \Box}}1
	\leq 
	\sum_{\substack{P\\ (s,\,t_0,\,r)\in \mathscr B}}1
	=\sum_{\substack{P=(u,\,v)^T\in\Omega\cap\mathcal C\\[0,4mm]
		(u/(d_0YM),\,v/N,\,c/u)\in\mathscr B\\[0,4mm]
		\ell_p(g')=\ell_p\ \text{for any }p\mid S}}1
	\leq
	\sum_{\Box\in\mathscr Q_{\mathrm{in}}\cup\mathscr Q_{\partial}} \sum_{\substack{P\\ (s,\,t_0,\,r)\in \Box}}1.
\end{align}
Applying \eqref{eq:position-count} to each inner sum over $P$ on the left-hand side of \eqref{sum-P}, one gets
\begin{equation*}
	\frac1{d_0YMN}\sum_{\Box\in\mathscr Q_{\mathrm{in}}}\sum_{\substack{P\\(s,\,t_0,\,r)\in\Box}}1
	=\frac{\sum_{\Box\in\mathscr Q_{\mathrm{in}}}\operatorname{vol}(\Box)}{3\zeta(2)\prod_{p\mid S}(p+1)}
	+O(x^{-5c_\delta/8}),
\end{equation*}
since there are at most $\#\mathscr Q\ll x^{3c_\delta/8}$ boxes.
The same asymptotic holds for the sum on the right-hand side, with $\mathscr Q_{\mathrm{in}}$ replaced by $\mathscr Q_{\mathrm{in}}\cup\mathscr Q_{\partial}$.
Moreover, by the inclusions above, we also have
\begin{equation*}
	0\le\operatorname{vol}(\mathscr B)-\sum_{\Box\in\mathscr Q_{\mathrm{in}}}\operatorname{vol}(\Box)
	\le\sum_{\Box\in\mathscr Q_{\partial}}\operatorname{vol}(\Box)
	\ll x^{-c_\delta/8+o(1)}.
\end{equation*}
Thus, the two asymptotic formulas above have the same main term, differing by an error term $O(x^{-c_\delta/8+o(1)})$:
\begin{equation}\label{eq:fixed-position-count}
	\frac1{d_0YMN}
	\sum_{\substack{P=(u,\,v)^T\in\Omega\cap\mathcal C\\[0,4mm]
			(u/(d_0YM),\,v/N,\,c/u)\in\mathscr B\\[0,4mm]
			\ell_p(g')=\ell_p\ \text{for any }p\mid S}}1
	=\frac{\operatorname{vol}(\mathscr B)}{3\zeta(2)\prod_{p\mid S}(p+1)}
	+O(x^{-c_\delta/8+o(1)}),
\end{equation}
uniformly in $S$ and in the choices of the lines $\ell_p$.

We return to the inner sum over $P$ in $S(K,R)$ defined in \eqref{def-S(K,R)}, with  $z_j$, $\boldsymbol p^{(j)}$, and the permutations $\boldsymbol\pi,\,\boldsymbol\sigma$ in  $a_{\boldsymbol\ell}$ fixed.
Recall $s=u/(d_0YM)$, $t_0=v/N$ and $r=c/u$ for each $P=(u,\,v)^T$. By the definition of $\mathscr B$, the conditions $\widetilde gz_j\in\Omega'$ and \eqref{eq:goodone} and \eqref{eq:goodtwo} at each $(\widetilde gz_j,\,\boldsymbol p^{(j)})$ with $0\le j\le2R$, are equivalent to $(s,\,t_0,\,r)\in\mathscr B$.
Moreover, the summation condition $(z_j,\,\boldsymbol p^{(j)})\in\mathscr S_{\boldsymbol\ell}$ also requires $	[z_j]_{p_{i,h}^{(j)}}=\ell_{p_{i,h}^{(j)}}$ for $0\le j\le2R,\ 1\le i\le K,\ 1\le h\le J+1.$ With  the factors $\rho^{(\omega_{\boldsymbol\ell}(z_j)+\omega_{\boldsymbol\ell}(z_{j+1}))/2}$ coming from the $2R$ kernels $K_{1,\boldsymbol\ell}$, we define the product $F(\boldsymbol\ell)$ by
\begin{equation}\label{def-F(ell)}
	F(\boldsymbol\ell):=
	\biggl(\prod_{j=0}^{2R}\prod_{i=1}^K\prod_{h=1}^{J+1}
	\mathbf1_{[z_j]_{p_{i,h}^{(j)}}=\ell_{p_{i,h}^{(j)}}}\biggr)
	\prod_{j=0}^{2R-1}\rho^{\frac{\omega_{\boldsymbol\ell}(z_j)+\omega_{\boldsymbol\ell}(z_{j+1})}{2}}.
\end{equation}
The other factors in the kernels are independent of $P$ and $\boldsymbol\ell$. 
In the sum over $P$ in $S(K,R)$, we take $\boldsymbol\ell=\boldsymbol\ell(P)$ with $\boldsymbol\ell(P)$ defined in \eqref{def-ell(P)}.
Thus, we need to treat the sum
\begin{equation*}
	\frac1{d_0YMN}
	\sum_{\substack{P=(u,\,v)^T\in\Omega\cap\mathcal C\\[0,1mm](s,\,t_0,\,r)\in\mathscr B}} F\big(\boldsymbol\ell(P)\big).
\end{equation*}

The expression $F(\boldsymbol\ell(P))$ involves $\ell_p(g')$ for every prime $p\in\bigcup_i\mathcal P_i$. We first write explicitly the contribution of each prime to $F(\boldsymbol\ell(P))$.
By the definition of $\omega_{\boldsymbol\ell(P)}$, we have
\begin{equation*}
	\prod\nolimits_{0\leq j\leq 2R-1} \rho^{\frac{\omega_{\boldsymbol\ell(P)}(z_j)+\omega_{\boldsymbol\ell(P)}(z_{j+1})}{2}}
	=\prod\nolimits_{p\in\bigcup_i\mathcal P_i} \rho^{\frac12\sum_{j=0}^{2R-1}\big(\mathbf1_{\ell_p(g')=[z_j]_p}
	+\mathbf1_{\ell_p(g')=[z_{j+1}]_p}\big)}.
\end{equation*}
For each $p\in\bigcup_i\mathcal P_i$, define
\begin{equation*}
	H_p(\ell):=1-\rho^{\frac12\sum_{j=0}^{2R-1}
	\big(\mathbf1_{\ell=[z_j]_p}+\mathbf1_{\ell=[z_{j+1}]_p}\big)}
	\qquad\text{for }\ell\in\mathbb P^1(\mathbb F_p).
\end{equation*}
Let $S_0$ be the product of all distinct primes among $p_{i,h}^{(j)}$, $0\le j\le2R$, $1\le i\le K$ and $1\le h\le J+1$, and let
\begin{equation*}
	F_0(\boldsymbol\ell(P)):=\biggl(\prod_{j=0}^{2R}\prod_{i=1}^K\prod_{h=1}^{J+1}\mathbf1_{\ell_{p_{i,h}^{(j)}}(g')=[z_j]_{p_{i,h}^{(j)}}}\biggr)\prod_{p\mid S_0} \Big(1-H_p\big(\ell_p(g')\big)\Big)\in [0,\,1],
\end{equation*}
then we have
\begin{align*}
	F(\boldsymbol\ell(P))&=F_0\big(\boldsymbol\ell(P)\big)
	\prod\nolimits_{\substack{p\in\bigcup_i\mathcal P_i,\,p\nmid S_0}}\Big(1-H_p\big(\ell_p(g')\big)\Big)
\\
	&=F_0\big(\boldsymbol\ell(P)\big)\sum_{k\ge0}(-1)^k
	\sum_{\substack{q_1<\cdots<q_k\\[0,3mm] (q_1\cdots q_k,\, S_0)=1}} \prod_{\nu=1}^k H_{q_\nu}\big(\ell_{q_\nu}(g')\big),
\end{align*}
where $q_\nu \in\bigcup_i\mathcal P_i,\, 1\leq \nu\leq k$. Note that the above sum is finite.

To apply \eqref{eq:fixed-position-count}, in which the conditions $\ell_p(g')=\ell_p$ with $p\mid S,\, S\le\exp(O(L^{0.98}))$ need to be verified, we truncate the above sum over $k$. Take $B=\lceil L^{3/4}\rceil$ and define
\begin{equation*}
	F_{<B}(\boldsymbol\ell(P)):
	=F_0(\boldsymbol\ell(P))\sum_{k=0}^{B-1}(-1)^k
	\sum_{q_1<\cdots<q_k}\prod_{\nu=1}^k H_{q_\nu}(\ell_{q_\nu}(g'))
	=F_0(\boldsymbol\ell(P))\sum_{k=0}^{B-1}(-1)^k
	E_k(\boldsymbol\ell(P)),
\end{equation*}
say, and we also use the notation $E_k(\boldsymbol\ell(P))$ for $k=B$. By inclusion--exclusion, one has
\begin{equation}\label{inclu-exclu}
	\big|F(\boldsymbol\ell(P))-F_{<B}(\boldsymbol\ell(P))\big|\le F_0(\boldsymbol\ell(P))E_B(\boldsymbol\ell(P))\le E_B(\boldsymbol\ell(P)),
\end{equation}
Thus, we only need to estimate the sums with respects to $F_{<B}(\boldsymbol\ell(P))$ and $E_B(\boldsymbol\ell(P))$, respectively.
For $F_{<B}(\boldsymbol\ell(P))$, we shall take modulus $S=S_0q_1\cdots q_k$, whence we have $\Omega(S)\ll RKJ+k$, $P^+(S)\leq \exp(2L^{0.2})$, and for $0\le k\le B$,
\begin{equation*}
	\log S\ll \bigl(RKJ+k\bigr)L^{0.2}\ll_K R(J+1)L^{0.2}+BL^{0.2}\ll_K L^{0.96}.
\end{equation*}
Thus \eqref{eq:fixed-position-count} applies with $S=S_0q_1\cdots q_k$. As for $E_B(\boldsymbol\ell(P))$,  the moduli is $S=q_1\cdots q_B$, which also satisfies the required bound in \eqref{eq:fixed-position-count}.

We first consider the sum of $E_B(\boldsymbol\ell(P))$ over $P$. For each fixed tuple $q_1<\cdots<q_B$, we group the vectors $P$ according to the tuple $(\ell_{q_1}(g'),\,\ldots,\,\ell_{q_B}(g'))$, giving
\begin{equation}\label{group}
	\begin{aligned}
		\sum_{\substack{P=(u,\,v)^T\in\Omega\cap\mathcal C\\(s,\,t_0,\,r)\in\mathscr B}}E_B(\boldsymbol\ell(P))
		=\sum_{q_1<\cdots<q_B}\,
		\sum_{\substack{\ell_{q_\nu}\in\mathbb P^1(\mathbb F_{q_\nu})\\1\le\nu\le B}}
		\biggl(\prod_{\nu=1}^B H_{q_\nu}(\ell_{q_\nu})\biggr)
		\sum_{\substack{P=(u,\,v)^T\in\Omega\cap\mathcal C\\(s,\,t_0,\,r)\in\mathscr B\\
				\ell_{q_\nu}(g')=\ell_{q_\nu}\ \text{for }1\le\nu\le B}}1.
	\end{aligned}
\end{equation}
Now we can apply \eqref{eq:fixed-position-count} to the inner sum over $P$. By calculate, one may find that the error term is at most $x^{-c_\delta/16}$. Consequently, we get
\begin{equation}\label{E-B-estimate}
	\begin{aligned}
	\frac1{d_0YMN}\sum_{\substack{P\in\Omega\cap\mathcal C\\ (s,\,t_0,\,r)\in\mathscr B}}E_B(\boldsymbol\ell(P))
	=\frac{\operatorname{vol}(\mathscr B)}{3\zeta(2)}
	\sum_{q_1<\cdots<q_B}\,\prod_{\nu=1}^B
	\biggl(\sum_{\ell_{q_\nu}\in\mathbb P^1 (\mathbb F_{q_\nu})} \frac{H_{q_\nu}(\ell_{q_\nu})}{q_\nu+1} \biggr)
		+O(x^{-\frac{c_\delta}{16}}).
	\end{aligned}
\end{equation}
To estimate the sums over lines in the main term, we note that the definition of $H_p$ gives
\begin{equation*}
	0\le H_p(\ell_p)\le\mathbf1_{\ell_p\in\{[z_j]_p:\ 0\le j\le2R\}},\qquad
	\frac1{p+1}\sum_{\ell_p\in\mathbb P^1(\mathbb F_p)}H_p(\ell_p)\le\frac{2R+1}{p+1}.
\end{equation*}
This, together with the estimate $\operatorname{vol}(\mathscr B)\leq 15$, shows that the main term in \eqref{E-B-estimate} is 
\begin{equation}\label{E-B-estimate-main}
	\begin{aligned}
	\ll \sum_{q_1<\cdots<q_B}\,\prod_{\nu=1}^B
	\biggl(\sum_{\ell_{q_\nu}} \frac{H_{q_\nu}(\ell_{q_\nu})}{q_\nu+1} \biggr)
	\ll\frac1{B!}\biggl(\sum_{\substack{p\in\bigcup_i\mathcal P_i\\p\nmid S_0}}\frac{R}{p}\biggr)^B
	\ll\frac{(C_KR)^B}{B!} \ll\exp(-cL^{3/4})
	\end{aligned}
\end{equation}
for some $c>0$. Here we have used $\sum_{p\in\bigcup_i\mathcal P_i}(p+1)^{-1}\le\sum_iV_i\ll_K1$, $B!\ge(B/e)^B$ and $R\asymp L^{1/2}$.
Combining \eqref{E-B-estimate-main}, \eqref{E-B-estimate} and \eqref{inclu-exclu}, one has
\begin{equation}\label{F-F<}
	\frac1{d_0YMN}
	\sum_{\substack{P\in\Omega\cap\mathcal C\\ 
			(s,\,t_0,\,r)\in\mathscr B}}
	\bigl|F(\boldsymbol\ell(P))-F_{<B}(\boldsymbol\ell(P))\bigr|
	\ll  \exp(-cL^{3/4}).
\end{equation}

We next treat the sum of $F_{<B}(\boldsymbol\ell(P))$ over $P$ with $S=S_0q_1\cdots q_k$. As in \eqref{group}, we may write
%we group the $P$ according to the values of $\ell_p(g')$ for 
%$p\mid S$, giving
\begin{equation*}
	\begin{aligned}
	\sum_{\substack{P=(u,\,v)^T\in\Omega\cap\mathcal C\\
	(s,\,t_0,\,r)\in\mathscr B}}F_{<B}(\boldsymbol\ell(P))
	&=\sum_{k=0}^{B-1}(-1)^k\sum_{q_1<\cdots<q_k}
	\sum_{\substack{\ell_p\in\mathbb P^1(\mathbb F_p)\\p\mid S}}
	\biggl(\prod_{j=0}^{2R}\prod_{i=1}^K\prod_{h=1}^{J+1}
	\mathbf1_{\ell_{p_{i,h}^{(j)}}=[z_j]_{p_{i,h}^{(j)}}}\biggr)
	\\&\quad\times\prod_{p\mid S_0}\bigl(1-H_p(\ell_p)\bigr)
	\prod_{\nu=1}^kH_{q_\nu}(\ell_{q_\nu})
	\sum_{\substack{P=(u,\,v)^T\in\Omega\cap\mathcal C\\
	(s,\,t_0,\,r)\in\mathscr B\\
	\ell_p(g')=\ell_p\ \text{for any }p\mid S}}1.
	\end{aligned}
\end{equation*}
We then apply \eqref{eq:fixed-position-count} to the innermost
sum over $P$. By calculation, the total error term is again
$\ll x^{-c_\delta/16}$. Consequently, we obtain
\begin{equation*}
	\begin{aligned}
	\frac1{d_0YMN}\sum_{\substack{P\in\Omega\cap\mathcal C\\
	(s,\,t_0,\,r)\in\mathscr B}}F_{<B}(\boldsymbol\ell(P))
    =\frac{\operatorname{vol}(\mathscr B)}{3\zeta(2)}
	\frac{\sum_{\boldsymbol\ell}F_{<B}(\boldsymbol\ell)}
	{\prod_{p\in\bigcup_i\mathcal P_i}(p+1)}
	+O(x^{-c_\delta/16}),
	\end{aligned}
\end{equation*}
where
\begin{equation*}
	\begin{aligned}
		F_{<B}(\boldsymbol\ell):={}&
		\biggl(
		\prod_{j=0}^{2R}\prod_{i=1}^K\prod_{h=1}^{J+1}
		\mathbf1_{\ell_{p_{i,h}^{(j)}}=[z_j]_{p_{i,h}^{(j)}}}
		\biggr)
		\prod_{p\mid S_0}\bigl(1-H_p(\ell_p)\bigr)
		\sum_{k=0}^{B-1}(-1)^k
		\sum_{q_1<\cdots<q_k}
		\prod_{\nu=1}^kH_{q_\nu}(\ell_{q_\nu}).
	\end{aligned}
\end{equation*}
Inserting this asymptotic into \eqref{F-F<}, one has \begin{equation*}%\label{F-asy}
	\frac1{d_0YMN} \sum_{\substack{P\in\Omega\cap\mathcal C\\
	(s,\,t_0,\,r)\in\mathscr B}}F(\boldsymbol\ell(P))
	=\frac{\operatorname{vol}(\mathscr B)}{3\zeta(2)}
	\frac{\sum_{\boldsymbol\ell}F_{<B}(\boldsymbol\ell)}
	{\prod_{p\in\bigcup_i\mathcal P_i}(p+1)}
	+O\bigl(\exp(-cL^{3/4})\bigr).
\end{equation*}

We next replace $F_{<B}(\boldsymbol\ell)$ by $F(\boldsymbol\ell)$ defined by \eqref{def-F(ell)} in the preceding asymptotic formula.
Similar to \eqref{E-B-estimate-main}, one may verify  that $\prod_{p\in\cup_i\mathcal P_i}(p+1)^{-1}\sum_{\boldsymbol\ell}
|F(\boldsymbol\ell)-F_{<B}(\boldsymbol\ell)|
\ll\exp(-cL^{3/4}),$ which is admissible with the replacement. Thus, by the above argument, we  conclude that
\begin{equation}\label{fixed-P-comparison}
	\begin{aligned}
	\frac1{d_0YMN}
	\sum_{\substack{P\in\Omega\cap\mathcal C\\
	(s,\,t_0,\,r)\in\mathscr B}}F(\boldsymbol\ell(P))
	=\frac{\operatorname{vol}(\mathscr B)}{3\zeta(2)}
	\frac{\sum_{\boldsymbol\ell}F(\boldsymbol\ell)}
	{\prod_{p\in\bigcup_i\mathcal P_i}(p+1)}
	+O\bigl(\exp(-cL^{3/4})\bigr).
	\end{aligned}
\end{equation}

The error term can be handled easily. 
By the bounds $|h_j|,|k_j|\ll Rd_0Y \ll \exp\big(O_K(L^{0.21})\big)$ and $|\mathcal P_i|\le\exp(2L^{0.2})$, one deduce that there are at most $\exp(O_K(L^{0.71}))$
choices of these vectors, prime lists and permutations.
Moreover, the remaining coefficients contribute at most $\exp(O_{K,\rho}(R(J+1)))$, whence the total error terms is bounded,
for some $c'>0$, by
\begin{align*}
	\exp(O_{K,\rho}(L^{0.71}))\exp(-cL^{3/4})\ll \exp(-c'L^{3/4}).
\end{align*}

%For the main term in \eqref{eq:fixed-P-comparison}, recall that $s=u/(d_0YM)$ and $t_0=v/N$, thus we have
%\begin{equation*}
%	\operatorname{vol}(\mathscr B)=
%	\int_{\mathscr B}\,\mathrm ds\,\mathrm dt_0\,\mathrm dr =\frac1{d_0YMN}\int_{(u/(d_0YM),\,v/N,\,r)\in\mathscr B}\,\mathrm du\,\mathrm dv\,\mathrm dr.
%\end{equation*}

For the main term above, write $\operatorname{vol}(\mathscr B)=\int_{\mathscr B}\,\mathrm ds\,\mathrm dt_0\,\mathrm dr$.
We substitute \eqref{fixed-P-comparison} in the expression of $S(K,R)/(d_0YMN)$ in \eqref{def-S(K,R)}, and then interchange the summations and integration. For each $(s,\,t_0,\,r)\in[1,\,16]\times[1,\,2]\times[0,\,1)$, we  sum over $z_j$, $\boldsymbol p^{(j)}$ and $\boldsymbol\pi,\boldsymbol\sigma$ with $(s,\,t_0,\,r)\in\mathscr B$. The conditions in \eqref{def-mathscrB} and factors in \eqref{def-F(ell)}, together with the remaining coefficients, give exactly the summands in \eqref{moment-expansion}. Hence, with the above bound for total error terms, we get
\begin{align*}
	\frac{S(K,R)}{d_0YMN}
	&=\frac1{3\zeta(2)}\int_1^{16}\int_1^2\int_0^1
	\frac{\sum_{\boldsymbol\ell}\bigl\langle u_{z_0},\,
	(\mathcal A_{1,\boldsymbol\ell}
	\mathcal A_{1,\boldsymbol\ell}^*)^R u_{z_0}
	\bigr\rangle_{\boldsymbol\ell}}{\prod_{p\in\bigcup_i\mathcal P_i}(p+1)}\,\mathrm dr\,\mathrm dt_0\,\mathrm ds +O(\exp(-cL^{3/4}))
	\\
	&=\frac1{3\zeta(2)}\int_1^{16}\int_1^2\int_0^1\mathcal M_R(d_0YM s,\,Nt_0,\,r)\,\mathrm dr\,\mathrm dt_0\,\mathrm ds+O(\exp(-cL^{3/4}))
\end{align*}
where $\langle\cdot,\,\cdot\rangle_{\boldsymbol\ell}$ is taken with $u=d_0YM s$, $v=Nt_0$; and the second equality follows from the definition of $\mathcal M_R$ in \eqref{def-M_R}.
Finally, by \eqref{operator-transf} and \eqref{def-S(K,R)}, one finds that $S(K,R)/(d_0YMN)$ above is just the sum in \eqref{eq:moment-integral}. This completes the proof of \textup{(i)}.

We now prove \textup{(ii)} using the same counting and truncation arguments as in \textup{(i)}. We first prove \eqref{eq:one-edge-integral} and then treat the two additional restrictions separately.

Fix $z_1=(h,\,k)^T$, $\boldsymbol p$ and $p'_1,\ldots,p'_K$ in the sum defining $\Phi$. The coordinate estimate in \textup{(i)} is applicable to $z_1$, whence $|h|,\,|k|\ll d_0Y$. Like $\mathscr B$ in (i), we also define
\begin{equation*}
	\mathscr B_0:=\Bigl\{(s,\,t_0,\,r)\in[1,\,16]\times[1,\,2]\times[0,\,1):\ \widetilde gz_0\in\Omega',\ \widetilde gz_1\in\Omega'\Bigr\},
\end{equation*}
where $\widetilde g=\widetilde g(d_0YM s,\,Nt_0,\,r)$. Applying the argument based on \eqref{faces} and \eqref{different} to $z_0$ and $z_1$, we can obtain \eqref{eq:fixed-position-count} with $\mathscr B$ replaced by $\mathscr B_0$.

Recall that  $bD=\prod_{i=1}^K\prod_{j=1}^{J+1}p_{i,j},\, aD=\prod_{i=1}^K\bigl(p'_i\prod_{j=1}^Jp_{i,j}\bigr)$.
The conditions $(z_0,\,\boldsymbol p)\in\mathscr S_{\boldsymbol\ell}$ and $(z_1,\,\boldsymbol p')\in\mathscr S_{\boldsymbol\ell}$ require $\ell_p=[z_0]_p$ for any $p\mid bD$ and $\ell_p=[z_1]_p$ for any $p\mid aD$. As in \eqref{def-F(ell)},  define
\begin{equation*}
	F_1(\boldsymbol\ell):=\biggl(\prod_{p\mid bD}\mathbf1_{\ell_p=[z_0]_p}\biggr)\biggl(\prod_{p\mid aD}\mathbf1_{\ell_p=[z_1]_p}\biggr)\rho^{\frac{\omega_{\boldsymbol\ell}(z_0)+\omega_{\boldsymbol\ell}(z_1)}2}.
\end{equation*}
We estimate the sum of $F_1(\boldsymbol\ell(P))$ over $P$ by the argument proving \eqref{fixed-P-comparison}. For each $p\in\bigcup_i\mathcal P_i$ and $\ell\in\mathbb P^1(\mathbb F_p)$, let  $H_{1,p}(\ell):=1-\rho^{\frac12\big(\mathbf1_{\ell=[z_0]_p}+\mathbf1_{\ell=[z_1]_p}\big)}$.
Then
\begin{equation*}
	\rho^{\frac{\omega_{\boldsymbol\ell}(z_0)+\omega_{\boldsymbol\ell}(z_1)}2}=\prod_{p\in\bigcup_i\mathcal P_i} \bigl(1-H_{1,p}(\ell_p)\bigr),\qquad \frac1{p+1}\sum_{\ell\in\mathbb P^1(\mathbb F_p)}H_{1,p}(\ell)\le\frac2{p+1}.
\end{equation*}

The product $abD$ contains all the distinct primes in $\boldsymbol p$ and $\boldsymbol p'$, and has $K(J+2)$ prime factors. As in \textup{(i)}, we expand the product over $p\nmid abD$ and truncate the resulting sum at $B=\lceil L^{3/4}\rceil$. For the moduli $S=abDq_1\cdots q_k$ with $q_\nu\nmid abD$ and $0\le k\le B$, we  have the bound $\log S\ll_K (J+2+B)L^{0.2}\ll L^{0.96}.$
Thus we may use \eqref{eq:fixed-position-count} with $\mathscr B_0$ as we have mentioned above. As in \eqref{E-B-estimate-main}, the truncation error in the average over $\boldsymbol\ell$ is $\ll\exp(-cL^{3/4})$. Applying the same grouping and truncation argument as in \textup{(i)}, we therefore obtain
\begin{equation*}
	\frac1{d_0YMN}\sum_{\substack{P=(u,\,v)^T\in\Omega\cap\mathcal C\\ (u/(d_0YM),\,v/N,\,c/u)\in\mathscr B_0}}F_1(\boldsymbol\ell(P))=\frac{\operatorname{vol}(\mathscr B_0)}{3\zeta(2)}\frac{\sum_{\boldsymbol\ell}F_1(\boldsymbol\ell)}{\prod_{p\in\bigcup_i\mathcal P_i}(p+1)}+O\bigl(\exp(-cL^{3/4})\bigr).
\end{equation*}
We next multiply this formula by the remaining factor in the definition of $\Phi$, i.e.
\begin{equation*}
	\biggl(\prod_{i=1}^K\frac1{V_i^{J+2}}\biggr)\frac{d_0}{D}\Bigl|
	\eta\Bigl(\frac aY\Bigr)\eta\Bigl(\frac bY\Bigr)\psi_1 \Bigl(\frac{t'}Y\Bigr)\Bigr|\rho^{-K(J+1)}W(t'-b+a),
\end{equation*}
with $t'=k/D$, and then sum over $z_1$, $\boldsymbol p$ and $p'_1,\ldots,p'_K$ as in the definition of $\Phi$. 
Using $|h|,|k|\ll d_0Y$, $|\mathcal P_i|\le\exp(2L^{0.2})$ and $0\le W\le1$, one finds that the total error terms after summation, as in \textup{(i)}, is also bounded by $O\big(\exp(-cL^{3/4})\big)$ for some $c>0$.

For the main terms, write $\operatorname{vol}(\mathscr B_0)=\int_{\mathscr B_0}\mathrm ds\,\mathrm dt_0\,\mathrm dr$ and interchange the finite sums with the integrals. The conditions defining $\mathscr B_0$ and the indicators in $F_1$ give precisely the conditions in $\sum^*$ defining $\Phi$. Hence, also as in \textup{(i)}, we get
\begin{align*}
	\frac1{d_0YMN}\sum_{P=(u,\,v)^T\in\Omega\cap\mathcal C}\Phi\Bigl(u,\,v,\,\frac cu;\,\boldsymbol\ell(P)\Bigr) =&\frac1{3\zeta(2)}\int_1^{16}\int_1^2\int_0^1\frac{\sum_{\boldsymbol\ell}\Phi(d_0YM s,\,Nt_0,\,r;\,\boldsymbol\ell)}{\prod_{p\in\bigcup_i\mathcal P_i}(p+1)}\,\mathrm dr\,\mathrm dt_0\,\mathrm ds
	\\
	&\ +O\big(\exp(-cL^{3/4})\big),
\end{align*}
which completes the proof of \eqref{eq:one-edge-integral} without the additional restrictions.

For the restriction at $(\widetilde gz_0,\,\boldsymbol p)$, we replace $\mathscr B_0$ in $\sum_{P}F_1(\boldsymbol\ell(P))$
%the sum of $F_1(\boldsymbol\ell(P))$ over $P$ 
and in the integral for its main term by the set $\mathscr B_1=\bigl\{(s,\,t_0,\,r)\in\mathscr B_0:\ \text{at least one of \eqref{eq:goodone} and \eqref{eq:goodtwo} fails at }(\widetilde gz_0,\,\boldsymbol p)\bigr\}.$
The additional restriction involves the equations $a\tau(w_0)/\tau(z_0)=n\pm\vartheta$ which are the same as in \textup{(i)} after \eqref{eq:rderiv}, whence one has $\#\mathscr Q_{\partial}\ll x^{c_\delta/4}$ with $\mathscr B$ replaced by $\mathscr B_1$ in the definition of $\mathscr Q_{\partial}$. Then applying \eqref{eq:position-count} to the boxes as in \textup{(i)}, we get \eqref{eq:fixed-position-count} for $\mathscr B_1$.
Finally, by the same arguments as in (i) for $F_1$ and for the subsequent summation, we get the desired asymptotic.

For the restriction that at least one of \eqref{eq:goodone} and \eqref{eq:goodtwo} fails at $(\widetilde gz_1,\,\boldsymbol p')$, the argument is almost identical, and hence we omit the details. This completes the proof of \textup{(ii)}.

\end{proof}

In the following Subsections \ref{sec:memops} and \ref{sec:memory}, we will bound $\mathcal M_R(u,\,v,\,r)$ uniformly for  $r\in[0,\,1)$, $u/(d_0YM)\in[1,\,16]$, $v/N\in[1,\,2]$. Then we can derive \eqref{root-aver} 
by Proposition \ref{prop:position-integrals}.

\vskip 3mm

\subsection{Auxiliary operators and norm estimates}\label{sec:memops}

Let $(u,\,v)\in\Omega$ and $0\le r<1$ as in Subsection \ref{sec:roots}. In this subsection, we define two auxiliary operators $\mathscr{G}_j$ and $\mathscr{O}_j$ to estimate $\mathcal M_R(u,\,v,\,r)$, and we shall show that, for any $A>0$, 
\begin{equation*}
	\|\mathscr{G}_j\|\ll_{K,\,\rho}1\quad(0\le j\le2R),\qquad \|\mathscr{O}_j\|\le L^{-A}\quad(0\le j<2R).
\end{equation*}

Choose a fixed constant $C_1>0$ such that $|z_j|\le C_1Rd_0Y$ for all vectors $z_j$ occurring in \eqref{moment-expansion}, as established in Subsection \ref{sec:roots}. We then define
\begin{equation}\label{eq:relset}
	\mathcal Z_0:=\biggl\{z\in\mathbb Z^2:\ z\text{ primitive},\ |z|\le2C_1Rd_0Y,\ \frac{d_0YM}{u}\le\tau(z)\le\frac{16d_0YM}{u}\biggr\}.
\end{equation}
Here $\tau(z)=z_1-2rz_2$, as before. We write $\boldsymbol p=(\boldsymbol p_1,\,\ldots,\,\boldsymbol p_K)$, where $\boldsymbol p_i=(p_{i,1},\,\ldots,\,p_{i,J+1})$ is the ordered tuple of primes in $\mathcal P_i$ paired with $z$. We also write $\mathfrak m=(\mathfrak m_1,\,\ldots,\,\mathfrak m_K)$, where $\mathfrak m_i$ is an ordered $k_i$-tuple of pairs $(p,\,\ell)$ with $p\in\mathcal P_i$ and $\ell\in\mathbb P^1(\mathbb F_p)$; repeated pairs are allowed. We use ``append $(p,\,\ell)$ to $\mathfrak m_i$'' to mean the replacement
\begin{equation*}
\bigl((p_1,\,\ell_1),\,\ldots,\,(p_{k_i},\,\ell_{k_i})\bigr)\longmapsto\bigl((p_1,\,\ell_1),\,\ldots,\,(p_{k_i},\,\ell_{k_i}),\,(p,\,\ell)\bigr).
\end{equation*}
We append an ordered tuple of pairs by placing its coordinates at the end in their given order. 
For functions $F(z,\,\boldsymbol p,\,\mathfrak m)$, define the projection operator $\mathsf O$ by
\begin{equation}\label{eq:posproj}
	(\mathsf O F)(z,\,\boldsymbol p,\,\mathfrak m):=\mathbf1_{\widetilde gz\in\Omega'}\mathbf1_{2\nmid z_1}F(z,\,\boldsymbol p,\,\mathfrak m).
\end{equation}
Since the two indicator functions depend only on $z$, $\mathsf O$ commutes with list permutations and operations that leave $z$ fixed. Moreover, $\|\mathsf O\|\le1$ and $1/16\le\tau(z')/\tau(z)\le16$ for $z,\,z'\in\mathcal Z_0$.

We set $q_0=q_{2R}=\sqrt{\rho}$, $q_j=\rho$ for $0<j<2R$, and $\eta'_j=1-q_j$ for $0\le j\le2R$. For $p\in\mathcal P_i$, define
\begin{equation}\label{eq:nudef}
	b'_p=1-\frac{\sum_{j=0}^{2R}\eta'_j}{p+1},\qquad \nu_i(p)=\frac1{V_i(p+1)b'_p}.
\end{equation}
Let $p_{\min}=\min\bigcup_i\mathcal P_i$. The reciprocal prime sums give $\Delta:=\max_{i,\,p}\nu_i(p)\ll_K e^{-L^{0.1}}$ and
\begin{equation}\label{eq:numass}
	pV_i\nu_i(p)=1+O\Big(\frac{2R+1}{p}\Big),\quad s_i:=\sum_{p\in\mathcal P_i}\nu_i(p)=1+O_K\Big(\frac{R}{p_{\min}}\Big). 
\end{equation}
In particular, $b'_p>0$ for large $x$, and $0<\prod_p b'_p\le1$. Products over $p$ in this subsection and in Subsection \ref{sec:memory} are over $\bigcup_i\mathcal P_i$.

To deal with repeated occurrences of the same prime in estimating $\mathcal M_R(u,\,v,\,r)$, we follow the construction of OpenAI \cite{PD} by introducing auxiliary operators $\mathscr G_j$, $\mathscr O_j$. 
We note that the high-moment method in \cite{PD} builds on the work of Helfgott--Radziwi{\l}{\l} \cite{HelRad21} and its further development by Pilatte \cite{Pil26}, including their treatment of dependencies caused by repeated primes.

For each $i$, we define the set of pairs and its measure by
\begin{equation*}
	\mathcal X_i=\{(p,\,\ell):\ p\in\mathcal P_i,\ \ell\in\mathbb P^1(\mathbb F_p)\},\qquad \lambda_i(p,\,\ell)=\nu_i(p).
\end{equation*}
We give each $z\in\mathcal Z_0$ weight $1$, each prime tuple $\boldsymbol p_i$ weight $\prod_{h=1}^{J+1}\nu_i(p_{i,h})$, and each auxiliary tuple $\mathfrak m_i=((p_h,\,\ell_h))_{h=1}^{k_i}$ weight $\prod_{h=1}^{k_i}\lambda_i(p_h,\,\ell_h)/k_i!$. We multiply these weights over $i$. The functions are invariant under permutations of the coordinates of each $\mathfrak m_i$. We allow repetitions in the underlying space and impose the distinctness conditions on $\boldsymbol p_i,\,\boldsymbol p'_i$ in the operators, as in \eqref{def-T1}.
Fix $0<\vartheta<1$ with $\vartheta^2>\sqrt\rho$. We define $\mathscr{G}_j$ and $\mathscr{O}_j$  as follows.

 The operator $\mathscr{G}_j$ fixes $z$ and the prime tuples $\boldsymbol p_i$.  In each $\mathfrak m_i$, remove any subset of the coordinates $(p,\,\ell)$ satisfying $\ell=[z]_p$, and append $l\ge0$ ordered pairs $(p_a,\,[z]_{p_a})$, where $p_a\in\mathcal P_i$ and $1\le a\le l$. If $d$ coordinates are removed and $h$ of the pre-existing coordinates that remain satisfy $\ell=[z]_p$, the weight is
\begin{equation*}
	\biggl(-\frac{\eta'_j}{\vartheta}\biggr)^d\biggl(\frac{q_j}{\vartheta^2}\biggr)^h\frac1{l!}\prod_{a=1}^l\biggl(-\frac{\eta'_jV_i}{\vartheta}\nu_i(p_a)\biggr).
\end{equation*}
Sum over all subsets of coordinates to be removed and all ordered tuples $(p_1,\,\ldots,\,p_l)$, independently in the $K$ groups; repeated pairs are distinguished by their indices.

 We define $\mathscr{O}_j$ for $z,\,z'\in\mathcal Z_0$. In the forward direction, we keep the factor $(d_0/D)\eta(a/Y)\eta(b/Y)\psi_1(t'/Y)$ and the difference between the two terms in braces in $K_{1,\boldsymbol\ell}(z,\,z';\,D,\,a,\,b)$, where $t'=\det(z,\,z')/D$. We also keep the permutations in \eqref{def-S}, conditions \eqref{eq:goodone} and \eqref{eq:goodtwo} at $(\widetilde gz,\,\boldsymbol p)$ and $(\widetilde gz',\,\boldsymbol p')$, and the restrictions on the prime tuples in \eqref{def-T1}. 

For the shared entries, let
\begin{equation*}
D=\prod_{i=1}^K\prod_{h=1}^J p_{i,h},\qquad
 d_0\le D<2d_0,\qquad D\mid\det(z,\,z').
\end{equation*}
Both restrictions on $D$ are retained in each of the two terms
of the geometric kernel. Since $D$ is squarefree and $z,z'$ are
primitive, the divisibility condition is equivalent to
\begin{equation*}
[z]_p=[z']_p\qquad(p\mid D).
\end{equation*}
We retain $[z]_p=[z']_p$ for all $p\mid D$ when we replace the other divisibility indicators below.

For each $1\le i\le K$, we retain $p_{i,1},\,\ldots,\,p_{i,J}$
from the list at $z$ as the first $J$ entries of the list at $z'$.
Sum over all subsets $S\subseteq\{1,\,\ldots,\,K\}$,
with coefficient $1$ for each choice. For $i\in S$, append
$(p_{i,J+1},\,[z]_{p_{i,J+1}})$ to $\mathfrak m_i$.
For $i\notin S$, discard the departing entry $p_{i,J+1}$
without appending a pair to $\mathfrak m_i$.
We choose the unshared entry $p'_i$ of each tuple $\boldsymbol p'_i$ at $z'$ either by removing an existing coordinate $(p,\,\ell)$ of $\mathfrak m_i$ satisfying $\ell=[z']_p$ and using $p$, or by summing over $p'_i\in\mathcal P_i$ with measure $\nu_i$. Let $I$ be the set of groups using the first option. Let $h_s$ count the coordinates of $\mathfrak m$ before the operation satisfying $\ell=[z]_p$, and let $h_t$ count the coordinates after the operation satisfying $\ell=[z']_p$. Replace the original damping and divisibility factors by
\begin{equation*}
	\vartheta^{h_s+h_t}\prod_{i\in I}\frac1{V_i}\prod_{\substack{1\le i\le K\\i\notin I}}\nu_i(p'_i),
\end{equation*}
where the sums over $p'_i$ apply to $i\notin I$. For $i\in I$, only coordinates present before the operation may be used; a pair appended during this operation cannot supply the entry $p'_i$. On alternate edges, transpose and conjugate the geometric kernel and use the same rules for $\mathfrak m$. At an internal index $j$, a pair retained in $\mathfrak m$ and satisfying $\ell=[z_j]_p$ contributes
\begin{equation}\label{eq:thetaid}
	\vartheta\,(q_j/\vartheta^2)\,\vartheta=q_j.
\end{equation}
A pair appended by $\mathscr{G}_j$ contributes $(-\eta'_jV_i/\vartheta)\vartheta=-\eta'_jV_i$, and a pair removed by $\mathscr{G}_j$ contributes $\vartheta(-\eta'_j/\vartheta)=-\eta'_j$. 

We now estimate the individual operators without requiring primes independently selected at different steps to be distinct.

For an operator $\mathcal{T}$, we write $\mathcal{T}f(s)=\sum_t a(s,\,t)f(t)m(t)$, where $s=(z,\,\boldsymbol p,\,\mathfrak m)$ and $t=(z',\,\boldsymbol p',\,\mathfrak m')$ are its arguments and $m(t)$ is the measure specified above. We call the function $a(s,\,t)$ the kernel of $\mathcal{T}$. For $\mathcal{T}=\mathscr{G}_j$ or $\mathscr{O}_j$, we restrict its kernel as follows:
\begin{equation*}
a(s,\,t)\longmapsto a(s,\,t)\mathbf1_{\sum_i|\mathfrak m_i|\le \lceil L^2\rceil}\mathbf1_{\sum_i|\mathfrak m_i'|\le \lceil L^2\rceil},
\end{equation*}
where $|\mathfrak m_i|$ denotes the number of coordinates of $\mathfrak m_i$. We keep the notation $\mathscr{G}_j,\mathscr{O}_j$ for these restricted operators. For a positive weight $w$, we define
\begin{equation}\label{eq:schur}
	R(\mathcal{T})=\sup_s\frac1{w(s)}\sum_t|a(s,\,t)|w(t)m(t),\quad C(\mathcal{T})=R(\mathcal{T}^*),\quad \|\mathcal{T}\|\le\sqrt{R(\mathcal{T})C(\mathcal{T})}.
\end{equation}
Here $m(t)$ is the product of the weights of $z',\,\boldsymbol p',\,\mathfrak m'$ specified above.  The factorial identities
\begin{equation*}
	\frac{\binom{k+a}{a}}{(k+a)!}=\frac1{k!a!},\quad \frac{\binom{k+a}{a}a!}{(k+a)!}=\frac1{k!}
\end{equation*}
show that the adjoint of $\mathscr{G}_j$ interchanges removing and appending coordinates, while the adjoint of $\mathscr{O}_j$ reverses the transfer of entries between the prime tuples $\boldsymbol p_i$ and the auxiliary tuples $\mathfrak m_i$. In $\mathscr{O}_j^*$, the factor $1/V_i$ for an entry moved  from $\mathfrak m_i$ to the prime tuple becomes  the factor for the reverse transfer. These identities remain valid when coordinates of $\mathfrak m_i$ repeat.

For large $x$ we have $V_i<1$. Choose a fixed $v_0>1$ so large that $\sqrt\rho/\vartheta^2+1/(\vartheta v_0)<1$, and take $w=v_0^{k_1+\cdots+k_K}$. If exactly $h_i$ coordinates $(p,\,\ell)$ of $\mathfrak m_i$ satisfy $\ell=[z]_p$, its weighted row and column factors for $\mathscr{G}_j$ are bounded, respectively, by
\begin{equation*}
	\exp\biggl(\frac{v_0\eta'_jV_i s_i}{\vartheta}\biggr)\biggl(\frac{q_j}{\vartheta^2}+\frac{\eta'_j}{\vartheta v_0}\biggr)^{h_i},\qquad \exp\biggl(\frac{v_0\eta'_j s_i}{\vartheta}\biggr)\biggl(\frac{q_j}{\vartheta^2}+\frac{\eta'_jV_i}{\vartheta v_0}\biggr)^{h_i}.
\end{equation*}
Both parentheses are at most $1$, and $s_i\asymp1$. Multiplication over the groups and the weighted Schur test therefore give $R(\mathscr{G}_j),\,C(\mathscr{G}_j),\,\|\mathscr{G}_j\|\ll_{K,\,\rho}1$.

For fixed $z\in\mathcal Z_0$, choose $w_z\in\mathbb Z^2$ with $\det(z,\,w_z)=1$. An admissible vector $z'$ with prescribed determinant $j_1=\det(z,\,z')$ has the form
\begin{equation}\label{eq:target}
	z'=j_1w_z+hz,\qquad \frac{\tau(z')}{\tau(z)}=h+j_1\frac{\tau(w_z)}{\tau(z)}\in[1/16,\,16].
\end{equation}
The kernel contains the difference between the indicator $\mathbf1_{j_1=D(b-a)}$ and the integral $\int_{\mathfrak M_1}e\bigl(\theta(j_1/D-b+a)\bigr)\,\mathrm d\theta$. Both terms have the same remaining factors and restrictions.

For fixed $j_1$, \eqref{eq:target} allows $O(1)$ integers $h$. The indicator with $j_1=D(b-a)$ therefore allows $O(1)$ vectors $z'$. For the integral over $\mathfrak M_1$, there are $O(Y)$ possible integers $j_1/D$, and the integral has absolute value at most $|\mathfrak M_1|\ll L^{3A_0}/Y$.  There are at most $\lceil L^2\rceil^K$ choices of coordinates to move from $\mathfrak m_i$ to the lists at $z'$, while the sums of the measures $\nu_i$ and the remaining factors are bounded. Applying the same count to the adjoint gives
\begin{equation}\label{eq:crudebd}
	R(\mathscr{G}_j),\,C(\mathscr{G}_j),\,\|\mathscr{G}_j\|\ll_{K,\,\rho}1,\quad\ R(\mathscr{O}_j),\,C(\mathscr{O}_j),\,\|\mathscr{O}_j\|\ll_{K,\,\rho}\lceil L^2\rceil^K(1+L^{3A_0})\le L^{C_*},
\end{equation}
for some fixed $C_*=C_*(K,\,A_0,\,\rho)$. These absolute bounds also hold after inserting choice multipliers of modulus at most one, provided they preserve symmetry in the coordinates of $\mathfrak m_i$ present before the operation.

We write $\mathscr{O}_j^{(0)}$ for the terms with $S=I=\varnothing$. For $S\cup I\ne\varnothing$, we denote the terms containing $\mathbf1_{j_1=D(b-a)}$ and the integral over $\mathfrak M_1$ by $\mathscr{O}_j^{(\mathrm{ind})}$ and $\mathscr{O}_j^{(\mathrm{arc})}$, respectively, so that $\mathscr{O}_j=\mathscr{O}_j^{(0)}+\mathscr{O}_j^{(\mathrm{ind})}-\mathscr{O}_j^{(\mathrm{arc})}$. We estimate the three operators separately. 

We first consider $\mathscr{O}_j^{(0)}$. Partition positions by intervals of length $d_0Y$ for $z_2/\tau(z)$, and then by the lines $[z]_p$ at the shared primes $p\mid D$. Since
\begin{equation*}
	\frac{z'_2}{\tau(z')}-\frac{z_2}{\tau(z)}=\frac{\det(z,\,z')}{\tau(z)\tau(z')},\qquad |\det(z,\,z')|\ll d_0Y,
\end{equation*}
for a fixed interval containing $z_2/\tau(z)$, only $O(1)$ intervals can contain $z'_2/\tau(z')$. We fix two such intervals, the shared entries $p_{i,1},\,\ldots,\,p_{i,J}$, the lines at the primes dividing $D$, and the auxiliary tuples $\mathfrak m_i$. If the resulting restriction of $\mathscr{O}_j^{(0)}$ is nonzero, we choose a primitive $z_*$ occurring in it with a prime tuple satisfying \eqref{eq:goodone} and \eqref{eq:goodtwo} at $\widetilde{g} z_*$, and choose $w_*\in\mathbb Z^2$ with $\det(z_*,\,w_*)=1$. The common line conditions give the lattice
\begin{equation*}
	\Lambda_D=\{hz_*+\ell Dw_*:\ h,\,\ell\in\mathbb Z\}.
\end{equation*}
Indeed, writing $z=hz_*+kw_*$, the condition at each shared prime $p$ is $p\mid k$, and hence $D\mid k$. Put $\alpha=D\tau(w_*)/\tau(z_*)$. The position bounds and \eqref{eq:goodone} with $\widetilde D=D$ and $r_{\tilde{g} z_*}\equiv\tau(w_*)/\tau(z_*)\,(\textrm{mod}\,1)$ imply that
\begin{equation}\label{eq:023}
	|\ell|\ll Y,\qquad \frac{\tau(z)}{\tau(z_*)}=h+\ell\alpha\asymp1,\qquad \|k\alpha\|>Y^{-0.7}\quad(1\le k\le Y^{0.2}).
\end{equation}

The lattice $\{(\ell/Y,\,h+\ell\alpha):\ h,\,\ell\in\mathbb Z\}$ has covolume $Y^{-1}$. Let $\lambda$ be the length of its shortest nonzero vector. The inequality $\lambda<Y^{-0.8}$ would contradict \eqref{eq:023}; the planar convex-body theorem gives $\lambda\ll Y^{-1/2}$. Thus
\begin{equation*}
	Y^{-0.8}\le\lambda\ll Y^{-1/2},\qquad |v_1|=\lambda,\qquad |v_2|\ll(Y\lambda)^{-1},
\end{equation*}
where $v_1,v_2$ are a lattice basis obtained by completing a shortest vector $v_1$ and reducing the component of $v_2$ along it. In this basis the bounded region in \eqref{eq:023} is contained in an integer-coordinate box with side parameters $A'\asymp\lambda^{-1}$ and $B'\asymp Y\lambda$. Consequently
\begin{equation}\label{eq:latbox}
	A'B'\asymp Y,\qquad Y^{0.1}\le A',\,B'\le Y^{0.9}.
\end{equation}
Changing the sign of one basis vector if necessary, let $\mathcal B_D$ be the positively oriented corresponding basis matrix of $\Lambda_D$. All relevant integer coordinates belong to $\mathcal R=\{\mathbf k\in\mathbb Z^2:\ |k_1|\le A',\ |k_2|\le B'\}$, and
\begin{equation}\label{eq:latbox-e}
	z=\mathcal B_D\mathbf k,\qquad z'=\mathcal B_D\mathbf k',\qquad \det\mathcal B_D=D,\qquad \frac{\det(z,\,z')}{D}=\det(\mathbf k,\,\mathbf k').
\end{equation}
In particular, $\mathcal B_D$ is invertible modulo every prime not dividing $D$.

We next prove the signed estimate on this box. For an integer $n=\prod_{i=1}^K q_i$ with $q_i\in\mathcal P_i$, we set $m(n)=\prod_{i=1}^K\nu_i(q_i)$, and set $m(n)=0$ otherwise.  The disjoint prime sets make the representation unique, and
\begin{equation*}
	m(a)\ll_K Y^{-1}\quad\text{when }\eta(a/Y)\ne0,\qquad \sqrt{m(a)m(b)}=\frac1Y\sqrt{Ym(a)}\sqrt{Ym(b)}.
\end{equation*}
Multiplication by $\sqrt{m(a)}$ converts the product measure to counting measure.  We extend functions by zero to $L^2(\mathcal R\times\mathbb Z)$ with counting measure. To bound $\mathscr{O}_j^{(0)}$, we use the operator $\mathcal E_0$ with kernel 
\begin{equation*}
	\mathcal E_0((\mathbf k,\,b),\,(\mathbf k',\,a))=\frac1Y\psi_1\biggl(\frac{\det(\mathbf k,\,\mathbf k')}Y\biggr)\int_{\mathbb T\setminus\mathfrak M_1}e\bigl(\theta(\det(\mathbf k,\,\mathbf k')-b+a)\bigr)\,\mathrm d\theta.
\end{equation*}
The factors $\eta(a/Y)\sqrt{Ym(a)}$ and $\eta(b/Y)\sqrt{Ym(b)}$,  the indicators of distinctness within each prime tuple, of conditions \eqref{eq:goodone} and \eqref{eq:goodtwo} and of primitivity, and the factors $\vartheta^{h_s}$ and $\vartheta^{h_t}$ all act by bounded multiplication at $z$ or $z'$; and $d_0/D\asymp1$ is fixed on the block. 
The position restriction is multiplication by $\mathbf1_{\widetilde g\mathcal B_D\mathbf k\in\Omega'}\mathbf1_{2\nmid(\mathcal B_D\mathbf k)_1}$.
Thus, if $\iota_s,\,\iota_t$ are the isometric inclusions after measure conversion and zero extension, and $D_s,\,D_t$ are the resulting endpoint multipliers, then
\begin{equation}\label{eq:pcontr}
	\|\iota_s^*D_s\mathcal E_0D_t\iota_t\|\le\|D_s\|\,\|\mathcal E_0\|\,\|D_t\|\ll_{K,\,\rho}\|\mathcal E_0\|.
\end{equation}
For the moment, this allows a prime to occur among the unshared entries at both $z$ and $z'$; we estimate that correction below.

Fourier transformation in the integer product variable reduces $\mathcal E_0$ to the matrices
\begin{equation*}
	\frac1Y\biggl[\psi_1\biggl(\frac{\det(\mathbf k,\,\mathbf k')}Y\biggr)e\bigl(\theta\det(\mathbf k,\,\mathbf k')\bigr)\biggr]_{\mathbf k,\,\mathbf k'},\qquad \theta\notin\mathfrak M_1.
\end{equation*}
Define $\widehat\psi_1(v)=\int_{\mathbb R}\psi_1(t)e(-vt)\,\mathrm dt$ and use Fourier inversion. For $|v|\le L^{A_0}/2$, Dirichlet approximation applied to $\xi=\theta+v/Y$ gives integers $c,\,d$ with
\begin{equation*}
	\gcd(c,\,d)=1,\qquad L^{A_0}<d\le\lceil Y/L^{A_0}\rceil,\qquad |\xi-c/d|\le\frac1{d\lceil Y/L^{A_0}\rceil}\le d^{-2}.
\end{equation*}

Indeed, a denominator $d\le L^{A_0}$ would place $\theta$ in $\mathfrak M_1$. For the matrix $T_{A',B'}(\xi)=(e(\xi hk))_{|h|\le A',\,|k|\le B'}$, its Gram matrix gives
\begin{equation*}
	\|T_{A',B'}(\xi)\|^2\ll\sum_{|h|\le2A'}\min\bigl(B',\,\|h\xi\|^{-1}\bigr)\ll(A'/d+1)(B'+d\log(2d)).
\end{equation*}
At $h=0$ the minimum means $B'$. For the last bound, split the $h$-range into intervals of length at most $d/2$; within each interval the fractional parts are separated by at least $1/(2d)$, and sum their reciprocal distances.

Up to a permutation of columns, the determinant matrix is the tensor product of $T_{A',B'}(\xi)$ and its transposed conjugate. By \eqref{eq:latbox},
\begin{equation*}
	\frac1Y\bigl\|[e(\xi\det(\mathbf k,\,\mathbf k'))]_{\mathbf k,\,\mathbf k'}\bigr\|\ll\frac1d+Y^{-0.1}\log(2Y)+\frac{d\log(2d)}Y\ll_K L^{0.2-A_0}+Y^{-0.1}\log Y.
\end{equation*}
The range $|v|>L^{A_0}/2$ is negligible by rapid decay of $\widehat\psi_1$ and the trivial matrix bound $O(Y)$. This proves the same bound for $\|\mathcal E_0\|$.

Let $\mathcal E_{\mathrm{coll}}$ denote the terms in the enlarged estimate for $\mathscr{O}_j^{(0)}$ with $p'_i=p_{i,J+1}$ for some $i$. Each equality prescribes one prime summed with measure $\nu_i$, of mass at most $\Delta$. The count of admissible vectors in \eqref{eq:target}, applied in both directions, gives
\begin{equation*}
	R(\mathcal E_{\mathrm{coll}}),\,C(\mathcal E_{\mathrm{coll}}),\,\|\mathcal E_{\mathrm{coll}}\|\ll_{K,\,\rho}(1+L^{3A_0})\Delta.
\end{equation*}
For the fixed block in \eqref{eq:pcontr}, the difference $\iota_s^*D_s\mathcal E_0D_t\iota_t-\mathcal E_{\mathrm{coll}}$ consists precisely of the terms satisfying $p'_i\ne p_{i,J+1}$ for every $1\le i\le K$. The bounded number of neighboring boxes and the list symmetrizations therefore give
\begin{equation}\label{eq:cleanbd}
	\|\mathscr{O}_j^{(0)}\|\ll_{K,\,\rho}L^{0.2-A_0}+Y^{-0.1}\log Y.
\end{equation}

We now fix $S$ and $I$ with $S\cup I\ne\varnothing$ and estimate $\mathscr{O}_j^{(\mathrm{ind})}$. If $I\ne\varnothing$, put $i_*=\max I$ and fix a coordinate $(p_*,\,\ell_*)$ present in $\mathfrak m_{i_*}$ before the operation to supply the entry at $z'$. There are at most $\lceil L^2\rceil$ such choices. For the unshared products $b$ at $z$ and $a$ at $z'$, write
\begin{equation*}
	Db=\prod_{i=1}^K\prod_{h=1}^{J+1}p_{i,h},\qquad Z_I=\prod_{i\in I\setminus\{i_*\}}p'_i,\qquad T_f=\prod_{i\notin I}p'_i,\qquad a=p_*Z_IT_f.
\end{equation*}

For $i\in I$, $p'_i$ comes from $\mathfrak m_i$; for $i\notin I$, it is summed with weight $\nu_i(p'_i)$. These products are $Z_I$ and $T_f$ of \eqref{eq:goodtwo} with $q_i=p'_i$. Since $p_*=p'_{i_*}$ is distinct from every entry of the tuple at $z$ in $\mathcal P_{i_*}$, we have $p_*\nmid Db$. For the terms with $j_1=D(b-a)$, we therefore have 
\begin{equation}\label{eq:rawdet}
	j_1=Db-Dp_*Z_IT_f,\qquad j_1\equiv Db\not\equiv0\,(\textrm{mod}\,p_*).
\end{equation}

The prescribed line $[j_1w_z+hz]_{p_*}=\ell_*$ either is impossible or determines $h\equiv h_*\,(\textrm{mod}\,p_*)$. For fixed $z,\,\boldsymbol p$ and $(p_*,\,\ell_*)$, the residue $h_*$ is independent of the omitted entries defining $D$ and of $Z_I$, since $j_1\equiv Db\,(\textrm{mod}\,p_*)$ and $Db$ is the product of all entries of $\boldsymbol p$.  Substitution in \eqref{eq:target} gives
\begin{equation*}
	\biggl\|DZ_IT_f\frac{\tau(w_z)}{\tau(z)}-\frac{Db\tau(w_z)/\tau(z)+h_*}{p_*}\biggr\|\le\frac{16}{p_*}.
\end{equation*}

By \eqref{eq:numass}, the weights $\prod_{i\notin I}\nu_i(p'_i)$ are comparable to $\prod_{i\notin I}(p'_iV_i)^{-1}$. By \eqref{eq:goodtwo} at $(\widetilde gz,\,\boldsymbol p)$, outside a set of tuples $(p'_i)_{i\notin I}$ of total weight $O_K(e^{-cL^{0.1}})$, the points $DZ_IT_fr_{\tilde{g}z}$ in $\mathbb T$ for distinct products $DZ_I$ are separated by at least $100e^{-L^{a_{i_*}}}$. Here $D$ is one of the products $\widetilde D$ in \eqref{eq:goodtwo}, obtained by omitting the unshared entry in each tuple. Since $r_{\tilde{g}z}\equiv\tau(w_z)/\tau(z)\,(\textrm{mod}\,1)$ and $32/p_*\le32e^{-L^{a_{i_*}}}$, the preceding inequality allows at most one such product $DZ_I$. Its prime factors belonging to the tuples at $z$ are exactly those dividing $D$, since no prime dividing $Z_I$ occurs in these tuples. Thus at most one choice of the omitted entry in each tuple is possible, and its weight in \eqref{def-S} is $(J+1)^{-K}$. 

The resulting determinant permits $O(1)$ vectors $z'$. Summing over repeated eligible coordinates of $\mathfrak m_i$ costs only a constant, because $h\vartheta^{h-1}\ll_\vartheta1$ for $h\ge1$. For the exceptional choices of the primes $p'_i$ with $i\notin I$, use the crude count $\lceil L^2\rceil^K$. Define
\begin{equation}\label{eq:epsk}
	\varepsilon_K=C_{K,\,\rho}\bigl(\lceil L^2\rceil(J+1)^{-K}+\lceil L^2\rceil^K e^{-cL^{0.1}}\bigr).
\end{equation}
When $I=\varnothing$, we obtain the row bound $O_{K,\,\rho}(1)$ from \eqref{eq:target}. For an operator $\mathcal{T}$ with kernel $a(s,\,t)$, we write the column quantity in \eqref{eq:schur} explicitly as
\begin{equation*}
C(\mathcal{T})=\sup_t\frac1{w(t)}\sum_s|a(s,\,t)|w(s)m(s)
=R(\mathcal{T}^*).
\end{equation*}
We estimate $C(\mathcal{T})$ by applying the row estimate to $\mathcal{T}^*$. In the adjoint, we transfer an entry from $\mathfrak m_i$ to $\boldsymbol p_i$ for $i\in S$, and from $\boldsymbol p_i$ to $\mathfrak m_i$ for $i\in I$, so the roles of $S$ and $I$ are interchanged. We therefore obtain the following bounds for $R$ and $C$ for  the part of $\mathscr{O}_j^{(\mathrm{ind})}$ with fixed $S$ and $I$:
\begin{equation*}
	\begin{array}{c|c|c|c}
		S&I&R&C\\ \hline
		\varnothing&I\ne\varnothing&\varepsilon_K&O_{K,\,\rho}(1)\\
		S\ne\varnothing&\varnothing&O_{K,\,\rho}(1)&\varepsilon_K\\
		S\ne\varnothing&I\ne\varnothing&\varepsilon_K&\varepsilon_K
	\end{array}
\end{equation*}
For sufficiently large fixed $K$, weighted Schur and $J+1\asymp L^{0.01}$ give, for any $A>0$
\begin{equation*}
	\|\mathscr{O}_j^{(\mathrm{ind})}\|\ll_{K,\,\rho}\sqrt{\varepsilon_K}\ll_{A,\,K,\,\rho}L^{1-0.005K}+L^{-A}.
\end{equation*}

For $\mathscr{O}_j^{(\mathrm{arc})}$ with $I\ne\varnothing$, a coordinate $(p,\,\ell)$ moved from $\mathfrak m_i$ to a list at $z'$ imposes $[z']_p=\ell$. Since $p\nmid D$, \eqref{eq:latbox-e} converts this into a nonzero linear congruence in the coordinates of $\mathcal R$. For integers $a_0,b_0$ not both divisible by $p$,
\begin{equation}\label{eq:dirty-hit-count}
	\#\{(h,\,\ell)\in\mathcal R:\ a_0h+b_0\ell\equiv0\,(\textrm{mod}\,p)\}\ll\frac{A'B'}p+\max(A',\,B')\ll Y(p^{-1}+Y^{-0.1}).
\end{equation}
To prove this, fix the coordinate whose coefficient may vanish and count one residue class in the other coordinate. Multiplying by $L^{3A_0}/Y$ and summing over at most $\lceil L^2\rceil^K$ choices of coordinates from $\mathfrak m$ gives a row bound $O_{K,\,\rho}(L^{3A_0}\lceil L^2\rceil^K(p_{\min}^{-1}+Y^{-0.1}))$. Use \eqref{eq:crudebd} for the opposite direction; when $I=\varnothing$ and $S\ne\varnothing$, apply the same argument to the adjoint. Thus, for a fixed $C_\dagger=C_\dagger(K,\,A_0,\,\rho)$, one has
\begin{equation}\label{eq:dirty-comparison-bound}
	\|\mathscr{O}_j^{(\mathrm{arc})}\|\ll_{K,\,\rho}L^{C_\dagger}(p_{\min}^{-1}+Y^{-0.1})^{1/2}\ll_{A,\,K,\,A_0,\,\rho}L^{-A}\qquad(A>0).
\end{equation}
Here $p_{\min}\ge e^{L^{0.1}}$ and $\log Y\gg L^{0.1}$. Combining the three edge estimates, choosing $A_0>A+3$ and then $K>200(A+3)$, we get
\begin{equation}\label{eq:esmall}
	\|\mathscr{O}_j\|\le L^{-A}\qquad(0\le j<2R).
\end{equation}

These choices do not depend on the particular exponents $a_i$; those exponents affect only the threshold for $x$ and the implied constants.

\subsection{Proof of the moment estimate}\label{sec:memory}

In this subsection, we prove \eqref{eq:moment-assumption} by expressing $\mathcal M_R(u,\,v,\,r)$ using the auxiliary operators $\mathscr{G}_j,\mathscr{O}_j$ and applying the estimates of Subsection \ref{sec:memops}.
\begin{proposition}\label{prop:par}
For fixed $E_0>0$ and $\delta>0$, suppose $M,\,N\ge x^\delta$ and $MN\asymp x$. We can choose $A_0$ and then $K$, independently of the fixed exponents $a_i$, such that the following bound holds uniformly for the values of $d_0$ in \eqref{eq:operator-reduction}:
	\begin{equation}\label{eq:moment}
	\frac1{d_0YMN}\sum_{P\in\Omega\cap\mathcal C}\bigl\langle u_{_P},\,\bigl(\mathcal A_{-2}^{\mathcal C}(\mathcal A_{-2}^{\mathcal C})^*\bigr)^Ru_{_P}\bigr\rangle\ll L^{-2RE_0}.
	\end{equation}
\end{proposition}

\begin{proof}
Take $A=E_0+5$ and fix $A_0,K$ so that \eqref{eq:esmall} holds; $C_*$ in \eqref{eq:crudebd} is now fixed. We first express $\mathcal M_R(u,\,v,\,r)$ using the definitions of $\mathscr{G}_j,\mathscr{O}_j$ before imposing the restriction $\sum_i k_i\le\lceil L^2\rceil$.

For $p\in\mathcal P_i$, put $A_p:=\{0\le j\le2R:\ p=p_{i,h}^{(j)}\text{ for some }1\le h\le J+1\}$. The sum over $\ell\in\mathbb P^1(\mathbb F_p)$ in the definition of $\mathcal M_R(u,\,v,\,r)$ gives the factor
\begin{equation}\label{eq:pfactor}
	\frac1{p+1}\sum_{\ell\in\mathbb P^1(\mathbb F_p)}\prod_{j\in A_p}\mathbf1_{\ell=[z_j]_p}\prod_{j\notin A_p}\bigl(1-\eta'_j\mathbf1_{\ell=[z_j]_p}\bigr).
\end{equation}
When $A_p=\varnothing$, expanding the product and grouping terms by the least and greatest selected indices gives
\begin{equation*}
	\begin{aligned}
		&\frac1{p+1}\sum_{\ell\in\mathbb P^1(\mathbb F_p)}\prod_{j=0}^{2R}\bigl(1-\eta'_j\mathbf1_{\ell=[z_j]_p}\bigr)=b'_p+\frac1{p+1}\sum_{0\le a<c\le2R}\eta'_a\eta'_c\mathbf1_{[z_a]_p=[z_c]_p}\prod_{a<j<c}\bigl(1-\eta'_j\mathbf1_{[z_j]_p=[z_a]_p}\bigr).
	\end{aligned}
\end{equation*}
When $A_p\ne\varnothing$, \eqref{eq:pfactor} vanishes unless $[z_j]_p=\ell$ for one common line $\ell$ and every $j\in A_p$. If $a=\min A_p$ and $c=\max A_p$, the factors with $j<a$ or $j>c$ satisfy
\begin{equation*}
	\begin{aligned}
		\prod_{j<a}\bigl(1-\eta'_j\mathbf1_{\ell=[z_j]_p}\bigr)&=1-\sum_{h<a}\eta'_h\mathbf1_{\ell=[z_h]_p}\prod_{h<j<a}\bigl(1-\eta'_j\mathbf1_{\ell=[z_j]_p}\bigr),\\
		\prod_{j>c}\bigl(1-\eta'_j\mathbf1_{\ell=[z_j]_p}\bigr)&=1-\sum_{h>c}\eta'_h\mathbf1_{\ell=[z_h]_p}\prod_{c<j<h}\bigl(1-\eta'_j\mathbf1_{\ell=[z_j]_p}\bigr).
	\end{aligned}
\end{equation*}
Thus the selected indices for each prime lie between one least and one greatest index. After extracting $b'_p$, the factor $1/((p+1)b'_p)$ occurs once, and each maximal interval of consecutive indices contained in $A_p$ contributes $1/V_i$.

We require pairwise distinct values for all primes in the prime tuples at $z_0$, all primes summed with $\nu_i$ in $\mathscr{O}_j$, and all prime coordinates of pairs appended by $\mathscr{G}_j$. Moving an existing pair from $\mathfrak m_i$ to the prime tuple $\boldsymbol p_i$ introduces no additional prime variable. For each fixed $p\in\mathcal P_i$, we now compute its contribution to the auxiliary operator expansion.

If $A_p=\varnothing$, no selection contributes $1$, while
appending at $a$ and removing at $c>a$ give the remaining terms in
\begin{equation*}
	1+\frac{1}{(p+1)b'_p}
	\sum_{0\le a<c\le2R}\eta'_a\eta'_c
	\mathbf1_{[z_a]_p=[z_c]_p}
	\prod_{a<j<c}
	\bigl(1-\eta'_j\mathbf1_{[z_j]_p=[z_a]_p}\bigr),
\end{equation*}
since $V_i\nu_i(p)=1/((p+1)b'_p)$.

If $A_p\ne\varnothing$, let $\kappa_p$ be the number of its maximal intervals of consecutive indices. Between two successive intervals the pair for $p$ must remain in $\mathfrak m_i$, since a second independent selection of $p$ is excluded. The contribution vanishes unless $[z_j]_p$ has the same value $\ell$ for every $j\in A_p$. The factors for selecting $p$ first in a prime tuple or first in $\mathscr{G}_j$ are, respectively,
\begin{equation*}
	\nu_i(p)V_i^{-(\kappa_p-1)}
	=V_i\nu_i(p)V_i^{-\kappa_p}
	=\frac{V_i^{-\kappa_p}}{(p+1)b'_p}.
\end{equation*}
In the suffix expansion, the constant $1$ corresponds to
immediate discarding, or an empty suffix. The prefix and suffix
identities above, together with \eqref{eq:thetaid}, therefore give
\begin{equation*}
	\frac{V_i^{-\kappa_p}}{(p+1)b'_p}
	\prod_{\substack{0\le j\le2R\\j\notin A_p}}
	\bigl(1-\eta'_j\mathbf1_{\ell=[z_j]_p}\bigr).
\end{equation*}
For $A_p=\varnothing$ we recover \eqref{eq:pfactor} divided by $b'_p$; for $A_p\ne\varnothing$ we also have the factor $V_i^{-\kappa_p}$ from the prime tuples.

The $l!$ orderings of distinct appended primes cancel $1/l!$.
Empty initial and final auxiliary tuples leave no unmatched
$\vartheta$-factors. Multiplication over $p$ proves the identity below.

We define $\mathbf b(z,\,\boldsymbol p,\,\mathfrak m)=\mathbf1_{z=(1,\,0)^T}\mathbf1_{k_1=\cdots=k_K=0}$. Then
\begin{equation*}
	\|\mathbf b\|^2=\prod_i s_i^{J+1}=1+o(1),\quad 
	\mathcal M_R(u,\,v,\,r)=\Bigl(\prod_p b'_p\Bigr)
	\bigl\langle \mathbf b,\,\mathscr{G}_0\mathscr{O}_0\mathscr{G}_1\cdots \mathscr{O}_{2R-1}\mathscr{G}_{2R}\mathbf b\bigr\rangle_{\mathrm{distinct}}.
\end{equation*}
The function is denoted by $\mathbf b$ to distinguish it from the prime product $b$ in \eqref{Dba-definition}.

The subscript $\mathrm{distinct}$ requires the pairwise distinctness condition specified above.

We now estimate the error caused by restricting the total number of pairs in $\mathfrak m$ to $\lceil L^2\rceil$.
	
We expand the matrix element using the definitions of $\mathscr{G}_j$ and $\mathscr{O}_j$ without the restriction $\sum_i k_i\le \lceil L^2\rceil$, and write
\begin{equation*}
\langle \mathbf b,\,\mathscr{G}_0\mathscr{O}_0\mathscr{G}_1\cdots \mathscr{O}_{2R-1}\mathscr{G}_{2R}\mathbf b\rangle_{\mathrm{distinct}}=\sum_\gamma c_\gamma.
\end{equation*}

Let $\gamma$ specify all vectors, prime tuples, permutations, and coordinate choices in the defining sums for $\mathscr{G}_j$ and $\mathscr{O}_j$, subject to the stated distinctness conditions; $c_\gamma$ includes all measure weights and factorial factors. We write $k_i^{(t)}(\gamma)$ for the number of coordinates of $\mathfrak m_i$ after the first $t$ factors, counted from the left, and define
\begin{equation*}
\mathcal E_{\mathrm{mem}}=\biggl(\prod_p b'_p\biggr)\sum_\gamma|c_\gamma|\mathbf1_{\max_{0\le t\le4R+1}\sum_{i=1}^K k_i^{(t)}(\gamma)>\lceil L^2\rceil}.
\end{equation*}
The $2R$ operators $\mathscr{O}_j$ append at most $2KR$ coordinates in total, so every term counted in $\mathcal E_{\mathrm{mem}}$ has at least $\lceil L^2\rceil-2KR$ coordinates appended by the operators $\mathscr{G}_j$. We multiply $|c_\gamma|$ by $2$ for each coordinate appended by an operator $\mathscr{G}_j$.
	
For a prime $p\in\mathcal P_i$ that occurs in none of the prime tuples at $z_0,\,\ldots,\,z_{2R}$, we have $A_p=\varnothing$. We sum over the indices $0\le a<c\le2R$ for which $\mathscr{G}_a$ appends the coordinate $(p,\,[z_a]_p)$ to $\mathfrak m_i$ and $\mathscr{G}_c$ removes that coordinate. We retain this coordinate between those two factors with the inserted factor $2$, getting
	\begin{equation*}
		\begin{aligned}
		\frac2{(p+1)b'_p}\sum_{0\le a<c\le2R}\eta'_a\eta'_c\mathbf1_{[z_a]_p=[z_c]_p}
		\prod_{a<j<c}\Bigl(1-\eta'_j\mathbf1_{[z_j]_p=[z_a]_p}\Bigr)\le\frac{2(2R+1)}{(p+1)b'_p}\ll\frac Rp.
		\end{aligned}
	\end{equation*}
Indeed, for fixed $a$, the sum over $c$ telescopes to at most $1$. Enlarging the product to all primes in $\bigcup_i\mathcal P_i$, their total factor is bounded by
	\begin{equation*}
		\prod_{p\in\bigcup_i\mathcal P_i}\biggl(1+\frac{C_\rho R}p\biggr)\le\exp\biggl(C_\rho R\sum_{p\in\bigcup_i\mathcal P_i}\frac1p\biggr)\le\exp(C_{K,\,\rho}R).
	\end{equation*}
There are $O_K(J+R)$ distinct primes in the tuples at $z_0,\,\ldots,\,z_{2R}$. For each such prime there are at most $(2R+2)^2$ choices of the indices at which its pair is appended or removed by the operators $\mathscr{G}_j$. Using the bound $\exp(O_{K,\rho}(L^{0.71}))$ for the sum of the absolute coefficients over the vectors $z_j$, prime lists $\boldsymbol p^{(j)}$ and permutations $\boldsymbol\pi,\boldsymbol\sigma$ in Subsection \ref{sec:roots}, we obtain
	\begin{equation}\label{eq:memtail}
		\begin{aligned}
		\mathcal E_{\mathrm{mem}}&\le2^{-\lceil L^2\rceil+2KR} \exp\bigl(O_{K,\,\rho}(L^{0.71}+(J+R)\log(2R+2)+R)\bigr)\\
		&\ll2^{-\lceil L^2\rceil+2KR}\exp(O_{K,\,\rho}(L^{0.74}))
		\ll_A L^{-2AR}.
		\end{aligned}
	\end{equation}
	
Under the restriction $\sum_i k_i\le \lceil L^2\rceil$, let $\mathcal B$ index the initial list entries, the unshared list coordinate in each group for each $\mathscr{O}_j$, and the first $\lceil L^2\rceil$ coordinates of the tuple added by each $\mathscr{G}_j$ in each group. Order the initial indices first and the others according to the product $\mathscr{G}_0\mathscr{O}_0\mathscr{G}_1\cdots \mathscr{O}_{2R-1}\mathscr{G}_{2R}$, with a fixed order within each factor. Then
	\begin{equation}\label{eq:birthct}
		|\mathcal B|\le K(J+1)+2KR+K(2R+1)\lceil L^2\rceil\le L^{O_K(1)}.
	\end{equation}
For each term in the expansion, let $\mathcal B_{\mathrm{act}}\subseteq\mathcal B$ consist of the initial indices, the indices where $\mathscr{O}_j$ selects a prime without taking it from $\mathfrak m$, and the indices of the coordinates actually added by $\mathscr{G}_j$. Write $p_a$ for the prime selected at $a\in\mathcal B_{\mathrm{act}}$. Thus transferring an existing pair from $\mathfrak m$ to a list introduces no new index in $\mathcal B_{\mathrm{act}}$. Let $\mathcal C_0$ range over collections of disjoint subsets $C\subseteq\mathcal B$ with $|C|\ge2$, and let $E_C$ mean that $C\subseteq\mathcal B_{\mathrm{act}}$ and $p_a=p_{a'}$ for all $a,\,a'\in C$. Then
	\begin{equation}\label{eq:025}
		\prod_{\substack{a,\,a'\in\mathcal B_{\mathrm{act}}\\a<a'}}\mathbf1_{p_a\ne p_{a'}}=\sum_{\mathcal C_0}\prod_{C\in\mathcal C_0}(-1)^{|C|-1}(|C|-1)!\prod_{C\in\mathcal C_0}\mathbf1_{E_C}.
	\end{equation}
Indeed, on the indices $a\in\mathcal B_{\mathrm{act}}$ having one common prime value, the right side is the sum of permutation signs: a cycle on $l$ specified indices has $(l-1)!$ orderings and sign $(-1)^{l-1}$. This sum is $1$ for a singleton and $0$ otherwise.
	
For each collection define its rank by $\kappa(\mathcal C_0)=\sum_{C\in\mathcal C_0}(|C|-1)$. It involves at most $2\kappa(\mathcal C_0)$ indices, and
\begin{equation}\label{eq:rankdef}
		\sum_{\kappa(\mathcal C_0)=\kappa}\prod_{C\in\mathcal C_0}(|C|-1)!\le |\mathcal B|^{2\kappa}.
	\end{equation}

To obtain this bound, decompose each counted permutation canonically into $\kappa$ transpositions; the ordered list of pairs of indices determines the permutation. Define
\begin{equation*}
		\mathcal I(\mathcal C_0)=\bigl\langle \mathbf b,\,\mathscr{G}_0\mathscr{O}_0\mathscr{G}_1\cdots \mathscr{O}_{2R-1}\mathscr{G}_{2R}\mathbf b\bigr\rangle_{\bigcap_{C\in\mathcal C_0}E_C},
	\end{equation*}
where we retain only terms satisfying all the conditions $E_C$ and $\sum_i k_i\le \lceil L^2\rceil$ after each factor $\mathscr{G}_j$ or $\mathscr{O}_j$.
	
For large rank, sum the prime choices in the specified order. In each $C$, leave the first choice unrestricted; each of the other $|C|-1$ choices then has a prescribed value. Let $R_{\mathrm{fix},s}(\mathscr{O}_j)$ be the weighted absolute row bound when $s$ specified primes selected by $\mathscr{O}_j$, without being taken from $\mathfrak m$, have prescribed values. The count in \eqref{eq:target} retains the corresponding factors $\nu_i(p)\le\Delta$ and gives
\begin{equation}\label{eq:eatom}
		R_{\mathrm{fix},s}(\mathscr{O}_j)\ll_{K,\,\rho}\lceil L^2\rceil^K(1+L^{3A_0})\Delta^s\qquad(0\le s\le K).
	\end{equation}

For the group-$i$ part $\mathscr{G}_{j,i}$ of $\mathscr{G}_j$, define $R_{\mathrm{fix},s}(\mathscr{G}_{j,i})$ similarly by prescribing $s$ ordered coordinates of the new tuple. The factor from removing or retaining the existing pairs is at most $1$, so
	\begin{equation}\label{eq:gatom}
		R_{\mathrm{fix},s}(\mathscr{G}_{j,i})\le\sum_{l\ge s}\frac{(v_0\eta'_jV_i/\vartheta)^l\Delta^s s_i^{l-s}}{l!}\ll_{K,\,\rho}(C_{K,\,\rho}\Delta)^s.
	\end{equation}
This follows from $1/l!\le1/(l-s)!$; requiring the tuple length to reach a specified index only decreases the sum in \eqref{eq:gatom}. Initial entries have the same factor $\Delta$. These bounds are uniform in all previously selected values, including earlier coordinates of the same tuple.
	
We omit the condition $z_{2R}=z_0$ and iterate the absolute row bounds, using $\mathbf b\le w$. We absorb the bounds for the two occurrences of $\mathbf b$ in $\mathcal I(\mathcal C_0)$ and for all $\mathscr{G}_j$ into a fixed enlargement of $C_*$. We obtain
\begin{equation*}
		|\mathcal I(\mathcal C_0)|\le L^{C_*(2R+1)}(C_{K,\,\rho}\Delta)^{\kappa(\mathcal C_0)}.
	\end{equation*}

Choose $C_0>2(C_*+E_0+7)$ and put
\begin{equation}\label{eq:rankcut}
		\kappa_0=\biggl\lceil C_0\frac{2R\log L}{L^{0.1}}\biggr\rceil,\qquad \frac{\kappa_0}{2R}=O\biggl(\frac{\log L}{L^{0.1}}\biggr)=o(1).
	\end{equation}

Since $C_{K,\,\rho}|\mathcal B|^2\Delta\le e^{-L^{0.1}/2}$ for large $x$, the high-rank contribution satisfies
\begin{equation}\label{eq:hirank}
		\begin{aligned}
			&\sum_{\kappa(\mathcal C_0)\ge \kappa_0}\prod_{C\in\mathcal C_0}(|C|-1)!\,|\mathcal I(\mathcal C_0)|\le L^{C_*(2R+1)}\sum_{\kappa\ge \kappa_0}(C_{K,\,\rho}|\mathcal B|^2\Delta)^\kappa\ll L^{-2R(E_0+4)}.
		\end{aligned}
	\end{equation}
	
For $\kappa<\kappa_0$, retain the signed estimate \eqref{eq:esmall} for each $\mathscr{O}_j$ unaffected by the equality constraints. Write $C=\{a_0,\,\ldots,\,a_s\}$, with $a_0$ its first index in the specified order. Orthogonality on $\mathbb T$ gives
	\begin{equation}\label{eq:orth}
		\mathbf1_{E_C}=\begin{cases}\displaystyle\int_{\mathbb T^s}e\biggl(\sum_{h=1}^s\theta_h(p_{a_h}-p_{a_0})\biggr)\,\mathrm d\boldsymbol\theta,&C\subseteq\mathcal B_{\mathrm{act}},\\0,&C\not\subseteq\mathcal B_{\mathrm{act}}.\end{cases}
	\end{equation}
When $C\subseteq\mathcal B_{\mathrm{act}}$, we factor the exponential in \eqref{eq:orth} as follows; we call the individual exponentials on the right-hand side its phase factors:
	\begin{equation*}
		e\biggl(\sum_{h=1}^s\theta_h(p_{a_h}-p_{a_0})\biggr)=e\biggl(-p_{a_0}\sum_{h=1}^s\theta_h\biggr)\prod_{h=1}^s e(\theta_hp_{a_h}).
	\end{equation*}
Thus each phase factor multiplies only the initial entry or the operator $\mathscr{O}_j$ or $\mathscr{G}_j$ at which its prime is selected.
	
For an ordered tuple $\mathbf p=(p_1,\,\ldots,\,p_l)\in\mathcal P_i^l$ appended by $\mathscr{G}_j$, we define $\Phi_l(\mathbf p)$ as the product, over all $C\in\mathcal C_0$, of those factors $e(-p_{a_0}\sum_{h=1}^s\theta_h)$ and $e(\theta_hp_{a_h})$ in the preceding factorization whose selection indices refer to coordinates of this tuple. We set $\Phi_l=0$ if any specified coordinate index exceeds $l$. We write $\mathfrak S_l$ for the permutation group on $\{1,\,\ldots,\,l\}$ and replace $\Phi_l$ by
	\begin{equation*}
		\Phi_l^{\mathrm{sym}}(\mathbf p)=\frac1{l!}\sum_{\sigma\in\mathfrak S_l}\Phi_l(p_{\sigma(1)},\,\ldots,\,p_{\sigma(l)}),\qquad |\Phi_l^{\mathrm{sym}}(\mathbf p)|\le1.
	\end{equation*}
We sum over all ordered tuples $\mathbf p=(p_1,\,\ldots,\,p_l)\in\mathcal P_i^l$, allowing repeated primes and assigning each tuple the weight $\frac1{l!}\prod_{a=1}^l\nu_i(p_a)$. For every symmetric function $F$ on $\mathcal P_i^l$, we have
	\begin{equation*}
		\frac1{l!}\sum_{\mathbf p\in\mathcal P_i^l}\Phi_l^{\mathrm{sym}}(\mathbf p)F(\mathbf p)\prod_{a=1}^l\nu_i(p_a)=\frac1{l!}\sum_{\mathbf p\in\mathcal P_i^l}\Phi_l(\mathbf p)F(\mathbf p)\prod_{a=1}^l\nu_i(p_a).
	\end{equation*}
	
For an index $a$ associated with $\mathscr{O}_j$, we have $\mathbf1_{a\in\mathcal B_{\mathrm{act}}}=1$ precisely when its prime is chosen with weight $\nu_i(p)$ rather than taken from $\mathfrak m_i$. For the $h$th coordinate of a tuple of length $l$ appended by $\mathscr{G}_j$, we have $\mathbf1_{a\in\mathcal B_{\mathrm{act}}}=\mathbf1_{h\le l}$. We may permute the existing coordinates of $\mathfrak m_i$ without changing these indicators, and the subsequent operators commute with those permutations. We may therefore replace $\Phi_l$ by $\Phi_l^{\mathrm{sym}}$ without changing $\mathcal I(\mathcal C_0)$ or the bounds in \eqref{eq:crudebd}.
	
At most $2\kappa$ operators $\mathscr{O}_j$ are modified; all the others satisfy \eqref{eq:esmall}. The norms of $\mathscr{G}_j$ remain bounded, and the exponential factors at the initial prime tuples do not increase $\|\mathbf b\|$. Consequently,
\begin{equation*}
		|\mathcal I(\mathcal C_0)|\le C_{K,\,\rho}^{2R+1}\|\mathbf b\|^2L^{-A(2R-2\kappa)+2C_*\kappa}.
	\end{equation*}

Summing over the collections of rank $\kappa<\kappa_0$ and using \eqref{eq:rankdef}, we get
\begin{equation}\label{eq:lorank}
	\begin{aligned}
	\sum_{\kappa(\mathcal C_0)<\kappa_0}\prod_{C\in\mathcal C_0} (|C|-1)!\,|\mathcal I(\mathcal C_0)|
	&\le\sum_{\kappa<\kappa_0}|\mathcal B|^{2\kappa} C_{K,\,\rho}^{2R+1}\|\mathbf b\|^2L^{-A(2R-2\kappa)+2C_*\kappa}
	\\
	&\le L^{-2AR+o(R)}\ll L^{-2R(E_0+4)}.
	\end{aligned}
	\end{equation}
Here $\kappa_0=o(R)$ and all parameters in $C_*$ are fixed before $x\to\infty$.

Combining \eqref{eq:memtail}, \eqref{eq:hirank}, and \eqref{eq:lorank}, and restoring $\prod_p b'_p\le1$, we find that
	\begin{equation*}
		|\mathcal M_R(u,\,v,\,r)|\ll L^{-2R(E_0+4)}
	\end{equation*}
	holds uniformly for $1\le u/(d_0YM)\le16$, $1\le v/N\le2$, and $0\le r<1$. By \eqref{eq:moment-integral}, this gives \eqref{root-aver}, since the integration region has volume $15$ and $\exp(-cL^{3/4})\ll L^{-2RE_0}$ for $R\asymp L^{1/2}$. The identity \eqref{oper-moment-iden} then gives \eqref{eq:moment}. 
\end{proof}

\vskip 3mm

\subsection{Removal of auxiliary restrictions}\label{sec:bad}\label{sec:pads}

In this subsection, we shall establish the following bound for the second sum in \eqref{eq:operator-reduction} for any $A>0$:
\begin{equation}\label{eq:second-sum-bound}
\sum_{d_0}\frac{\bigl|\langle f,\,\mathcal T_{-2}^{\mathcal C}f \rangle-\langle f,\,\mathcal A_{-2}^{\mathcal C}f\rangle\bigr|}{d_0}
	\ll_{A,\,K,\,\rho} XYL^{-A}.
\end{equation}
Once \eqref{eq:second-sum-bound} is proved, inserting it and the estimate \eqref{eq:first-sum-bound} for the
first sum into \eqref{eq:operator-reduction}, we immediately have
$|Q_Y-Q_Y^{\mathrm{maj}}|\ll XYL^{-A}$, completing the proof of
Proposition \ref{prop:D}.

For each $d_0$ in \eqref{eq:second-sum-bound}, we take $f$ as in \eqref{def-f} and write $\mathcal K=\mathcal S\mathcal T_{-2}^{\mathcal C}\mathcal S$. 
Note that the value $f(P,\,\boldsymbol p)$ depends only on $P$, whence $\mathcal Sf=f$ by the definition of $\mathcal S$ in \eqref{def-S}. This, together with $\mathcal S^*=\mathcal S$ shows that
\begin{equation*}
	\langle f,\,\mathcal Kf\rangle=\langle\mathcal Sf,\,\mathcal T_{-2}^{\mathcal C}\mathcal Sf\rangle=\langle f,\,\mathcal T_{-2}^{\mathcal C}f\rangle.
\end{equation*}
Therefore, to get \eqref{eq:second-sum-bound}, we only need to consider $|\langle f,\,\mathcal Kf\rangle-\langle f,\,\mathcal A_{-2}^{\mathcal C}f\rangle|$.
Noting that $\mathcal G_{-2}^*=\mathcal G_{-2}$ and  $\mathcal A_{-2}^{\mathcal C}=\mathcal G_{-2}\mathcal K\mathcal G_{-2}$ by \eqref{eq:aone}, we have
\begin{equation}\label{eq:badends}
\begin{aligned}
\big\langle f,\,\mathcal Kf\big\rangle
-\big\langle f,\,\mathcal A_{-2}^{\mathcal C}f\big\rangle
=\big\langle(1-\mathcal G_{-2})f,\,\mathcal Kf\big\rangle +\big\langle\mathcal G_{-2}f,\,\mathcal K(1-\mathcal G_{-2})f \big\rangle.
\end{aligned}
\end{equation}
Note that in $\big\langle(1-\mathcal G_{-2})f,\,\mathcal Kf\big\rangle$, the summand is nonzero only if at least one of the two conditions \eqref{eq:goodone} and \eqref{eq:goodtwo} defining $\mathcal G_{-2}$ in \eqref{def-G_-2} fails at $(P,\,\boldsymbol p)$; and in $\big\langle\mathcal G_{-2}f,\,\mathcal K(1-\mathcal G_{-2})f \big\rangle$, both conditions must hold at $(P,\,\boldsymbol p)$ and at least one must fail at $(Q,\,\boldsymbol p')$. 
We shall apply the estimate \eqref{eq:one-edge-integral} in Proposition \ref{prop:position-integrals} to the these two terms respectively.

We first consider  $\langle(1-\mathcal G_{-2})f,\,\mathcal Kf\rangle$. 
Observe that $\langle(1-\mathcal G_{-2})f,\,\mathcal Kf\rangle=\langle(1-\mathcal G_{-2})f,\,\mathcal T_{-2}^{\mathcal C}f\rangle$, since $\mathcal S^*=\mathcal S$, $\mathcal Sf=f$
and $\mathcal S(1-\mathcal G_{-2})f=(1-\mathcal G_{-2})f$.
We use the coordinates $TP=g'z$ and $TQ=g'z'$, where $z=z_0=(1,\,0)^T$, $z'=z_1$ and $g'$ is defined by \eqref{eq:gprime}.
In \eqref{eq:one-edge-integral}, we take $r=c/u\in[0,\, 1)$ in the sum over $P=(u,\,v)^T$, with $r\equiv r_{_P}\,(\textrm{mod}\,1)$.
For each of the two choices of $W$ in \eqref{W-choices}, we denote by $\Phi_W$ the sum defining $\Phi$ with the following restriction in $\sum^{*}$: for the initial prime tuples $\boldsymbol p$, at least one of \eqref{eq:goodone} and \eqref{eq:goodtwo} fail with $r_{_P}$ replaced by $-2r$, since otherwise $\bigl((1-\mathcal G_{-2})f\bigr)(P,\boldsymbol p)=0$.
Recall that $|\overline{f(P)}f(Q)|\le L^{4C}$. Hence, we deduce by calculation that
\begin{align}\label{first-term-esti}
	\frac{\bigl|\langle(1-\mathcal G_{-2})f,\,\mathcal Kf \rangle\bigr|}{d_0YMN}
	\le L^{4C}\sum_W\frac1{d_0YMN} \sum_{\substack{P=(u,\,v)^T\\P\in\Omega\cap\mathcal C}} \Phi_W\Bigl(u,\,v,\,\frac cu;\,(\ell_p(g'))_p\Bigr),
\end{align}
where $\sum_W$ is over the two terms in \eqref{W-choices}.
For each $W$, by \eqref{eq:one-edge-integral} in Proposition \ref{prop:position-integrals}, we get
\begin{equation}\label{sum-integral}
	\begin{aligned}
	\frac1{d_0YMN}&\sum_{\substack{P=(u,\,v)^T\\P\in\Omega\cap\mathcal C}}\Phi_W\Bigl(u,\,v,\,\frac cu;\,(\ell_p(g'))_p\Bigr)
	\\[0,6mm]
	&=\frac1{3\zeta(2)}\int_1^{16}\int_1^2\int_0^1
	\frac{\sum_{\boldsymbol\ell}\Phi_W(d_0YM s,\,Nt_0,\,r;\, \boldsymbol\ell)}{\prod_{q\in\bigcup_i\mathcal P_i}(q+1)}\,\mathrm dr\,\mathrm dt_0\,\mathrm ds+O_A(L^{-A}).
\end{aligned}
\end{equation}
Note that the sum $\sum_{\boldsymbol\ell}$ above runs over all choices of line $\ell_p\in\mathbb P^1(\mathbb F_p)$ for each prime $p\in\bigcup_{i=1}^K\mathcal P_i$.
We then interchange  $\sum_{\boldsymbol\ell}$ with the sum $\sum^{*}$ over $z',\,\boldsymbol p,\,p'_1,\,\ldots,\,p'_K$ in $\Phi_W$. 

For fixed $z'$ and prime tuples $\boldsymbol p,\,\boldsymbol p'$, the sum $\sum^{*}$  requires that $\ell_p=[z]_p$ for any $p\mid Db$ and $\ell_p=[z']_p$ for any $p\mid Da$. Under these restrictions, we have 
$$\omega_{\boldsymbol\ell}(z)=\sum_{i=1}^K\,\sum_{p\in\mathcal P_i}\mathbf1_{\ell_p=[z]_p}\ge\sum_{p\mid Db}1=K(J+1),$$ 
and similarly $\omega_{\boldsymbol\ell}(z')\ge K(J+1)$.
Since $0<\rho<1$, thus the sum over $\boldsymbol\ell$ can be estimated by
\begin{equation*}
	\begin{aligned}
		\sum_{\substack{\boldsymbol\ell\\[0,1mm] \ell_p=[z]_p\ (p\mid Db)\\[0,7mm] \ell_{p}=[z']_{p}\ (p\mid Da)}} \rho^{\frac{\omega_{\boldsymbol\ell}(z)+\omega_{\boldsymbol\ell}(z')-2K(J+1)}{2}}
		\le\sum_{\substack{\boldsymbol\ell\\[0,1mm] \ell_p=[z]_p\ (p\mid Db)\\[0,7mm] \ell_{p}=[z']_{p}\ (p\mid Da)}}1 =\sum_{\boldsymbol\ell}\prod_{p\mid Db} \mathbf1_{\ell_p=[z]_p}\prod_{p\mid Da} \mathbf1_{\ell_{p}=[z']_{p}}.
	\end{aligned}
\end{equation*}
Therefore, the integrand in \eqref{sum-integral} is estimated by
\begin{align}\label{two-product-factor}
	%\begin{aligned}
		\prod_{q\in\bigcup\limits_i\mathcal P_i}\frac1{q+1}
		\sum_{\boldsymbol\ell}\Phi_W(u,\,v,\,r;\,\boldsymbol\ell)
		&\le \prod_{i=1}^K \frac{1}{V_i^{J+2}} 
		\mathop{{\sum}^{\dagger}}_{\substack{z',\,\boldsymbol p\\ p'_1,\, \ldots,\, p'_K}}\frac{d_0}{D}\biggl|\eta\Bigl(\frac aY\Bigr) \eta\Bigl(\frac bY\Bigr) \psi_1\Bigl(\frac{t'}Y\Bigr)\biggr|W(t'-b+a)
		\\[-3mm]\nonumber
		&\quad\times \prod_{q\in\bigcup_i\mathcal P_i}\frac1{q+1} \sum_{\boldsymbol\ell}\,
		\prod_{p\mid Db} \mathbf1_{\ell_p=[z]_p}
		\prod_{p\mid Da} \mathbf1_{\ell_p=[z']_p},
	%\end{aligned}
\end{align}
where $\sum^\dagger$ runs over $z',\,\boldsymbol p,\, p'_1,\,\ldots,\,p'_K$ for which the restrictions in the sum defining $\Phi_W$ hold for at least one choice of $\boldsymbol\ell=(\ell_p)_{p\in\cup_{i=1}^K\mathcal P_i}$.

We first consider the factor in the second line in \eqref{two-product-factor}. 
For a fixed prime $p'\mid b$, we have $p'\nmid Da$. Hence the only factor in the two products involving $\ell_{p'}$ is $\mathbf1_{\ell_{p'}=[z]_{p'}}$, and
\begin{equation*}
	\sum_{\ell_{p'}\in\mathbb P^1(\mathbb F_{p'})}\prod_{p\mid Db}\mathbf1_{\ell_p=[z]_p}\prod_{p\mid Da}\mathbf1_{\ell_p=[z']_p}=\biggl(\prod_{\substack{p\mid Db\\p\ne p'}}\mathbf1_{\ell_p=[z]_p}\biggr)\biggl(\prod_{p\mid Da}\mathbf1_{\ell_p=[z']_p}\biggr)\sum_{\ell_{p'}\in\mathbb P^1(\mathbb F_{p'})}\mathbf1_{\ell_{p'}=[z]_{p'}}.
\end{equation*}
The sum over $\ell_{p'}$ on the right equals $1$, since only the term with $\ell_{p'}=[z]_{p'}$ is nonzero. For $p'\mid a$, the same calculation applies to the indicator $\mathbf1_{\ell_{p'}=[z']_{p'}}$.
For $p'\mid D$, the line $\ell_{p'}$ must satisfy both $\ell_{p'}=[z]_{p'}$ and $\ell_{p'}=[z']_{p'}$. Hence we obtain
\begin{equation*}
	\sum_{\ell_{p'}\in\mathbb P^1(\mathbb F_{p'})}\prod_{p\mid Db}\mathbf1_{\ell_p=[z]_p}\prod_{p\mid Da}\mathbf1_{\ell_p=[z']_p}=\biggl(\prod_{\substack{p\mid Db\\p\ne p'}}\mathbf1_{\ell_p=[z]_p}\biggr)\biggl(\prod_{\substack{p\mid Da\\p\ne p'}}\mathbf1_{\ell_p=[z']_p}\biggr)\mathbf1_{[z]_{p'}=[z']_{p'}}.
\end{equation*}
Since $D$ is squarefree and $z,\,z'$ are primitive, the conditions $[z]_p=[z']_p$ for all $p\mid D$ are equivalent to $D\mid\det(z,\,z')$. 
In addition, for $q\in\bigcup_i\mathcal P_i$ with $q\nmid Dab$, neither product involves $\ell_q$, so all $q+1$ choices of $\ell_q$ are allowed and $(q+1)^{-1}\sum_{\ell_q\in\mathbb P^1(\mathbb F_q)}1=1$.
Thus, for fixed prime tuples $\boldsymbol p,\,\boldsymbol p'$, the above arguments yield
\begin{equation}\label{eq:line-sum}
	\prod_{q\in\bigcup_i\mathcal P_i}\frac1{q+1}\ \sum_{\boldsymbol\ell}\prod_{p\mid Db} \mathbf1_{\ell_p=[z]_p}\prod_{p\mid Da} \mathbf1_{\ell_p=[z']_p}
	=\mathbf1_{D\mid\det(z,\,z')}\prod_{p\mid Dab}\frac1{p+1}.
\end{equation}
%We retain $D\mid\det(z,\,z')$ in the sum over $z'$.

We next sum the weights $\prod_{q\mid Dab}(q+1)^{-1}$ over the prime tuples $\boldsymbol p_1,\,\ldots,\,\boldsymbol p_K$ and the primes $p'_1,\,\ldots,\,p'_K$ in $\sum^\dagger$.
By \eqref{Dba-definition}, the primes $p$ with $p\mid Dab,$ $p\in\mathcal P_i$ are just the $J+2$ distinct primes $p_{_{i,1}},\,\ldots,\,p_{_{i,J+1}},\,p'_{_i}$. Hence, we get
\begin{equation*}%\label{eq:prime-tuple-sum}
	\sum_{\substack{p_{_{1,1}},\,\ldots,\,p_{_{1,J+1}},\,p'_{_1}\in \mathcal P_1\\[0,4mm] \mu(p_{_{1,1}}\ldots p_{_{1,J+1}}p'_{_1}) \neq 0}}
	\cdots \sum_{\substack{p_{_{K,1}},\,\ldots,\,p_{_{K,J+1}},\,p'_{_K}\in \mathcal P_K\\[0,4mm] \mu(p_{_{K,1}}\ldots p_{_{K,J+1}}p'_{_K}) \neq 0}}\,
	\prod_{q\mid Dab}\frac1{q+1} \le\prod_{i=1}^K\biggl(\sum_{q\in\mathcal P_i} \frac1{q+1}\biggr)^{J+2}
	\le\prod_{i=1}^K V_i^{J+2},
\end{equation*}
since $V_i=\sum_{q\in\mathcal P_i}1/q$. 
%Multiplying this bound by the factor $\prod_{i=1}^K V_i^{-(J+2)}$ in front of $\sum^\dagger$ gives an upper bound of $1$.
Substituting this upper bound, together with \eqref{eq:line-sum}, \eqref{two-product-factor} and \eqref{sum-integral} into \eqref{first-term-esti}, we deduce by calculation that
\begin{equation}\label{first-term-esti-2}
	\frac{\langle(1-\mathcal G_{-2})f,\,\mathcal Kf\rangle}{d_0YMN}
	\ll L^{4C}\sum_W\sup_{\substack{\boldsymbol p\\[-0,1mm] p'_1,\ldots,p'_K}} \int_1^{16}\int_1^2\int_{E(\boldsymbol p)}\sum_{\substack{\widetilde gz'\in\Omega' \\[0,4mm] D\mid\det(z,\,z')}}\Bigl|\psi_1\Bigl(\frac{t'}Y\Bigr)\Bigr|W(t'-b+a)\,\mathrm dr\,\mathrm dt_0\,\mathrm ds+L^{-A}.
\end{equation}
Here $E(\boldsymbol p)$ denotes the set of $r\in[0,\,1)$ for which \eqref{eq:goodone} or \eqref{eq:goodtwo} fails for the initial prime tuples $\boldsymbol p$ with $r_{_P}$ replaced by $-2r$, the matrix $\widetilde g=\widetilde g(d_0YM s,\,Nt_0,\,r)$ is defined by \eqref{def-g'(uvr)}, 
and $z'$ ranges over primitive integer vectors with odd first coordinate.
%We have also used the estimate $(d_0/D)|\eta(a/Y)\eta(b/Y)|\ll1$.

We now estimate the sum  over $z'$ uniformly in $D,\,a,\,b$ and $u,\,v,\,r$. Let $t'=-t/2$. By  \eqref{eq:019}, the two choices of $W(t'-b+a)$ in \eqref{W-choices} are as follows:
\begin{equation*}
	\mathbf1_{t'=b-a}=\mathbf1_{t=2(a-b)},\qquad \biggl|\int_{\mathfrak M_1}e\bigl(\phi(t'-b+a)\bigr)\,\mathrm d\phi\biggr|=\biggl|\int_{\mathfrak M_{-2}}e\bigl(\theta(t-2(a-b))\bigr)\,\mathrm d\theta\biggr|;
\end{equation*}
%and we shall bound the sum over $z'$ in $\sum^\dagger$ for the two choices of $W$ in \eqref{W-choices}. The bounds below are uniform in $u,\,v,\,r$ and $D,\,a,\,b$.
Since $\det\widetilde g=1$ and $z=(1,\,0)^T$, we have
\begin{equation*}
	-\frac t2=t'=\frac{\det(\widetilde gz,\,\widetilde gz')}{D}=\frac{\det(z,\,z')}{D},
\end{equation*}
and hence we may write $z'=(h,\,-Dt/2)^T$ with $h\in\mathbb Z$.

By \eqref{eq:target}, we have $1/16\le h+rDt\le16$. Thus, for any fixed $t\in2\mathbb Z$ and $r\in[0,\,1)$,
\begin{equation}\label{number-z'}
	\#\Bigl\{h\in\mathbb Z:\ 1/16-rDt\le h\le16-rDt\Bigr\}\le16.
\end{equation}
Applying \eqref{number-z'} with $t=2(a-b)$, together with the bound $|\psi_{-2}|\le1$, one has
\begin{equation*}
	\sum_{z'}\Bigl|\psi_{-2}\Bigl(\frac tY\Bigr)\Bigr| \mathbf1_{t=2(a-b)}\le
	\#\Bigl\{h\in\mathbb Z:\ 1/16-2rD(a-b)\le h\le16-2rD(a-b)\Bigr\}
	\le16.
\end{equation*}

Next, for the major arc integral, the condition $\psi_{-2}(t/Y)\ne0$ implies $|t|<16Y$. Since $t\in2\mathbb Z$, there are $O(Y)$ possible values of $t$. For each such $t$, there are at most $16$ choices of $z'$ by \eqref{number-z'}. Using $|\mathfrak M_{-2}|\ll L^{3A_0}/Y$, we therefore obtain
\begin{equation*}
	\sum_{z'}\biggl|\psi_{-2}\Bigl(\frac tY\Bigr)\biggr| \biggl|\int_{\mathfrak M_{-2}} e\bigl(\theta(t-2(a-b))\bigr)\,\mathrm d\theta\biggr|
	\le 16\big|\mathfrak M_{-2}\big|\sum_{\substack{t\in2\mathbb Z,\,|t|<16Y}}1\ll  L^{3A_0}.
\end{equation*}

Finally, since the map $r\mapsto-2r$ preserves Lebesgue measure on $\mathbb T$, it follows from \eqref{meas-upperbound} that
\begin{equation*}
	\sup_{\boldsymbol p}\operatorname{meas}\Bigl\{r\in[0,\,1):\ \text{\eqref{eq:goodone} or \eqref{eq:goodtwo} fails for }\boldsymbol p\text{ with }r_{_P}\text{ placed by }-2r\Bigr\} \ll_K\exp(-cL^{0.1}).
\end{equation*}
Therefore, inserting the preceding three upper bounds into \eqref{first-term-esti-2}, we conclude that
\begin{equation}\label{first-term-esti-3}
	\bigl|\langle(1-\mathcal G_{-2})f,\,\mathcal Kf\rangle\bigr| \ll_{A,\,K,\,\rho} d_0YMN\big(L^{4C+3A_0}e^{-cL^{0.1}}+L^{-A}\big).
\end{equation}

For the second term of \eqref{eq:badends}, the restriction from $1-\mathcal G_{-2}$ is at $Q$, so we take $(Q,\,\boldsymbol p')$ as the initial pair in \eqref{eq:one-edge-integral}. The bound for the sum over $z'$ is unchanged, and the same measure estimate applies to $E(\boldsymbol p')$. 
Thus the same upper bound as \eqref{first-term-esti-3} holds for the second term $\bigl|\langle\mathcal G_{-2}f,\,\mathcal K(1-\mathcal G_{-2})f\rangle\bigr|$. This, together with \eqref{eq:badends} and \eqref{first-term-esti-3}, yields 
\begin{equation}\label{eq:baderr}
\frac{\bigl|\langle f,\,\mathcal T_{-2}^{\mathcal C}f\rangle
	-\langle f,\,\mathcal A_{-2}^{\mathcal C}f\rangle\bigr|}{d_0}
\ll_{A,\,K,\,\rho} XYL^{-A-2},
\end{equation}
since $MN=X$. Finally, we note that there are $O_K(L)$ values of $d_0$. Hence, summing \eqref{eq:baderr} over $d_0$, we get \eqref{eq:second-sum-bound}, thereby completing the proof of Proposition \ref{prop:D}.

\vskip 5mm

\section{\large The major arcs and completion of the proof}
\label{sec:major}

In this section, we first estimate $Q_Y^{\mathrm{maj}}$ defined in \eqref{def-Q_Y^maj}, and then in Subsection \ref{sec:final}, we  will prove Theorem \ref{thm:main} by applying this estimate together with Propositions \ref{prop:typeI} and \ref{prop:D}.

As before, let $X=MN\asymp x$, with $M,\,N\ge x^\delta$ for some $\delta>0$. The prime sets $\mathcal P_i$ are as in \eqref{eq:prime-sets}, and $V_i=\sum_{p\in\mathcal P_i}1/p$. Recall that $\mathcal A:=\big\{p_1\cdots p_K:\ p_i\in\mathcal P_i\big\},$ whence
\begin{equation}\label{eq:aharm}
	\biggl(\prod_{i=1}^K \frac1{V_i}\biggr)\sum_{a\in\mathcal A}\frac1a
	=\prod_{i=1}^K\biggl(\frac1{V_i}\sum_{p\in\mathcal P_i}\frac1p\biggr)=1.
\end{equation}
With the notation above, $Q_Y^{\mathrm{maj}}$ in \eqref{def-Q_Y^maj} becomes
\begin{equation*}
	Q_Y^{\mathrm{maj}}
	=\biggl(\prod_{i=1}^K\frac1{V_i^2}\biggr)
	\sum_{a,\,b\in\mathcal A}
	\eta\Bigl(\frac aY\Bigr)\eta\Bigl(\frac bY\Bigr)
	\sum_{\substack{M<m,\,r\le2M\\[0,3mm]N<n,\,s\le2N}}
	\alpha_m\overline{\alpha_r}\beta_n\overline{\beta_s}
	H_{-2}(bmn-ars;\,a,\,b).
\end{equation*}

\subsection{Fundamental lemmas}

Let $I_M\subseteq(M,\,2M]$ and $I_N\subseteq(N,\,2N]$ be two  intervals.
For any two Dirichlet characters $\chi_1$ and $\chi_2$, let
\begin{equation}\label{def-M_alpha-N_beta}
	M_\alpha(t):=\frac1M\sum_{m\in I_M}\alpha_m\chi_1(m)m^{it},
	\qquad
	N_\beta(t):=\frac1N\sum_{n\in I_N}\beta_n\chi_2(n)n^{it},
\end{equation}
where $|\alpha_m|\le1$ and $|\beta_n|\le L^C$ for some fixed $C>0$.

\begin{lemma}\label{lem-5.1}
Define the coefficient $c_u$ by
	\begin{equation*}
		c_u:=\frac1X\sum_{\substack{m\in I_M,\,n\in I_N \\[0,3mm]mn=u}}
		\alpha_m\beta_n\chi_1(m)\chi_2(n).
	\end{equation*}
	Then, for any interval $I=(t_0,\, t_0+T]$ of length $T\ge1$,
	\begin{equation}\label{eq:prodms}
		\int_I |M_\alpha(t)N_\beta(t)|^2\,\mathrm dt
		\ll (T+X)\sum_{X<u\le4X}|c_u|^2
		\ll L^{2C+3}(1+T/X).
	\end{equation}
\end{lemma}

\begin{proof}
	The first inequality follows from  Montgomery--Vaughan \cite[Corollary 3]{MonVau74}  with a change of variables $t=t_0-v$; and 
	the second is a standard divisor function estimate.
\end{proof}

\begin{lemma}\label{lem:sparse}
	Let $c,\,C_1>0$ be fixed constants and suppose 
	$cL^{1/10}\le\log P\le C_1L^{1/5}$. Let
	$\mathcal P\subseteq[P,\,2P)$ be a set of primes. For any Dirichlet character $\chi$, $\xi\in\mathbb R$ and any complex coefficient $|v_p|\le1$, we set
	\begin{equation*}
	\mathscr P(t):=\frac1P\sum_{p\in\mathcal P}v_p \chi(p)p^{i(t+\xi)}.
	\end{equation*}
	For fixed large constants $B,\, A_1$, define the set $\mathcal I$ by 
	\begin{equation*}
		\mathcal I:=\Big\{[j,\,j+1]:\ j\in\mathbb Z, 
		\text{ there exists } t\in[j,\,j+1] \text{ such that }
		|t|\le2XL^B,\, |\mathscr P(t)|>L^{-A_1}\Big\}.
	\end{equation*}
	Then we have
	\begin{equation}\label{eq:sparse}
		|\mathcal I|\le\exp(O(L^{0.91})),
		\qquad
		\sum_{I\in\mathcal I}\sup_{t\in I}|N_\beta(t)|^2\ll L^{2C+1}.
	\end{equation}
	%The implied constants may depend on the fixed parameters.
	The estimates are uniform in $\xi$, $\mathcal P$, the characters
	and the coefficients.
\end{lemma}
\begin{remark}
Almost the same time, Matom\"aki, Pandey, Pliego, Ter\"av\"ainen
and Wang \cite[Lemma 2.4]{MPPTW26} obtained a more
general sparse mean square estimate. This, combined with the first
estimate in \eqref{eq:sparse}, could imply the second
estimate with a slightly sharper bound  $L^{2C}$.
\end{remark}

\begin{proof}
For each interval $I\in\mathcal I$, we choose one $t_I\in I$ with
$|t_I|\le2XL^B$ and $|\mathscr P(t_I)|>L^{-A_1}$.
We then partition $\mathcal I$ into two classes according as $2\mid j$ or $2\nmid j$. For distinct intervals $I=[j,\,j+1]$ and $I'=[k,\,k+1]$ in the same class, we have $|j-k|\ge2$, whence $|t_I-t_{I'}|\ge1$.
	
For $p\in\mathcal P$, put $a_p=pv_p\chi(p)p^{i\xi}/(2P)$,
and put $a_p=0$ for $p\notin\mathcal P$, $p\in [P,\,2P]$. Then
	\begin{equation*}
		|a_p|\le1,
		\qquad
		\sum_{P\le p\le2P}a_p p^{-1+it}=\frac12\mathscr P(t).
	\end{equation*}
	Now we apply \cite[Lemma 8]{MR16} to the points $-t_I$ in each class
	with $T=2XL^B+1$ and $V=2L^{A_1}$, yielding
	\begin{equation*}
		\log(2+|\mathcal I|)
		\ll \log V+\frac{\log T}{\log P}(\log V+\log\log T)
		\ll L^{0.9}\log L.
	\end{equation*}
This proves the first estimate in \eqref{eq:sparse}.
	
We now treat the mean square estimate in \eqref{eq:sparse}. For each $I\in\mathcal I$,
choose $u_I\in I$ at which $|N_\beta|$ attains its supremum.
Let $\mathcal J=\{I_1,\cdots, I_R\}\subseteq\mathcal I$ be a fixed parity class  mentioned above. We write $u_r=u_{I_r}$ for $1\le r\le R$. Then the corresponding Gram matrix is $G=(G_{r,s})_{1\le r,s\le R}$, where
\begin{equation*}
	G_{r,s}=\sum_{N<n\le2N}n^{i(u_r-u_s)}.
\end{equation*}
Note that for distinct indices $r,s$,  we have
$1\le|u_r-u_s|\le5XL^B\le X^2$ for sufficiently large $X$.

Let $S_N(u)=\sum_{N<n\le2N}n^{iu}$ and $T_*=\exp(L/(\log L)^2)$ with $1\leq |u|\leq X^2$. We next estimate $S_N(u)$ in the following three ranges of $|u|$ separately.

For $1\le|u|\le T_*$, by partial summation and by the bound $N\ge x^\delta$, one may deduce that
\begin{equation*}%\label{u-1}
	\frac{|S_N(u)|}{N}
	\ll\frac1{|u|}+\frac{1+|u|}{N}
	\ll\frac1{|u|}+x^{-\delta/2}.
\end{equation*}
	
For $T_*\le|u|\le N$, by applying van der Corput's method (see e.g. \cite[Corollary 8.13]{IwaKow04}) with $f(v)=u\log v/(2\pi)$ with $|f''(v)|\asymp |u|/N^2$, we get
	\begin{equation*}%\label{u-2}
		\frac{|S_N(u)|}{N}
		\ll\frac{\sqrt{|u|}}N+\frac1{\sqrt{|u|}}
		\le N^{-1/2}+T_*^{-1/2}.
	\end{equation*}

For $N\le |u|\le X^2$, choose a fixed integer $k\ge2$
such that $k\delta>2$. Applying \cite[Theorem 8.4]{IwaKow04}
to $f(v)$ as above, we find that
\begin{equation*}%\label{u-3}
	\frac{|S_N(u)|}{N}
	\ll_k
	\biggl(\frac{|u|}{N^k}+\frac1{|u|}\biggr)^{4/(k2^k)}
	\log(3N)
	\ll_\delta X^{-\eta_\delta}
\end{equation*}
holds for some $\eta_{_\delta}>0$.

By these three upper bounds, we get, for any fixed $1\le r\le R$,
\begin{equation*}
	\begin{aligned}
	\frac1N\sum_{\substack{1\leq s\leq R\\ s\neq r}} |G_{r,s}|& =\frac1N\sum_{\substack{1\leq s\leq R\\ s\neq r}}|S_N(u_r-u_s)|
	\ll\sum_{\substack{1\leq s\leq R\\ s\neq r}} \Big(\frac1{|u_r-u_s|}+x^{-\delta/2}
	+T_*^{-1/2}+X^{-\eta_{_\delta}}\Big)
	\end{aligned}
\end{equation*} 
Since $u_1,\ldots,u_R$ are pairwise at distance $\geq1$, the first term above can be estimated by
\begin{equation*}
	\sum_{\substack{1\le s\le R\\s\ne r}}\frac1{|u_r-u_s|}
	\le 2\sum_{1\leq h\leq R}\frac1h\ll\log(2+R).
\end{equation*}
Moreover, we note that $R\le|\mathcal I|\le\exp(O(L^{0.91}))$ the first estimate in \eqref{eq:sparse}, and $G_{r,r}\ll N$ trivially. Thus, by the above arguments, we conclude for any $1\leq r\leq R$ that
\begin{equation}\label{eq:row-sum}
	\sum_{s=1}^R
	\biggl|\sum_{N<n\le2N}n^{i(u_r-u_s)}\biggr|=\sum_{1\leq s\leq R}|G_{r,s}|	\ll_\delta N\bigl(1+\log(2+R)+o(1)\bigr)
	\ll_\delta NL.
\end{equation}

Since $G$ is Hermitian, by duality, the bound \eqref{eq:row-sum}
implies that for any complex $a_n$,
\begin{equation*}
	\sum_{r=1}^R
	\biggl|\sum_{N<n\le2N}a_n n^{iu_r}\biggr|^2
	\ll_\delta NL\sum_{N<n\le2N}|a_n|^2.
\end{equation*}
Taking $a_n=\beta_n\chi_2(n)$ for $n\in I_N$
and $a_n=0$ otherwise, we obtain
\begin{align*}
	\sum_{r=1}^R\sup_{t\in I_r}|N_\beta(t)|^2
	=\frac1{N^2}\sum_{r=1}^R\biggl|\sum_{n\in I_N}\beta_n\chi_2(n)n^{iu_r}\biggr|^2
	\ll_\delta \frac{L}{N}\sum_{n\in I_N}|\beta_n\chi_2(n)|^2
	\ll_\delta L^{2C+1},
\end{align*}
where we used the choice of $u_r$ and the bounds $|\beta_n|\le L^C$,
$|\chi_2(n)|\le1$.
Summing over the two parity classes $\mathcal{J}$, we obtain the second estimate in \eqref{eq:sparse}.
\end{proof}

\vskip 3mm

\subsection{Estimate for the major arc approximation}

In this subsection, we shall prove the following Proposition \ref{prop:maj} by combining the estimates of the preceding subsection with the character--Mellin estimate in Proposition \ref{prop:mel}.

\begin{proposition}\label{prop:maj}
Recall the definition of $Q_Y^{\mathrm{maj}}$ in \eqref{def-Q_Y^maj}. Suppose that $\alpha_m=\mu(m)R_Z(m)$ for $M<m\le2M$ and that $|\beta_n|\le L^C R_Z(n)$ for $N<n\le2N$, where $C>0$ is some fixed constant. Then for $M,\,N\ge x^\delta,\, MN=X$ and any $A>0$, we have 
\begin{equation}\label{eq:majbd}
\begin{aligned}
|Q_Y^{\mathrm{maj}}|\ll_A XYL^{-A}
\end{aligned}
\end{equation}
This estimate also holds with $\alpha_m$, $\beta_n$ multiplied by factors $m^{iv}$ and $n^{iv}$ respectively, as in \eqref{eq:011}, uniformly for $|v|\le x^{C_0}$ for each fixed $C_0>0$. We may also multiply either sequence by a fixed smooth weight of bounded supremum norm and total variation.
\end{proposition}

\begin{proof}
	
We shall apply Lemmas \ref{lem-5.1} and \ref{lem:sparse} to estimate $Q_Y^{\mathrm{maj}}$. To do this, we first separate the variables in the factor $H_{-2}(bmn-ars;\,a,\,b)$ appearing in $Q_Y^{\mathrm{maj}}$. Recall by \eqref{eq:major-arcs-kernel} that 
\begin{equation*}
	H_{-2}(bmn-ars;\,a,\,b)
	=\psi_{-2}\biggl(\frac{bmn-ars}{Y}\biggr)
	\int_{\mathfrak M_{-2}}
	e\bigl(\theta(bmn-ars+2b-2a)\bigr)\,\mathrm d\theta,
\end{equation*}
where
\begin{equation*}
	\mathfrak M_{-2}
	=\bigcup_{\substack{1\le q\le L^{A_0}\\
			0\le j<2q,\,(j,\,q)=1}}
	\Biggl\{\theta\in\mathbb T:
	\biggl\|\theta-\frac{j}{2q}\biggr\|
	\le\frac{L^{A_0}}Y\Biggr\}.
\end{equation*}
We fix $q,\,j$ satisfying the above conditions in the union, and write
\begin{equation*}
	\frac{j}{2q}=\frac hk,\quad\ \text{with}\quad
	(h,\,k)=1,\quad 0\le h<k\le2q\le2L^{A_0}.
\end{equation*}
In the integral defining $H_{-2}$, we substitute
$\theta\equiv h/k+\lambda/Y\,(\textrm{mod}\,1)$ with $|\lambda|\le L^{A_0}$. Since $\mathrm d\theta=\mathrm d\lambda/Y$, we obtain
\begin{equation}\label{separate-varia-1}
	\begin{aligned}
	&\int_{\substack{\theta\in\mathbb T\\ 
		\|\theta-h/k\|\le L^{A_0}/Y}}
		e\bigl(\theta(bmn-ars+2b-2a)\bigr)\,\mathrm d\theta
		\\[1mm]
		&\qquad=\frac1Y
		e\biggl(\frac{h(bmn-ars+2b-2a)}k\biggr)
		\int_{-L^{A_0}}^{L^{A_0}}
		e\biggl(\frac{\lambda(bmn-ars+2b-2a)}Y\biggr)
		\,\mathrm d\lambda.
	\end{aligned}
\end{equation}

We now express the factor
$e(h(bmn-ars+2b-2a)/k)$ as a sum of products of
Dirichlet characters.
To apply the orthogonality of characters, we need to verify the coprime condition. First, $P^-(mnrs)>Z$ whenever 
$\alpha_m\overline{\alpha_r}\beta_n \overline{\beta_s}\ne0$. In addition, $a$ and $b$ are products of primes from the sets
$\mathcal P_i$, whence $P^-(ab)>\exp(L^{1/10})$. Since $k<\exp(L^{1/10})<Z$, thus we derive that
%for every nonzero summand in $Q_Y^{\mathrm{maj}}$,
\begin{equation*}
	(abmnrs,\,k)=1.
\end{equation*}
This allows us to obtain
\begin{equation}\label{separ-varia-2}
	e\biggl(\frac{h(bmn-ars+2b-2a)}k\biggr)
	=\sum_{\boldsymbol\chi}c_{\boldsymbol\chi}(h,\,k)
	\chi_3(a)\chi_4(b)\chi_1(m)\chi_2(n)\chi_5(r)\chi_6(s),
\end{equation}
where $\boldsymbol\chi=(\chi_1,\ldots,\chi_6)$ ranges over
all six  Dirichlet characters modulo $k$, and
\begin{align*}
	c_{\boldsymbol\chi}(h,k)
	:= &\frac1{\varphi(k)^6}
	\sum_{\substack{1\le a',\,b',\,m',\,n',\,r',\,s'\le k\\
			(a'b'm'n'r's',\,k)=1}}
	e\biggl(\frac{h(b'm'n'-a'r's'+2b'-2a')}k\biggr)\\[0,3mm]
	&\times	\overline{\chi_3(a')\chi_4(b')\chi_1(m')
		\chi_2(n')\chi_5(r')\chi_6(s')}.
\end{align*}
%Note that $|c_{\boldsymbol\chi}(h,k)|\le1,$ whence $\sum_{\boldsymbol\chi}|c_{\boldsymbol\chi}(h,k)|\le\varphi(k)^6\le k^6.$

For fixed $\boldsymbol\chi$, define
\begin{align*}
	\mathcal S_{\boldsymbol\chi}(\lambda)
	:= &\biggl(\prod_{i=1}^K \frac{1}{V_i^2}\biggr)
	\sum_{a,\,b\in\mathcal A}
	\eta\Big(\frac{a}{Y}\Big)\eta\Big(\frac{b}{Y}\Big)
	\chi_3(a)\chi_4(b)
	\sum_{\substack{M<m,\,r\le2M\\N<n,\,s\le2N}}
	\alpha_m\overline{\alpha_r}\beta_n\overline{\beta_s}\\[1mm]
	&\times\chi_1(m)\chi_2(n)\chi_5(r)\chi_6(s)
	\psi_{-2}\biggl(\frac{bmn-ars}{Y}\biggr)
	e\biggl(\frac{\lambda(bmn-ars+2b-2a)}Y\biggr).
\end{align*}
Substituting the above expression for $\mathcal S_{\boldsymbol\chi}(\lambda)$, \eqref{separ-varia-2} and \eqref{separate-varia-1} into \eqref{def-Q_Y^maj}, we therefore get
\begin{equation*}%\label{eq:Qmaj-S}
	Q_Y^{\mathrm{maj}}=\frac{1}{Y}\sum_{1\le q\le L^{A_0}}\sum_{\substack{0\le j<2q\\(j,\,q)=1}}\sum_{\boldsymbol\chi}c_{\boldsymbol\chi}(h,\,k)\int_{-L^{A_0}}^{L^{A_0}}\mathcal S_{\boldsymbol\chi}(\lambda)\,\mathrm d\lambda.
\end{equation*}
Here, for each pair $(q,\,j)$, we write $j/(2q)=h/k$ with $(h,\,k)=1$.
Noting that there are $O(L^{2A_0})$ pairs $(q,\,j)$, and $\sum_{\boldsymbol\chi}|c_{\boldsymbol\chi}(h,\,k)|\le k^6\ll L^{6A_0}$, we deduce
\begin{equation*}
	\big|Q_Y^{\mathrm{maj}}\big|\ll\frac{L^{9A_0}}{Y}
	\sup_{\substack{1\le k\le2L^{A_0},\ |\lambda|\le L^{A_0}\\[0,1mm]\chi_1,\ldots,\chi_6\,(\textrm{mod}\,k)}} \big|\mathcal S_{\boldsymbol\chi}(\lambda)\big|.
\end{equation*}
Therefore, to get \eqref{eq:majbd}, it suffices to prove that, for any $A>0$,
\begin{equation}\label{aim-S_chi}
	|\mathcal S_{\boldsymbol\chi}(\lambda)|\ll_A XY^2L^{-A}
\end{equation}
holds uniformly for $k\le2L^{A_0}$, all six characters modulo $k$, and $|\lambda|\le L^{A_0}$. 

%The character factors already depend on individual variables.
To get \eqref{aim-S_chi}, we shall first separate the variables in the factor $\psi_{-2}(\cdot)e(\cdot)$ of $\mathcal S_{\boldsymbol\chi}(\lambda)$.
%\begin{equation*}
%	\psi_{-2}\biggl(\frac{bmn-ars}{Y}\biggr)
%	e\biggl(\frac{\lambda(bmn-ars+2b-2a)}Y\biggr).
%\end{equation*}
We let $w=bmn/(YX)$ and $w'=ars/(YX)$, whence
\begin{equation}\label{psi_-2-e}
	\begin{aligned}
	&\psi_{-2}\biggl(\frac{bmn-ars}{Y}\biggr)
	e\biggl(\frac{\lambda(bmn-ars+2b-2a)}Y\biggr)=
	%e\Big(\frac{2\lambda b}{Y}\Big)e\Big(\frac{-2\lambda a}{Y}\Big)
	e\Big(\frac{2\lambda b-2\lambda a}{Y}\Big)
	\psi_{-2}\big(X(w-w')\big)e\big(\lambda X(w-w')\big).
	\end{aligned}
\end{equation}
with  $1\le w,\,w'\le16$, since $\eta(a/Y)\eta(b/Y)$ is supported in $Y\le a,\,b\le4Y$.
%The first factors on the right-hand side depend only on $b$ and $a$, and we shall separate the dependence on $w$ and $w'$ in $\psi_{-2}(X(w-w'))e(\lambda X(w-w'))$.

Next, we shall express $\psi_{-2}(X(w-w'))e(\lambda X(w-w'))$ as an integral of products of powers of $w$ and $w'$ by Fourier transform. To do this, we introduce a smooth function  $\psi_0:\mathbb R\to[0,1]$, just as \cite[Page 242]{BomFriIwa86}, such that $\psi_0(u)=1$ for $1\le u\le16$, and $\psi_0(u)=0$ for $u\le1/2\ \text{or}\ u\ge32$.
Since $1\le w,\,w'\le16$, it follows that 
\begin{equation}\label{smoothfunction}
	\psi_{-2}(X(w-w'))e(\lambda X(w-w')) =\psi_0(w)\psi_0(w')\psi_{-2}(X(w-w'))e(\lambda X(w-w')).
\end{equation}
Put $z=\log w'$ and $\zeta=X\log(w/w')$. Then we have $w=e^{z+\zeta/X}$ and
\begin{equation*}
   X(w-w')=Xe^z(e^{\zeta/X}-1).
\end{equation*}
Hence, the factor on the right-hand side of \eqref{smoothfunction}, say $K_\lambda(z,\,\zeta)$, becomes
\begin{equation}\label{def-K_lambda}
	K_\lambda(z,\,\zeta)=\psi_0(e^z)\psi_0(e^{z+\zeta/X})\psi_{-2}\bigl(Xe^z(e^{\zeta/X}-1)\bigr)e\bigl(\lambda Xe^z(e^{\zeta/X}-1)\bigr).
\end{equation}
Note that $|z|\ll 1$, and $|w-w'|\ll X^{-1}$ by the support of $\psi_{-2}$, whence  $|\zeta|\ll1$. Therefore, $K_\lambda$ is smooth and supported in a fixed compact set independent of $X$, and for any integers $j_1,j_2\geq 0$,
\begin{equation}\label{eq:K-derivatives}
	\big\|\partial_z^{j_1}\partial_\zeta^{j_2}K_\lambda\big\|_\infty\ll_{j_1,\,j_2}(1+|\lambda|)^{j_1+j_2}.
\end{equation}

Define the Fourier transform $F_\lambda$ of $K_\lambda$ by
\begin{equation*}
	F_\lambda(v,\,u):=\frac1{(2\pi)^2}\iint_{\mathbb R^2}K_\lambda(z,\,\zeta)e^{-i(v\zeta+uz)}\,\mathrm dz\,\mathrm d\zeta.
\end{equation*}
whence
\begin{equation}\label{eq:K-inversion}
	K_\lambda(z,\,\zeta)=\iint_{\mathbb R^2}F_\lambda(v,\,u)e^{i(v\zeta+uz)}\,\mathrm dv\,\mathrm du.
\end{equation}
Substituting $z=\log w'$, $\zeta=X\log(w/w')$, and using $e^{i(v\zeta+uz)}=w^{iXv}(w')^{i(u-Xv)}$ in \eqref{eq:K-inversion}, one may deduce, from \eqref{def-K_lambda}, that
\begin{equation}\label{eq:K-separation}
	\psi_{-2}(X(w-w'))e(\lambda X(w-w'))=\frac1X\iint_{\mathbb R^2}F_\lambda(t/X,\,u)w^{it}(w')^{i(u-t)}\,\mathrm dt\,\mathrm du.
\end{equation}
%Here we used $\psi_0(w)\psi_0(w')=1$ on $[1,\,16]^2$ to obtain the left-hand side.

We next rewrite $\mathcal S_{\boldsymbol\chi}(\lambda)$ as an integral involving $M_\alpha(t)N_\beta(t)$ defined in \eqref{def-M_alpha-N_beta} for specific intervals $I_M=(M,\,2M]$ and $I_N=(N,\,2N]$, so that we can apply Lemmas \ref{lem-5.1} and \ref{lem:sparse}. Note that
%Since $w=bmn/(YX)$ and $w'=ars/(YX)$, we have
\begin{equation*}
	w^{it}(w')^{i(u-t)}=(YX)^{-iu}a^{i(u-t)}b^{it}m^{it}n^{it}r^{i(u-t)}s^{i(u-t)}.
\end{equation*}
Substituting this equality, \eqref{eq:K-separation} and \eqref{psi_-2-e} into $\mathcal S_{\boldsymbol\chi}(\lambda)$ and interchanging the finite sums with the integrals, we obtain the following expression
\begin{equation}\label{eq:S-separation}
\begin{aligned}
	\mathcal S_{\boldsymbol\chi}(\lambda)= 
	&\prod_{i=1}^K \frac{1}{V_i^2}
	\iint_{\mathbb R^2}F_\lambda(t/X,\,u)(YX)^{-iu}\\[1mm]
	&\times\biggl(\sum_{b\in\mathcal A}\eta\Bigl(\frac bY\Bigr) e\Bigl(\frac{2\lambda b}{Y}\Bigr) \chi_4(b)b^{it}\biggr)\biggl(\sum_{a\in\mathcal A}\eta\Bigl(\frac aY\Bigr)e\Bigl(\frac{-2\lambda a}{Y}\Bigr) \chi_3(a)a^{i(u-t)}\biggr)\\[1mm]
	&\times\biggl(\sum_{M<r\le2M}\overline{\alpha_r}\chi_5(r) r^{i(u-t)}\biggr)\biggl(\sum_{N<s\le2N}\overline{\beta_s} \chi_6(s)s^{i(u-t)}\biggr)M_\alpha(t)N_\beta(t)\,\mathrm dt\,\mathrm du
	\end{aligned}
\end{equation}
by the definitions of $M_\alpha(t),\,N_\beta(t)$ and $MN=X.$

We next group the sums over $(b,\,m,\,n)$ and $(a,\,r,\,s)$ in $\mathcal S_{\boldsymbol\chi}(\lambda)$. Define
\begin{equation*}
	B_\lambda(t):=\frac{1}{Y}\prod_{i=1}^K\frac{1}{V_i}\sum_{b\in\mathcal A}\eta\biggl(\frac bY\biggr)e\biggl(\frac{2\lambda b}{Y}\biggr)\chi_4(b)b^{it},\quad
	\widetilde B_\lambda(t):=\frac{1}{Y}\prod_{i=1}^K\frac{1}{V_i}\sum_{a\in\mathcal A}\eta\biggl(\frac aY\biggr)e\biggl(-\frac{2\lambda a}{Y}\biggr)\chi_3(a)a^{it},
\end{equation*}
and then put
\begin{equation*}
	Q(t):=B_\lambda(t)M_\alpha(t)N_\beta(t),\quad
	\widetilde Q(t):=\frac{1}{MN}\widetilde B_\lambda(t)\biggl(\sum_{M<r\le2M}\overline{\alpha_r}\chi_5(r)r^{it}\biggr)\biggl(\sum_{N<s\le2N}\overline{\beta_s}\chi_6(s)s^{it}\biggr).
\end{equation*}
With the notation above, we obtain from \eqref{eq:S-separation} that
\begin{equation}\label{eq:S-normalized}
	\begin{aligned}
		|\mathcal S_{\boldsymbol\chi}(\lambda)|&\le XY^2\iint_{\mathbb R^2}\big|F_\lambda(t/X,\,u)Q(t)\widetilde Q(u-t)\big|\,\mathrm dt\,\mathrm du.
	\end{aligned}
\end{equation}

Let $B_0>A_0$ and $B_1=B_0+2$ be fixed constants.
We shall first show that the double integral in \eqref{eq:S-normalized}, restricted to $|t|>XL^{B_1}$ or $|u|>L^{B_1}$, is $O(L^{-A})$ for any $A>0$. 
%We can then restrict the subsequent estimates to $|t|\le XL^{B_1}$ and $|u|\le L^{B_1}$.
To do this, we estimate $Q(t)$ uniformly in $|\lambda|\le L^{A_0}$; and the same argument applies to $\widetilde Q(t)$, since
\begin{equation*}
	\overline{\alpha_r}=\alpha_r=\mu(r)R_Z(r),\qquad |\overline{\beta_s}|=|\beta_s|\le L^C R_Z(s).
\end{equation*}
Note that $\eta(b/Y)=0$ unless $Y\le b\le4Y$, thus Rankin's trick, together with \eqref{eq:aharm}, gives
\begin{equation*}
	|B_\lambda(t)|\ll\biggl(\prod_{i=1}^K V_i^{-1}\biggr)\sum_{b\in\mathcal A}\frac1b=1.
\end{equation*}
This, together with \eqref{eq:prodms}, implies that for
 any interval $J$ of length $U\ge1$
\begin{equation*}
	\int_J\big|Q(t)\big|^2\,\mathrm dt\ll \int_J\big|M_\alpha(t)N_\beta(t)\big|^2\,\mathrm dt
	\ll L^{2C+3}\Big(1+\frac{U}{X}\Big),
\end{equation*}
and this upper bound is also true for 
$\int_J|\widetilde Q(t)|^2\,\mathrm dt$.
Thus, by Cauchy--Schwarz inequality, we obtain uniformly for all real $u$ that
\begin{equation}\label{eq:Q-local-product}
	\int_J\big|Q(t)\widetilde Q(u-t)\big|\,\mathrm dt  \ll L^{2C+3}\Big(1+\frac{U}{X}\Big).
\end{equation}

Now we estimate the factor $F_\lambda(v,\,u)$ in \eqref{eq:S-normalized}. Since $K_\lambda$ is supported in a fixed bounded set, an integration by parts in $z$ and  $\zeta$, together with \eqref{eq:K-derivatives}, yields
\begin{equation*}
	|u|^{j_1}|v|^{j_2}|F_\lambda(v,\,u)|\ll
	\iint_{\mathbb R^2} \big|\partial_z^{j_1}\partial_\zeta^{j_2}K_\lambda(z,\,\zeta)\big|\,\mathrm dz\,\mathrm d\zeta\ll_{j_1,\,j_2} (1+|\lambda|)^{j_1+j_2}.
\end{equation*}
Thus, for any $j\geq0$, we apply this estimate with $(j_1,j_2)=(0,0),(j,0),(0,j),(j,j)$, getting
\begin{equation}\label{eq:decay-lambda}
	\big|F_\lambda(v,\,u)\big|\ll_j\Bigl(1+\frac{|v|}{1+|\lambda|} \Bigr)^{-j}\Bigl(1+\frac{|u|}{1+|\lambda|}\Bigr)^{-j}
	\ll_j\Bigl(1+\frac{|v|}{L^{B_0}}\Bigr)^{-j}
	\Bigl(1+\frac{|u|}{L^{B_0}}\Bigr)^{-j},
\end{equation}
where we have used the upper bound $|\lambda|\le L^{A_0}$.

We first consider the case $|t|>XL^{B_1}$, with $u\in\mathbb R$ by dividing this range into subintervals $2^\ell XL^{B_1}<|t|\le2^{\ell+1}XL^{B_1}$ for $\ell\ge0$. Then we apply \eqref{eq:decay-lambda} with $j>2$ and \eqref{eq:Q-local-product}, getting
%\begin{equation}\label{eq:S-separation}
	\begin{align}\label{t>}
	\int_{\mathbb R}\int_{|t|>XL^{B_1}}\Big|F_\lambda\Big(
	\frac{t}{X},\,u\Big)Q(t)\widetilde Q(u-t)\Big|\mathrm dt
	\,\mathrm du &\ll 
	\int_{\mathbb R} \Bigl(1+\frac{|u|}{L^{B_0}}\Bigr)^{-j}
	\int_{|t|>XL^{B_1}}\frac{\big|Q(t)\widetilde Q(u-t)\big|} {\bigl(\frac{|t|}{XL^{B_0}}\bigr)^{j}}\,\mathrm dt\,\mathrm du
	\nonumber \\[1,6mm]\nonumber
	&\ll L^{2C+3+B_1-j(B_1-B_0)}\int_{\mathbb R} \Bigl(1+\frac{|u|}{L^{B_0}}\Bigr)^{-j}\mathrm du
	\\[2,3mm]
	&\ll L^{2C+3+B_0+B_1-j(B_1-B_0)}.
	\end{align}
%\end{equation}

We next consider the cases $|t|\le XL^{B_1}$ and $|u|>L^{B_1}$. By applying \eqref{eq:Q-local-product} to the interval $[-XL^{B_1},\,XL^{B_1}]$ and by integrating directly, respectively,  we obtain 
\begin{equation*}
	\int_{|t|\le XL^{B_1}}\big|Q(t)\widetilde Q(u-t)\big|\,\mathrm dt\ll L^{2C+3+B_1},\quad \int_{|u|>L^{B_1}} \Bigl(1+\frac{|u|}{L^{B_0}}\Bigr)^{-j}\,\mathrm du
	%=\frac{2L^{B_0}}{j-1}\bigl(1+L^{B_1-B_0}\bigr)^{1-j}
	\ll_j L^{B_0-(j-1)(B_1-B_0)}.
\end{equation*}
These two estimates imply that
\begin{equation}\label{u>t<}
	\int_{|u|>L^{B_1}}\int_{|t|\le XL^{B_1}}\big|F_\lambda(t/X,\,u)Q(t)\widetilde Q(u-t)\big|\,\mathrm dt\,\mathrm du\ll_j L^{2C+3+B_0+B_1-(j-1)(B_1-B_0)}.
\end{equation}

For any fixed $A>0$, we choose $j$ sufficiently large such that $2(j-1)\ge A+2C+3+B_0+B_1$. Combining \eqref{t>} and \eqref{u>t<}, we obtain
\begin{equation}\label{t>u>}
	\iint_{\substack{|t|>XL^{B_1}\ \text{or}\ |u|>L^{B_1}}} \Big|F_\lambda\Big(\frac{t}{X},\, u\Big)Q(t)
	\widetilde Q(u-t)\Big|\,\mathrm dt\,\mathrm du\ll_A L^{-A}.
\end{equation}

It remains to consider the case $|t|\le XL^{B_1}$, $|u|\le L^{B_1}$. For any $|u|\le L^{B_1}$, by the Cauchy--Schwarz inequality one has
\begin{equation}\label{cauchy}
	\int_{|t|\le XL^{B_1}}\big|Q(t)\widetilde Q(u-t)\big|\,\mathrm dt\le\biggl(\int_{|t|\le2XL^{B_1}}|Q(t)|^2\,\mathrm dt\biggr)^{1/2}\biggl(\int_{|t|\le2XL^{B_1}}|\widetilde Q(t)|^2\,\mathrm dt\biggr)^{1/2}.
\end{equation}

Now, we assume that for every fixed $A>0$,
\begin{equation}\label{eq:energy}
	\int_{|t|\le2XL^{B_1}}\Big(|Q(t)|^2+|\widetilde Q(t)|^2\Big)\,\mathrm dt\ll L^{-A-B_1},
\end{equation}
holds uniformly in $|\lambda|\le L^{A_0}$ and in the characters.
Then with the upper bound $|F_\lambda(t/X,\,u)|\ll1$ by \eqref{eq:decay-lambda}, we insert \eqref{eq:energy} into \eqref{cauchy}, getting
\begin{equation*}
	\int_{|u|\le L^{B_1}}\int_{|t|\le XL^{B_1}}\big|F_\lambda(t/X,\,u)Q(t)\widetilde Q(u-t)\big|\,\mathrm dt\,\mathrm du\ll_A L^{-A}.
\end{equation*}
By this upper bound, together with \eqref{eq:S-normalized} and \eqref{t>u>}, we derive that
\begin{equation*}
	|\mathcal S_{\boldsymbol\chi}(\lambda)|\ll_A XY^2L^{-A},
\end{equation*}
as desired.
Thus, it remains to prove \eqref{eq:energy}, and we shall employ \eqref{eq:prodms}, Lemma \ref{lem:sparse} and Proposition \ref{prop:mel}.

To apply Lemma \ref{lem:sparse}, we first express $B_\lambda$ in terms of a sum over primes $p\in\mathcal P_1$.
We write $b=p_1\cdots p_K\in\mathcal A$ with $p_i\in\mathcal P_i$, 
%We shall use Mellin inversion to separate the sums over $p_1,\ldots,p_K$. 
and define
\begin{equation*}
	g_\lambda(y):=y\eta(y)e(2\lambda y),\qquad 
	\widehat g_\lambda(\xi):=\frac{1}{2\pi}\int_0^\infty g_\lambda(y)y^{-i\xi}\,\frac{\mathrm dy}{y}.
\end{equation*}
Then by Mellin inversion, $g_\lambda(y)=\int_{\mathbb R}\widehat g_\lambda(\xi)y^{i\xi}\,\mathrm d\xi,$ whence we can write
\begin{equation}\label{eq:bfactor}
	B_\lambda(t)=\prod_{i=1}^K\frac{1}{V_i}\sum_{b\in\mathcal A} \frac{1}{b}g_\lambda\biggl(\frac bY\biggr)\chi_4(b)b^{it}
	=\int_{\mathbb R}\widehat g_\lambda(\xi)Y^{-i\xi} \prod_{i=1}^K\biggl(\frac{1}{V_i}\sum_{p\in\mathcal P_i}\frac{\chi_4(p)p^{i(t+\xi)}}{p}\biggr)\,\mathrm d\xi.
\end{equation}
Moreover for $|\lambda|\le L^{A_0}$, integrating by parts in $z$ with $z=\log y$,  one derives  that
\begin{equation}\label{eq:g-L1}
	\int_{\mathbb R}|\widehat g_\lambda(\xi)|\mathrm d\xi\ll
	\int_{\mathbb R} \Bigl(1+\frac{|\xi|}{1+|\lambda|}\Bigr)^{-2}\mathrm d\xi
	\ll1+|\lambda|\ll L^{A_0}.
\end{equation}

Note that for every $i$, we have $\big|\sum_{p\in\mathcal P_i}\chi_4(p)p^{i(t+\xi)}/p\big|\leq \sum_{p\in\mathcal P_i}1/p=V_i$.
We apply this bound to the factors with $i\ge2$ in \eqref{eq:bfactor}, getting
\begin{equation}\label{eq:B-prime-square}
	\begin{aligned}
	\big|B_\lambda(t)\big|^2 &\le\frac{1}{V_1^2}
	\biggl(\int_{\mathbb R}\big|\widehat g_\lambda(\xi)\big| \biggl|\sum_{p\in\mathcal P_1} \frac{\chi_4(p)p^{i(t+\xi)}}{p}\biggr|\,\mathrm d\xi\biggr)^2
	\\[1mm]
	&\leq \frac{1}{V_1^2}\int_{\mathbb R}\big|\widehat g_\lambda(\xi)\big|\,\mathrm d\xi
	\int_{\mathbb R}\big|\widehat g_\lambda(\xi)\big|\biggl| \sum_{p\in\mathcal P_1}\frac{\chi_4(p)p^{i(t+\xi)}}{p}\biggr|^2 \,\mathrm d\xi
	\\[1mm]
	&\ll \frac{L^{A_0}}{V_1^2}
	\int_{\mathbb R}\big|\widehat g_\lambda(\xi)\big|\biggl| \sum_{p\in\mathcal P_1}\frac{\chi_4(p)p^{i(t+\xi)}}{p}\biggr|^2 \,\mathrm d\xi,
    \end{aligned}
\end{equation}
where we have used the bound \eqref{eq:g-L1}. To rewrite the sum over $p$ above in the form required by Lemma \ref{lem:sparse}, we partition $\mathcal P_1$ into at most $O(L)$ sets $\mathcal P_1\cap[P,\,2P)$, where $P=2^k, k\geq 1$. Define
\begin{equation*}
	\mathscr P_{P,\xi}(t):=\sum_{p\in\mathcal P_1\cap[P,\,2P)}\frac{\chi_4(p)p^{i(t+\xi)}}{p}=\frac{1}{P}\sum_{p\in\mathcal P_1\cap[P,\,2P)}\frac{P}{p}\chi_4(p)p^{i(t+\xi)}.
\end{equation*}
Note that the size of $P$ with $\mathcal P_1\cap[P,\,2P)\neq\varnothing$ satisfies the condition in Lemma \ref{lem:sparse}.

We now reduce \eqref{eq:energy} to a uniform mean square estimate for $\mathscr P_{P,\xi}(t)M_\alpha(t)N_\beta(t)$. Noting that $V_1^{-1}\ll1$, we then apply the Cauchy--Schwarz inequality to the sum over $p$ in \eqref{eq:B-prime-square}, getting
\begin{equation*}
	\big|B_\lambda(t)\big|^2\ll L^{1+A_0}\int_{\mathbb R}\big| \widehat g_\lambda(\xi)\big|\,\sum_P\big|\mathscr P_{P,\xi}(t)\big|^2\,\mathrm d\xi.
\end{equation*}
Recall that $Q(t)=B_\lambda(t)M_\alpha(t)N_\beta(t).$
Multiplying the above upper bound by $|M_\alpha(t)N_\beta(t)|^2$ and then integrating over $|t|\le2XL^{B_1}$, we obtain
\begin{equation*}
	\int_{|t|\le2XL^{B_1}}|Q(t)|^2\,\mathrm dt\ll L^{2+2A_0}\sup_{P,\,\xi}\int_{|t|\le2XL^{B_1}}|\mathscr P_{P,\xi}(t)M_\alpha(t)N_\beta(t)|^2\,\mathrm dt.
\end{equation*}
Therefore, to get the upper bound for $Q(t)$ in \eqref{eq:energy}, it suffices to show that for any $A>0$,
\begin{equation}\label{eq:penergy}
	\int_{|t|\le2XL^{B_1}}|\mathscr P_{P,\xi}(t) M_\alpha(t)N_\beta(t)|^2\,\mathrm dt\ll_A L^{-A}
\end{equation}
holds uniformly in $P$, $\xi$. 
The same argument applies to $\widetilde Q$, whence  \eqref{eq:energy} follows from \eqref{eq:penergy}.

To prove \eqref{eq:penergy}, We first consider the range $|t|\le2XL^{B_1}$, $|\mathscr P_{P,\xi}(t)|\le L^{-A_1}$ with $A_1>0$ sufficiently large. By \eqref{eq:prodms}, we deduce that
\begin{equation}\label{mathscrP<}
	\int_{\substack{|t|\le2XL^{B_1}\\|\mathscr P_{P,\xi}(t)|\le L^{-A_1}}}\big|\mathscr P_{P,\xi}(t) M_\alpha(t)N_\beta(t)\big|^2\,\mathrm dt\le L^{-2A_1}\int_{|t|\le2XL^{B_1}}\big|M_\alpha(t)N_\beta(t)\big|^2\,\mathrm dt\ll L^{-A}.
\end{equation}

We next treat the range $|t|\le2XL^{B_1}$, $|\mathscr P_{P,\xi}(t)|>L^{-A_1}$. 
First, note that by Proposition \ref{prop:mel}, the following bound
\begin{equation*}
	\sup_{|t|\le2XL^{B_1}}|M_\alpha(t)|\ll_{A_2}L^{-A_2}
\end{equation*}
holds for any $A_2>0$. Let $\mathcal I$ be the set of intervals $[j,\,j+1]$, $j\in\mathbb Z$, such that there exists $t\in[j,\,j+1]$ satisfying $|t|\le2XL^{B_1}$ and $|\mathscr P_{P,\xi}(t)|>L^{-A_1}$.
Since $|\mathscr P_{P,\xi}(t)|\ll1$, we apply Lemma \ref{lem:sparse} with $\chi=\chi_4$ and $v_p=P/p$, obtaining
\begin{equation}\label{mathscrP>}
\int_{\substack{|t|\le2XL^{B_1}\\|\mathscr P_{P,\xi}(t)| >L^{-A_1}}}\big|\mathscr P_{P,\xi}(t)M_\alpha(t)N_\beta(t)\big|^2 \,\mathrm dt
\ll L^{-2A_2}\sum_{I\in\mathcal I}\sup_{t\in I}|N_\beta(t)|^2
\ll L^{-A}.
\end{equation}
Combining \eqref{mathscrP<} and \eqref{mathscrP>}, we obtain \eqref{eq:penergy}, thereby establishing \eqref{eq:majbd}.

The same argument as above permits a factor $m^{iv}w(m)$ in $\alpha(m)$, where $|v|\le x^{C_0}$ and $w$ is a smooth weight  on $[M,\,2M]$. To apply Proposition \ref{prop:mel}, we need to verify that $t+v$ appearing in $m^{it}\cdot m^{iv}$ is bounded by a fixed power of $x$, which is true since $|t+v|\le2XL^{B_1}+x^{C_0}\ll x^2+x^{C_0}.$
Moreover, the transformation $\beta_n\mapsto\beta_n n^{iv}$ preserves the required bound for coefficients. Hence the preceding argument also gives \eqref{eq:penergy}, \eqref{eq:energy}, and hence \eqref{eq:majbd}, with these additional factors.
We therefore complete the proof of Proposition \ref{prop:maj}.
\end{proof}

\subsection{Proof of Theorem \ref{thm:main}}
\label{sec:final}

\begin{proof}[Proof]
For any $A>0$, let $A'=2A+8$ .
We first apply Vaughan's identity \eqref{eq:vaughan}  by taking $F(n)=\mu(n+2)\mathcal W(n+2)$. With the Type I sums $T_1(F),\, T_2(F)$ estimated in \eqref{eq:typeI}, we have
\begin{equation}\label{thm1-W}
	\sum_{x<p\le2x}(\log p)\mu(p+2)\mathcal W(p+2)=T_3(F)
	+O\big(x (\log x)^{-A}\big),
\end{equation}
where the contribution of $p^k,\,k\geq 2$ is $o(x^{1-\varepsilon})$ because $\mathcal W$ is bounded. 
%It remains to estimate the Type II sum $T_3(F)$ defined in \eqref{eq:remain}.

To estimate the Type II sum $T_3(F)$, by \eqref{eq:finalII}, we need an upper bound for $Q_Y$. 
Choose $A_0$ and the positive integer $K$ as required by Proposition \ref{prop:D}.
Then for coefficients $\alpha_m,\, \beta_n$ defined in \eqref{eq:011}, applying Propositions \ref{prop:D} and \ref{prop:maj}, one finds that
\begin{equation*}
	Q_Y=Q_Y^{\mathrm{maj}}+O(MNYL^{-A'})\ll MNYL^{-A'}
\end{equation*}
holds uniformly over the dyadic segments in \eqref{eq:finalII} and $|v|\le L^{2B}$. Since $MN\asymp x$, it follows from \eqref{eq:finalII} that
\begin{equation}\label{T3-upper}
	\begin{aligned}
	|T_3(F)|&\ll\sum_{M,\,N}\int_{|v|\le L^{2B}}|\widehat\psi(v)| \sum_Y\biggl(\frac{MN}{Y}\cdot MNYL^{-A}\biggr)^{1/2}
	\,\mathrm dv + xL^{-A'}
	\\
	&\ll xL^{-A'/2}\sum_{M,\,N}\int_{|v|\le L^{2B}}|\widehat\psi(v)| \sum_Y1\,\mathrm dv+xL^{-A'}
	\\
	&\ll xL^{-A},
	\end{aligned}
\end{equation} 
since the number of dyadic segments is bounded by $L^3$, and  $\int_{|v|\le L^{2B}}|\widehat\psi(v)|\,\mathrm dv\ll\log L$ by \eqref{mellin-psi}. Substituting this bound into \eqref{thm1-W}, we conclude that, for any $A>0$
\begin{equation}\label{eq:marked}
	\sum_{x<p\le2x}(\log p)\mu(p+2)\mathcal W(p+2)
	=O\big(x L^{-A}\big).
\end{equation}

We now remove the weight $\mathcal W$ defined by \eqref{eq:mark}. For the given integer $K$ which is chosen previously, we fix an  integer $b>0$ and then choose $bK$ pairwise distinct exponents satisfying $1/10<a_{j,\,1}<\cdots<a_{j,\,K}<1/5$ for every $1\le j\le b$.
Let $\mathcal P_{j,\,i}$, $\mathcal W_j$ and $m_j$ be as in Proposition \ref{prop:rem}, where the sets $\mathcal P_{j,\,i}$ are pairwise disjoint and $\mathcal W_j$ are defined by $\mathcal W$ in \eqref{eq:mark} with $\mathcal P_i$ replaced by $\mathcal P_{j,\,i}$. Hence \eqref{eq:marked} also holds with $\mathcal W$ replaced by $\mathcal W_j$ for every $1\le j\le b$. 
Since $m_j=2^{K(\rho-1)}+o(1)\gg 1$ for $0<\rho<1$, we therefore obtain
	\begin{equation*}
		\frac1b\sum_{j=1}^b\frac1{m_j}\sum_{x<p\le2x}(\log p)\mu(p+2)\mathcal W_j(p+2)=O\big(x L^{-A}\big).
	\end{equation*}
By this estimate and Proposition \ref{prop:rem}, we obtain
\begin{align*}
\biggl|\sum_{x<p\le2x}\mu(p+2)\log p\biggr|
&\le \biggl|\sum_{x<p\le2x}\mu(p+2)(\log p) \biggl(1-\frac1b\sum_{j=1}^b\frac{\mathcal W_j(p+2)}{m_j}\biggr)
\biggr|+O\big(x L^{-A}\big)
\\[1mm]
&\le\biggl(\sum_{x<p\le2x}\log p\biggr)^{1/2} \biggl(\sum_{x<p\le2x}(\log p) \biggl(1-\frac1b\sum_{j=1}^b \frac{\mathcal W_j(p+2)}{m_j}\biggr)^2\biggr)^{1/2}+o(x)
\\[1mm]
&\le x\bigg(\sqrt{\frac{C_{K,\,\rho}}b}+o(1)\bigg).
\end{align*}
Since $K$ is fixed and $C_{K,\,\rho}$ is independent of $b$, we let $x\to\infty$ with $b$ fixed, and then let $b\to\infty$, yielding 
$\sum_{x<p\le2x}\mu(p+2)\log p=o(x).$
%\begin{equation*}%\label{2-dy}
%	\sum_{x<p\le2x}\mu(p+2)\log p=o(x).
%\end{equation*}
Finally, summing this estimate over dyadic intervals, and then applying partial summation, we conclude that
\begin{equation*}
	\sum_{p\le x}\mu(p+2)=o\biggl(\frac{x}{\log x}\biggr),
\end{equation*}
which completes the proof of Theorem \ref{thm:main}.

\end{proof}

\vskip 5mm

\section{\large Squarefree shifted primes and the distribution of signs}
\label{sec:squarefree-signs}

The squarefree-density estimate holds for every fixed nonzero shift;
its constant depends on the prime divisors of the shift.

\begin{proposition}\label{prop:squarefree-density}
For any fixed integer $h\ne0$ and any $B>0$,
\begin{equation}\label{eq:sf-asymptotic}
 \sum_{\substack{1\leq p\le x\\p+h\geq 1}}\mu^2(p+h)\log p
=A_hx+O_{B,h}\bigg(\frac{x}{(\log x)^B}\bigg),
\quad\text{with}\quad A_h:=\prod_{p\nmid h}\bigg(1-\frac1{p(p-1)}\bigg).
% \sum_{\substack{1\leq n\le x\\n+h\geq 1}}\Lambda(n)\mu^2(n+h)
% =A_hx+O_{A,h}\bigg(\frac{x}{(\log x)^A}\bigg),
% \quad\text{with}\quad A_h:=\prod_{p\nmid h}\bigg(1-\frac1{p(p-1)}\bigg).
\end{equation}
\end{proposition}

\begin{proof}
We first note that  for notation simplicity, we may omit the condition $n+h\geq 1$ with a negligible error term. One may find easily that it suffices to consider the sum
\begin{equation}\label{eq:sf-mobius-inversion}
    \sum_{\substack{p\le x,\,p\,\nmid\, h}}\mu^2(p+h)\log p
    =\sum_{\substack{p\le x,\,p\,\nmid\, h\\p+h\ge1}}(\log p)
    \sum_{\substack{d^2\mid p+h}}\mu(d),
\end{equation}
where we have used the M\"{o}bius inversion formula for $\mu^2$. Note that $\gcd(d,\,h)=1$ since $p\,\nmid\, h$.

Recall that $L=\log x$.
The terms with $d>L^{B+2}$ can be estimated trivially by
\begin{equation}\label{eq:sf-large-tail}
\begin{aligned}
 \biggl|\sum_{\substack{p\le x,\,p\,\nmid\, h\\p+h\ge1}}(\log p)
        \sum_{\substack{d^2\mid p+h\\d>L^{B+2}}}\mu(d)\biggr|
 &\le L\sum_{d>L^{B+2}}\biggl\lfloor\frac{x+|h|}{d^2}\biggr\rfloor
 \ll_{B,\,h}\frac{x}{L^{B+1}}.
\end{aligned}
\end{equation}
which is admissible.

For $\gcd(d,\,h)=1$ and $d\le L^{B+2}$, the Siegel--Walfisz theorem (see \cite[Theorem 12.1]{kouk}) gives
\begin{equation}\label{eq:sf-small-moduli}
 \sum_{\substack{p\le x\\p\equiv-h\,(\mathrm{mod}\,d^2)}}\log p
 =\frac{x}{\varphi(d^2)}+O_B(xL^{-2B-4}).
\end{equation}
Thus \eqref{eq:sf-mobius-inversion}--\eqref{eq:sf-small-moduli} imply that
\begin{equation}\label{eq:sf-small-main}
 \sum_{\substack{p\le x,\,p\,\nmid\, h}}\mu^2(p+h)\log p
 =x\sum_{\substack{d\le L^{B+2}\\\gcd(d,\,h)=1}}
      \frac{\mu(d)}{\varphi(d^2)}+O_{B,h}(xL^{-B}).
\end{equation}

Since $1/\varphi(d^2)\le\tau(d)/d^2$ and
$\sum_{d\le t}\tau(d)\ll t\log(2t)$, partial summation gives
\begin{equation}\label{eq:sf-coefficient-tail}
 \sum_{d>L^{B+2}}\frac{|\mu(d)|}{\varphi(d^2)}
 \ll_B \frac{\log L}{L^{B+2}}.
\end{equation}
Furthermore, the Euler product with $d\geq 1$ is
\begin{equation}\label{eq:sf-euler-product}
 \sum_{\substack{d\ge1\\\gcd(d,\,h)=1}}\frac{\mu(d)}{\varphi(d^2)}
 =\prod_{p\nmid h}\bigg(1-\frac1{p(p-1)}\bigg)=A_h>0.
\end{equation}

 Combining \eqref{eq:sf-small-main}--\eqref{eq:sf-euler-product}, we 
complete the proof of Proposition \ref{prop:squarefree-density}.
\end{proof}

\begin{proof}[Proof of Theorem \ref{thm2}]
	Theorem \ref{thm2} is an immediately consequence of Theorem \ref{thm:main} and  Proposition \ref{prop:squarefree-density}. To see this, we first apply partial	summation to the asymptotic \eqref{eq:sf-asymptotic} with $h=2$, getting
	\begin{equation}\label{eq:sf-prime-count}
		\sum_{p\le x}\mu^2(p+2)=\big(A_2+o(1)\big)\pi(x).
	\end{equation}
	Next, note that for $\kappa\in\{-1,1\}$, we have 
	$$\mathbf{1}_{\{\mu(m)=\kappa\}}=\frac{\mu^2(m)+\kappa\mu(m)}{2}.$$
	This identity, together with \eqref{eq:sf-prime-count} and
	Theorem \ref{thm:main}, gives
	\begin{align*}
		\Big|\big\{p\le x:\ \mu(p+2)=\kappa\big\}\Big|
		=\frac12\sum_{p\le x}\mu^2(p+2)
		+\frac{\kappa}{2}\sum_{p\le x}\mu(p+2)=\bigg(\frac{A_2}{2}+o(1)\bigg)\pi(x).
	\end{align*}
\end{proof}
\vskip 3mm

\appendix
\section{\large A counting estimate for primitive vectors}\label{Append}

We prove the estimate used in Subsection \ref{sec:roots} to count primitive vectors $(u,\,v)^T$ with a prescribed matrix congruence.
Compared with \cite[Lemma 3.3]{PD}, we establish Lemma \ref{lem:app-root-count} below with an explicit error term for arbitrary squarefree $S$, allowing both coordinate intervals to lie in the interval $[1,16]$. The main term is unchanged, and the proof follows the argument of \cite{PD}.

\begin{lemma}\label{lem:app-root-count}
	Suppose that $d_0YM,N\ge2$, and let $S$ be a squarefree positive integer.
	For positive coprime integers $u,v$, let
	$g=\Big(\begin{smallmatrix}u&c\\[0,9mm]v&d\end{smallmatrix}\Big)\in\SL_2(\Z)$	satisfying $0\le c<u$. For
	$g_0\in\SL_2(\Z/S\Z)$ and intervals $I_1,I_2\subseteq[1,16]$ and
	$I_3\subseteq[0,1]$, we have
\begin{equation}\label{eq:A1}
	\sum_{\substack{u,v\ge1,\ \gcd(u,\,v)=1\\[0,7mm]
		\frac{u}{d_0YM}\in I_1,\,\, \frac{v}{N}\in I_2,\,\, \frac{c}{u}\in I_3\\[0,1mm]
		 g\equiv g_0\,({\rm{mod}}\,S)}}1
		=\frac{d_0YMN\,|I_1||I_2||I_3|}
		{\zeta(2)|\SL_2(\Z/S\Z)|}
		+O\Big(d_0YMN\,S^2
		\min\big\{d_0YM,N\}^{-\frac{1}{100}}\big\}\Big).
\end{equation}
\end{lemma}

\begin{proof}
For $m\ge2$ and integers $h,k$, define the Kloosterman sum
$$
K_m(h,k):=\sum_{\substack{y\,(\textrm{mod}\,m)\\\gcd(y,\,m)=1}}
e\bigg(\frac{hy+ky^{-1}}m\bigg),
$$
where $y^{-1}$ denotes the inverse of $y$ modulo $m$. By the Weil bound \cite[Corollary 11.12]{IwaKow04} and $\tau(m)\ll_\varepsilon m^\varepsilon$, we obtain
\begin{equation}\label{eq:A2}
		|K_m(h,k)|\le\tau(m)m^{1/2}\gcd(h,\,k,\,m)^{1/2}
		\ll_\varepsilon m^{1/2+\varepsilon}\gcd(h,\,k,\,m)^{1/2}.
	\end{equation}
%This includes even moduli and frequencies not coprime to the modulus.
	
Suppose first that $d_0YM\le N$. Choose 
$g_0=\Big(\begin{smallmatrix}u_0&c_0\\[0,7mm]v_0&d_0'\end{smallmatrix}\Big)\in\SL_2(\Z/S\Z)$. We fix $u\equiv u_0\,(\textrm{mod}\,S)$ and define $R_u\subseteq(\Z/uS\Z)^2$ as the set of pairs $(c+uS\Z,\,v+uS\Z)$ satisfying
\begin{equation*}
	c\equiv c_0\,(\textrm{mod}\,S),\qquad v\equiv v_0\,(\textrm{mod}\,S),\qquad cv\equiv ud'_0-1\,(\textrm{mod}\,uS).
\end{equation*}

These congruences are equivalent to the prescribed matrix congruence once $d=\frac{1+cv}{u}$. The last congruence implies $\gcd(u,\,c)=\gcd(u,\,v)=1$. At primes dividing both $u$ and $S$, coprimality of $c$ and $u$ follows from
$u_0d'_0-c_0v_0\equiv1\,(\textrm{mod}\,S)$; at the other primes dividing $u$, it excludes the proportion $1/p$ of possible lifts. For each admissible $c$, writing $v=v_0+Sj$ gives a unique $j$ modulo $u$. Therefore
	$$
	\#R_u=u\prod\nolimits_{\substack{p\mid u,\, p\nmid S}}
	\Bigl(1-\frac1p\Bigr).
	$$
	
We next estimate how many points $\big(c/(uS),\,v/(uS)\big)$(mod 1) with $(c,\,v)\in R_u$ lie in a given rectangle in $\mathbb T^2$. We write $S_1=\gcd(u,\,S)$ and $S_2=S/S_1$. Since $S$ is squarefree, $\gcd(S_2,\,uS_1)=1$, so we may choose $\gamma$ with $\gamma S_2\equiv1\,(\textrm{mod}\,uS_1)$. Modulo $S_2$, both coordinates are fixed. Modulo $uS_1$, we have $cv\equiv ud'_0-1\,(\textrm{mod}\,uS_1)$ and $c\equiv c_0\,(\textrm{mod}\,S_1)$; these congruences imply $v\equiv v_0\,(\textrm{mod}\,S_1)$. The Chinese remainder theorem and additive orthogonality give 
$$
\bigg|\sum_{(c,\,v)\in R_u}e \Big(\frac{hc+kv}{uS}\Big)\bigg|
\le\frac1{S_1}\sum_{j\,(\textrm{mod}\,S_1)}
\Big|K_{uS_1}\bigl(\gamma h+ju,\gamma k(ud'_0-1)\bigr)\Big|.
$$
Both $\gamma$ and $ud'_0-1$ are coprime to $uS_1$. For
	$0<\max\{|h|,|k|\}\le u^{0.02}$, we have
$$
\gcd\bigl(\gamma h+ju,\,\gamma k(ud'_0-1),\,uS_1\bigr)
	\le S_1\gcd(h,\,k,\,u)\le S_1u^{0.02}.
$$
Applying \eqref{eq:A2} with $\varepsilon=1/200$ we have for
	$0<\max\{|h|,|k|\}\le u^{0.02}$
	\begin{equation}\label{eq:A3}
	\bigg|\sum_{(c,v)\in R_u}e\Big(\frac{hc+kv}{uS}\Big)\bigg|
		\ll u^{0.515}S_1^{1.005}\le u^{0.515}S^2.
	\end{equation}
	
To pass from \eqref{eq:A3} to interval counts, take
$\Delta=\frac1{100}u^{-0.01}$. For any interval $J$ in $\R/\Z$,
enlargement or contraction by $2\Delta$, followed by convolution with a smooth probability density supported within distance $\Delta$, shows that the functions $f_J^-,f_J^+$ satisfy
\[
0\le f_J^-\le\ind_J\le f_J^+\le1,\qquad
\widehat f_J^\pm(0)=|J|+O(\Delta),\qquad
|\widehat f_J^\pm(n)|\ll_j(1+\Delta|n|)^{-j}.
\]
Here $\widehat f(n)=\int_0^1 f(t)e(-nt)\,dt$. If $|J|<4\Delta$, take
$f_J^-=0$; if the enlarged interval fills the circle, take $f_J^+=1$. The same bounds remain valid. For either choice of sign,
\[
\sum_{n\in\Z}|\widehat f_J^\pm(n)|\ll\Delta^{-1},\qquad
\sum_{|n|>u^{0.02}}|\widehat f_J^\pm(n)|
\ll\Delta^{-10}(u^{0.02})^{-9}.
\]
Apply the products of these upper and lower functions to two torus intervals $J_1,J_2$. The zero frequency contributes
$(\#R_u)|J_1||J_2|+O(u\Delta)$. There are $O(u^{0.04})$ nonzero frequency pairs with $\max\{|h|,|k|\}\le u^{0.02}$. Since
$|\widehat f_J^\pm(n)|\le1$, \eqref{eq:A3} gives
\begin{align*}
	\sum_{0<\max\{|h|,|k|\}\le u^{0.02}}
	\big|\widehat f_{J_1}^\pm(h)\widehat f_{J_2}^\pm(k)\big|
	\bigg|\sum_{(c,v)\in R_u}e\Big(\frac{hc+kv}{uS}\Big)\bigg| 
	\ll u^{0.04}u^{0.515}S^2=u^{0.555}S^2.
\end{align*}
For the remaining frequencies,
	\[
	\sum_{\max\{|h|,\,|k|\}>u^{0.02}}
	|\widehat f_{J_1}^\pm(h)\widehat f_{J_2}^\pm(k)|
	\ll\Delta^{-11}(u^{0.02})^{-9}\ll u^{-0.07}.
	\]
Using the trivial bound $\#R_u\le u$ for these remaining frequencies, we obtain
\begin{equation}\label{eq:A4}
	\bigg|\sum_{(c,v)\in R_u}\ind_{J_1}\Big(\frac c{uS}\Big)
	\ind_{J_2}\Big(\frac v{uS}\Big)-(\#R_u)|J_1||J_2|\bigg|
	\ll u^{0.99}+u^{0.555}S^2+u^{0.93}\ll u^{0.99}S^2.
\end{equation}
The arguments of the indicators are taken modulo one. In particular, the error is independent of the lengths and locations of $J_1,J_2$.
	
For fixed $u$, the restriction $\frac cu\in I_3$ corresponds to a torus interval of length $\frac{|I_3|}{S}$. Split the interval $NI_2$ for $v$ at multiples of $uS$; it has $O(\frac Nu+1)$ pieces, with total normalized length $\frac{N|I_2|}{uS}$. Thus \eqref{eq:A4} gives
$$
\sum_{\substack{v,c\in\Z,\ 0\le c<u\\ \frac{v}{N}\in I_2,\ \frac{c}{u}\in I_3\\ (c+uS\Z,\,v+uS\Z)\in R_u}}1
=\frac{N|I_2||I_3|}{S^2} \prod_{\substack{p\mid u\\p\nmid S}} \Big(1-\frac1p\Big)
+O\bigg(\Big(\frac Nu+1\Big)u^{0.99}S^2\bigg).
$$

Summing the errors over $u\in d_0YM I_1$ gives
$O(d_0YMN(d_0YM)^{-0.01}S^2)$. To sum the main terms, expand the product into divisors.  Since $\gcd(\ell,\,S)=1$, the conditions $\ell\mid u$ and $u\equiv u_0\,(\textrm{mod}\,S)$ specify one residue class modulo $\ell S$, whether or not $\gcd(u_0,\,S)=1$. Consequently,
\begin{align*}
\sum_{\substack{\frac{u}{d_0YM}\in I_1\\u\equiv u_0\, (\textrm{mod}\,S)}}
\prod_{\substack{p\mid u\\p\nmid S}}\Big(1-\frac1p\Big)
&=\sum_{\substack{\ell\le16d_0YM\\\gcd(\ell,\,S)=1}}\frac{\mu(\ell)}\ell \bigg(\frac{d_0YM|I_1|}{\ell S}+O(1)\bigg)
\\
&=\frac{d_0YM|I_1|}{S\zeta(2)}\prod_{p\mid S}\big(1-p^{-2}\big)^{-1}
+O\Big(\log\big(2d_0YM\big)\Big).
\end{align*}

Here extending the series in $\ell$ to infinity costs $O(1/S)$. Since
	\[
	|\SL_2(\Z/S\Z)|=\prod_{p\mid S}p(p^2-1)
	=S^3\prod_{p\mid S}(1-p^{-2}),
	\]
we obtain the main term in \eqref{eq:A1}. The additional error is absorbed because
	\[
	N\log(2d_0YM)/S^2	\ll d_0YMN(d_0YM)^{-0.01}S^2.
	\]
	
Finally suppose that $N<d_0YM$. Since $u,v>1$, the normalization $0\le c<u$	implies $0<d<v$. We interchange $u$ and $v$ and replace $g$ by the matrix
\begin{equation*}
	\widetilde g=\begin{pmatrix}v&v-d\\u&u-c\end{pmatrix},\qquad \det\widetilde g=v(u-c)-u(v-d)=ud-vc=1.
\end{equation*}
Thus $(v-d,\,u-c)^T$ is a complement to $(v,\,u)^T$, and its first coordinate satisfies $0\le v-d<v$. Moreover, $ud-vc=1$ gives $(v-d)/v=1-c/u-1/(uv)$.

If we apply the same rule to $\widetilde g$, we recover $g$:
\begin{equation*}
	\widetilde{\widetilde g}=\begin{pmatrix}u&u-(u-c)\\v&v-(v-d)\end{pmatrix}=g.
\end{equation*}
Since this rule involves only subtraction and rearrangement of the entries, we have $g\equiv g_0\,(\textrm{mod}\,S)$ if and only if $\widetilde g\equiv\widetilde g_0\,(\textrm{mod}\,S)$, where $\widetilde g_0$ is obtained from $g_0$ by the same rule.

 Apply the case already proved with the two
coordinate scales and intervals interchanged. The condition $\frac cu\in I_3$ is bounded above and below by requiring $\frac{v-d}{v}$ to lie in the interval $1-I_3$ enlarged or contracted by $\frac1{d_0YMN}$, with intersection with $[0,1]$. These changes alter the predicted main term by $O(1)$, which is absorbed by the error in \eqref{eq:A1}. This proves the	lemma for both orders of the coordinate scales.
\end{proof}

\vskip 3mm

\noindent\textbf{Acknowledgements}\hspace{2mm} 

The authors thank Professor Harald Helfgott for pointing out the methodological contributions of Helfgott--Radziwi{\l}{\l} \cite{HelRad21} and Pilatte \cite{Pil26}  to the graph and high-moment arguments, which is prior to OpenAI \cite{PD}.

The authors would like to thank ChatGPT for its assistance in
understanding the methods in \cite{PD}, and in establishing the Appendix \ref{Append}. The authors also thank ChatGPT for helping to check the manuscript.

\nocite{MV}
\bibliographystyle{amsplain}
\bibliography{Wang}

@article{BomFriIwa86,
  title={Primes in arithmetic progressions to large moduli},
  author={Bombieri, E. and Friedlander, J. B. and Iwaniec, H.},
  journal={Acta Math.},
  volume={156},
  pages={203--251},
  year={1986},
}

@book{Chowla,
  title={The {R}iemann hypothesis and {H}ilbert's tenth problem},
  author={Chowla, S.},
  series={Mathematics and its Applications},
  volume={4},
  publisher={Gordon and Breach Science Publishers},
  year={1965},
}

@article{Hil,
  title={Additive and multiplicative functions on shifted primes},
  author={Hildebrand, A.},
  journal={Proc. London Math. Soc. (3)},
  volume={59},
  number={2},
  pages={209--232},
  year={1989},
}

@article{HelRad21,
  title={Expansion, divisibility and parity},
  author={Helfgott, H. A. and Radzi\l\l, M.},
  journal={arXiv: 2103.06853},
  year={2021},
}

@book {IwaKow04,
    AUTHOR = {Iwaniec, H. and Kowalski, E.},
     TITLE = {Analytic number theory},
    VOLUME = {53},
 PUBLISHER = {American Mathematical Society, Providence, RI},
      YEAR = {2004},
}

@article{Lichtman,
  title={Averages of the {M}\"obius function on shifted primes},
  author={Lichtman, J. D.},
  journal={Q. J. Math.},
  volume={73},
  number={2},
  pages={729--757},
  year={2022},
}

@article{LichtmanTeravainen2022,
  title={On the {H}ardy--{L}ittlewood--{C}howla conjecture on average},
  author={Lichtman, J. D. and Ter{\"a}v{\"a}inen, J.},
  journal={Forum Math. Sigma},
  volume={10},
  pages={Paper No. e57, 17 pp.},
  year={2022},
}

@article{MSTT,
  title={Higher uniformity of arithmetic functions in short intervals {I}.
         {A}ll intervals},
  author={Matom{\"a}ki, K. and Shao, X. and Tao, T. and Ter{\"a}v{\"a}inen, J.},
  journal={Forum Math. Pi},
  volume={11},
  pages={Paper No. e29},
  year={2023},
}

@article{MR16,
  author  = {Matom{\"a}ki, K. and Radziwi{\l}{\l}, M.},
  title   = {Multiplicative functions in short intervals},
  journal = {Ann. of Math. (2)},
  volume  = {183},
  number  = {3},
  year    = {2016},
  pages   = {1015--1056},
}

@misc{MPPTW26,
  author = {Matom{\"a}ki, K. and Pandey, M. and Pliego, J. and
            Ter{\"a}v{\"a}inen, J. and Wang, M. D.},
  title  = {Higher order uniformity of the primes and cancellation of the
            {M\"obius} function in shorter intervals},
  year   = {2026},
  note   = {arXiv:2610.09567}
}

@article{MurVat,
  title={Twin primes and the parity problem},
  author={Murty, M. R. and Vatwani, A.},
  journal={J. Number Theory},
  volume={180},
  pages={643--659},
  year={2017},
}

@article{MurVat18,
  title={A remark on a conjecture of {C}howla},
  author={Murty, M. R. and Vatwani, A.},
  journal={J. Ramanujan Math. Soc.},
  volume={33},
  number={2},
  pages={111–123},
  year={2018},
}

@article{MonVau74,
  author  = {Montgomery, H. L. and Vaughan, R. C.},
  title   = {Hilbert's inequality},
  journal = {J. London Math. Soc. (2)},
  volume  = {8},
  year    = {1974},
  pages   = {73--82},
}

@misc{PD,
  title={The {P}oisson--{D}irichlet law for prime predecessors},
  author={{OpenAI}},
  howpublished={\href{https://github.com/openai/math/blob/main/preprints/The-Poisson-Dirichlet-Law-for-Prime-Predecessors-September-24-2026/paper.pdf}{OpenAI Math Release preprint}},
  note={September, 2026 version, 74 pp.},
  year={2026},
  url={https://github.com/openai/math/blob/main/preprints/The-Poisson-Dirichlet-Law-for-Prime-Predecessors-September-24-2026/paper.pdf},
}

@misc{OpenAImobius,
  author       = {{OpenAI}},
  title        = {{Prime Predecessors with an Even Number of Prime Factors}},
  year         = {2026},
  note={September, 2026 version, 45 pp.},
  howpublished = {OpenAI Math Release preprint},
  url          = {https://github.com/openai/math/blob/main/preprints/Prime-Predecessors-with-an-Even-Number-of-Prime-Factors-September-17-2026/paper.pdf}
}

@article{Pin,
  title={An approximation to the twin prime conjecture and the parity phenomenon},
  author={Pintz, J.},
  journal={Indag. Math. (N.S.)},
  volume={26},
  number={5},
  pages={883--896},
  year={2015},
}

@article{Pil26,
  title={Improved bounds for the two-point logarithmic Chowla conjecture},
  author={Pilatte, C.},
  journal={ J. Amer. Math. Soc.},
  year={2026},
  note   = {https://doi.org/10.1090/jams/1084.}
}

@misc{Sar,
  title={{S}arnak's conjecture},
  author={{AIM Problem Lists}},
  howpublished={\url{https://aimpl.org/sarnakconjecture/}},
  note={Section 5, Problem 5.2},
  url={https://aimpl.org/sarnakconjecture/5/},
}

@article{TaoTeravainen2021,
  title={The {H}ardy--{L}ittlewood--{C}howla conjecture in the presence of a {S}iegel zero},
  author={Tao, T. and Ter{\"a}v{\"a}inen, J.},
  journal={J. Lond. Math. Soc. (2)},
  volume={106},
  number={4},
  pages={3317--3378},
  year={2022},
}

@book{MV,
  title={Multiplicative number theory. {I}. {C}lassical theory},
  author={Montgomery, H. L. and Vaughan, R. C.},
  series={Cambridge Stud. Adv. Math.},
  volume={97},
  publisher={Cambridge University Press},
  address={Cambridge},
  year={2007},
}

@book{kouk,
  title={The distribution of prime numbers},
  author={Koukoulopoulos, D.},
  series={Grad. Stud. Math.},
  volume={203},
  publisher={American Mathematical Society},
  address={Providence, RI},
  year={2019},
  url={https://www.dms.umontreal.ca/~koukoulo/documents/publications/primes.pdf},
}

\end{document}